\documentclass[twoside,11pt]{article}

\usepackage[abbrvbib, preprint]{jmlr2e}

\usepackage{amsmath,amsfonts}
\usepackage{mathtools}
\usepackage{enumerate,color,xcolor}
\usepackage{subfigure}
\usepackage{url}
\usepackage[textsize=footnotesize]{todonotes}

\newcommand{\E}{\mathbb{E}}
\newcommand{\tkappa}{\tilde{\kappa}}
\newcommand{\ones}{\textbf{1}}
\newcommand{\calW}{\mathcal{W}}

\newcommand{\bigO}{\mathcal{O}}

\newcommand{\beq}{\begin{eqnarray}}
\newcommand{\eeq}{\end{eqnarray}}
\newcommand{\beqs}{\begin{eqnarray*}}
\newcommand{\eeqs}{\end{eqnarray*}}
\DeclarePairedDelimiter\ceil{\lceil}{\rceil}

\newcommand{\R}{\mathbb{R}}
\usepackage{multirow}

\usepackage{xcolor}
\usepackage{comment}
\usepackage[ruled]{algorithm2e}
\usepackage[normalem]{ulem} 
\newcommand\str{\bgroup\markoverwith
{\textcolor{red}{\rule[0.5ex]{2pt}{1.5pt}}}\ULon} 
\usepackage{pifont}
\newcommand{\rev}[1]{{\color{black}#1}} 

\ShortHeadings{Robust Distributed Accelerated Stochastic Gradient Methods}{Fallah, G\"{u}rb\"{u}zbalaban, Ozdaglar, \c{S}im\c{s}ekli and Zhu }
\firstpageno{1}

\begin{document}

\title{Robust Distributed Accelerated Stochastic Gradient Methods for Multi-Agent Networks}

\author{\name Alireza Fallah* \email afallah@mit.edu \\
       \addr Department of Electrical Engineering and Computer Science\\ 
       Massachusetts Institute of Technology\\
       Cambridge, MA 02139, United States of America
       \AND
       \name Mert G\"{u}rb\"{u}zbalaban*$^\clubsuit$ \email mg1366@rutgers.edu \\
       \addr Department of Management Science and Information Systems \\
       Rutgers Business School\\
       Piscataway, NJ 08854, United States of America
       \AND
       \name Asuman Ozdaglar* \email asuman@mit.edu \\
       \addr Department of Electrical Engineering and Computer Science\\
       Massachusetts Institute of Technology\\
       Cambridge, MA 02139, United States of America
       \AND
       \name Umut \c{S}im\c{s}ekli* \email umut.simsekli@inria.fr \\
       \addr 
       INRIA - D\'{e}partement d'Informatique de l'\'{E}cole Normale Sup\'{e}rieure \\
PSL Research University \\
Paris, France
       \AND
       \name Lingjiong Zhu* \email zhu@math.fsu.edu \\
       \addr Department of Mathematics \\
       Florida State University \\
       Tallahassee, FL 32306, United States of America\\
       ~\\
       \name * \addr The authors are in alphabetical order.\\
       \name $^\clubsuit$ \addr Corresponding author.\\
       }

\editor{}

\maketitle




\begin{abstract}
  We study distributed stochastic gradient (D-SG) method and its accelerated variant (D-ASG) for solving decentralized strongly convex stochastic optimization problems where the objective function is distributed over several computational units, lying on a fixed but arbitrary connected communication graph, subject to local communication constraints where noisy estimates of the gradients are available. We develop a framework which allows to choose the stepsize and the momentum parameters of these algorithms in a way to optimize performance by systematically trading off the bias, variance and dependence to network effects. When gradients do not contain noise, we also prove that D-ASG can \emph{achieve acceleration}, in the sense that it requires $\mathcal{O}(\sqrt{\kappa} \log(1/\varepsilon))$ gradient evaluations and $\mathcal{O}(\sqrt{\kappa} \log(1/\varepsilon))$ communications to converge to the same fixed point with the non-accelerated variant where $\kappa$ is the condition number and $\varepsilon$ is the target accuracy. 
  For quadratic functions, we also provide finer performance bounds that are tight with respect to bias and variance terms. Finally, we study a multistage version of D-ASG with parameters carefully varied over stages to ensure exact convergence to the optimal solution. It achieves optimal and accelerated $\bigO(-k/\sqrt{\kappa})$ linear decay in the bias term as well as optimal $\bigO(\sigma^2/k)$ in the variance term. We illustrate through numerical experiments that our approach results in accelerated practical algorithms that are robust to gradient noise and that can outperform existing methods.
\end{abstract}


\section{Introduction} 
Advances in sensing and processing technologies, communication capabilities and smart devices have enabled deployment of systems where a massive amount of data is collected by many distributed autonomous units to make decisions. There are numerous such examples including a set of sensors collecting and processing information about a time-varying spatial field (e.g., to monitor temperature levels or chemical concentrations) \rev{\citep{blatt2007convergent}, a collection of mobile robots performing dynamic tasks spread over a region \citep{decentralized-applications}, federated learning on edge devices \citep{konevcny2016federated,mcmahan2017communication}, on-device peer-to-peer learning \citep{koloskova2019decentralized}} and distributed model training across a network or computers \citep{arjevani2020ideal,gurbuzbalaban2020decentralized,scaman2018optimal}.
In such systems, most of the information is often collected in a decentralized, distributed manner, and processing of information has to go hand-in-hand with its communication and sharing across these units over an undirected network $\mathcal{G}=(\mathcal{V},\mathcal{E})$ defined by the set of (computational units) agents $\mathcal{V}=\{1,2,\dots,N\}$ connected by the edges $\mathcal{E}\subseteq \mathcal{V}\times \mathcal{V}$. 
In such a setting, we consider the group of agents (i.e., the nodes) collaboratively solving the following optimization problem:
\begin{equation}\label{opt-pbm}
\min_{x\in\mathbb{R}^{d}}f(x):=\frac{1}{N}\sum_{i=1}^{N}f_{i}(x), 
\end{equation}
where each $f_i:\mathbb{R}^d\to \mathbb{R}$ is known by agent $i$ only and therefore referred to as its \textit{local objective function}. We assume each $f_i$ is $\mu$-strongly convex with $L$-Lipschitz gradients (hence $f$ is also $\mu$-strongly convex with $L$-Lipschitz gradient and we refer to $\kappa = L/\mu$ as its condition number). We also use $x_*$ to denote the unique optimal solution of \eqref{opt-pbm}. In addition, we denote the local model of node $i$ at iteration $k$ by $x_{i}^{(k)}\in\mathbb{R}^{d}$.

We consider the setting where each agent $i$ has access to noisy estimates $\tilde{\nabla} f_i(x)$ of the actual gradients satisfying the following assumption:
\begin{assumption}\label{assump_1}
Recall that $x_i^{(k)}$ denotes the decision variable of node $i$ at iteration $k$. We assume at iteration $k$, node $i$ has access to $\tilde{\nabla} f_i \left(x_i^{(k)}, w_i^{(k)}\right)$ which is an estimate of $\nabla f_i \left(x_i^{(k)}\right)$ where $w_{i}^{(k)}$ is a random variable independent of $\left\{w_{j}^{(t)}\right\}_{j=1,\ldots,N, t=1,\ldots,k-1}$ and $\left\{w_{j}^{(k)}\right\}_{j \neq i}$. Moreover, we assume 

\begin{equation*}
\mathbb{E}\left[\tilde{\nabla} f_i \left(x_i^{(k)}, w_i^{(k)}\right)\Big|x_i^{(k)}\right]= \nabla f_i \left(x_i^{(k)}\right),
\quad
\mathbb{E}\left[\left\Vert \tilde{\nabla} f_i \left(x_i^{(k)}, w_i^{(k)}\right) - \nabla f_i \left(x_i^{(k)}\right) \right\Vert^{2}\Big|x_i^{(k)}\right]\leq\sigma^{2}. 
\end{equation*}	
To simplify the notation, we suppress the $w_i^{(k)}$ dependence, and denote $\tilde{\nabla} f_i \left(x_i^{(k)}, w_i^{(k)}\right)$ by $\tilde{\nabla} f_i \left(x_i^{(k)}\right)$.
\end{assumption}
This arises naturally in distributed learning problems where $f_i(x)$ represents the expected loss $\E_{\eta_i}\left [ f_i(x,\eta_i) \right]$ where $\eta_i$ are independent data points collected at node $i$ (see e.g. \cite{pu2018distributed,lan2017communication, olshevsky2019non}). For this setting, $\tilde{\nabla} f_i (x)$ is an unbiased estimator of $\nabla f_i(x)$ which we assume satisfies the bounded variance assumption of Assumption~\ref{assump_1}. {\color{black}In Appendix~\ref{sec:general:noise}, we will discuss
the unbounded variance assumption (Assumption~\ref{assump:unbounded}) that extends Assumption~\ref{assump_1}, 
and show that all the main results in the paper can be extended.}

Note that in our setting, a master node that can coordinate the computations is not available unlike the master/slave architecture studied in the literature (see e.g. \cite{mishchenko18a,duchi11,hakimi2019dana,lee2018distributed,meng2016asynchronous,jaggi2014communication,xin2018distributed}). Furthermore, our setting covers an arbitrary network topology that is more general than particular network topologies such as the complete graph or ring graph.

Deterministic variants of problem \eqref{opt-pbm} have been studied extensively in the literature. Much of the work builds on the Distributed Gradient (DG) method proposed in \cite{nedic2009distributed} where each agent keeps local estimates of the optimal solution of \eqref{opt-pbm} and updates by a combination of weighted average of neighbors' estimates and a gradient step (normalized by 
the stepsize $\alpha_k$) of the local objective function. \cite{nedic2009distributed} analyzed the case with convex and possibly nonsmooth local objective functions, constant stepsize $\alpha_k = \alpha >0$, and agents linked over an undirected connected graph and showed that the ergodic average of local estimates of the agents converge at rate $\bigO(1/k)$ to an $\bigO(\alpha)$ neighborhood of the optimal solution of problem \eqref{opt-pbm} (where $k$ denotes the number of iterations). \cite{yuan2016convergence} considered this algorithm for the case that local functions are smooth, i.e., $\nabla f_i(x)$ are Lipschitz continuous, and when $f_{i}(x)$ are either convex, restricted strongly convex or strongly convex. For the convex case, they show the network-wide mean estimate converges at rate $\bigO(1/k)$ to an $\bigO(\alpha)$ neighborhood of the optimal solution, and for the strongly convex case, all local estimates converge at a linear rate $\bigO(\exp(- k/\Theta(\kappa)))$ to an $\bigO(\alpha)$ neighborhood of $x_*$.
\footnote{For two real-valued functions $f$ and $g$, we say $f = \Theta(g)$ if there exist {positive} constants $C_\ell$ and $C_u$ such that $C_\ell  g(x)\leq f(x) \leq C_u g(x)$ for every $x$ in the domain of $f$ and $g$ with $\Vert x\Vert$ being sufficiently large.}

There have been many recent works on developing new distributed deterministic algorithms with faster convergence rate and exact convergence to the optimal solution $x_*$. We start by summarizing the literature in this area that are most relevant to this work. First, \cite{shi2015extra} provides a novel algorithm which can be viewed as a primal-dual algorithm for the constrained reformulation of problem \eqref{opt-pbm} (see \cite{mokhtari2016dsa} for this interpretation) that achieves \textit{exact convergence} with linear rate to the optimal solution; {\color{black} however the linear convergence rate with the recommended stepsize is $\rho = 1 - \mathcal{O}(\frac{1}{\kappa^2})$ where $\kappa$ is the condition number (see Table \ref{Dist_Stoch_Table}). This convergence guarantee will be slow for ill-conditioned problems when $\kappa$ is large}. Second, \cite{QuLi} proposes to update the DG method such that agents also maintain, exchange, and combine estimates of gradients of the global objective function of \eqref{opt-pbm}. This update is based on a technique called ``gradient tracking'' (see e.g. \cite{di2015distributed,di2016next}) which enables better control on the global gradient direction and yields a linear rate of convergence to the optimal solution (see \cite{Jakovetic19} for a unified analysis of these two methods). In a follow up paper, \cite{qu-lina-accelerated} also considered an acceleration of their algorithm and achieved a linear 
convergence rate $\bigO(\exp(- k/\Theta(\kappa^{5/7})))$ to the optimal solution. To our best knowledge, whether an accelerated primal variant of the DG algorithm can
achieve the non-distributed $\bigO(\exp(- k/\Theta(\sqrt{\kappa})))$ linear rate to a neighborhood of the optimum solution with $\sqrt{\kappa}$ dependence 
has been an open problem. Alternative distributed first-order methods besides DG have also been studied. In particular, if 
additional assumptions are made such as the explicit characterization of Fenchel dual of the local objective functions, referred to as {the} \emph{dualable} setting as in \cite{scaman2018optimal,uribe2018dual}), then it is known that the multi-step dual accelerated (MSDA) method of \cite{scaman2018optimal} achieves the $\bigO(\exp(- k/\Theta(\sqrt{\kappa})))$ linear rate to the optimum with $\sqrt{\kappa}$ dependence. For deterministic distributed optimization problems under smooth and strongly convex objectives, \cite{dvinskikh2019decentralized} proposed the PSTM algorithm and provided accelerated convergence guarantees. Recently, \cite{scaman2019optimal} provided lower bounds which matches the upper bounds of \cite{dvinskikh2019decentralized} up to logarithmic factors (see also \cite{scaman2019optimal} for a discussion of deterministic optimal algorithms under different assumptions (Lipschitz continuity, strong convexity, smoothness, and a combination of strong convexity and smoothness)).

This paper focuses on the Distributed Stochastic Gradient (D-SG) method (which is a stochastic version of the DG method) and its momentum enhanced variant, Distributed Accelerated Stochastic Gradient (D-ASG) method.
These methods are relevant for solving distributed learning problems and are natural decentralized versions of the stochastic gradient and its variant based on Nesterov's momentum averaging \citep{nesterov_convex,can-gur-ling19}. In this paper, we focus on strongly convex and smooth objectives. Several works studied D-SG under these assumptions although D-ASG remains relatively understudied except the deterministic case (see e.g. \cite{jakovetic2014fast,xi2017add,li2018sharp,qu2016accelerated}). \rev{The performance of distributed algorithms such as D-SG and their deterministic versions depend on the connectivity of the underlying network structure as expected. In particular, when D-SG and D-ASG are run on undirected graphs, the propagation of information among neighbors is governed by a symmetric mixing matrix $W$ which depend on the network structure and its eigenvalues affect the convergence rates. In particular; the largest eigenvalue of the matrix $W$ is one, and the second largest (in modulus) of the eigenvalues of $W$, which we refer to as $\gamma$ in this paper (formally defined in \eqref{eqn:asymp_suboptim}), arises in the study of distributed algorithms such as D-SG}. We summarize the existing convergence rate results for D-SG in Table~\ref{Dist_Stoch_Table}.\footnote{See also \cite{shamir2014distributed} for a different noise model than ours in the mini-batch setting, where each objective $f_i$ can be expressed as a finite sum.} Among these, \cite{rabbat2015multi} studied composite stochastic optimization problems and showed a $\bigO(\sigma^2/k)$ convergence rate for D-SG and its mirror descent variant. \cite{koloskova2019decentralized} studied  decentralized stochastic gradient algorithms when the nodes compress (e.g. quantize or sparsify) their updates. 
\cite{olshevsky2019non} provided an asymptotic network independent sublinear rate. 
In our approach, we use a dynamical system representation of these iterative algorithms (presented in \cite{lessard2016analysis} and further used in \cite{hu2017dissipativity, StrConvex, aybat2019universally}) to provide rate estimates for convergence of the local agent iterates to a neighborhood of the optimal solution of problem \eqref{opt-pbm}. Our bounds are presented in terms of three components: (\romannumeral 1) a \textit{bias term} that shows the decay rate of the initialization error (i.e., distance of the initial estimates to the optimal solution) independent of gradient noise, (\romannumeral 2) a \textit{variance term} that depends on the error level $\sigma^2$ of local objective functions' gradients, measuring the ``robustness'' of the algorithm to noise (in a sense that we will define precisely later), (\romannumeral 3) a \textit{network effect} that highlights the dependence on the structure of the network. In this paper, in addition to the convergence analysis for D-SG and D-ASG, our purpose is to study the trade-offs and interplays between these three terms that affect the performance.

\begin{table}[h!]
\begin{scriptsize}
\hspace{0.5cm}
\begin{tabular}{|l|l|l|l|}
\hline  
Algorithm  & \begin{tabular}[c]{@{}l@{}l} Extra \\ Assumption \end{tabular}  & \begin{tabular}[c]{@{}l@{}l} \rev{Stepsize} \\ $\alpha_k$ \end{tabular}  & Convergence Rate \\
\hline 
\begin{tabular}[c]{@{}l@{}l}{ {\color{black}EXTRA}} \\  {\tiny\cite{shi2015extra}} \end{tabular}
& {\color{black}$\sigma=0$} & \begin{tabular}[c]{@{}l@{}l} {\color{black}$\alpha$ as in}\\ {\tiny\cite{shi2015extra}}\\   \end{tabular}    &   {\color{black}$\mathbb{E} \left\|x_i^{(k)}-x_*\right\|^2 \leq \mathcal{O}\left(\rho^k \right)$ where $\rho = 1 - \mathcal{O}(1/\kappa^2)$} \\ 
\hline 
\begin{tabular}[c]{@{}l@{}l}{ D-SG} \\  {\tiny\cite{tsianos2012distributed}} \end{tabular}
& Yes$^{\ddagger}$ &  $\mathcal{O}\left(\frac{1}{k}\right)$       &   $\mathbb{E} f\left(x_i^{(k)}\right)-f_* \leq \mathcal{O}\left(G^2 \frac{\log(\sqrt{N}k)}{(1-\sqrt{\gamma})k}\right)$ \\ 
\hline 
\begin{tabular}[c]{@{}l@{}l}{D-SG} \\   {\tiny\cite{rabbat2015multi}} \end{tabular}
&  No & $\mathcal{O}\left(\frac{1}{k}\right)$ &  $\mathbb{E} f\left(x_i^{(k)}\right)-f_* \leq \mathcal{O} \left( \frac{\kappa^2 \|x_i^{(0)} - x_*\|^2}{k^2} + \frac{\sigma^2}{N \mu^2 k} \right) + o\left(\frac{1}{k}\right)$ \\ 
\hline
\begin{tabular}[c]{@{}l@{}l}D-SG \\  {\tiny\cite{olshevsky2019non}{$^\dagger$}} \end{tabular} 
&  No  & $\mathcal{O}\left(\frac{1}{k}\right)$ & $\mathbb{E}\left\|{\bar x}^{(k)} - x^*\right\|^2 \leq \mathcal{O}\left(\frac{\sigma^2 }{\mu^2 N k} + \frac{\kappa^2}{(1-\gamma)^2 \mu^2 k^2}
+\frac{\sigma^{2}\kappa^{2}}{(1-\gamma)\mu^{2}k^{2}}\right)$ \\ 
\hline
\begin{tabular}[c]{@{}l@{}l}D-SG \\  {\tiny\cite{koloskova2019decentralized}}{$^\dagger$} \end{tabular}
 & Yes$^{\ddagger}$ & $\mathcal{O}\left(\frac{1}{k}\right)$  & $\mathbb{E} f\left(x_{avg}^{(k)}\right)-f_* \leq  \mathcal{O}\left (\frac{\sigma^2}{\mu N k} + \frac{\kappa G^2}{\mu (1-\gamma)^4 k^2} + \frac{G^2}{\mu (1-\gamma)^6 k^3}  \right)$ \\ 
 \hline
\begin{tabular}[c]{@{}l@{}l@{}l} \textbf{D-SG} \\ \\  {\color{black}\textbf{Proposition~\ref{prop:dsg-average-optimal-rate}}}\\ \textbf{in this paper} \end{tabular}
 &  No &  $\alpha$ &   \begin{tabular}[c]{@{}l@{}l@{}l} {\color{black}$\mathbb{E}\left\|\bar{x}^{(k)} - x_*\right\|^2$} \\ {\color{black}$\leq \mathcal{O}\left((1-\alpha\mu)^{2k}\left\Vert \bar{x}^{(0)}-x_{\ast}\right\Vert^{2}+  \frac{\alpha\sigma^{2}}{\mu N}
+\kappa^{2}\alpha^{2}\left(\frac{D^{2}}{(1-\gamma)^{2}N}+\frac{\sigma^{2}}{1-\gamma^{2}}\right)\right)$}
\\
{\color{black}$\mathbb{E} \left\|x_{i}^{(k)}-x_*\right\|^2 $}
\\
{\color{black}$\leq \mathcal{O}\Big((1-\alpha\mu)^{2k}\left\Vert \bar{x}^{(0)}-x_{\ast}\right\Vert^{2}+  \frac{\alpha\sigma^{2}}{\mu N}
+\kappa^{2}\alpha^{2}\left(\frac{D^{2}}{(1-\gamma)^{2}N}+\frac{\sigma^{2}}{1-\gamma^{2}}\right)$}
\\
{\color{black}$\qquad\qquad\qquad
+\gamma^{k}\Vert x^{(0)}\Vert^{2}+\frac{D^{2}\alpha^{2}}{(1-\gamma)^{2}}+\frac{\sigma^{2}N\alpha^{2}}{1-\gamma^{2}}\Big)$}
\end{tabular} \\ 
\hline                                     
\begin{tabular}[c]{@{}l@{}l@{}l} \textbf{D-ASG} \\ \\ {\color{black}\textbf{Proposition~\ref{prop:dsg-average-optimal-rate-DASG}}} \\ \textbf{in this paper}\end{tabular} 
&   No      &   $\alpha    $ &   \begin{tabular}[c]{@{}l@{}l@{}l} 
{\color{black}$\mathbb{E}\left\|\bar{x}^{(k)} - x_*\right\|^2 $}  
\\ 
{\color{black}$\leq \mathcal{O}\left(\left(1-\frac{\sqrt{\alpha\mu}}{2}\right)^{k}\frac{\Vert x^{(0)}-x^{\ast}\Vert^{2}}{N}+\frac{\sigma^{2}\sqrt{\alpha}}{\mu\sqrt{\mu}N}+\frac{\kappa^{2}C_{0}\alpha}{N(1-\gamma)^{2}}\right) $} 
\\
{\color{black}$\mathbb{E} \left\|x_{i}^{(k)}-x_*\right\|^2 $}
\\
{\color{black} $\leq \mathcal{O}\left(\left(1-\frac{\sqrt{\alpha\mu}}{2}\right)^{k}\frac{\Vert x^{(0)}-x^{\ast}\Vert^{2}}{N}+\frac{\sigma^{2}\sqrt{\alpha}}{\mu\sqrt{\mu}N}+\left(\frac{\kappa^{2}}{N}+1\right)\frac{C_{0}\alpha}{(1-\gamma)^{2}}\right) $} 
\end{tabular} \\ 
\hline
\begin{tabular}[c]{@{}l@{}l@{}l} \textbf{D-MASG} \\ \\ {\color{black}\textbf{Corollary~\ref{cor_D-MASG:bar}}} \\ \textbf{in this paper} \end{tabular} 
&   No &  $\mathcal{O}\left(\frac{1}{k}\right)$    & \begin{tabular}[c]{@{}l@{}l@{}l} {\color{black}$\mathbb{E}\left\|\bar{x}^{(k)} - x_*\right\|^2$} 
\\  {\color{black}$ \leq \bigO\left ( \exp\left(-\frac{k}{\Theta(\sqrt{\tilde{\kappa}})}\right)\frac{\left\| x^{(0)} - x^* \right\|^2}{N} + \frac{  \sigma^2}{N\mu\sqrt{\mu} k}  + \frac{\kappa^2C_{0}}{N( 1-\gamma)^{2} k^4} \right )$} 
\\
{\color{black}$\mathbb{E} \left\|x_{i}^{(k)}-x_*\right\|^2$}
\\  {\color{black}$ \leq \bigO\left ( \exp\left(-\frac{k}{\Theta(\sqrt{\tilde{\kappa}})}\right)\frac{\left\| x^{(0)} - x^* \right\|^2}{N} + \frac{  \sigma^2}{N\mu\sqrt{\mu} k}  + \left(\frac{\kappa^{2}}{N}+1\right)\frac{C_{0}}{( 1-\gamma)^{2} k^4} \right )$}  
\end{tabular} \\ 
\hline  
\end{tabular}
\end{scriptsize}
\caption{\label{Dist_Stoch_Table}Summary for D-SG and D-ASG. ${\bar x}^{(k)}$ denotes the average of nodes' estimates at time $k$, i.e., $\bar{x}^{(k)}:=\frac{1}{N}\sum_{i=1}^{N}x_{i}^{(k)}$, and, $x_{avg}^{(k)}$ is a weighted average defined in \cite{koloskova2019decentralized}. 
Also, \rev{$\gamma \in (0,1)$ is the second largest modulus of the eigenvalues of the mixing matrix $W$ (formally defined in \eqref{eqn:asymp_suboptim}). In the table, $\mu$ denotes the strong convexity constant, $L$ is the gradient Lipschitz constant and $\kappa=L/\mu$ is the condition number, whereas
$\tilde{\kappa}:= \frac{\kappa + 1}{\lambda_N^W}$ is a scaled condition number (formally introduced in \eqref{tilde_kappa}), where $\lambda_N^W$ is the smallest positive eigenvalue of $W$,  $D^{2}$ is defined in \eqref{eqn:D1} such that $D^{2}=\mathcal{O}\left(L^{2}\mathbb{E}\Vert x^{(0)}-x^{\ast}\Vert^{2}+\Vert\nabla F(x^{\ast})\Vert^{2}\right)$ as $\alpha\rightarrow 0$, and $C_{0}$ is an explicitly computable constant such that $C_{0}=\mathcal{O}(1)$ as $\alpha\rightarrow 0$}.
\newline $\dagger$: The authors analyze a D-SG method with a slightly different update then ours. \newline $\ddagger$: The authors make the extra assumption $\sup_{i,j} \mathbb{E} \left\|\nabla f_i\left(x_i^{(j)}\right)\right\|^2 \leq G^2$.}
\end{table}

\textbf{Contributions.} 
We have three sets of contributions. 

First, we study the convergence rate of DSG with constant stepsize which is used in many practical applications \citep{sayed2019distributed,alghunaim2018distributed,dieuleveut2017bridging}. Our bounds provide tighter guarantees on the bias term as well as novel guarantees on the variance term for this algorithm. For quadratic functions, we provide sharper estimates for the bias, variance, and network effect terms that are tight, as there exist simple quadratic functions that achieve these bounds.

Second, we consider D-ASG with constant stepsize. We show that the bias term decays linearly with rate $\bigO(-k/\sqrt{\kappa})$ to a neighborhood of the optimal solution, and thus, it achieves an accelerated rate. We also provide an explicit characterization for this neighborhood, in terms of noise and network structure parameters, with the variance term dominating for small enough stepsize. 
When the objectives $f_i$ are all quadratic, we obtain non-asymptotic guarantees that are explicit in terms of their linear convergence rate and dependence to noise, generalizing available known guarantees for ASG to the distributed setting \citep{can-gur-ling19}. 

For both algorithms, following earlier work on non-distributed versions of these algorithms \citep{StrConvex}, we use our explicit characterization of bias, variance, and network effect terms to provide a computational framework that can choose algorithm parameters to trade-off these difference effects in a systematic manner. In the centralized setting, it has been observed and argued that accelerated algorithms are often more sensitive to noise than non-accelerated algorithms (see e.g.  \cite{flammarion2015averaging,d2008smooth,aybat2019universally,Hardt-blog}), however to our knowledge this behavior has not been systematically studied in the context of decentralized algorithms. We study the asymptotic variance of the D-SG and D-ASG iterates as a measure of robustness to random gradient noise and provide explicit expressions for this quantity for quadratic objectives as well as upper bounds for strongly convex objectives. This allows us to compare D-SG and D-ASG in terms of their robustness to random noise properties. Our results (see the discussion after Theorem~\ref{thm-rate-dasg}) show that indeed D-ASG can be less robust compared to D-SG depending on the choice of the momentum and stepsize parameters, shedding further light into the tuning of hyperparameters (stepsize and momentum) in the distributed setting.

Finally, we study a multistage version of D-ASG, building on the non-distributed method in \cite{aybat2019universally}, whereby a distributed accelerated stochastic gradient method with constant stepsize and momentum parameter is used at every stage, with parameters carefully varied over stages to ensure exact convergence to the optimal solution $x_*$. Similar to \cite{aybat2019universally}, a momentum restart is used to enable stitching the improvement obtained over consecutive stages. We show that
our proposed method achieves an accelerated $\bigO(-k/\sqrt{\kappa})$ linear decay in the bias term as well as a $\bigO(\sigma^2/k)$ term in the variance term and $\bigO((1-\gamma)^{-2} /k^4)$ in terms of network effect, where $1-\gamma$ is the spectral gap of the network, see \eqref{eqn:asymp_suboptim} for a formal definition. \rev{We also show that the node averages also achieves $O(\frac{1}{Nk})$ for the variance term with a tight dependency to the number of nodes $N$.} \rev{This dependency to $k$ and $\sqrt{\kappa}$ is optimal in the context of centralized black-box stochastic optimization. This suggests that our analysis is tight in terms of its $k$ and $\sqrt{\kappa}$ dependency, although the problems we consider is not black-box optimization but finite-sum problems. Such a dependency {on} $k$ and $\sqrt{\kappa}$ was obtained previously for the PBSTM algorithm of \cite{dvinskikh2019decentralized} which is optimal up to logarithmic terms. To the best of our knowledge, our analysis provides the best bounds for the D-ASG algorithm. Our results show that D-ASG without noise converges to a fixed point with the accelerated rate, i.e. the rate has a $\sqrt{\kappa}$ dependency to the condition number.} A summary of all our convergence results is provided in Table~\ref{Dist_Stoch_Table}. \rev{We also provide numerical experiments that show the efficiency of the D-ASG method in a number of decentralized optimization settings.}

\rev{\textbf{Other Related work.} There has been a growing recent interest in the dynamical system representation of distributed optimization algorithms to facilitate their analysis and design. In particular, \cite{sundararajan2020analysis} provides a framework to design a broad class of distributed algorithms for deterministic decentralized optimization for time-varying graphs. This framework provides worst-case certificates of linear convergence via semi-definite programming. Other related papers  \citep{sundararajan2017robust,sundararajan2019canonical} allow analysis and design of deterministic distributed optimization algorithms. However, these results and approaches are targeted for deterministic distributed algorithms and they do not directly apply to the stochastic algorithms we consider in this paper. 
}
\rev{Robustness of stochastic optimization algorithms to stepsize have also been considered in the literature. In particular, the accelerated gradient methods of \citet[Theorem 2, Corollary 1]{lan2012optimal} do enjoy various robustness properties to noise; in particular, for appropriate
stepsize choices, if $L$ is a Lipschitz constant of the gradient, $\sigma^2$ the noise, and $D$ the
diameter of the underlying domain, one may achieve rates roughly
\begin{equation*}
\mathbb{E}\left[f\left(x^{k}\right)-f(x_{\ast})\right]
\leq\frac{LD^{2}}{k^{2}}+\frac{\sigma D}{\sqrt{k}}\left(\gamma^{-1}+\gamma\right),
\end{equation*}
where $\gamma>0$ is a particular stepsize multiplier choice. Thus, misspecifying $\gamma$ does not force
a massive degradation in convergence rates, which reflects the robustness considerations
of \cite{nemirovski2009robust}. The work of \citet[Theorem 2.1.]{duchi2012randomized} also shows a similar robustness result to stepsize specification. 

}

\textbf{Notation.}
Let $\mathcal{S}_{\mu,L}(\mathbb{R}^{d})$ denote the set
of functions from $\mathbb{R}^{d}$ to $\mathbb{R}$ that are
$\mu$-strongly convex and $L$-smooth, that is, for every $x,y\in\mathbb{R}^{d}$,
\begin{align*}
\frac{L}{2}\Vert x-y\Vert^{2} \geq f(x)-f(y) - \nabla f(y)^{T}(x-y) \geq \frac{\mu}{2}\Vert x-y\Vert^{2},
\end{align*}
where we have the condition number $\kappa=L/\mu$.
Let $0_{a\times b}$ denote the zero matrix with $a$ rows and $b$ columns. Given a collection of square matrices $[A_i]_{i=1}^m$, the matrix $\textbf{\mbox{diag}}([A_i]_{i=1}^m)$ denotes the block diagonal square matrix with $i$-th diagonal block equal to $A_i$. For two matrices $A \in \mathbb{R}^{m \times n}$ and $B \in \mathbb{R}^{p \times q}$, we denote their Kronecker product by $A\otimes B$. For two functions $g,h$ defined over positive integers, we say $f = \bigO(g)$ if there exists a constant $C_u$ and a positive integer $n_0$ such that $f(n) \leq C_u g(n)$ for every positive integer $n\geq n_0$. \rev{We say $f = \tilde\bigO(g)$ if there exists a constant $C_u$ and a positive integer $n_0$ such that $f(n) \leq C_u g(n) \log(n)$ for every positive integer $n\geq n_0$. We use the notation $\|A\|_2$ to denote the 2-norm (largest singular value) of a matrix $A$, whereas we use $\|A\|_F$ to denote the Frobenius norm of $A$. For two real-valued functions $f$ and $g$, we say $f = \Theta(g)$ as $x\to 0$ if there exist {positive} constants $C_\ell$ and $C_u$ such that $C_\ell  g(x)\leq f(x) \leq C_u g(x)$ for every $x$ in a neighborhood of $0$ and lying in the domains of $f$ and $g$.}
\section{Distributed Stochastic Gradient and Its Accelerated Variant}\label{sec:strongly:convex}
We will first study the distributed stochastic gradient (D-SG) method which is the stochastic version of the distributed gradient (DG) method introduced in \cite{nedic2009distributed}, and then focus on its accelerated variant.

Consider {\color{black}an undirected} network $\mathcal{G}=(\mathcal{V},\mathcal{E})$ that is connected
by edges {\color{black}$\mathcal{E}\subseteq\mathcal{V}\times\mathcal{V}$}, 
where $\mathcal{V}=\{1,\ldots,N\}$ denotes the set of vertices. 
We associate this network with an $N\times N$ symmetric, doubly stochastic weight matrix $W$. 
We have $W_{ij}=W_{ji}>0$ if $(i,j)\in \mathcal{E}$ and $i\neq j$,
and $W_{ij}=W_{ji}=0$ if $(i,j)\not\in \mathcal{E}$ and $i\neq j$,
and finally $W_{ii}=1-\sum_{j\neq i}W_{ij}>0$
for every\footnote{We adopt the convention that the node is a neighbor of itself, i.e. $(i,i) \in \mathcal{E}$.} $1\leq i\leq N$.
{\color{black}It is known that the eigenvalues of a doubly stochastic matrix $W$ can be} ordered
in a descending manner satisfying:
\begin{equation*}
1=\lambda_{1}^{W}>\lambda_{2}^{W}\geq\cdots\geq\lambda_{N}^{W}>-1,
\end{equation*}
{\color{black}where the largest eigenvalue is $1$ with an all-one eigenvector, i.e. $W\ones = \ones$, and the smallest eigenvalue is greater than $-1$. The eigenvalues of $W$ can be used to study the properties of the network associated with the weight matrix $W$ (see e.g. \cite{chung1997spectral}). For example, if $W$ represents the transition matrix of a Markov chain, then $1-\max\{|\lambda_{2}^{W}|,|\lambda_{N}^{W}|\}$, known as the spectral gap, can be used to measure the mixing time
of the Markov chain, i.e. how fast the Markov chain converges to its stationary distribution (see e.g. \cite{Peres2009}).} 
Such a matrix $W$ always exists (see e.g. \cite{boyd2006randomized}) if the graph is not bi-partite and there can be different choices of $W$ \citep{shi2015extra}. For bi-partite graphs, one can also construct such a matrix $W$ by considering the transition matrix of a lazy random walk on the graph (see e.g. \cite{chung1997spectral}).

Next, we make a few definitions for the sake of subsequent analysis. First define the average iterates
\begin{equation}\label{defn:average}
\bar{x}^{(k)}:=\frac{1}{N}\sum_{i=1}^{N}x_{i}^{(k)}\in\mathbb{R}^{d}.
\end{equation}
Next we define the column vector
\begin{equation}\label{column_x-nonave}
x^{(k)}=\left[\left(x_{1}^{(k)}\right)^{T},\left(x_{2}^{(k)}\right)^{T},\ldots,\left(x_{N}^{(k)}\right)^{T}\right]^{T}\in\mathbb{R}^{Nd},
\end{equation}
which concatenates the local decision variables into a single vector. We also define $x^* \in \mathbb{R}^{Nd}$ as
\begin{equation}\label{x_*_def}
x^* = \begin{bmatrix} x_*^T & x_*^T & \cdots & x_*^T \end{bmatrix}^T,	
\end{equation}
which is the column vector of length $Nd$ that concatenates $N$ copies of the optimizer $x_*$ to the problem \eqref{opt-pbm}.

In addition, we define $F:\mathbb{R}^{Nd}\rightarrow\mathbb{R}$ as
\begin{equation*}
F(x):=F(x_{1},\ldots,x_{N})
=\sum_{i=1}^{N}f_{i}(x_{i}),
\end{equation*}
where
\begin{equation*}
\tilde{\nabla}F\left(x^{(k)}\right)=\left[\left(\tilde{\nabla} f_1 \left(x_1^{(k)}\right)\right)^{T},\left(\tilde{\nabla} f_2 \left(x_2^{(k)}\right)\right)^{T},\ldots,\left(\tilde{\nabla} f_N \left(x_N^{(k)}\right)\right)^{T}\right]^{T},
\end{equation*} 
which obeys 
\begin{equation}\label{assump}
\mathbb{E}\left[\tilde{\nabla}F\left(x^{(k)}\right)\Big|x^{(k)}\right]= \nabla F\left(x^{(k)}\right),
\quad
\mathbb{E}\left[\left\Vert \tilde{\nabla}F\left(x^{(k)}\right) - \nabla F\left(x^{(k)}\right) \right\Vert^{2}\Big|x^{(k)}\right]
\leq\sigma^{2}N,
\end{equation}
\noindent due to Assumption~\ref{assump_1}. Furthermore, $F\in\mathcal{S}_{\mu,L}(\mathbb{R}^{Nd})$
is $\mu$-strongly convex and $L$-smooth.  
\subsection{Distributed stochastic gradient (D-SG)}\label{sec:DSGD}
Recall that $x_i^{(k)}$ denotes the decision variable of node $i$ at iteration $k$. The D-SG iterations update this variable by performing a stochastic gradient descent update with respect to the local cost function $f_i$ together with a weighted averaging with the decision variables $x_j^{(k)}$ of node $i$'s immediate neighbors $j\in\Omega_{i}:=\{j : (i,j) \in \mathcal{E}\}$:
\begin{equation}\label{eqn:dsg_update}
x_{i}^{(k+1)}=\sum_{j\in\Omega_{i}}W_{ij}x_{j}^{(k)}-\alpha {\tilde{\nabla} f_{i}\left(x_{i}^{(k)}\right)},
\end{equation}
where $\alpha>0$ is the stepsize. Note that we can express the D-SG iterations as
\begin{equation}\label{eq-iter-dgd}
x^{(k+1)}=\mathcal{W}x^{(k)}-\alpha \tilde{\nabla} F\left(x^{(k)}\right),
\end{equation}
where $\mathcal{W} := W \otimes I_d$. 

Without noise, i.e., when $\tilde{\nabla} F(x^{(k)}) = \nabla F (x^{(k)})$, D-SG reduces to the DG algorithm. In this case,  \cite{yuan2016convergence} show that 
DG algorithm is \emph{inexact} in the sense that
the iterates $x_i^{(k)}$ of the DG algorithm do not converge to the optimum $x_*$ in general with constant stepsize, but instead converge linearly to a fixed point $x_i^{\infty}$ that is in a neighborhood of the solution satisfying 
\begin{equation}\label{eqn:asymp_suboptim}
\left\| x_i^{\infty} - x_*\right\| \leq C_1 \frac{\alpha}{1-\gamma}=  \mathcal{O}\left(\frac{\alpha}{1-\gamma}\right), \quad \mbox{where} 
\quad \gamma := \max\left\{\left|\lambda_2^W\right|, \left|\lambda_N^W\right|\right\},
\end{equation}
for some constant $C_1$ with the explicit expression
\begin{equation}\label{defn:C:1}
C_{1}:= \sqrt{2L\sum_{i=1}^N \left(f_i\left(0\right) -f_i^*\right)}\cdot \left(1 + \frac{2(L+\mu)}{\mu   }\right), \quad f_i^* := \min_{x\in\mathbb{R}^d} f_i(x),
\end{equation}
provided that the stepsize $\alpha$ satisfies some conditions \citep{yuan2016convergence}
{(see Lemma~\ref{lem:C1:gamma} in the Appendix for details). 

Similar to \eqref{x_*_def}, we define the column vector
\begin{equation}\label{def-fixed-pt}
x^\infty := 
\left[\left(x_{1}^{\infty}\right)^{T},\left(x_{2}^{\infty}\right)^{T},
\cdots,\left(x_{N}^{\infty}\right)^{T}\right]^{T} \in \mathbb{R}^{Nd},
\end{equation}
which is a concatenation of the fixed point $x_i^\infty$ of node $i$ over all the nodes. It can be checked that the unique fixed point $x^\infty$ to \eqref{eq-iter-dgd} in the noiseless setting is the solution to
\begin{equation}\label{eq-fixed-point}
(I_{Nd}-\calW) x^\infty +\alpha \nabla F(x^\infty)= 0.
\end{equation}

This means that the sequence $\xi_{k}:=x^{(k)}-x^{\infty}$ converges to zero with an appropriate choice of the stepsize. The performance of the algorithm can then be measured by the distance of $x^\infty$ to $x^* \in \mathbb{R}^{Nd}$ given by \eqref{x_*_def}.
\subsection{Distributed accelerated stochastic gradient (D-ASG)}\label{sec:DNGD}

Consider the following variant of D-SG:
\begin{equation}\label{def-dasg-iters-0}
\begin{aligned}
&x_{i}^{(k+1)}=\sum_{j\in\Omega_{i}}W_{ij}y_{j}^{(k)}-\alpha \tilde{\nabla} f_{i}\left(y_{i}^{(k)}\right),
\\
&{\color{black}y_{i}^{(k)}}=(1+\beta)x_{i}^{(k)}-\beta x_{i}^{(k-1)},
\end{aligned}
\end{equation}
where $\alpha>0$ is the stepsize and $\beta\geq 0$ is called the \emph{momentum parameter}. This algorithm has also been considered in the literature by \cite{jakovetic2014fast} in the noiseless setting.

We define the average iterates $\bar{x}^{(k)}$ and the column vector $x^{(k)}$ as in \eqref{defn:average} and \eqref{column_x-nonave}, respectively. Also, similar to \eqref{column_x-nonave}, we
define the column vector
\begin{align*}
&y^{(k)}=\left[\left(y_{1}^{(k)}\right)^{T},\left(y_{2}^{(k)}\right)^{T},\ldots,\left(y_{N}^{(k)}\right)^{T}\right]^{T}\in\mathbb{R}^{Nd}.
\end{align*}
Then, we can re-write the D-ASG iterates \eqref{def-dasg-iters-0} as:
\begin{equation}\label{def-dasg-iters}
\begin{aligned}
&x^{(k+1)}=\mathcal{W}y^{(k)}-\alpha \tilde{\nabla} F\left(y^{(k)}\right), \\
&y^{(k)}=(1+\beta)x^{(k)}-\beta x^{(k-1)},
\end{aligned}
\end{equation}
for $k\geq 0$ starting from the initial values $x_i^{(0)}\in\mathbb{R}$ and $x_i^{(-1)}\in\mathbb{R}$ for each node $i$. Here, $\alpha>0$ is the stepsize and $\beta\geq 0$ is the momentum parameter. Note that for $\beta=0$, D-ASG reduces to the D-SG algorithm. When there is a single node, i.e. $N=1$, D-ASG also reduces to the Nesterov's (non-distributed) accelerated stochastic gradient algorithm (ASG) \citep{nesterov_convex}. Note that this algorithm is also inexact in the sense that both $\{x^{(k)}\}$ and $\{y^{(k)}\}$ will also converge to the same point $x^{\infty}=y^{\infty}$ in the noiseless setting where $x^{\infty}$ is the fixed point of the distributed gradient (DG) algorithm defined by \eqref{eq-fixed-point}.

\subsection{Convergence Rates and Robustness to Gradient Noise}

Consider both D-SG and D-ASG algorithms,
subject to gradient noise satisfying Assumption~\ref{assump_1}. 
For this scenario, the noise is
\emph{persistent}, i.e., it does not
decay over time, and it is possible
that the limit of $x^{(k)}$ as $k\rightarrow\infty$
may not exist (even in the non-distributed setting), see \cite{can-gur-ling19}; therefore, one natural way\footnote{There are other possible ways
to define a robustness measure, see e.g. \cite{StrConvex}.}
of defining \emph{robustness}
of an algorithm to gradient noise is to consider
the worst-case limiting variance
along all possible subsequences, i.e.
\begin{equation}\label{def-robust}
J_{\infty}:= \frac{1}{\sigma^2 N} \limsup_{k\to\infty} \mbox{Var}\left(x^{(k)}\right).
\end{equation}
\rev{In the special case, when $F$ is a quadratic function and the gradient noise is i.i.d. with an isotropic Gaussian distribution, the quantity $J_\infty$ is equal to the square of the $H_2$ norm of the linear dynamical system corresponding to the D-ASG iterations \eqref{def-dasg-iters} (see e.g. \cite{zhou1996robust, StrConvex}). $H_2$ norm
is well-studied in the robust control theory as a robustness metric and has been considered in the distributed algorithms literature previously as a measure of robustness to white noise (see e.g. \cite{pirani,sarkar2018asymptotic,chapman2015semi}). Indeed, we observe from \eqref{def-robust} that $J_{\infty}$ is equal to the ratio of the output variance and the input noise variance $\sigma^2N$ (which is the variance of noise at the worst case), therefore it can be interpreted as a signal-to-noise ratio (SNR) measure, quantifying how robust the underlying algorithm is to white noise.  We also note that the same definition was recently applied to optimization to develop noise-robust non-distributed algorithms \citep{StrConvex}. Our definition \eqref{def-robust} of robustness is motivated by such connections to the robust control and optimization literature. }

In the next sections, we will provide bounds on the robustness level $J_{\infty}$ and the expected distance to both the fixed point and the optimum for the D-SG and D-ASG algorithms. In particular, in the non-distributed setting, it is known that ASG can be less robust to noise compared to gradient descent \citep{Hardt-blog,StrConvex}; we will later obtain bounds in Section~\ref{subsec:dasg} for the robustness of D-ASG and D-SG which suggests a similar behavior in the distributed setting when the stepsize is small enough.

For analysis purposes, we consider the penalized objective function $F_{\calW,\alpha}(x):\mathbb{R}^{Nd}\to \mathbb{R}$ defined as
\begin{equation}\label{def-pen-obj}
F_{\calW,\alpha}(x) := \frac{1}{2\alpha} x^T (I_{Nd}-\calW)x + F(x),\quad\alpha>0.
\end{equation}
Similar penalized objectives have also been considered in the past to analyze deterministic algorithms (see e.g. \cite[Section 2]{yuan2016convergence}, \cite{mansoori2017superlinearly}). It can be seen that its gradient (with respect to $x$) is 
$\nabla F_{\calW,\alpha}(x) = \frac{1}{\alpha}(I_{Nd}-\calW)x + \nabla F(x)$. Since  $0_{Nd} \preceq I_{Nd}-\mathcal{W}\preceq (1 - \lambda_N^W)I_{Nd}$, we have also
\begin{equation}\label{def-L-alpha} F_{\calW,\alpha}\in \mathcal{S}_{\mu,L_\alpha}\left(\mathbb{R}^{Nd}\right) \quad \mbox{with} \quad L_\alpha := \frac{1-\lambda_N^W}{\alpha} + L.\end{equation}
Furthermore, the unique minimizer $z^*$ of $F_{\calW,\alpha}$ satisfies the first-order conditions 
$$
\nabla F_{\calW,\alpha}(z^*)= (I_{Nd}-\calW)z^* + \alpha \nabla F(z^*) = 0.
$$
Then, it follows from \eqref{eq-fixed-point} that $z^* = x^\infty$, i.e. the minimizer of $F_{\calW,\alpha}$ coincides with the limit point $x^\infty$. 
In fact, 
we can re-write the D-SG iterations \eqref{eq-iter-dgd} as
\begin{equation}\label{eq-centr-dsg}
x^{(k+1)}=x^{(k)}-\alpha \tilde{\nabla} F_{\mathcal{W},\alpha}\left(x^{(k)}\right),
\end{equation}
which is equivalent to running a non-distributed stochastic gradient algorithm for minimizing an alternative objective $F_{\calW,\alpha}$ in dimension $Nd$.
We can also re-write the D-ASG iterations \eqref{def-dasg-iters} as
\begin{equation}\label{eq-centr-asg}
\begin{aligned}
&x^{(k+1)}= y^{(k)}-\alpha \tilde{\nabla} F_{\mathcal{W},\alpha}\left(y^{(k)}\right), \\
&y^{(k)}=(1+\beta)x^{(k)}-\beta x^{(k-1)}.
\end{aligned}
\end{equation}
These iterations are identical to the iterations of the (non-distributed) ASG. In other words, D-ASG applied to solve the problem \eqref{opt-pbm} in dimension $d$ is equivalent to running a non-distributed ASG algorithm for minimizing an alternative objective $F_{\calW,\alpha}$ in dimension $Nd$. 

This connection allows us to analyze both D-SG and D-ASG with existing techniques developed for non-distributed algorithms in \cite{StrConvex,aybat2019universally} that builds on dynamical system representation of optimization algorithms.

\subsubsection{Dynamical system representation}\label{sec:reformulate}
We first reformulate D-SG \eqref{eq-centr-dsg} and D-ASG update rules \eqref{eq-centr-asg} 
as a discrete-time dynamical system:

\begin{equation}\label{unified-dyn-sys}
\xi_{k+1}=A\xi_{k}+B \tilde{\nabla} F_{\calW,\alpha}(C\xi_{k}),
\end{equation} 
where $\xi_{k}$ is the state, and $A, B, C$ are system matrices that are appropriately chosen. For example, we can represent the D-SG iterates with the choice of
\begin{equation} \xi_{k}:=x^{(k)}-x^{\infty}, \quad A:=I_{Nd},\quad B:=-\alpha I_{Nd}, \quad C:=I_{Nd}.\label{DGD:xi} 
\end{equation}
Similarly, we can represent the D-ASG iterations as the dynamical system \eqref{unified-dyn-sys} with 
\begin{equation}\label{eqn:xi:DASG}
\xi_{k}:=\left[\left(x^{(k)}-x^{\infty}\right)^{T},\left(x^{(k-1)}-x^{\infty}\right)^{T}\right]^{T}, 
\end{equation}
and 
$A = \tilde{A}_{\text{dasg}}\otimes I_{Nd}, B:=\tilde{B}_{\text{dasg}}\otimes I_{Nd}, C:=\tilde{C}_{\text{dasg}}\otimes I_{Nd}$
where
\begin{equation}\label{ABC:DASG}
\tilde{A}_{\text{dasg}}
=\left[\begin{array}{cc}
1+\beta & -\beta
\\
1 & 0
\end{array}\right],
\quad
\tilde{B}_{\text{dasg}}
=\left[\begin{array}{c}
-\alpha
\\
0
\end{array}\right],
\quad
\tilde{C}_{\text{dasg}}
=\left[\begin{array}{cc}
1+\beta & -\beta
\end{array}\right].
\end{equation}
(see also \cite{lessard2016analysis} for such a dynamical system representation in the deterministic case). For studying the dynamical system \eqref{ABC:DASG}, we introduce the following Lyapunov function 
  \begin{equation}\label{def-lyap} V_{P,\alpha,c}(\xi) := \xi^T  P \xi+ c \left[F_{\calW,\alpha}(T\xi + x^\infty) - F_{\calW,\alpha}(x^\infty)\right],
  \end{equation}
where $c\geq 0$ is a scalar, $P$ is a positive semi-definite matrix and {\color{black}$T=I_{Nd}$ 
for D-SG and $T=[1\,\, 0]\otimes I_{Nd}$ for D-ASG}. Since $x^\infty$ is the minimum of $F_{\calW,\alpha}$, we observe that $V_{P,\alpha,c}(\xi)$ has non-negative values. In particular, $V_{P,\alpha,c}(0) = 0$. In the special case when $c=0$, we obtain 
 $$ V_{P}(\xi) := V_{P,\alpha,0}(\xi) =\xi^T  P \xi. $$
In the next section, we obtain convergence results for D-SG and D-ASG for constant stepsize and momentum which also implies guarantees on the robustness measure $J_\infty$. The analysis is based on studying the Lyapunov function \eqref{def-lyap} for different choices of the matrix $P$ and the scalar $c$. In particular, for D-SG we can choose $P$ to be the identity matrix and $c=0$, however for D-ASG, the choice of $P$ is less trivial and depends on the choice of the stepsize $\alpha$ and $\beta$ in general. Here, our choice of the Lyapunov function \eqref{def-lyap} is motivated by \cite{fazlyab2017analysis} which studied this Lyapunov function to analyze accelerated gradient methods in the centralized deterministic setting.  

\subsubsection{Analysis of Distributed Stochastic Gradient}\label{sec:DGD}
We next provide a performance bound for D-SG in Theorem~\ref{thm:DGD:explicit}. 
It shows that the expected distance square to the fixed point $\mathbb{E}\left[\left\Vert x^{(k)}-x^{\infty}\right\Vert^{2}\right]$ can be bounded as a sum of two terms: $i)$ A \emph{bias term} that depends on the initialization and decays with a linear rate $\rho^2(\alpha)$ where $ \rho(\alpha)= \max\left\{\left|1-\alpha\mu\right|, \left|\lambda_N^W - \alpha L\right|\right\}.
$ $(ii)$ A \emph{variance term} that scales linearly with the noise level $\sigma^2$ providing a bound on the asymptotic variance $\limsup_{k\to\infty} \mathbb{E}\left[\left\Vert x^{(k)}-x^{\infty}\right\Vert^{2}\right]$ and hence the robustness level $J_\infty$. When there is no noise (when $\sigma =0$), the variance term is zero, and we obtain a linear convergence rate for the (deterministic) DG algorithm with rate $\rho^2(\alpha)$.
This improves the previously best known convergence rate $\rho_{\delta}^{2}$ for DG obtained in \cite{yuan2016convergence}, 
where $\rho_{\delta}^{2}:=1-\frac{\alpha\mu L}{\mu+L}
+\alpha\delta-\alpha^{2}\delta\frac{\mu L}{\mu+L}$, 
which can get arbitrarily close to $1-\frac{\alpha\mu L}{\mu+L}$, 
see Theorem~7 in \cite{yuan2016convergence}. We also note that the convergence rate and robustness we provide in Theorem~\ref{thm:DGD:explicit} is tight for D-SG in the sense that they are attained for some quadratic choices of the objective (see Remark~\ref{remark-dgd-quad-tightness} in Appendix~\ref{sec:quadratic}). 

For proving Theorem~\ref{thm:DGD:explicit}, we exploit the above-mentioned fact that running D-SG on the objective $F$ is equivalent to running (non-distributed SG) on the modified objective $F_{\calW,\alpha}$ and we build on the existing results for non-distributed stochastic gradient \cite[Prop. 4.3]{StrConvex}; the proof is given in the Appendix. 

\begin{theorem}\label{thm:DGD:explicit}
Consider running D-SG method with stepsize $\alpha \in (0,\frac{1+\lambda_{N}^{W}}{L})$. Then,
for every $k\geq 0$,
\begin{align}
&\mathbb{E}\left[\left\Vert x^{(k)}-x^{\infty}\right\Vert^{2}\right]
\leq\rho(\alpha)^{2k}\left\Vert x^{(0)}-x^{\infty}\right\Vert^{2}
+\frac{1-\rho(\alpha)^{2k}}{1-\rho(\alpha)^{2}}\sigma^{2}\alpha^{2}N,
\end{align}
where $\rho(\alpha)= \max\left\{\left|1-\alpha\mu\right|, \left|\lambda_N^W - \alpha L\right|\right\}
\in[0,1)$. As a result, the robustness of the D-SG method satisfies
\begin{equation*}
J_{\infty}(\alpha)\leq\frac{\alpha^{2}}{1-\rho(\alpha)^{2}}.
\end{equation*}	
\end{theorem}

We recall that the penalized objective
$F_{\calW,\alpha}$ depends on the network and the stepsize. The fixed point $x^\infty$ is the minimum of the penalized objective $F_{\calW,\alpha}$. In general, the difference  $\|x^{(\infty)} - x^*\|$ is not zero and it depends on the network structure and the stepsize $\alpha$. We call this term the ``network effect"; it can be controlled by the the inequality \eqref{eqn:asymp_suboptim}. The following corollary is obtained by a direct application of the inequality \eqref{eqn:asymp_suboptim} to Theorem~\ref{thm:DGD:explicit}.   

\begin{corollary}\label{DSG_cor}
Consider running D-SG method with stepsize $\alpha \in (0,\frac{1+\lambda_{N}^{W}}{\mu + L})$. Then, for every $k \geq 0$,
\begin{align}
&\mathbb{E}\left[\left\Vert x^{(k)}-x^{\infty}\right\Vert^{2}\right]
\leq (1-\alpha\mu)^{2k}\left\Vert x^{(0)}-x^{\infty}\right\Vert^{2}
+\alpha\sigma^{2}N \frac{1-(1-\alpha\mu)^{2k}}{\mu (2-\alpha\mu)},
\end{align}
which implies that the robustness of the D-SG method satisfies
\begin{equation*}
J_{\infty}(\alpha)\leq\frac{\alpha}{\mu(2-\alpha\mu)}.
\end{equation*}	
In addition, if $\alpha \leq \frac{1}{L+\mu}$, we have
\begin{align}
\mathbb{E}\left[\left\Vert x^{(k)}-x^{\ast}\right\Vert^{2}\right]
\leq 2(1-\alpha\mu)^{2k}\left\Vert x^{(0)}-x^{\infty}\right\Vert^{2}
+2\alpha\sigma^{2}N\frac{1-(1-\alpha\mu)^{2k}}{\mu (2-\alpha\mu)}
+\frac{2\alpha^{2}C_{1}^{2}N}{(1-\gamma)^{2}},
\label{ineq-perf-ub-dgd}
\end{align}
where {\color{black}$\gamma,C_{1}$ are given in \eqref{eqn:asymp_suboptim}-\eqref{defn:C:1}.}
\end{corollary} 

{\color{black}Next, we provide the performance bound on the distance between the average of iterates $\bar{x}^{(k)}$
and the minimizer $x_{\ast}$. Here, we can show that the asymptotic variance of the averaged iterates $\bar{x}^{(k)}$ with constant stepsize is $\mathcal{O}(\sigma^2/N)$; this is because averaging the iterates also averages the noise over the nodes.
\begin{proposition}\label{prop:dsg-average-optimal-rate}
Assume $0<\alpha\leq\frac{2}{L+\mu}$,
$\alpha<\frac{1+\lambda_{N}^{W}}{L}$ and $\mu\alpha(1+\lambda_{N}^{W}-\alpha L)<1$.  
Then, for any $k$, we have
\begin{align*}
\mathbb{E}\left\Vert\bar{x}^{(k)}-x_{\ast}\right\Vert^{2}
&\leq
8\left(\frac{\alpha}{\mu(1-\frac{\alpha L}{2})}+\frac{(1+\alpha L)^{2}}{\mu^{2}(1-\frac{\alpha L}{2})^{2}}\right)
\left(\frac{L^{2}D^{2}\alpha^{2}}{N(1-\gamma)^{2}}
+\frac{L^{2}\sigma^{2}\alpha^{2}}{(1-\gamma^{2})}\right)
\\
&\qquad\qquad
+8\frac{\gamma^{2k}-
\left(1-\alpha\mu\left(1-\frac{\alpha L}{2}\right)\right)^{k}}
{\gamma^{2}-1+\alpha\mu\left(1-\frac{\alpha L}{2}\right)}
\frac{L^{2}\gamma^{2}}{N}\mathbb{E}\left\Vert x^{(0)}\right\Vert^{2}
\\
&\qquad
+2(1-\alpha\mu)^{2k}
\Vert x_{0}-x_{\ast}\Vert^{2}
+\frac{1-(1-\alpha\mu)^{2k}}{\mu(1-\frac{\alpha\mu}{2})}\frac{\alpha\sigma^{2}}{N},
\end{align*}
and for every $i=1,2,\ldots,N$ and any $k$,
\begin{align*}
\mathbb{E}\left\Vert x_{i}^{(k)}-x_{\ast}\right\Vert^{2}
&\leq
8\gamma^{2k}\mathbb{E}\left\Vert x^{(0)}\right\Vert^{2}
+\frac{8D^{2}\alpha^{2}}{(1-\gamma)^{2}}
+\frac{8\sigma^{2}N\alpha^{2}}{(1-\gamma^{2})}
\\
&\qquad
+16\left(\frac{\alpha}{\mu(1-\frac{\alpha L}{2})}+\frac{(1+\alpha L)^{2}}{\mu^{2}(1-\frac{\alpha L}{2})^{2}}\right)
\left(\frac{L^{2}D^{2}\alpha^{2}}{N(1-\gamma)^{2}}
+\frac{L^{2}\sigma^{2}\alpha^{2}}{(1-\gamma^{2})}\right)
\\
&\qquad\qquad
+16\frac{\gamma^{2k}-
\left(1-\alpha\mu\left(1-\frac{\alpha L}{2}\right)\right)^{k}}
{\gamma^{2}-1+\alpha\mu\left(1-\frac{\alpha L}{2}\right)}
\frac{L^{2}\gamma^{2}}{N}\mathbb{E}\left\Vert x^{(0)}\right\Vert^{2}
\\
&\qquad
+4(1-\alpha\mu)^{2k}
\Vert x_{0}-x_{\ast}\Vert^{2}
+2\frac{1-(1-\alpha\mu)^{2k}}{\mu(1-\frac{\alpha\mu}{2})}\frac{\alpha\sigma^{2}}{N},
\end{align*}
where
\begin{equation}\label{eqn:D1}
D^{2}:=4L^{2}\mathbb{E}\left\Vert x^{(0)}-x^{\ast}\right\Vert^{2}
+8L^{2}\frac{C_{1}^{2}\alpha^{2}N}{(1-\gamma)^{2}}
+\frac{2L^{2}\alpha\sigma^{2}N}{\mu(1+\lambda_{N}^{W}-\alpha L)}
+4\left\Vert \nabla F\left(x^{\ast}\right)\right\Vert^{2},
\end{equation} 
where $\gamma,C_{1}$ are given in \eqref{eqn:asymp_suboptim}-\eqref{defn:C:1}. 
\end{proposition}}

\begin{remark}[Convergence rate of the averaged D-SG iterates] 
{\color{black}Note that given iteration budget $K>0$, if we take $\alpha = \frac{\log(K)}{\mu K}$ in the setting of Proposition \ref{prop:dsg-average-optimal-rate}, then we have $(1-\alpha \mu)^{2K} = \mathcal{O}(1/K^2)$ and we obtain $\mathbb{E}\left\Vert\bar{x}^{(K)}-x_{\ast}\right\Vert^{2} = \tilde{\mathcal{O}}(\frac{1}{NK} + \frac{1}{K^2})$ where $\tilde{\mathcal{O}}(\cdot)$ hides a logarithmic factor in $K$. }
\end{remark}

\subsubsection{Analysis of Distributed Accelerated Stochastic Gradient}
\label{subsec:dasg}


Throughout this section, we state the results under the following assumption.
\begin{assumption}\label{assump_3}
We assume all eigenvalues of $W$ are positive, i.e., we assume that $\lambda_N^W > 0$. 	
\end{assumption}

We note that Assumption~\ref{assump_3} is not restrictive in the sense that even if the weight matrix $W$ does not satisfy this assumption, we can still apply the results in our paper 
by considering the modified weight matrix $W_\tau := \frac{\tau}{\tau+1} I + \frac{1}{\tau+1} W$ for $\tau>1$ instead of $W$. Because, we have $\lambda_N^{W_\tau}> \frac{\tau -1}{\tau+1}>0$ for $\tau>1$ and therefore $W_\tau$ satisfies Assumption~\ref{assump_3}.
{\color{black}We will elaborate this point further after Corollary~\ref{cor:rate:dasg} in Remark~\ref{remark-spec-gap}.}

The following result extends  \cite{StrConvex} from non-distributed ASG to D-ASG.
\begin{theorem}\label{thm:general:DASG}
Assume there exist $\rho\in(0,1)$ and a positive semi-definite $2\times 2$
matrix $\tilde{P}$ such that
\begin{equation}\label{MI_general}
\rho^{2}\tilde{X}_{1}+(1-\rho^{2})\tilde{X}_{2}
\succeq 
\left[\begin{array}{cc}
\tilde{A}_{\text{dasg}}^{T}\tilde{P}\tilde{A}_{\text{dasg}}-\rho^{2}\tilde{P} & \tilde{A}_{\text{dasg}}^{T}\tilde{P}\tilde{B}_{\text{dasg}}
\\
\tilde{B}_{\text{dasg}}^{T}\tilde{P}\tilde{A}_{\text{dasg}} & \tilde{B}_{\text{dasg}}^{T}\tilde{P}\tilde{B}_{\text{dasg}}
\end{array}\right],
\end{equation}
where $\tilde{A}_{\text{dasg}}$, $\tilde{B}_{\text{dasg}}$ and 
$\tilde{C}_{\text{dasg}}$ are defined in \eqref{ABC:DASG} and
\begin{equation*}
\tilde{X}_{1}:=
\left[\begin{array}{ccc}
\frac{\beta^{2}\mu}{2} & \frac{-\beta^{2}\mu}{2} & \frac{-\beta}{2}
\\
\frac{-\beta^{2}\mu}{2} & \frac{\beta^{2}\mu}{2} & \frac{\beta}{2}
\\
\frac{-\beta}{2} & \frac{\beta}{2} & \frac{\alpha(1+\lambda_{N}^{W}-L\alpha)}{2}
\end{array}\right],
\quad
\tilde{X}_{2}:=
\left[\begin{array}{ccc}
\frac{(1+\beta)^{2}\mu}{2} & \frac{-\beta(1+\beta)\mu}{2} & \frac{-(1+\beta)}{2}
\\
\frac{-\beta(1+\beta)\mu}{2} & \frac{\beta^{2}\mu}{2} & \frac{\beta}{2}
\\
\frac{-(1+\beta)}{2} & \frac{\beta}{2} & \frac{\alpha(1+\lambda_{N}^{W}-L\alpha)}{2}
\end{array}\right].
\end{equation*}
Let $P=\tilde{P}\otimes I_{Nd}$. Then, for every $k\geq 0$,
\begin{align}
&\mathbb{E}\left[\left\Vert x^{(k)}-x^{\infty}\right\Vert^{2}\right]\leq
\rho^{2k}\frac{2V_{P, \alpha, 1}(\xi_{0})}{\mu}
+\frac{1}{1-\rho^{2}}\frac{2\alpha^{2}\sigma^{2}N}{\mu }\left(\tilde{P}_{11}+\frac{1-\lambda_{N}^{W}+\alpha L}{2\alpha}\right).\label{ineq-perf-dasg}
\end{align}
Therefore, the robustness of D-ASG iterations defined in \eqref{def-robust} satisfies
\begin{equation*}
J_{\infty}\leq
\frac{2\alpha^{2}}{\mu(1-\rho^{2})}\left(\tilde{P}_{11}+\frac{1-\lambda_{N}^{W}+\alpha L}{2\alpha}\right).
\end{equation*}
\end{theorem}

{\color{black}With the additional assumption $\alpha \leq \frac{1}{L+\mu}$, we have the following corollary.}

\begin{corollary}\label{cor:general:DASG}
{\color{black}Under the assumptions in Theorem~\ref{thm:general:DASG}, if in addition, $\alpha \leq \frac{1}{L+\mu}$, 
then we have
\begin{align*}
\mathbb{E}\left[\left\Vert x^{(k)}-x^{\ast}\right\Vert^{2}\right]\leq
4\rho^{2k}\frac{V_{P, \alpha, 1}(\xi_{0})}{\mu}
+\frac{1}{1-\rho^{2}}\frac{4\alpha^{2}\sigma^{2}N}{\mu }\left(\tilde{P}_{11}+\frac{1-\lambda_{N}^{W}+\alpha L}{2\alpha}\right)
+\frac{2\alpha^{2}C_{1}^{2}N}{(1-\gamma)^{2}},
\end{align*}
where $\gamma,C_{1}$ are given in \eqref{eqn:asymp_suboptim}-\eqref{defn:C:1}.}
\end{corollary}

The results in Theorem~\ref{thm:general:DASG} are stated in terms of a $2\times 2$ matrix $\tilde P$ which solves the $3 \times 3$ matrix inequality \eqref{MI_general}. For any fixed $\alpha$, $\beta$ and $\rho$; this is a linear matrix inequality (LMI). Therefore, we can compute $\tilde P$ numerically by 
varying $\alpha$, $\beta$ and $\rho$ on a grid and then solving the resulting LMIs with a software such as CVX \citep{grant2008cvx} (see also \cite{lessard2016analysis} for a similar approach). However, in the next result, we obtain some explicit performance bounds in the special case when $\beta=\frac{1-\sqrt{\alpha\mu}}{1+\sqrt{\alpha\mu}}$; this choice of $\beta$ is motivated by the fact that it is a common choice in the non-distributed and noiseless setting.\footnote{Furthermore it can be shown that it gives the fastest rate for quadratic objectives in the non-distributed case when there is no noise \citep{aybat2019universally}.} The proof is deferred to the Appendix; it is based on the fact that when $\beta=\frac{1-\sqrt{\alpha\mu}}{1+\sqrt{\alpha\mu}}$, $\rho=1-\sqrt{\alpha\mu}$ and $\alpha \in (0,\frac{\lambda_N^W}{L}]$; $\tilde{P}=\tilde S_\alpha$ is an explicit solution to the matrix inequality \eqref{MI_general} where

\begin{equation*} 
\tilde{S}_{\alpha} :=
    \begin{bmatrix} \frac{1}{2\alpha} &  -\frac{1-\sqrt{\alpha \mu}}{2\alpha} \\ 
     -\frac{1-\sqrt{\alpha \mu}}{2\alpha} & -\left(\frac{1-\sqrt{\alpha \mu}}{2\alpha}\right)^2
    \end{bmatrix}= vv^T \quad \mbox{where} \quad 
    v := \begin{bmatrix}
     \frac{1}{\sqrt{2\alpha}} \\
     \sqrt{\frac{\mu}{2}} - \sqrt{\frac{1}{2\alpha}}
    \end{bmatrix}.
\end{equation*}
Then, plugging in $\tilde P = \tilde S_\alpha$ in Theorem~\ref{thm:general:DASG} and in the bound \eqref{thm:general:DASG}, we obtain performance guarantees in terms of the Lyapunov function $V_{S_\alpha,\alpha,1}$. To simplify the notation in this case, with slight abuse of notation, we let 
\begin{equation}\label{eqn:V:S:alpha}
V_{S,\alpha}(\xi) := V_{S_\alpha,\alpha,1}(\xi) = \xi^T  S_{\alpha} \xi+ F_{\calW,\alpha}(T\xi + x^\infty) - F_{\calW,\alpha}(x^\infty).
\end{equation}
We have the following explicit performance bounds on the convergence and the robustness of D-ASG.
\begin{theorem}\label{thm-rate-dasg} 
Consider running D-ASG method with $\alpha \in (0,\frac{\lambda_N^W}{L}]$ and $\beta = \frac{ 1-\sqrt{\alpha\mu}}{1+\sqrt{\alpha\mu}}$. Then, 
for any $k\geq 0$, we have
\begin{align}
&\E \left[\left\|  x^{(k)} - x^\infty \right\|^2\right]
\leq
 2\left(1-\sqrt{\alpha \mu}\right)^{k} \frac{ V_{S,\alpha}\left(\xi_{0}\right)}{\mu}
+ \frac{\sigma^2 N \sqrt{\alpha}}{\mu\sqrt{\mu}}\left(2-\lambda_N^W +\alpha L \right). \label{ineq-to-prove-1}
\end{align}
Therefore, the robustness measure (defined in \eqref{def-robust}) satisfies
\begin{equation*}
J_{\infty}(\alpha)\leq
\frac{\sqrt{\alpha}}{\mu\sqrt{\mu}}\left(2-\lambda_N^W +\alpha L \right).
\end{equation*}
\end{theorem}

{\color{black}With the additional assumption $\alpha \leq \frac{1}{L+\mu}$, we have the following corollary.}

\begin{corollary}\label{cor:rate:dasg}
{\color{black}Under the assumptions in Theorem~\ref{thm-rate-dasg}, if in addition, $\alpha \leq \frac{1}{L+\mu}$, 
then we have
\begin{align}
\E\left[ \left\| x^{(k)} -x^* \right\|^2\right]
\leq
 4\left(1-\sqrt{\alpha \mu}\right)^{k} \frac{ V_{S,\alpha}\left(\xi_{0}\right)}{\mu}
+ \frac{2\sigma^2 N \sqrt{\alpha}}{\mu\sqrt{\mu}}\left(2-\lambda_N^W +\alpha L \right) +  \frac{2C_{1}^{2}N\alpha^2}{(1-\gamma)^2}, \label{ineq-to-prove-2}
\end{align}
where $\gamma,C_{1}$ are given in \eqref{eqn:asymp_suboptim}-\eqref{defn:C:1}.}
\end{corollary}

{\color{black}
\begin{remark}[Dependency to the spectral gap]\label{remark-spec-gap} We observe from Corollary~\ref{cor:rate:dasg} that among the three error terms in our performance bounds for D-ASG, only the last error term is about the network effect which depends on the spectral gap and this last error term is linear in $1/(1-\gamma)^{2}$, where $1-\gamma$ is the spectral
gap of the matrix $W$. 
We discussed earlier that Assumption~\ref{assump_3} is not restrictive
because even if the matrix $W$ does not satisfy Assumption~\ref{assump_3}, 
one can consider the modified weight matrix $W_{1}:=\frac{1}{2}I+\frac{1}{2}W$ which will 
satisfy Assumption~\ref{assump_3}. When we use the modified weight matrix $W_1$, the spectral gap may get smaller (i.e. spectral gap of $W_1$ can be smaller than that of $W$) and consequently the error term $1/(1-\gamma)^{2}$ due to network effects
in Corollary~\ref{cor:rate:dasg} may get (worse) larger. However, the network error term
can only get larger by a constant factor of 4. 
To explain this point further, assume that $W$ does not satisfy 
Assumption~\ref{assump_3}. In this case the smallest eigenvalue $\lambda_N^W$ of the mixing matrix $W$ can be negative. 
We have two cases: (I) $|\lambda_N^W| > |\lambda_2^W|$; (ii) $|\lambda_N^W| \leq |\lambda_2^W|$.
In case (I), i.e. when $|\lambda_N^W| > |\lambda_2^W|$, the spectral gap is determined by $\lambda_N^W$ in the sense that we have the spectral gap $\Delta(W):=1 - |\lambda_N^W|$ and the spectral gap of the shifted matrix $\Delta(W_{1})=\Delta(\frac{I+W}{2}) =\frac{1 - |\lambda_2(W)| }{2}$ can be larger; for instance when $\lambda_N^W$ is close enough to $-1$. If that is the case, then shifting the $W$ matrix will result in an improved spectral gap and improved convergence guarantees. If on the other hand, $\lambda_N^W$ is sufficiently far away from $-1$, then the spectral gap of the shifted matrix can be smaller, but by a factor of at most $2$; in other words we would have
	 $\Delta(W) = 1 - |\lambda_N^W| \leq 2\Delta(W_{1})=2\Delta\left(\frac{I+W}{2}\right) = 1- |\lambda_2^W|$.
In case (II), i.e. when $|\lambda_N^W| \leq |\lambda_2^W|$, we have the spectral gap $\Delta(W)=1 - |\lambda_2^W|$ whereas the spectral gap of the shifted matrix satisfies $\Delta(W_{1})=\Delta(\frac{I+W}{2}) =\frac{1 - |\lambda_2(W)|}{2} = \frac{\Delta(W)}{2}$. In this case, the spectral gap becomes worse, but only by a factor of $2$. 
To summarize, shifting the $W$ matrix to $W_{2}=\frac{1}{2}I+\frac{1}{2}$ so that Assumption~\ref{assump_3} can be satisfied
might lead to improved convergence results in some cases, and in some cases it can make the convergence bounds looser; but this looseness in the spectral gap is at most by a constant factor of $2$, and the last error term for D-ASG in Corollary~\ref{cor:rate:dasg} has the order of $O(\frac{\alpha^2}{1-\gamma^2}) =\mathcal{O}( \frac{\alpha^2}{\Delta^2(W)})$. Therefore, this term can become worse only by a constant factor of $4$. This shows that Assumption~\ref{assump_3} is not very restrictive in terms of iteration complexity results as it can hold for any graph topology and for a wide class of choices of $W$.
\end{remark}
}
{\color{black}With slight abuse of notation, we let
\begin{equation}\label{eqn:V:S:alpha:bar}
V_{\bar{S},\alpha}(\bar{\xi}) := \bar{\xi}^T  \bar{S}_{\alpha} \bar{\xi}+ f(\bar{T}\bar{\xi} + x_{\ast}) - f(x_{\ast}),
\end{equation}
where $\bar{S}_{\alpha}=\tilde{S}_{\alpha}\otimes I_{d}$ and $\bar{T}=[1\,\, 0]\otimes I_{d}$. Using this Lyapunov function, the next result establishes a performance bound for the node averages $\bar{x}^{(k)}$. We see that the variance term of our bound is proportional to $\frac{\sigma^2}{N}$ due to the averaging effect and is decreasing with $N$.}
\begin{proposition}\label{prop:dsg-average-optimal-rate-DASG}
{\color{black} Consider the node averages $\bar{x}^{(k)}$ for the D-ASG algorithm
with $0<\alpha\leq\min\left\{\frac{1}{L+\mu},\frac{\lambda_N^W}{L}\right\}$ and $\beta = \frac{ 1-\sqrt{\alpha\mu}}{1+\sqrt{\alpha\mu}}$ and the initialization $x^{(0)}=x^{(-1)}=0$.
For any $k$, we have
\begin{align*}
&\mathbb{E}\left\Vert\bar{x}^{(k)}-x_{\ast}\right\Vert^{2}
\\
&\leq
\left(1-\frac{\sqrt{\alpha\mu}}{2}\right)^{k}\frac{2V_{\bar{S},\alpha}\left(\bar{\xi}_{0}\right)}{\mu}
+\frac{8}{\gamma^{2}\mu\sqrt{\mu}}\sqrt{\alpha}H_{1}H_{3}\frac{\gamma^{2k}-(1-\sqrt{\alpha\mu}/2)^{k}}{\gamma^{2}-(1-\sqrt{\alpha\mu}/2)}
+\frac{2}{\mu^{2}}\alpha H_{1}H_{2}
\\
&\qquad\qquad\qquad
+\frac{2\sigma^2\sqrt{\alpha}}{\mu\sqrt{\mu}N}\left(1+\frac{\sqrt{\alpha\mu}}{2}\right)\left(1+\alpha L \right),
\end{align*}
and for every $i=1,2,\ldots,N$ and any $k$,
\begin{align*}
&\mathbb{E}\left\Vert x_{i}^{(k)}-x_{\ast}\right\Vert^{2}
\\
&\leq
\left(1-\frac{\sqrt{\alpha\mu}}{2}\right)^{k}\frac{4V_{\bar{S},\alpha}\left(\bar{\xi}_{0}\right)}{\mu}
+\frac{16}{\gamma^{2}\mu\sqrt{\mu}}\sqrt{\alpha}H_{1}H_{3}\frac{\gamma^{2k}-(1-\sqrt{\alpha\mu}/2)^{k}}{\gamma^{2}-(1-\sqrt{\alpha\mu}/2)}+\frac{4}{\mu^{2}}\alpha H_{1}H_{2}
\\
&\qquad
+16\gamma^{2k}\left(4\frac{ V_{S,\alpha}\left(\xi_{0}\right)}{\mu}
+ \frac{2\sigma^2 N \sqrt{\alpha}}{\mu\sqrt{\mu}}\left(2-\lambda_N^W +\alpha L \right) +  \frac{2C_{1}^{2}N\alpha^2}{(1-\gamma)^2}
+\Vert x^{\ast}\Vert^{2}\right)
\\
&\qquad\qquad
+\frac{8D_{y}^{2}\alpha^{2}}{(1-\gamma)^{2}}
+\frac{8\sigma^{2}N\alpha^{2}}{(1-\gamma)^{2}}
+\frac{16C_{0}\alpha}{(1-\gamma)^{2}}
+\frac{4\sigma^2\sqrt{\alpha}}{\mu\sqrt{\mu}N}\left(1+\frac{\sqrt{\alpha\mu}}{2}\right)\left(1+\alpha L \right),
\end{align*}
where $C_{0}$ is a positive constant\footnote{\color{black}An exact expression for the constant $C_0$ can be obtained from our proof technique. However, for the simplicity of the presentation, we did not specify the constant $C_0$ explicitly.} and
\begin{align}\label{D:y:eqn}
D_{y}^{2}
&:=4L^{2}\left((1+\beta)^{2}+\beta^{2}\right)
\left(4\frac{ V_{S,\alpha}\left(\xi_{0}\right)}{\mu}
+ \frac{2\sigma^2 N \sqrt{\alpha}}{\mu\sqrt{\mu}}\left(2-\lambda_N^W +\alpha L \right) +  \frac{2C_{1}^{2}N\alpha^2}{(1-\gamma)^2}\right)
\nonumber
\\
&\qquad\qquad
+2\left\Vert\nabla F\left(x^{\ast}\right)\right\Vert^{2},
\end{align}
and
\begin{align*}
&H_{1}:=8\left(1 + \frac{\mu\alpha}{2}-\sqrt{\alpha\mu}\right)+\frac{2L^{2}\alpha}{\mu} + \left(L\alpha+2 + \mu\alpha\right)\sqrt{\alpha\mu}-2\mu\alpha,
\\
&H_{2}:=\frac{2}{N}L^{2}
\left((1+\beta)^{2}+\beta^{2}\right)\left(\frac{4D_{y}^{2}\alpha}{(1-\gamma)^{2}}
+\frac{4\sigma^{2}N\alpha}{(1-\gamma)^{2}}+\frac{8C_{0}}{(1-\gamma)^{2}}\right),
\\
&H_{3}:=\frac{2}{N}L^{2}
\left((1+\beta)^{2}+\beta^{2}\right)
\left(4\frac{ V_{S,\alpha}\left(\xi_{0}\right)}{\mu}
+ \frac{2\sigma^2 N \sqrt{\alpha}}{\mu\sqrt{\mu}}\left(2-\lambda_N^W +\alpha L \right) +  \frac{2C_{1}^{2}N\alpha^2}{(1-\gamma)^2}
+\Vert x^{\ast}\Vert^{2}\right).
\end{align*}
}
\end{proposition}
\begin{remark}[Convergence rate of the averaged D-ASG iterates]
{\color{black}For a {given~iteration} budget $K>0$, if we take $\alpha =\frac{1}{\mu}\left( \frac{4\log(K)}{K}\right)^2$ in the setting of Proposition \ref{prop:dsg-average-optimal-rate-DASG}, then we have $(1-\frac{\sqrt{\alpha \mu}}{2})^{K} = \mathcal{O}(1/K^2)$ as well as  $\frac{\gamma^{2k}-(1-\sqrt{\alpha\mu}/2)^{K}}{\gamma^{2}-(1-\sqrt{\alpha\mu}/2)} = \mathcal{O}(1/K^2)$. Consequently, we obtain $\mathbb{E}\left\Vert\bar{x}^{(K)}-x_{\ast}\right\Vert^{2} = \tilde{\mathcal{O}}(\frac{1}{NK}  + \frac{1}{K^2}  )$ where $\tilde{\mathcal{O}}(\cdot)$ hides a logarithmic factor in $K$. }
\end{remark}


\textbf{Constants in Theorem~\ref{thm-rate-dasg}.} $\lambda_N^W$ and $\gamma$ can typically be estimated with a distributed algorithm; for instance when $W = I-L$ (see e.g. \cite{tran2014distributed}).  
For regularized problems of the form $f_i(x)=\tilde f_i(x) + \frac{\lambda}{2}\|x\|^2$ with $\tilde f_i$ convex, the parameter $\mu$ of strong convexity can be taken as the regularization parameter $\lambda$ and therefore is known. The Lipschitz constant $L$ can be estimated with a line search similar to \cite{beck2009fast,schmidt2015non}. The constant $C_1$ depends on $L,\mu$ and $\sigma$ explicitly. 

\rev{We note that if we possess a lower bound \underbar{$\mu$} on the strong convexity parameter $\mu$ and an upper bound $\bar{L}$ on the strong convexity constant, our results in Theorem~\ref{thm:DGD:explicit} and Theorem~\ref{thm-rate-dasg} will hold if replace \underbar{$\mu$} with $\mu$ and $\bar{L}$ with $L$. If a lower bound on the strong convexity constant cannot be estimated and if the strong convexity constant is instead over-estimated, it is known that this can lead to slower convergence, even for (centralized) SG and ASG. For example, if the strong convexity constant is overestimated by a factor of $c>1$, i.e. the estimated constant is $\bar{\mu} = c\mu$ where $\mu$ is the actual strong convexity constant; convergence rate of (centralized) SG on some quadratic examples can be as slow as $\mathcal{O}(\frac{1}{k^{1/c}})$ (compared to the $\mathcal{O}(\frac{1}{k})$ rate that can be achieved if the strong convexity constant can be accurately estimated) (see e.g. \cite{nemirovski2009robust}). Our bounds reflect a similar behavior. For example, for D-SG, with perfect knowledge of the strong convexity constant, for a given iteration budget $K$, we can choose the stepsize $\alpha = \frac{\log(K)}{2\mu K}$ and our Corollary~\ref{DSG_cor} will lead to the bound $\mathbb{E}\left[\left\Vert x^{(k)}-x^{\infty}\right\Vert^{2}\right] = \tilde{\mathcal{O}}(1/K)$ where $\tilde{\mathcal{O}}(\cdot)$ hides some logarithmic factors in $K$. If we were to overestimate the strong convexity constant by a factor of $c$, the same stepsize choice will lead to a slower convergence rate of
$\mathbb{E}\left[\left\Vert x^{(k)}-x^{\infty}\right\Vert^{2}\right]  = {\mathcal{O}}(1/K^{1/c})$. Similar observations also hold for D-ASG. That being said, it is worth noting that, even for the deterministic and centralized case,
\cite{arjevani2016iteration} 
have shown that for a wide class of algorithms including accelerated gradient methods, it is not possible to obtain accelerated rates, i.e. bounds of the form $L \|x_0 - x_*\|^2 \exp(\mathcal{O}(1)\frac{k}{\sqrt{\kappa}})$ after $k$ iterations where $\kappa=L/\mu$ is the condition number, without having a good estimate (lower bound) of the strong convexity parameter. Therefore, it is somehow expected that to get the accelerated convergence rates, one needs to have some information about the problem constants such as $\mu$ and $L$. }

\rev{We also note that in practice, for regularized problems such as $L_2$ regularized logistic regression or ridge regression, the regularizer $\frac{\lambda}{2}\|x\|^2$ provides a lower bound on $\mu$ directly (where we can simply take $\mu=\lambda$). If $L$ and $\mu$ are known approximately, the stepsize can be set to $\alpha \in (0, (1+\lambda_N^W)/L)$ for D-SG (Theorem~\ref{thm:DGD:explicit}) 
and the stepsize can be set to $\alpha\in(0,\lambda_N^W/L]$ and the momentum parameter can be set to $\beta = \frac{1-\sqrt{\alpha\mu}}{1+\sqrt{\alpha\mu}}$ for D-ASG (Theorem~\ref{thm-rate-dasg}) as an initial guess and can be further tuned to the dataset.}


\textbf{Robustness of D-SG vs D-ASG.} 
We derived in Theorem~\ref{thm:DGD:explicit}
that for D-SG, for small stepsize $\alpha$, the rate of convergence is $1-\alpha\mu$ while $J_{\infty}(\alpha)\leq\frac{\alpha}{\mu(2-\alpha\mu)}$, and 
in Theorem~\ref{thm-rate-dasg} that
for D-ASG, for small stepsize $\alpha$ and $\beta=\frac{1-\sqrt{\alpha\mu}}{1+\sqrt{\alpha\mu}}$,
the rate of convergence is $1-\sqrt{\alpha\mu}$, 
while $J_{\infty}(\alpha)\leq\frac{\sqrt{\alpha}}{\mu\sqrt{\mu}}(2-\lambda_{N}^{W}+\alpha L) = \bigO\big(\frac{\sqrt{\alpha}}{\mu\sqrt{\mu}}\big)$. Hence, for a fixed $\alpha$, D-ASG converges faster than D-SG,
but is less robust and more sensitive to noise for the same stepsize that is small enough,
and this suggests that there is a trade-off between convergence rate
and robustness. Next, we discuss how one can trade between convergence rate and robustness in a more systematic manner.

\textbf{Trading off convergence rate with the robustness and the network term.} Equation \eqref{ineq-to-prove-2} shows that large stepsize leads to faster rate $1-\sqrt{\alpha\mu}$, but the variance term (that is proportional to robustness $J_\infty$) and the network term in our bounds get larger. Consider minimizing the sum of variance and network terms there, subject to a constraint on the rate:
\begin{align}\label{eqn:trade}
&\min J_{tot}(\alpha):=
 \frac{2\sigma^2 N \alpha}{\mu\sqrt{\alpha\mu}}\left(2-\lambda_N^W +\alpha L \right)+\frac{2C_{1}^{2}N\alpha^2}{(1-\gamma)^2},
\\
&\mbox{subject to} \quad 0\leq \alpha \leq \bar\alpha, \quad 1-\sqrt{\alpha\mu}\leq \rho_{*}(1+\delta),
 \nonumber
\end{align}
where $\bar{\alpha}:=\min\big(\frac{\lambda_N^W}{L},\frac{1}{L+\mu}\big)$ and $\rho_* := 1 - \sqrt{\bar\alpha \mu}$ is the best rate we can certify with \eqref{ineq-to-prove-2} and $\delta \in [0, \frac{1}{\rho_*}-1]$ is the percentage of the best achievable rate we would like to trade with robustness and network effects. The constraints specify an interval for the stepsize to lie in, and the objective $J_{tot}$ can be optimized in this interval explicitly by calculating the first-order conditions. 
By letting $z:=\sqrt{\alpha}$, it can be checked that the optimization problem \eqref{eqn:trade} is equivalent to 
\begin{align*}
&\min_{z\geq 0}
G(z):=\frac{2\sigma^{2}N}{\mu\sqrt{\mu}}z\left(2-\lambda_{N}^{W}+z^{2}L\right)
+\frac{2C_{1}^{2}}{(1-\gamma)^{2}}z^{4},
\\
&\text{subject to}
\quad\sqrt{\bar\alpha} \geq z\geq\frac{1-\rho_{\ast}(1+\delta)}{\sqrt{\mu}}.
\end{align*}
We also have
\begin{align*}
G'(z)&=\frac{2\sigma^{2}N}{\mu\sqrt{\mu}}\left(2-\lambda_{N}^{W}\right)
+\frac{6\sigma^{2}N}{\mu\sqrt{\mu}}L z^{2}
+\frac{8C_{1}^{2}}{(1-\gamma)^{2}}z^{3} >0,
\end{align*}
for any $z> 0$ and hence $G(z)$ is strictly increasing. Therefore, the solution of the minimization problem is
$z^{\ast}=\frac{1-\rho_{\ast}(1+\delta)}{\sqrt{\mu}}$,
and the optimal stepsize is $\alpha^{\ast}=\frac{(1-\rho_{\ast}(1+\delta))^{2}}{\mu}$.
This choice of stepsize will lead to the tightest performance bounds in our analysis for the same rate and provides some guidance about how the stepsize can be chosen.

\subsection{Quadratic Objectives}
Our study so far has been focused on strongly convex objectives.
In the Appendix, we analyze the special case of strongly convex \textit{quadratic} objectives 
when $f_i$ is quadratic at every node $i$. 
Note that, in this case $F(x)$ is also quadratic. 
We obtain tight results in terms
of rate and robustness that improve 
upon current results. 
In particular, we obtain the same convergence rate $\rho{(\alpha)}^{2}=(1-\alpha\mu)^{2}$ for D-SG method
but better convergence rate $\rho_{\text{dasg}}^{2}=(1-\sqrt{\alpha\mu})^{2}$ 
for D-ASG method (instead of $\rho_{\text{dasg}}^{2}=1-\sqrt{\alpha\mu}$
for the strongly-convex setting). 
We also obtain explicit formulas for the robustness
measure $J_{\infty}$ for quadratic objectives
for both D-SG and D-ASG (instead of upper bounds
for the strongly-convex setting) under an additional assumption on the structure of the noise \rev{as well as explicit bounds on the asymptotic variance of the components of the node average vector $\bar{x}^{(k)}$.}


\section{An Exact Multistage Distributed Method}\label{sec:multi}

In the previous sections, we mainly focused on the D-SG and D-ASG methods with \textit{constant} step size and momentum parameters. For these algorithms, we studied the problem of tuning their parameters so that the iterates converge to a neighborhood of $x_*$ that depends on the stepsize $\alpha$. In this section, however, our focus is to design a distributed \emph{exact algorithm} that uses time-varying stepsize and momentum parameters and converges to the optimum $x_*$ when the number of iterations grows.
\begin{figure}[t]
\centering
\includegraphics[width=0.92\textwidth]{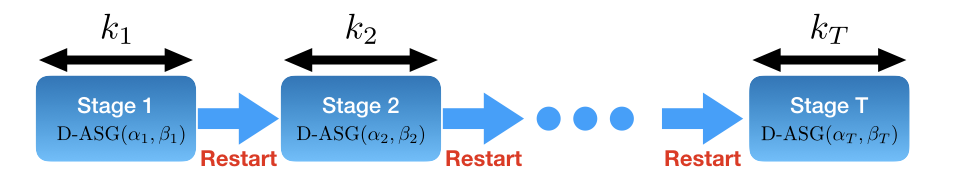}
\caption{\label{Fig_DMASG} The scheme of the Distributed Multistage ASG (D-MASG) method}
\end{figure}

We propose the Distributed Multistage ASG (D-MASG) method which is a distributed version of M-ASG proposed in \cite{aybat2019universally}. As illustrated in Figure~\ref{Fig_DMASG}, D-MASG consists of $T$ stages where at each stage $t\in \{1,2,\dots,T\},$ we run D-ASG with parameters $\alpha_t$ and $\beta_t = \frac{1-\sqrt{\mu \alpha_t}}{1+\sqrt{\mu \alpha_t}}$ for $n_t$ iterations where $\alpha_t$ and $n_t$ will be chosen in a particular way. 
These stages are stitched together using a \textit{momentum restart} technique which means that the first two iterates of every stage are equal to the last iterate of the previous stage. The details of D-MASG are provided in Algorithm~\ref{Algorithm1} where the iterate $x^{t,m}$ denotes the $m$-th iterate of the $t$-th stage.

For any $t \leq T$, let $L_t$ denote the total number of iterations up to the end of stage $t$, i.e,
\begin{align}
L_t := \sum_{i=1}^t k_i,
\end{align}
\noindent with the convention that $L_{0}:=0$. Let $x^{(k)}$ be the sequence that records all the inner and outer iterations of the D-MASG algorithm, obtained by concatenating the sequences $\{x^{(t,m)}\}_{m=1}^{k_t}$ for all stages $t$ and inner iterates indexed by $m$. In other words, $k$ is the counter for the total number of stochastic gradient evaluations and for $L_{t-1} < k \leq L_t$, we have
\begin{equation}
x^{(k)} = x^{(t,k-L_{t-1})}.
\end{equation}
\begin{algorithm}
    \SetKwInOut{Input}{Input}
    \SetKwInOut{Output}{Output}
    \Input{Initial iterate $x^{(0)}$, The sequence $\{\alpha_i\}_{i=1}^T$ of stepsizes, The sequence $\{k_i\}_{i=1}^T$ of length of stages.}
    Set $x^{(0,k_0)} = x^{(0)}$;\\
    \For{$t = 1;\ t \leq T;\ t = t + 1$}{
    Set $x^{(t,-1)} = x^{(t,0)} = x^{(t-1,k_{t-1})}$;\\
    \For{$m = 0;\ m \leq k_t-1;\ m = m + 1$}{
    Set $\beta_t = \frac{1- \sqrt{\mu \alpha_t}}{1+ \sqrt{\mu \alpha_t}}$;\\
    Set $y^{(t,m)} = (1+\beta_t)x^{(t,m)} - \beta_t x^{(t,m-1)}$;\\
    Set $x^{(t,m+1)} = \mathcal{W} y^{(t,m)} - \alpha_t \tilde{\nabla} f\left(y^{(t,m)}\right)$
    } 
    }
    \caption{\label{Algorithm1} Distributed Multistage Accelerated Stochastic Gradient Algorithm (D-MASG)}
\end{algorithm}
To characterize the convergence rate of D-MASG, we first analyze the evolution of iterates over one single stage. To simplify our presentation, we define the \emph{scaled condition number} as 
\begin{equation}\label{tilde_kappa}
\tilde{\kappa}:= \frac{L+\mu}{\mu \lambda_N^W} = \frac{\kappa + 1}{\lambda_N^W},
\end{equation}
where we assume for the rest of this section that Assumption~\ref{assump_3} holds, i.e. $\lambda_N^W>0$. 

\begin{proposition}\label{thm_one_stage}
Consider running D-ASG with initialization $x^{(-1)}=x^{(0)}{\color{black}=0}$ and parameters $\alpha \in (0,\bar\alpha]$ where $\bar\alpha =\min\left\{\frac{\lambda_N^W}{L}, \frac{1}{L+\mu}\right\}$ and $\beta = \frac{ 1-\sqrt{\alpha\mu}}{1+\sqrt{\alpha\mu}}$. Then, for any $k\geq 0$,
\begin{align*}
\E\left[ \left\| x^{(k)} -x^* \right\|^2\right]
\leq
 4 \exp(-k \sqrt{\alpha \mu}) \left\|x^{(0)} -x^*\right\|^2
+ 6N \left ( \frac{\sqrt{\alpha}}{\mu\sqrt{\mu}} \sigma^2 +  \frac{C_{1}^{2}\alpha^2}{(1-\gamma)^2} \right),
\end{align*}
where {\color{black}$\gamma,C_{1}$ are given in \eqref{eqn:asymp_suboptim}-\eqref{defn:C:1}.}
\end{proposition}

D-MASG with one stage is equivalent to running the D-ASG algorithm. Based on the previous result, we immediately obtain the following corollary which provides performance bounds for one-stage D-MASG.  
\begin{corollary}\label{cor_one_stage}
Given $k$, consider running D-MASG for one stage with $k_1=k$ and $\alpha_1 = \frac{\lambda_N^W}{L+\mu} \left ({p \sqrt{\tilde{\kappa}} \log k}/{k}\right )^2$ for some $p \geq 1$, 
where $\tilde{\kappa}$ is given in \eqref{tilde_kappa}.
Then, for any 
$$k \geq p \sqrt{\tilde{{\kappa}}} \max\left\{2 \log\left(p \sqrt{\tilde{{\kappa}}}\right), e \right\} ,$$ 
we have
\begin{align*}
\E\left[ \left\| x^{(k)} -x^* \right\|^2\right]
\leq
 \frac{4}{k^p} \left\|x^{(0)} -x^*\right\|^2
+  \frac{6Np \log k}{\mu^2 k} \left (\sigma^2 +  \frac{C_{1}^{2}(p \log k)^3}{(1-\gamma)^2 k^3} \right),
\end{align*}
where {\color{black}$\gamma,C_{1}$ are given in \eqref{eqn:asymp_suboptim}-\eqref{defn:C:1}}.
\end{corollary}

In the next {\color{black}proposition}, we propose a particular way to choose the stepsize $\alpha_t$ and the stage length $k_t$ for every stage $t\in [1,T]$ and obtain performance guarantees for the distance to the optimum after $T$ stages. In our proposed approach, the length of stages is geometrically increasing whereas the stepsize of each stage is chosen in a geometrically decaying manner. The length of the first stage $k_1$ can be an arbitrary positive integer and our performance bounds depends on how it is chosen.

\begin{proposition}\label{main_DMASG_result}
Consider running D-MASG with the following parameters:
\begin{align*}
k_1 \geq 1, \quad \alpha_1 = \frac{\lambda_N^W}{L+\mu}, \quad
k_t = 2^{t} \left\lceil p\sqrt{\tilde{\kappa}}\log(2)\right\rceil, \quad \alpha_t = \frac{\lambda_N^W}{2^{2t}(L+\mu)},
\end{align*}
with $p \geq 7 $. 
Then, for any $t\geq 0$:
\begin{align*}
\E\left[ \left\| x^{(L_{t+1})} -x^* \right\|^2\right]
\leq \frac{4}{2^{(p-2)t}} \exp\left(-\frac{k_1}{\sqrt{\tilde{\kappa}}}\right) \left\| x^{(0)} - x^* \right\|^2 + \frac{12N  \sigma^2}{2^t \mu^2 \sqrt{\tilde{\kappa}}} + \frac{12N}{2^{4t}} \left (\frac{C_1 \lambda_N^W}{L(1-\gamma)} \right )^{2}, 
\end{align*}
where {\color{black}$\gamma,C_{1}$ are given in \eqref{eqn:asymp_suboptim}-\eqref{defn:C:1}} and $\tilde{\kappa}$ is given in \eqref{tilde_kappa}.
\end{proposition}

The previous result gives performance bounds for last iterate of every stage.
Using this result, we can also derive upper bounds for the error after $k$ iterations as follows.
\begin{proposition}\label{D-MASG_any_n}
Consider running D-MASG with the parameters given in Proposition~\ref{main_DMASG_result}.
Then, for any $k > k_1$:
\begin{align*}
&\E\left[ \left\| x^{(k)} -x^* \right\|^2\right] 
\\
&\leq \bigO(1) \left ( \left ( \frac{6p \sqrt{\tilde{\kappa}}}{k-k_1} \right )^{p-2} \exp\left(-\frac{k_1}{\sqrt{\tilde{\kappa}}}\right) \left\| x^{(0)} - x^* \right\|^2  + \frac{N p \sigma^2}{\mu^2(k-k_1)}  + \frac{N p^4 C_1^{2}(1-\gamma)^{-2}}{\mu^2 (k-k_1)^4} \right ),
\end{align*}
where {\color{black}$\gamma,C_{1}$ are given in \eqref{eqn:asymp_suboptim}-\eqref{defn:C:1}} and $\tilde{\kappa}$ is given in \eqref{tilde_kappa}.
\end{proposition}

Note that Proposition~\ref{D-MASG_any_n} provides us with a degree of freedom in choosing $k_1$. In the following corollary we characterize two special cases. We omit the proof as it is a straightforward consequence of Proposition~\ref{D-MASG_any_n}.
\begin{corollary}\label{cor_D-MASG}
Consider running D-MASG with the parameters given in Proposition~\ref{main_DMASG_result}. In particular, by choosing $k_1 = \ceil{(p-2) \log(6 p \tilde{\kappa})\sqrt{\tilde{\kappa}}}$, we have
\begin{align*}
\E \left[ \left\| x^{(k)} -x^* \right\|^2\right] \leq \bigO(1) \left ( \frac{1}{k^{p-2}} \left\| x^{(0)} - x^* \right\|^2 + \frac{N p \sigma^2}{\mu^2 k}  + \frac{N p^4 C_1^2( 1-\gamma)^{-2}}{\mu^2 k^4} \right ),
\end{align*}
for any $k \geq 2 k_1$.
Also, for a given number of iterations, $k$, by choosing $p=7$
and $k_1 = \ceil{\frac{k}{C}}$ for some constant $C \geq 2$, we have
\begin{align*}
\E \left[ \left\| x^{(k)} -x^* \right\|^2\right] \leq \bigO(1) \left ( \exp\left(-\frac{k}{C\sqrt{\tilde{\kappa}}}\right)\left\| x^{(0)} - x^* \right\|^2 + \frac{N  \sigma^2}{\mu^2 k}  + \frac{N C_1^2( 1-\gamma)^{-2}}{\mu^2 k^4} \right ),
\end{align*}
for any $k \geq 2 \sqrt{\tilde{\kappa}}$, {\color{black}where $\gamma,C_{1}$ are given in \eqref{eqn:asymp_suboptim}-\eqref{defn:C:1} and $\tilde{\kappa}$ is given in \eqref{tilde_kappa}}.
\end{corollary}
{\color{black}Note that our results also provide bounds on the number of iterations required to find
an $\epsilon$-solution, i.e. a point 
$x^\epsilon$ that satisfies
$\E \left[ \left\| x^\epsilon -x^* \right\|^2\right] \leq \epsilon$ 
for a given $\epsilon > 0$.} This is obtained in the next corollary. 
We omit the proof as it follows directly from the previous corollary; by bounding bias, variance, and network effect terms, each by $\epsilon/3$.
\begin{corollary}\label{cor:final}
{\color{black}Let $\epsilon >0$ be an arbitrary positive number. Consider running D-MASG with the parameters given in Proposition~\ref{main_DMASG_result}. Assume choosing $p=7$ and $k_1 = \ceil{\sqrt{\tkappa}\log \left ( \frac{\Delta}{ \epsilon}\right)}$ where $\Delta$ is the optimality gap, an upper bound on the initial error, i.e., $\Delta \geq \left\| x^{(0)} - x^* \right\|^2$. 
Then, D-MASG leads to an $\epsilon$-close solution $x^\epsilon$ after at most
\begin{equation}
\bigO(1) \left ( \sqrt{\tkappa} \log \left (\frac{\Delta}{ \epsilon} \right ) + \frac{N\sigma^2}{\mu^2 \epsilon} + \frac{N^{1/4}\sqrt{C_1 (1-\gamma)^{-1}}}{\sqrt{\mu} \sqrt[4]{\epsilon}}\right)    
\end{equation}
iterations, where $\gamma,C_{1}$ are given in \eqref{eqn:asymp_suboptim}-\eqref{defn:C:1} and $\tilde{\kappa}$ is given in \eqref{tilde_kappa}}.
\end{corollary}

{\color{black}
Previously, we obtained optimal convergence results
for the average iterates and individual iterates 
(Proposition~\ref{prop:dsg-average-optimal-rate-DASG}).
Similar to Corollary~\ref{cor_D-MASG}, we have the following result.
We omit the proof as it follows directly from the previous corollary; by bounding bias, variance, and network effect terms, each by $\epsilon/3$.

\begin{corollary}\label{cor_D-MASG:bar}
Consider running D-MASG with the parameters given in Proposition~\ref{main_DMASG_result}. In particular, by choosing $k_1 = \ceil{(p-2) \log(6 p \tilde{\kappa})\sqrt{\tilde{\kappa}}}$, we have
\begin{align*}
&\E \left[ \left\| \bar{x}^{(k)} -x_* \right\|^2\right] \leq \bigO(1) \left ( \frac{1}{k^{p-2}} \frac{\left\| x^{(0)} - x^* \right\|^2}{N} + \frac{p \sigma^2}{N\mu\sqrt{\mu} k}  + \frac{p^4 C_{0}L^{2}( 1-\gamma)^{-2}}{N\mu^2 k^4} \right ),
\\
&\E \left[ \left\| x_{i}^{(k)} -x_* \right\|^2\right] 
\\
&\leq \bigO(1) \left ( \frac{1}{k^{p-2}} \frac{\left\| x^{(0)} - x^* \right\|^2}{N} + \frac{p \sigma^2}{N\mu\sqrt{\mu} k}  + \left(\frac{L^{2}}{N\mu^{2}}+1\right)\frac{p^4 C_{0}( 1-\gamma)^{-2}}{k^4} \right ),
\end{align*}
for any $k \geq 2 k_1$ and $i=1,2,\ldots,N$.
Also, for a given number of iterations, $k$, by choosing $p=7$
and $k_1 = \ceil{\frac{k}{C}}$ for some constant $C \geq 2$, we have
\begin{align*}
&\E \left[ \left\| \bar{x}^{(k)} -x_* \right\|^2\right] \leq \bigO(1) \left ( \exp\left(-\frac{k}{C\sqrt{\tilde{\kappa}}}\right)\frac{\left\| x^{(0)} - x^* \right\|^2}{N} + \frac{  \sigma^2}{N\mu\sqrt{\mu}k}  + \frac{C_{0}L^{2}( 1-\gamma)^{-2}}{N\mu^2 k^4} \right ),
\\
&\E \left[ \left\| x_{i}^{(k)} -x_* \right\|^2\right] 
\\
&\leq \bigO(1) \left ( \exp\left(-\frac{k}{C\sqrt{\tilde{\kappa}}}\right)\frac{\left\| x^{(0)} - x^* \right\|^2}{N} + \frac{  \sigma^2}{N\mu\sqrt{\mu}k}  + \left(\frac{L^{2}}{N\mu^{2}}+1\right)\frac{C_{0}( 1-\gamma)^{-2}}{k^4} \right ),
\end{align*}
for any $k \geq 2 \sqrt{\tilde{\kappa}}$ and $i=1,2,\ldots,N$, {\color{black}where $\gamma,C_{1}$ are given in \eqref{eqn:asymp_suboptim}-\eqref{defn:C:1} and $\tilde{\kappa}$ is given in \eqref{tilde_kappa}}.
\end{corollary}

Similar to Corollary~\ref{cor:final}, we can also provide bounds on the number of iterations required to find
an $\epsilon$-solution for the average iterates
and an individual iterate, i.e. a point 
$\bar{x}^\epsilon :=\frac{1}{N}\sum_{i=1}^{N}x_{i}^{\epsilon}$ that satisfies
$\E \left[ \left\| \bar{x}^\epsilon -x_* \right\|^2\right] \leq \epsilon$ 
and $\E \left[ \left\| x_i^\epsilon -x_* \right\|^2\right] \leq \epsilon$ 
for a given $\epsilon > 0$. 
We have the following result.}

\begin{corollary}\label{cor:final:2}
{\color{black}Let $\epsilon >0$ be an arbitrary positive number. Consider running D-MASG with the parameters given in Proposition~\ref{main_DMASG_result}. Assume choosing $p=7$ and $k_1 = \ceil{\sqrt{\tkappa}\log \left ( \frac{\Delta}{N \epsilon}\right)}$ where $\Delta$ is the optimality gap, an upper bound on the initial error, i.e., $\Delta \geq \left\| x^{(0)} - x^* \right\|^2$. 
Then, D-MASG leads to an $\epsilon$-close solution $\bar x^\epsilon$ after at most
\begin{equation}
\bigO(1) \left ( \sqrt{\tkappa} \log \left (\frac{\Delta}{N \epsilon} \right ) + \frac{\sigma^2}{\mu\sqrt{\mu}N \epsilon} + \frac{C_{0}^{1/4}\sqrt{L (1-\gamma)^{-1}}}{N^{1/4}\sqrt{\mu} \sqrt[4]{\epsilon}}\right)    
\end{equation}
iterations, and for any $1\leq i\leq k$, D-MASG leads to an $\epsilon$-close solution $x_{i}^\epsilon$ after at most
\begin{equation}
\bigO(1) \left ( \sqrt{\tkappa} \log \left (\frac{\Delta}{N \epsilon} \right ) + \frac{\sigma^2}{\mu\sqrt{\mu}N \epsilon} + \left(\frac{\sqrt{L}}{N^{1/4}\sqrt{\mu}}+1\right)\frac{C_{0}^{1/4}\sqrt{(1-\gamma)^{-1}}}{\sqrt[4]{\epsilon}}\right)    
\end{equation}
iterations, where $C_{0}$ is an explicitly computable constant
such that $C_{0}=\mathcal{O}(1)$ as $\alpha\rightarrow 0$
and $\gamma$ is given in \eqref{eqn:asymp_suboptim} and $\tilde{\kappa}$ is given in \eqref{tilde_kappa}.}
\end{corollary}

\section{Numerical Results}

In this section, we conduct several experiments to validate our theory and assess the performance of D-SG and D-ASG. We consider a (regularized) logistic regression problem which is a common formulation to solve binary classification tasks: 
\begin{align}
\min_{x \in \mathbb{R}^d}\left( \frac1{n} \sum_{i=1}^n \log\left(1 + \exp \left(-y_i X_i^\top x\right)\right) + \lambda \| x\|_2^2\right), \label{eqn:logreg}
\end{align}
where $(X_i,y_i)$ denotes a data pair: $X_i \in \mathbb{R}^d$ is the feature vector and $y_i \in \{-1,1\}$ denotes the label, and $n$ denotes the number of data pairs. 

In all our experiments, we assume that each computation node has access to a subset of all data points, and the noisy gradient in \eqref{eqn:dsg_update} and \eqref{def-dasg-iters} basically becomes the \emph{stochastic gradient} that is computed on a random sub-sample of the data. More precisely, we will assume that at each iteration, each computation node will draw a random sub-sample from the data points that it has access to, and compute the stochastic gradient by using this subsample. The size of the subsample will be determined by a single parameter $b\in(0,1]$, which determines the ratio of the number of elements contained in the subsample to the total number of data points that are accessible to that node. For instance, if all the data points are evenly distributed to the nodes, i.e.\ each node has access to $n/N$ number of distinct data points, the size of the data sub-sample that will be used for computing the stochastic gradients is determined as $(bn)/N$. If $b=1$, the node will use all of its data points to compute the gradient, hence the variance of the gradient noise $\sigma^2$ will vanish. Similarly, a small $b \ll 1$ will result in a large $\sigma^2$. 


In the sequel, we first conduct experiments on a synthetic problem, which provides us a more sterilized environment where we have a direct control on the problem. Then, we conduct experiments on two binary classification datasets, where we implement the proposed algorithms and the competitors in C++ and run them on a real distributed environment. 
{
\color{black}
We will consider five different network architectures: (i) Connected: the network nodes can communicate with all the other nodes in the network, (ii) Star: the network nodes are only allowed to communicate with a central node (iii) Circular: the network notes are only allowed to communicate with their `right' and `left' neighbors, (iv) Grid: the network nodes are allowed to communicate with their upper, lower, left, and right neighbors, and finally (v) Disconnected: the network nodes are not allowed to communicate. These architectures are visualized in Figure~\ref{fig:net_archs}. 
}
We note that we replicate each experiment $5$ times and we report the average results. 
{
  \color{black} Finally, in our last experiments, we monitor the robustness of the proposed algorithms to potential inaccuracies in estimating the problem constants $L$ and $\mu$.
}

\newcommand{\figsize}{0.17}

\begin{figure}[t]
    \centering
    \subfigure[Fully-connected]{
    \includegraphics[width=\figsize\columnwidth]{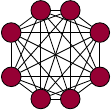} 
    }
    \hfill
    \subfigure[{\color{black}Star}]{
    \includegraphics[width=\figsize\columnwidth]{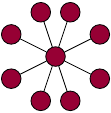} 
    }
    \hfill
    \subfigure[Circular]{
    \includegraphics[width=\figsize\columnwidth]{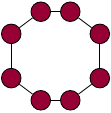} 
    }
    \hfill
    \subfigure[{\color{black}Grid}]{
    \includegraphics[width=\figsize\columnwidth]{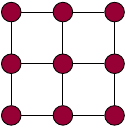} 
    }
    \hfill
    \subfigure[Disconnected]{
    \includegraphics[width=\figsize\columnwidth]{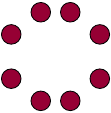} 
    }
    \caption{Illustration of the network architectures.}
    \label{fig:net_archs}
\end{figure} 



\subsection{Synthetic data experiments}

In this section, we present our experiments on a synthetic logistic regression problem, where our main goal is to validate Theorems~\ref{thm:DGD:explicit} and~\ref{thm:general:DASG} on the logistic regression task. In this set of experiments, we first generate synthetic data by simulating the following probabilistic model:
\begin{equation*}
x_0 \sim \mathcal{N}(0, I), 
\qquad
X_i \sim \mathcal{N}\left(0, \sigma_X^2 I\right), 
\qquad
y_i | X_i, x_0 \sim \delta\left( y_i - \mathrm{sign}\left(X_i^\top x_0\right) \right),
\end{equation*}
where $x_0$ denotes the data generating parameter and $\delta$ denotes the Dirac delta function to represent deterministic relations as a degenerate probability model. Once the set of pairs $(X_i, y_i)_{i=1}^n$ are generated, our goal becomes solving an $\ell_2$-regularized logistic regression problem defined in \eqref{eqn:logreg}.
In this set of experiments, we simulate the distributed environment in MATLAB and we provide our implementation in the supplementary material. Unless stated otherwise, we first generate $n=1000$ data points, set the dimension $d=100$, data variance $\sigma_X^2 = 5$, $\lambda = 0.05$, the number of nodes $N=10$, the batch proportion $b=0.1$, and we consider the circular network architecture. 

\begin{figure}[t]
    \centering
    \subfigure[Step-size]{
    \includegraphics[width=0.31\columnwidth]{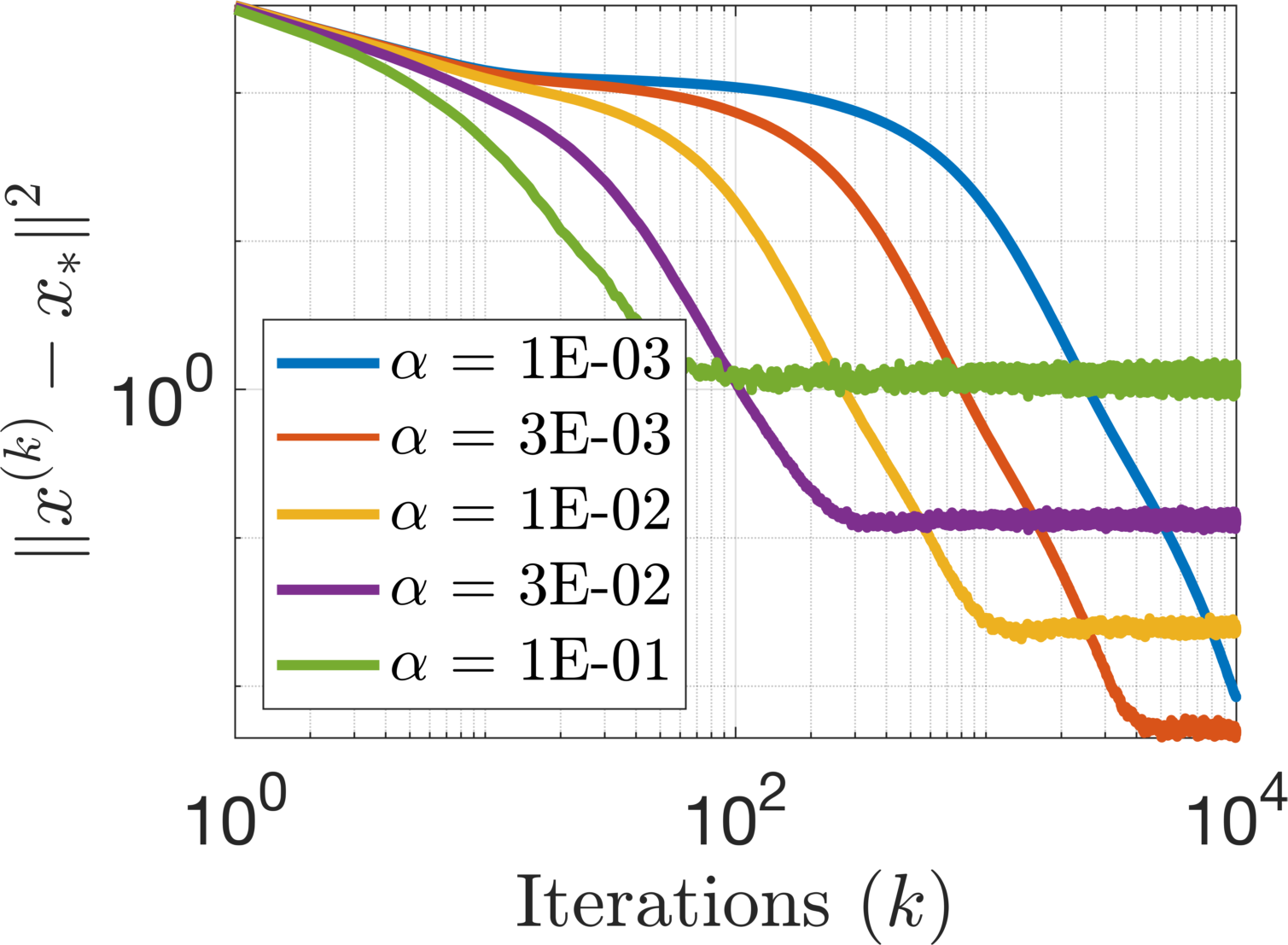} \label{fig:dsg_synth_stepsize}
    }
    \subfigure[{\color{black}Network architecture}]{\includegraphics[width=0.31\columnwidth]{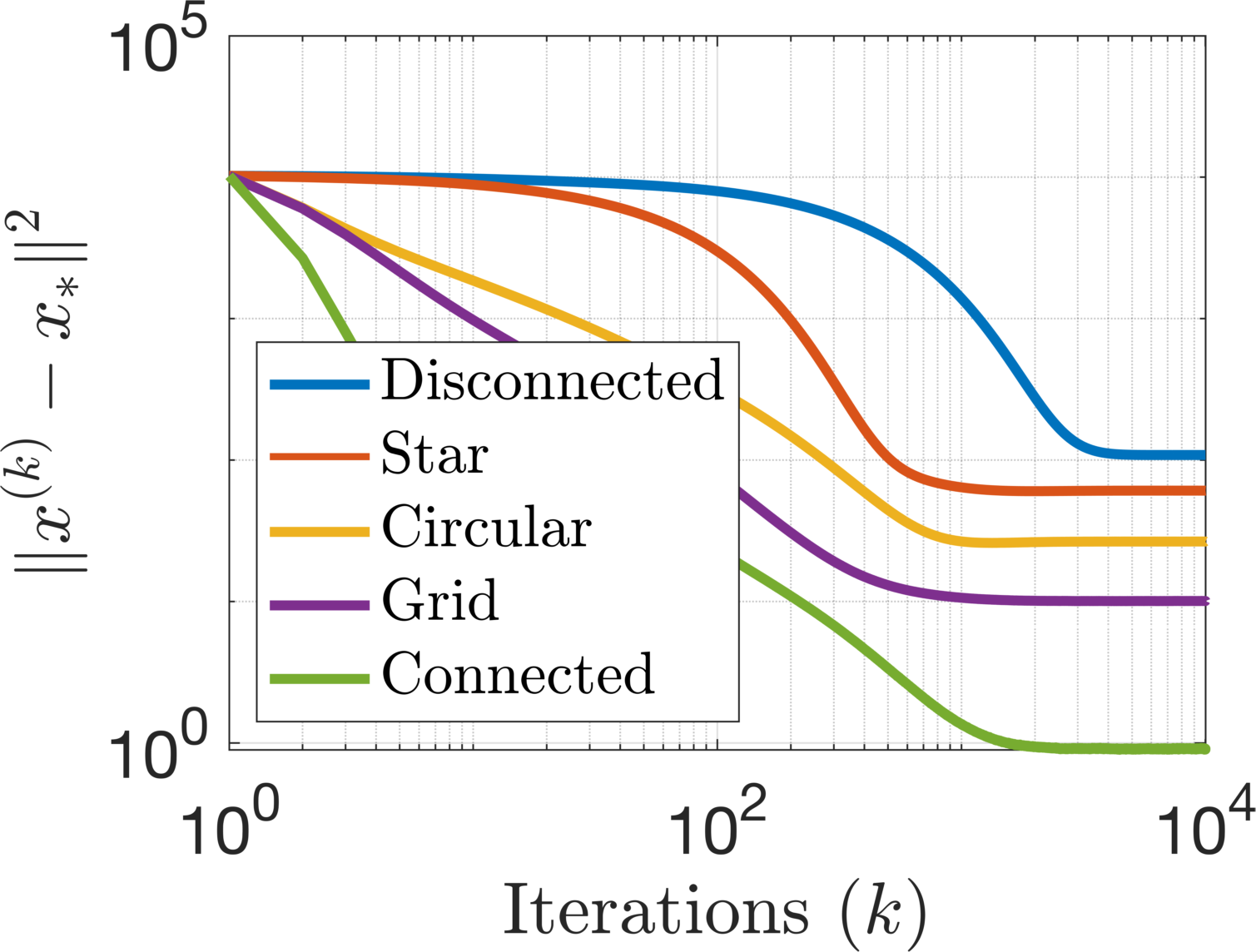} \label{fig:dsg_synth_network}}
    \subfigure[Minibatch size]{\includegraphics[width=0.31\columnwidth]{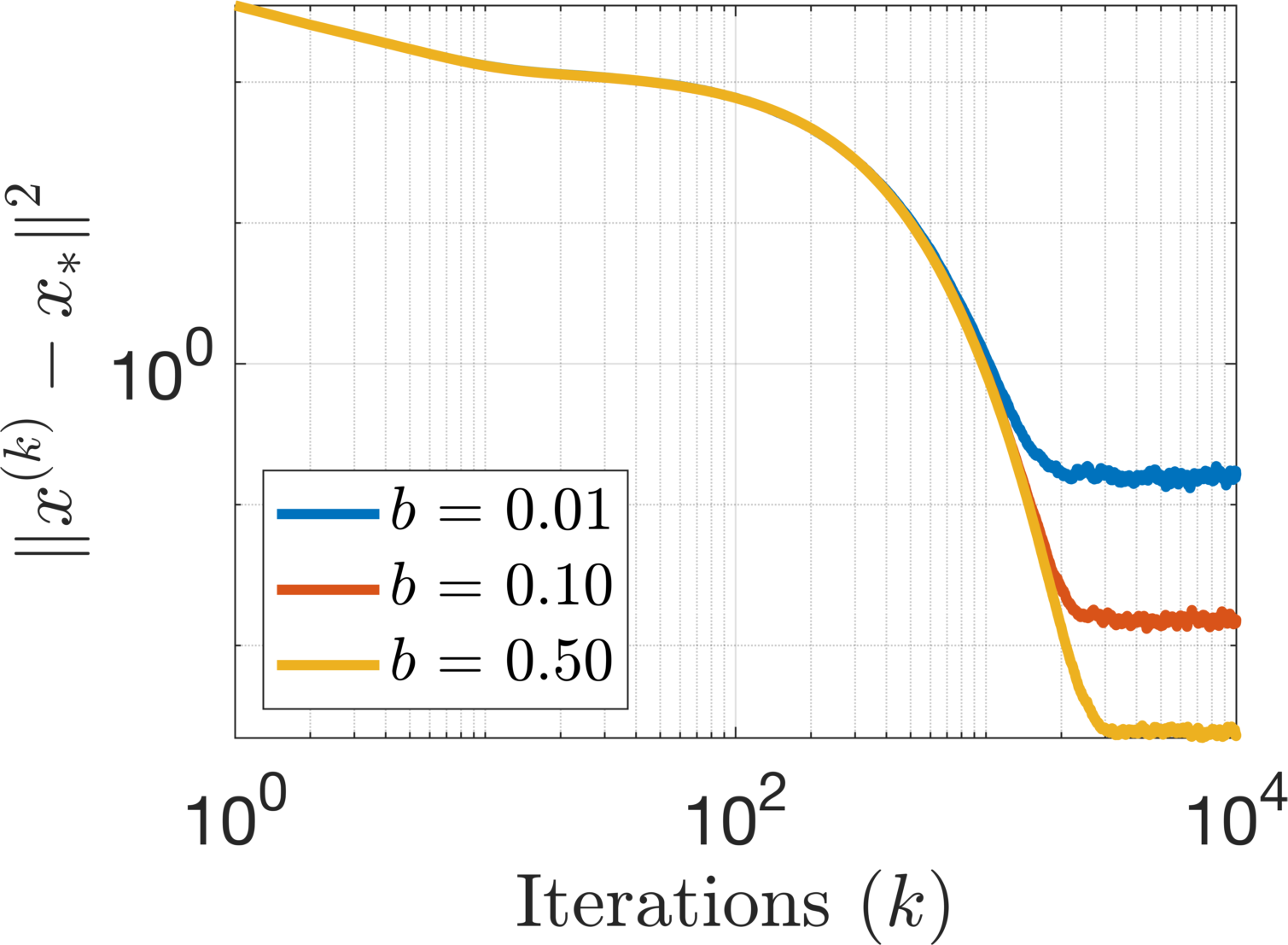} \label{fig:dsg_synth_batchsize}}
    \caption{Synthetic data experiments on D-SG.}
    \label{fig:exp_log_reg_synth_DSG}
\end{figure} 

\begin{figure}[t]
    \centering
    \subfigure[Step-size]{\includegraphics[width=0.31\columnwidth]{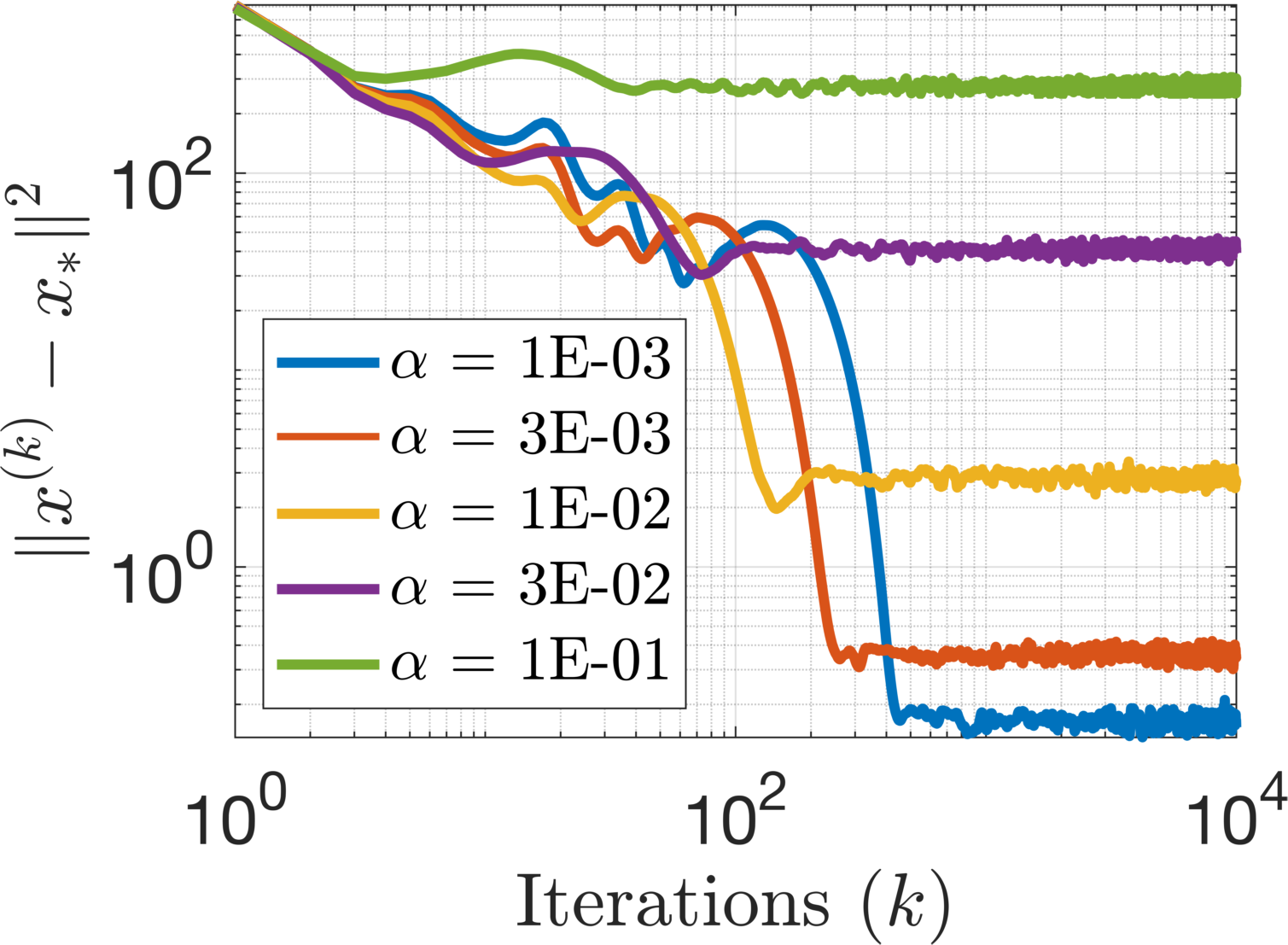} \label{fig:dasg_synth_stepsize}}
    \subfigure[{\color{black}Network architecture}]{\includegraphics[width=0.31\columnwidth]{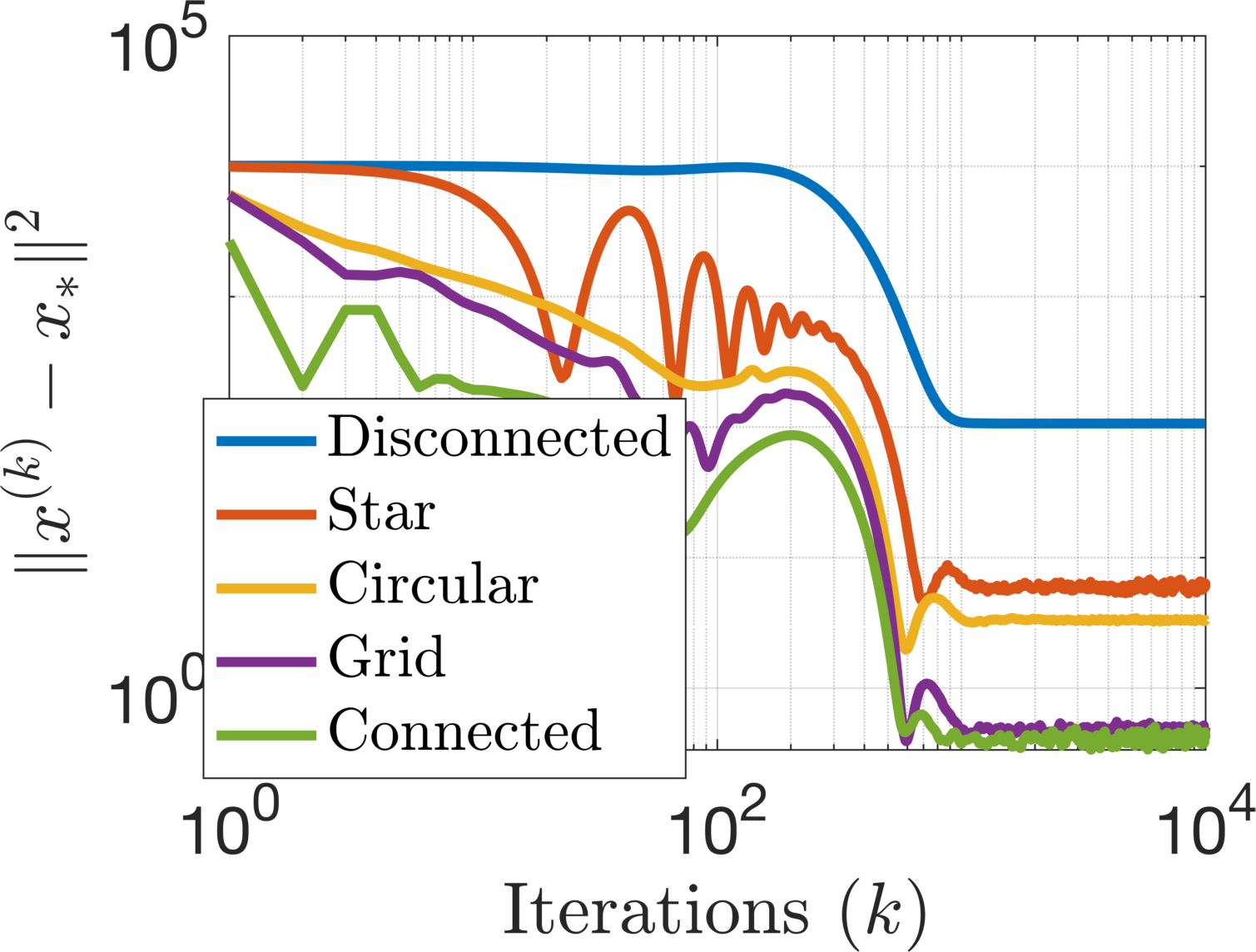} \label{fig:dasg_synth_network}}
    \subfigure[Minibatch size]{\includegraphics[width=0.31\columnwidth]{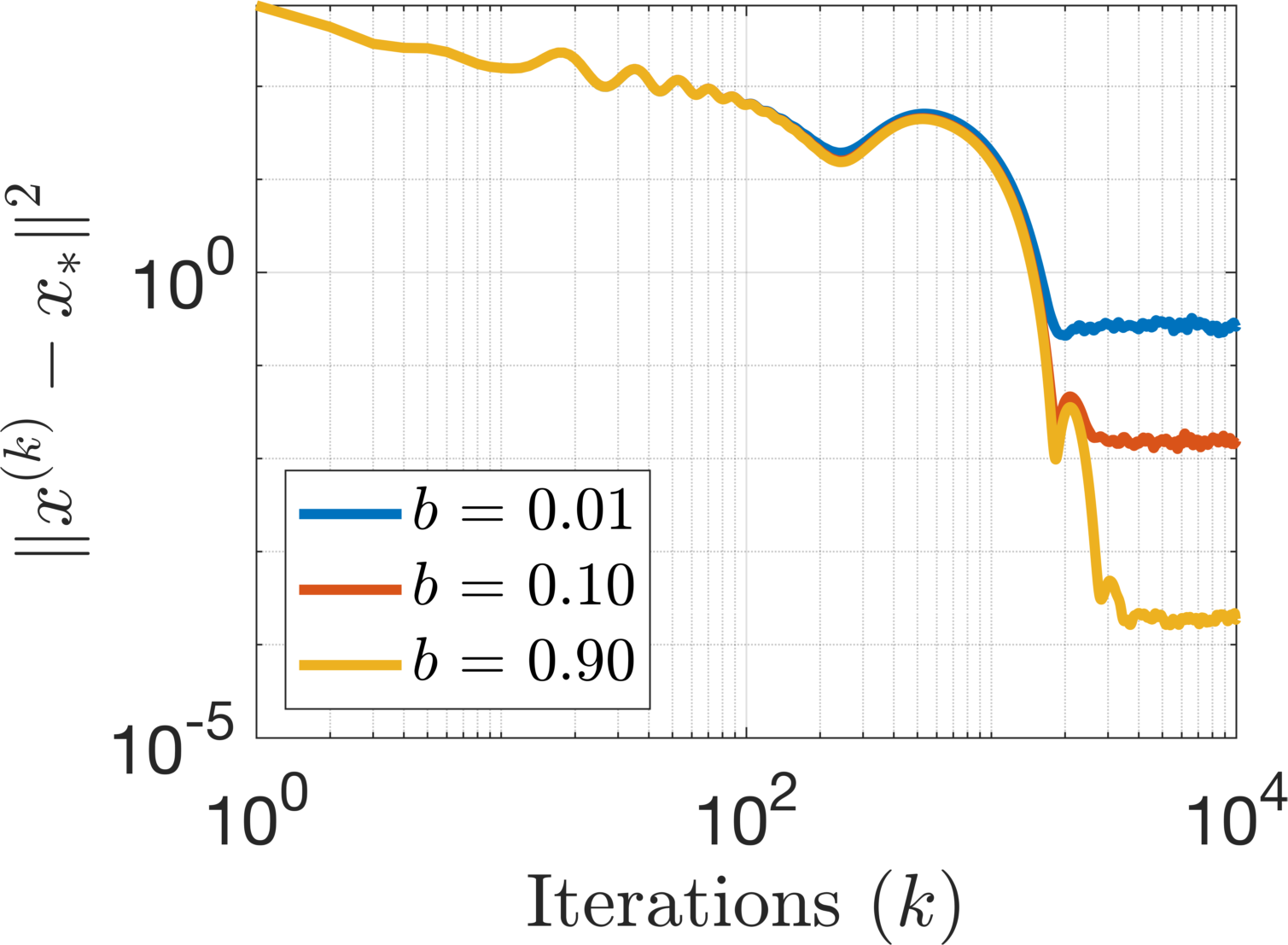} \label{fig:dasg_synth_batchsize}}
    \caption{Synthetic data experiments on D-ASG.  }
    \label{fig:exp_log_reg_synth_DASG}
\end{figure} 


Figure~\ref{fig:exp_log_reg_synth_DSG} illustrates the results for D-SG. In Figure~\ref{fig:dsg_synth_stepsize}, we investigate the convergence behavior of D-SG for varying step-size $\alpha$. The results clearly demonstrate the trade off between the convergence rate and the asymptotic variance: for larger $\alpha$ the algorithm attains a faster convergence rate but the resulting asymptotic variance becomes larger, as indicated by Theorem~\ref{thm:DGD:explicit}.

{
\color{black}
In the next experiment, we investigate the performance of D-SG for varying network architectures. In this setting we set $N=1000$ in order to illustrate the differences more clearly. 
As illustrated in Figure~\ref{fig:dsg_synth_network}, the results are intuitive: we observe that the disconnected graph non-surprisingly has the largest asymptotic variance. Furthermore, the performance improves as the graph becomes more connected: the performance of the (fully-connected) connected network is the best and degrades gradually as we go from the grid topology to the star topology. 
}

In our third experiment, we investigate the effect of the noise variance $\sigma^2$ by altering the batch proportion $b$. As shown in Figure~\ref{fig:dsg_synth_batchsize}, decreasing the batch size results in an increased asymptotic variance. This behavior is also correctly captured by Theorem~\ref{thm:DGD:explicit}: decreasing $b$ increases the noise variance $\sigma^2$ and hence the second term in \eqref{ineq-perf-ub-dgd} dominates for large number of iterations.  

In our next set of experiments, we replicate the previous three experiments by replacing D-SG with D-ASG. Figure~\ref{fig:exp_log_reg_synth_DASG} illustrates the results. We observe a similar outcome to the ones of the previous set of experiments. Figure~\ref{fig:dasg_synth_stepsize} verifies that the step-size determines the trade off between the convergence rate and the asymptotic variance as suggested by Theorem~\ref{thm:general:DASG}. 
{\color{black}
Figure~\ref{fig:dasg_synth_network} illustrates the behavior of the algorithm under different network settings with $N=1000$. We again observe that the disconnected network is performing worse than the other network architectures as expected; however, as opposed to Figure~\ref{fig:dsg_synth_network}, there is no significant difference between the grid and the connected networks. This result suggests that the usage of the momentum in D-ASG compensates the additional difficulty introduced by the sparsely connected network architecture. 
}
In our last experiment, we investigate the behavior of D-ASG for varying gradient noise variance. As illustrated in Figure~\ref{fig:dasg_synth_batchsize}, the asymptotic error increases with the decreasing batch proportion $b$. More importantly, compared to D-SG, the increase in the asymptotic variance turns out to be significantly larger for D-ASG, which illustrates that D-ASG is less robust to the gradient noise. This observation also supports our theory (cf.\ the remark about robustness in Section~\ref{subsec:dasg}).

\begin{figure}[t] 
    \centering
    \subfigure[D-SG]{\includegraphics[width=0.24\columnwidth]{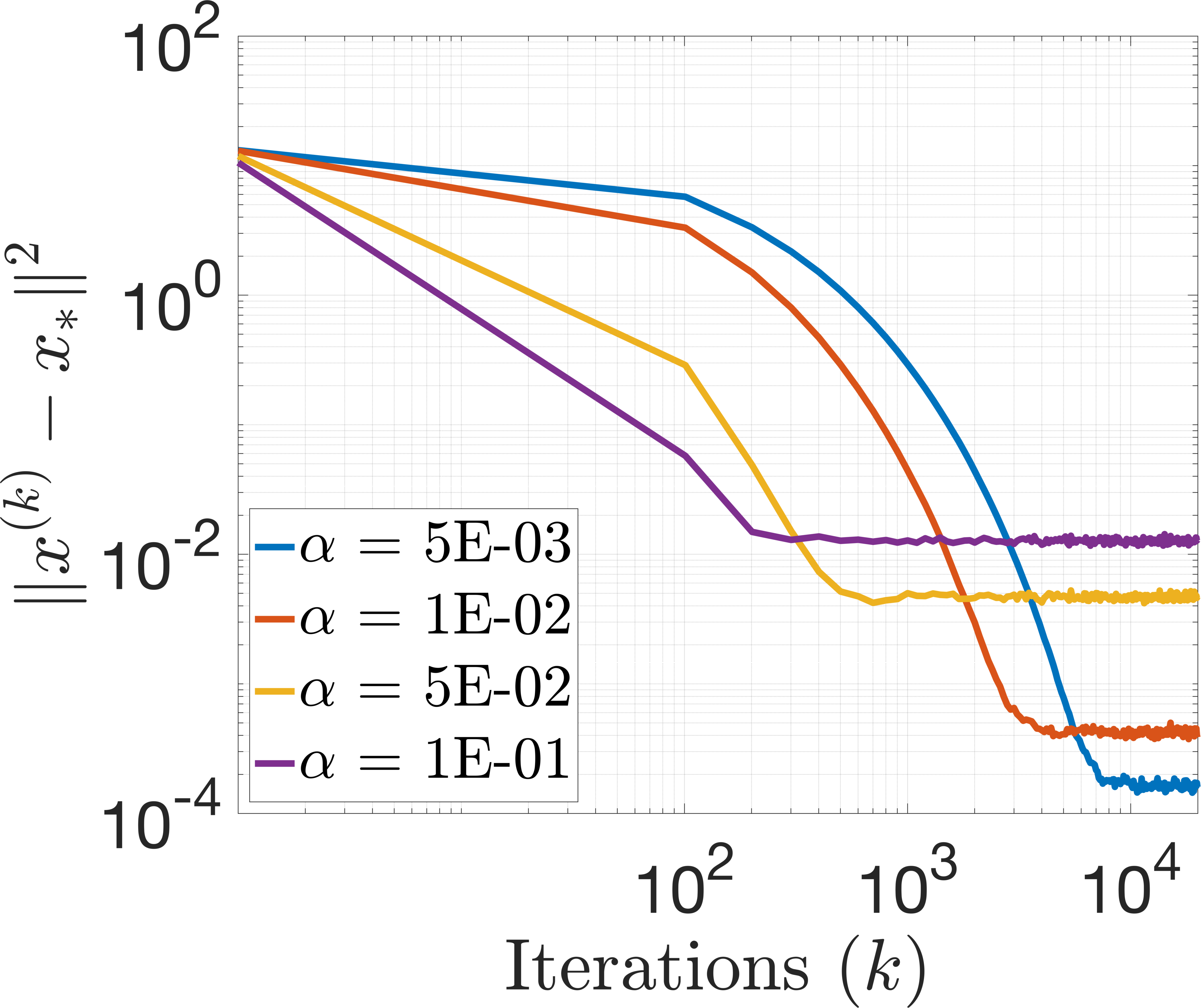}
    			 \includegraphics[width=0.24\columnwidth]{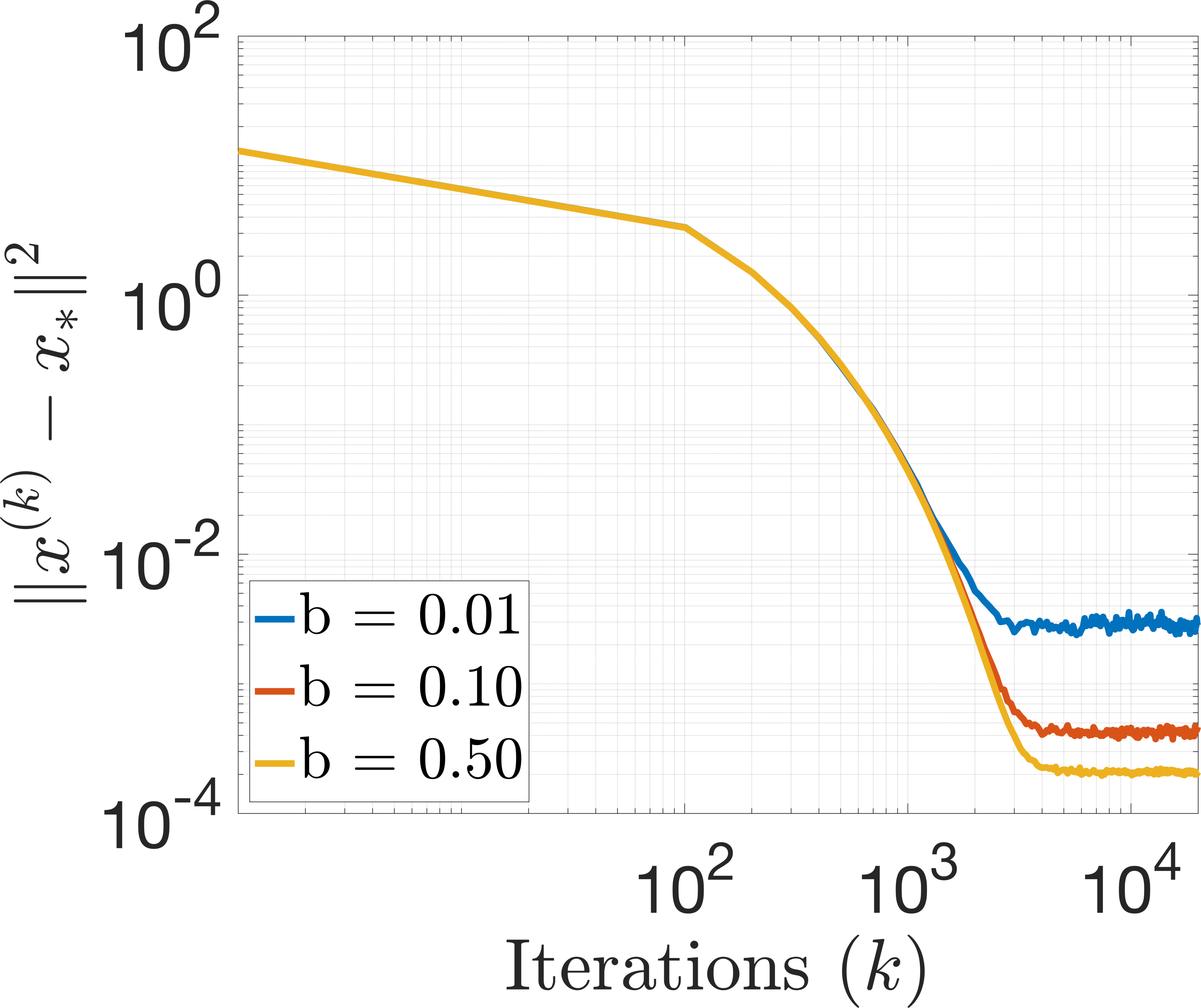}}
	\subfigure[D-ASG]{\includegraphics[width=0.24\columnwidth]{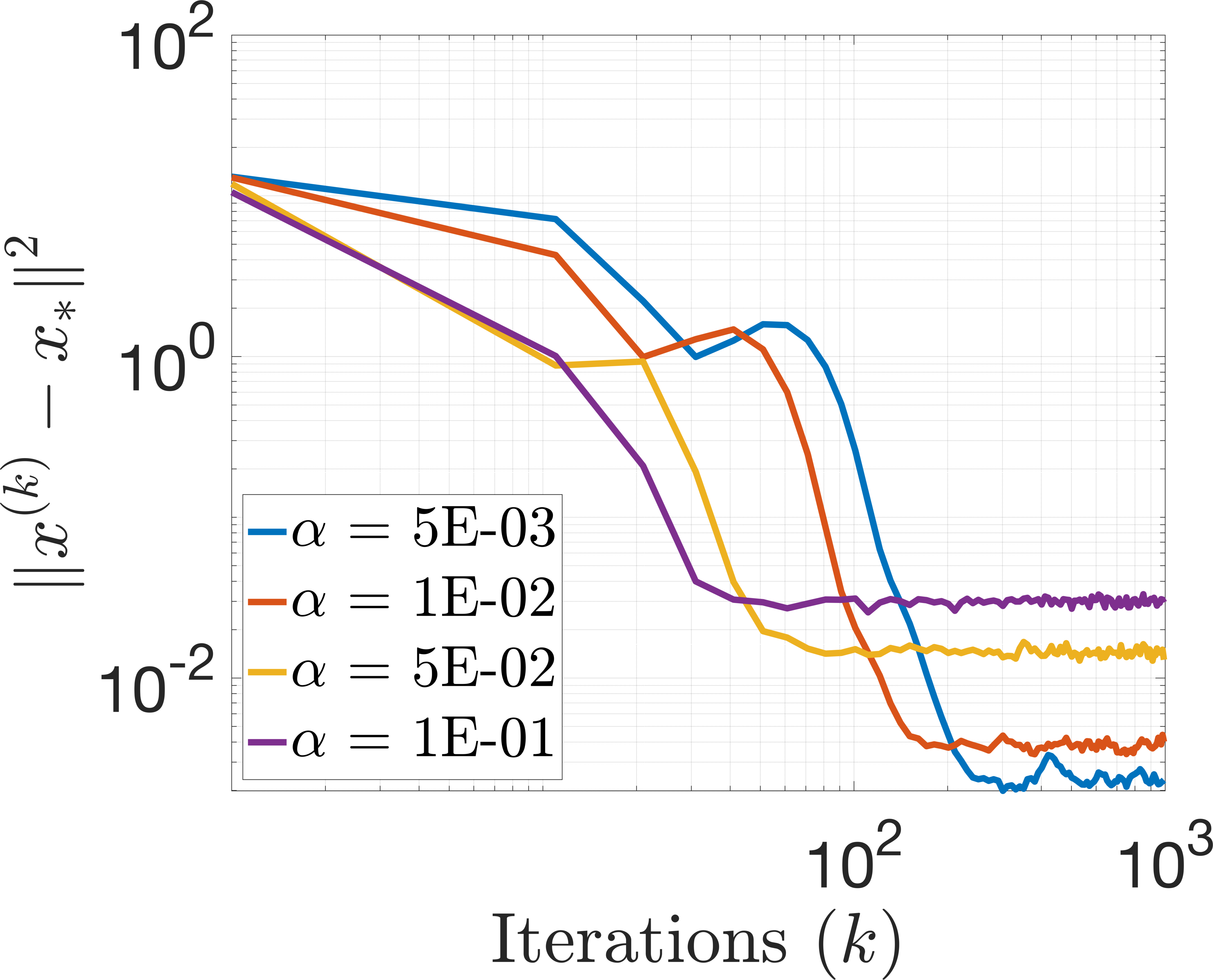}
    			 \includegraphics[width=0.24\columnwidth]{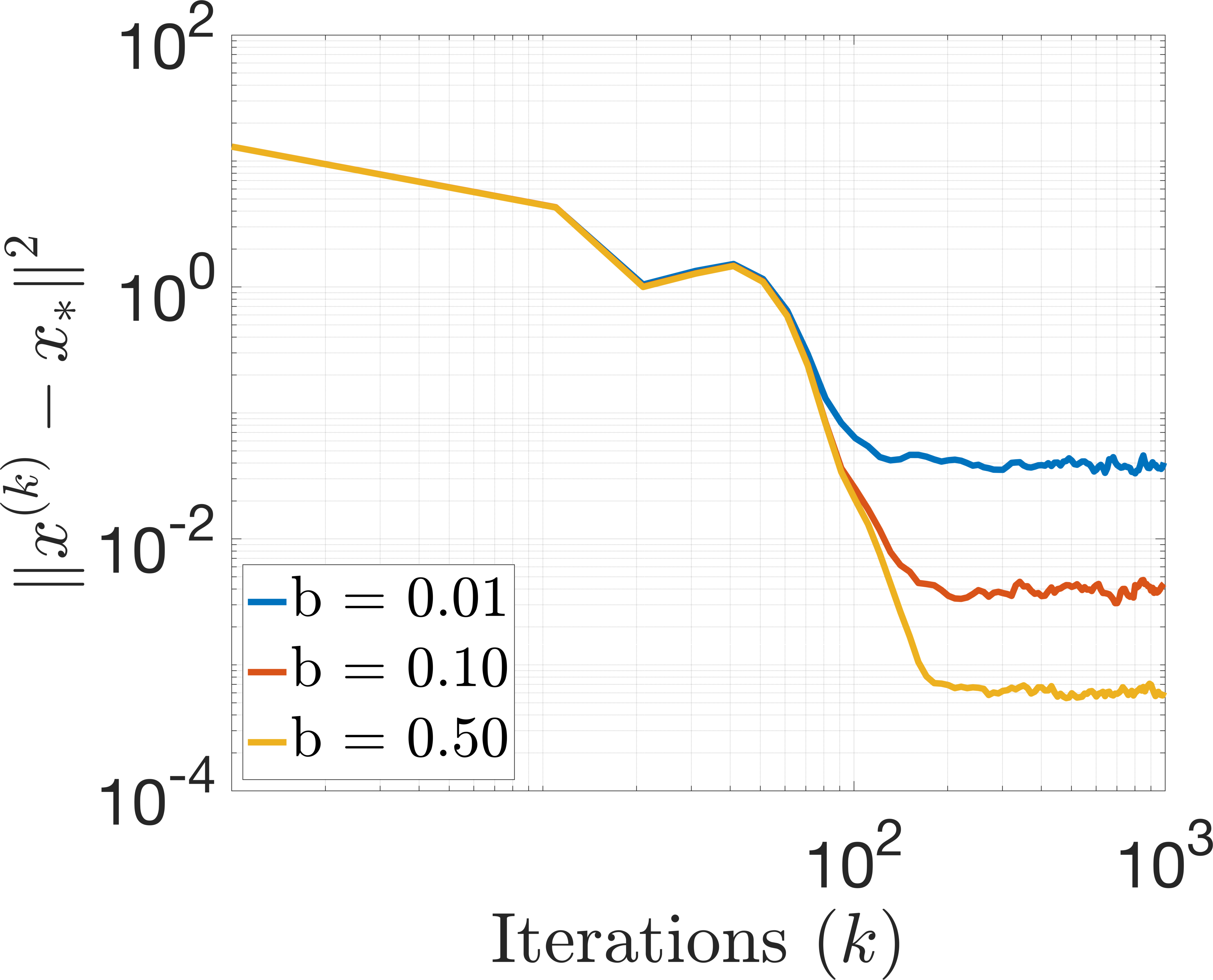}}

    \caption{Evaluation of D-SG and D-ASG on MNIST and a real distributed environment with $N=10$ interconnected computers. }
    \label{fig:exp_log_reg_mnist_DSG_DASG}
\end{figure}

\subsection{Real data experiments} 

In this section, we consider a real-data setting, where we evaluate the algorithms on a real distributed environment. We consider the same logistic regression problem on two binary classification datasets and compare the performance of D-SG and D-ASG with their natural competitors, namely distributed dual averaging (D-DA) \citep{duchi2012dual}, distributed stochastic gradient tracking (D-SGT) \citep{pu-nedich-journal}, and distributed communication sliding (D-CS) \citep{lan2017communication}. Among these algorithms D-CS is an exact algorithm, similar to D-MASG. As datasets, we use the MNIST, and the Epsilon datasets. The MNIST dataset contains $70$K binary images (of size $d = 20 \times 20$) corresponding to $10$ different digits\footnote{\url{http://yann.lecun.com/exdb/mnist}}. To obtain a binary classification problem, we extract the images corresponding to the digits $0$ and $8$, where we end up with $n=11774$ images in total. On the other hand, the Epsilon dataset is one of the standard binary classification datasets\footnote{\url{https://www.csie.ntu.edu.tw/~cjlin/libsvmtools/datasets/binary.html}} and contains $n=400$K samples with $d=2000$.

We have implemented all the algorithms in C++ by using a low-level message passing protocol for parallel processing, namely the OpenMPI library\footnote{\url{https://www.open-mpi.org}}. In order to have a realistic experimental environment, we have conducted these experiments on a cluster interconnected computers, each of which is equipped with different quality CPUs and memories. We set $b=0.1$ unless stated otherwise. 

In the first experiment, similar to the previous section, we monitor the behavior of D-SG and D-ASG with varying step-sizes and batch proportions in order to affirm that our theoretical results also hold in the real problem setting. Figure~\ref{fig:exp_log_reg_mnist_DSG_DASG} illustrates the results. We observe that, even under the real data/distributed environment setting the algorithms exhibit the same behavior. The trade off between the convergence rate and the asymptotic variance is still present and D-ASG is significantly less robust to the stochastic gradient noise.  

\begin{figure}[t] 
    \centering
    \subfigure[MNIST]{\includegraphics[width=0.24\columnwidth]{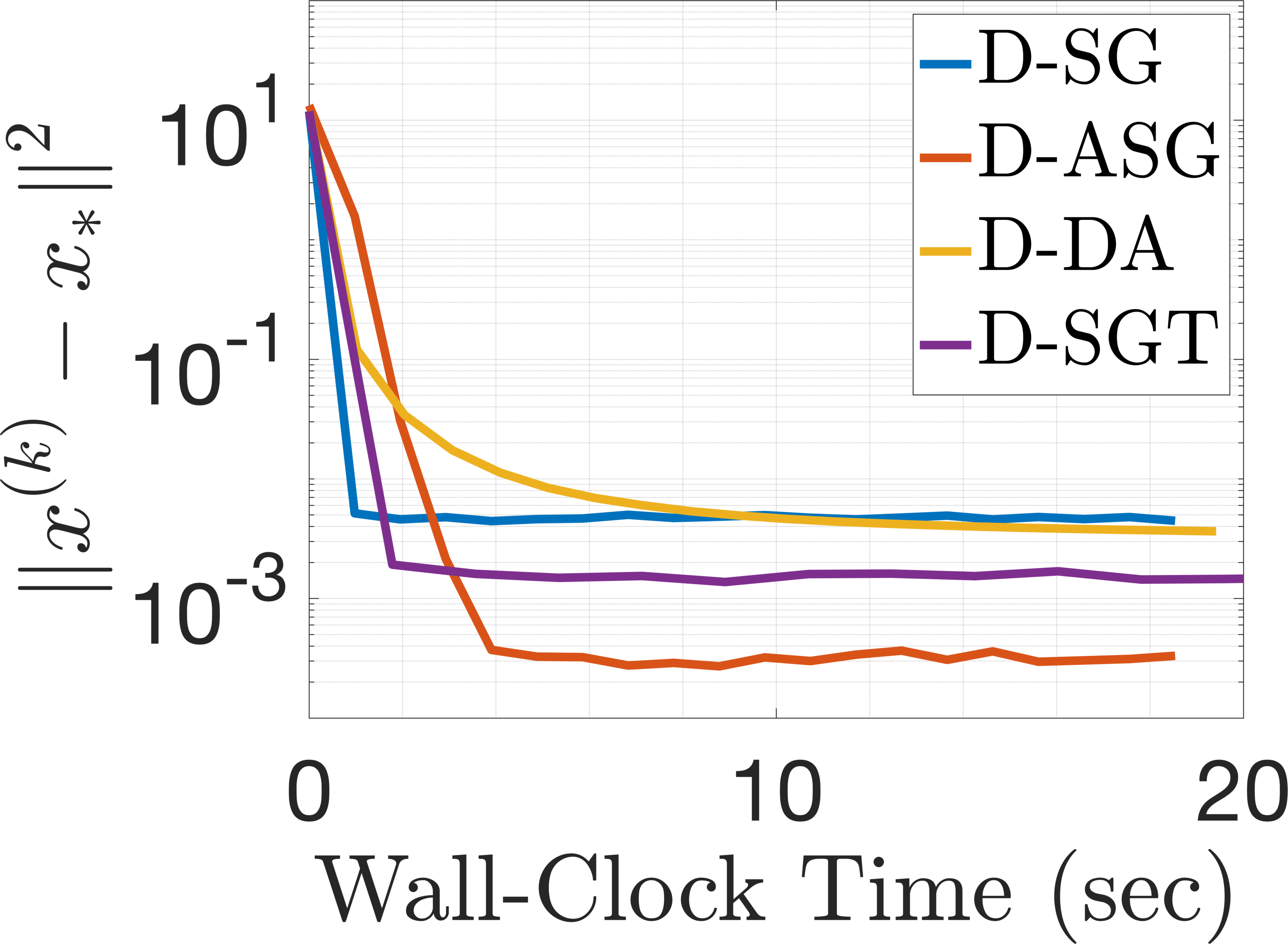}}
	\subfigure[Epsilon]{\includegraphics[width=0.24\columnwidth]{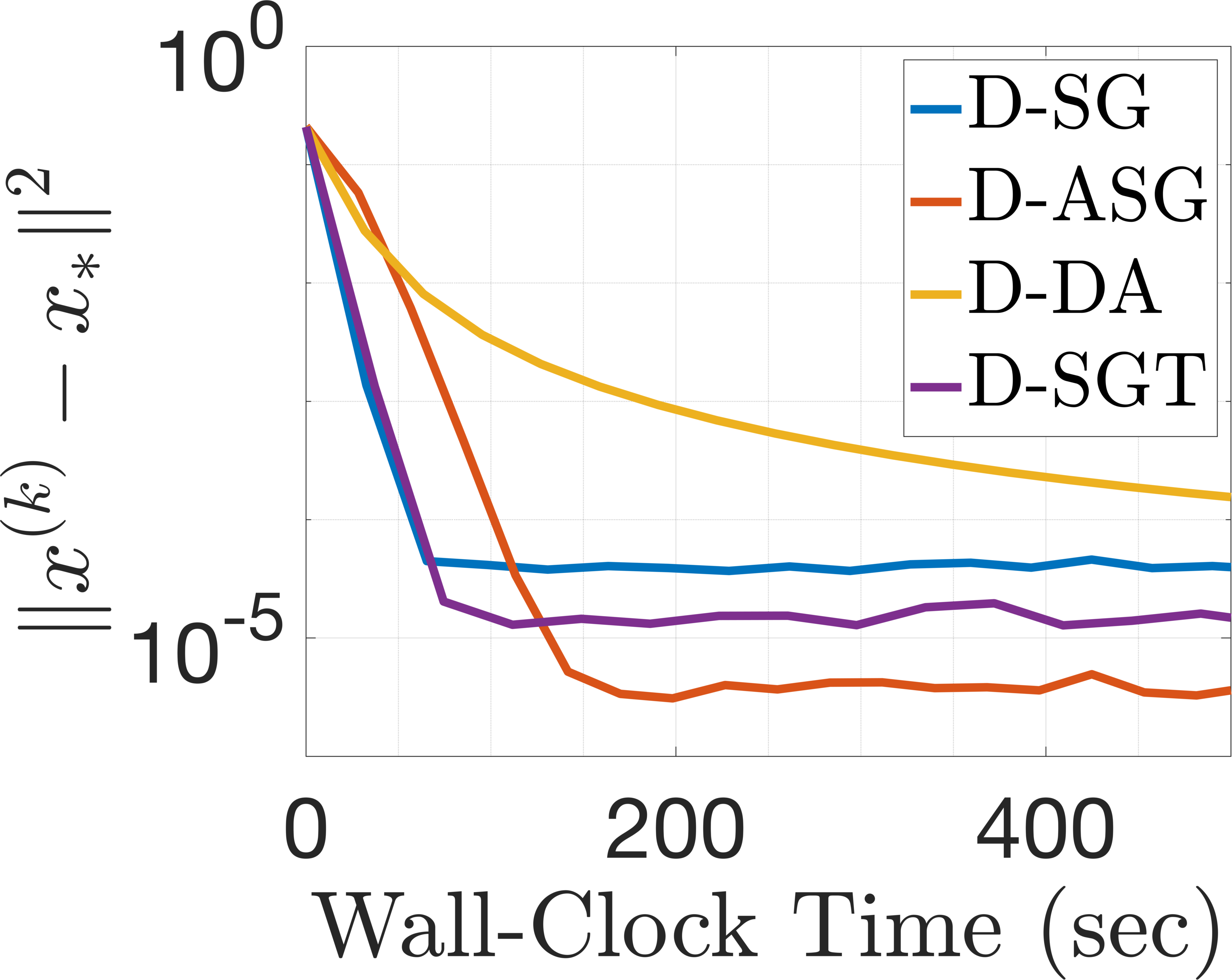}}
      \subfigure[MNIST]{\includegraphics[width=0.24\columnwidth]{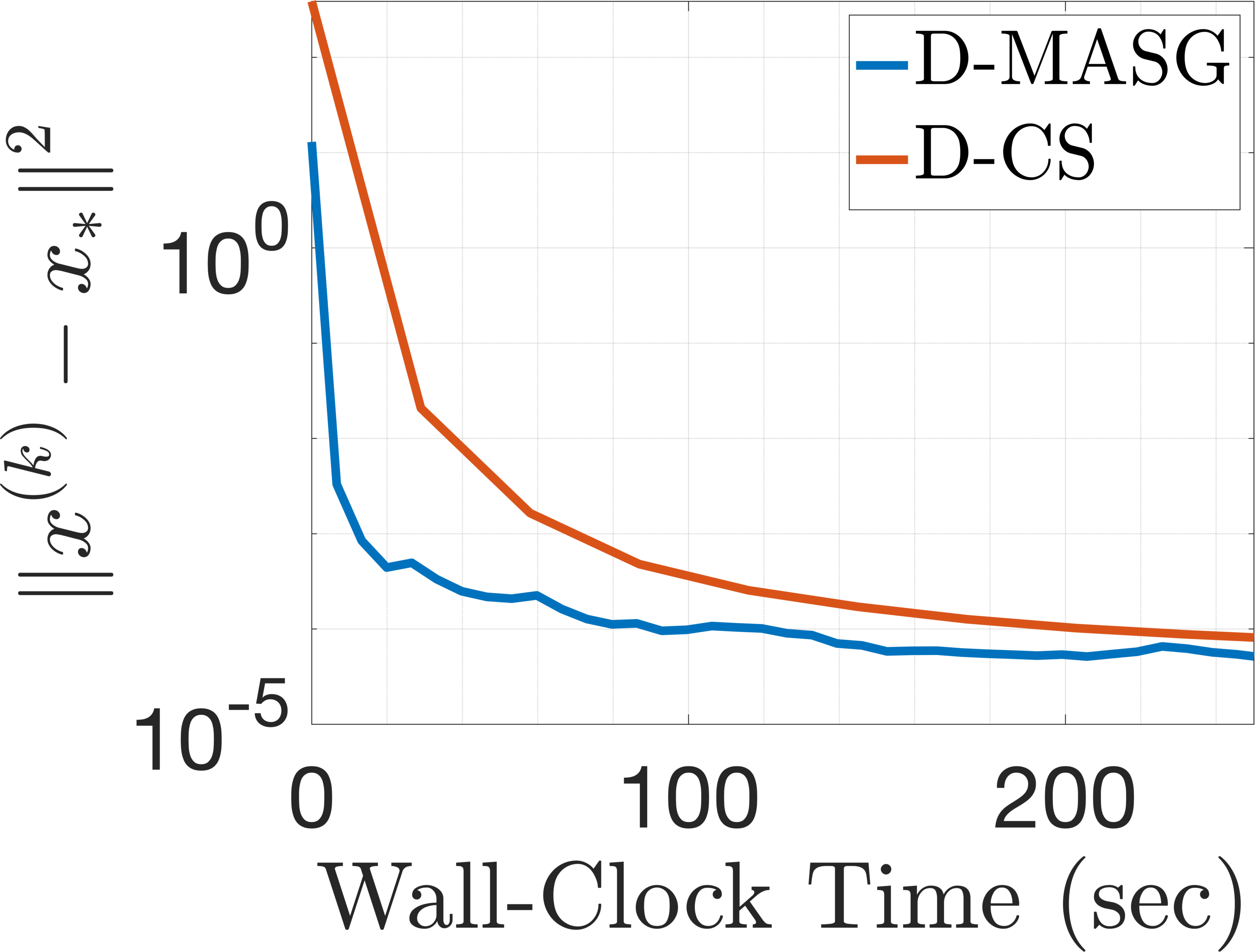}}
      \subfigure[Epsilon]{\includegraphics[width=0.24\columnwidth]{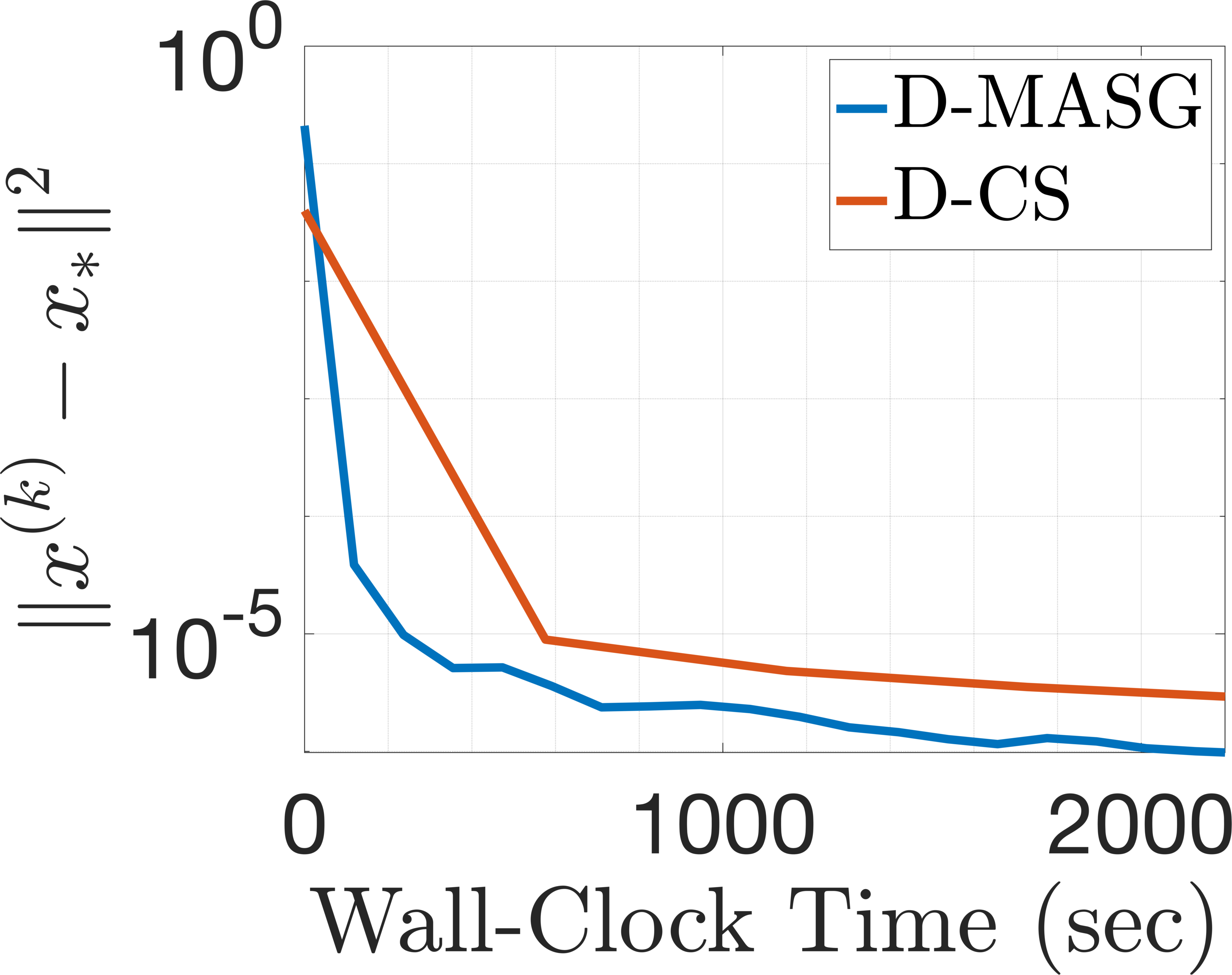}}
    \caption{(a)-(b) Comparison of D-SG and D-ASG with D-DA and D-SGT on the two datasets. (c)-(d) Comparison of D-MASG and D-CS on the two datasets.}
    \label{fig:exp_log_reg_real_sota} 
\end{figure}

In our next experiment, we compare the performances of the inexact algorithms, namely D-SG and D-ASG with D-DA and D-SGT on the two datasets. The results are illustrated in Figure~\ref{fig:exp_log_reg_real_sota}(a)-(b). In all settings, we observe that the performance of D-SG and D-DA are very similar, whereas the variance reduction step improves the performance of D-SGT over these two algorithms. The results show that D-ASG outperforms all these three algorithms and illustrate the acceleration brought by the use of momentum. 

We then proceed to comparing the exact algorithms D-MASG and D-CS. We note that the D-CS algorithm has two levels of nested iterations: an outer iteration and an inner iteration. At each outer iteration the algorithm makes the nodes communicate two times, whereas the actual optimization is done in the inner iteration and the number of inner iterations can be varied depending on the communication cost: if the communication cost is high, the number of inner iterations should be high as well in order to make the communications less often. In order to make the wall-clock-time comparison between D-CS and D-MASG fairer, we set the number of inner iterations to $2$, since D-MASG has only one round of communications at every iteration. We also note that the computational requirements of each inner iteration of D-CS are significantly higher than the one of D-MASG.

We first investigated the performance of D-CS and D-MASG under the circular network setting. As opposed to the previous experiments, we did not observe a significant performance improvement over D-CS. We suspect that the Polyak-Ruppert-type averaging of D-CS is providing some acceleration to D-CS. However, when we evaluate the two algorithms under the connected network setting, we obtain improved results, which are visualized in Figure~\ref{fig:exp_log_reg_real_sota} (c)-(d). The results show that, on the MNIST dataset D-MASG provides a slight improvement over D-CS, whereas on the Epsilon dataset the difference between the computational costs of D-CS and D-MASG become more prominent, which yields a significant improvement over D-CS.  



\begin{figure}[t] 
    \centering
    \subfigure[MNIST b=0.5]{\includegraphics[width=0.23\columnwidth]{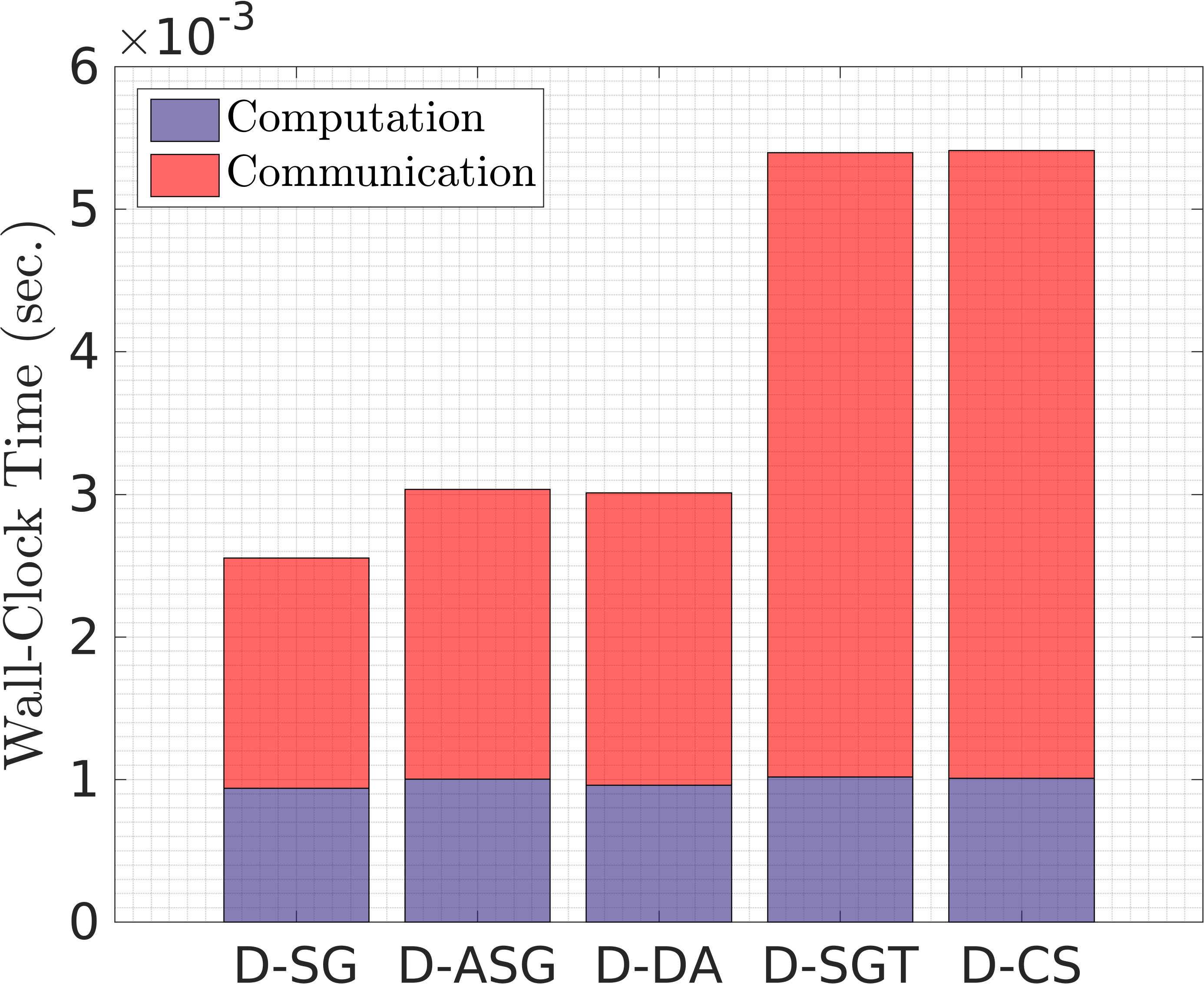} \label{fig:mnist_comp}}
	\subfigure[Epsilon b=0.1]{\includegraphics[width=0.23\columnwidth]{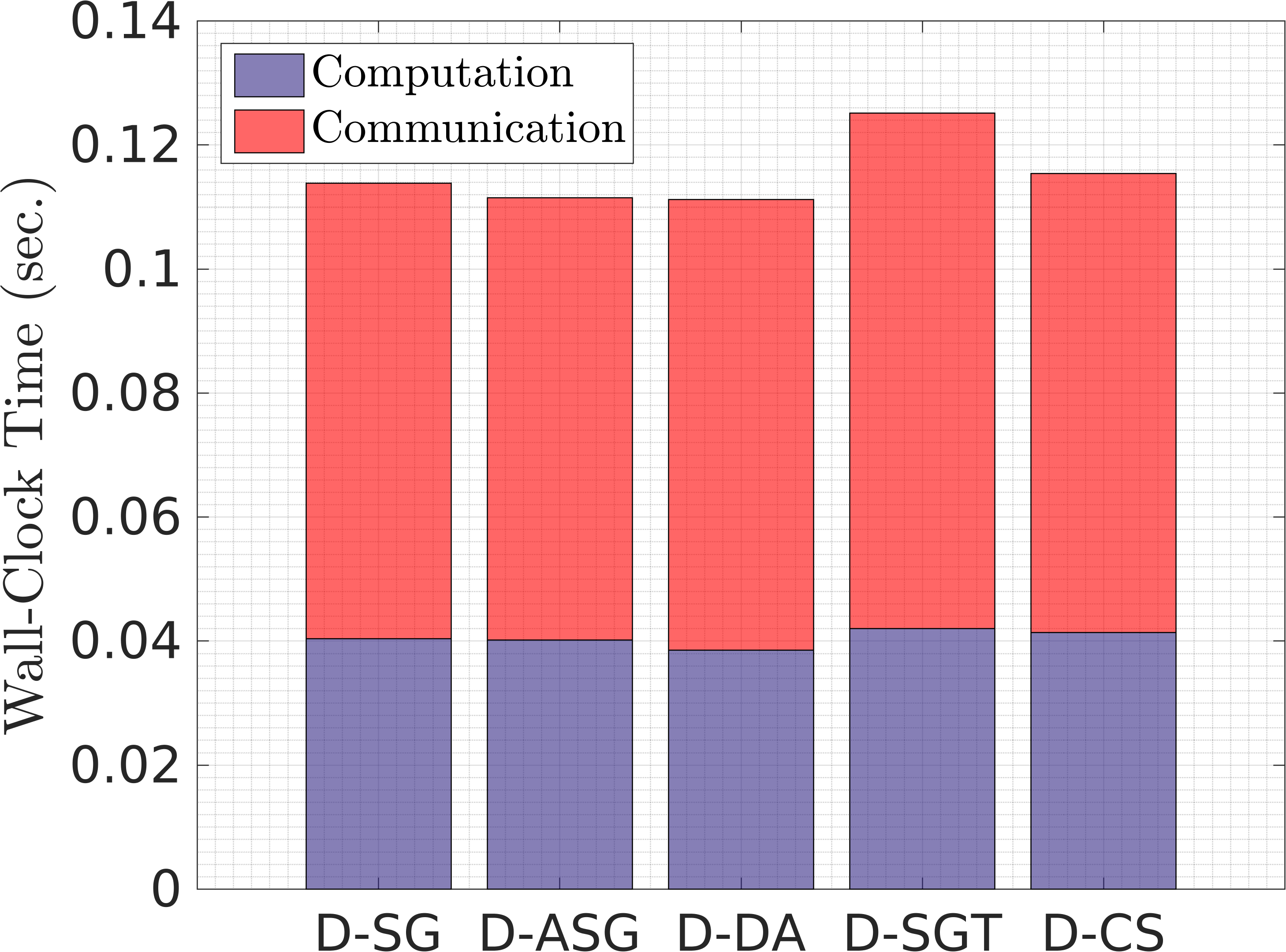} \label{fig:eps_comp}}
    \subfigure[MNIST D-SG ]
    {\includegraphics[width=0.23\columnwidth]{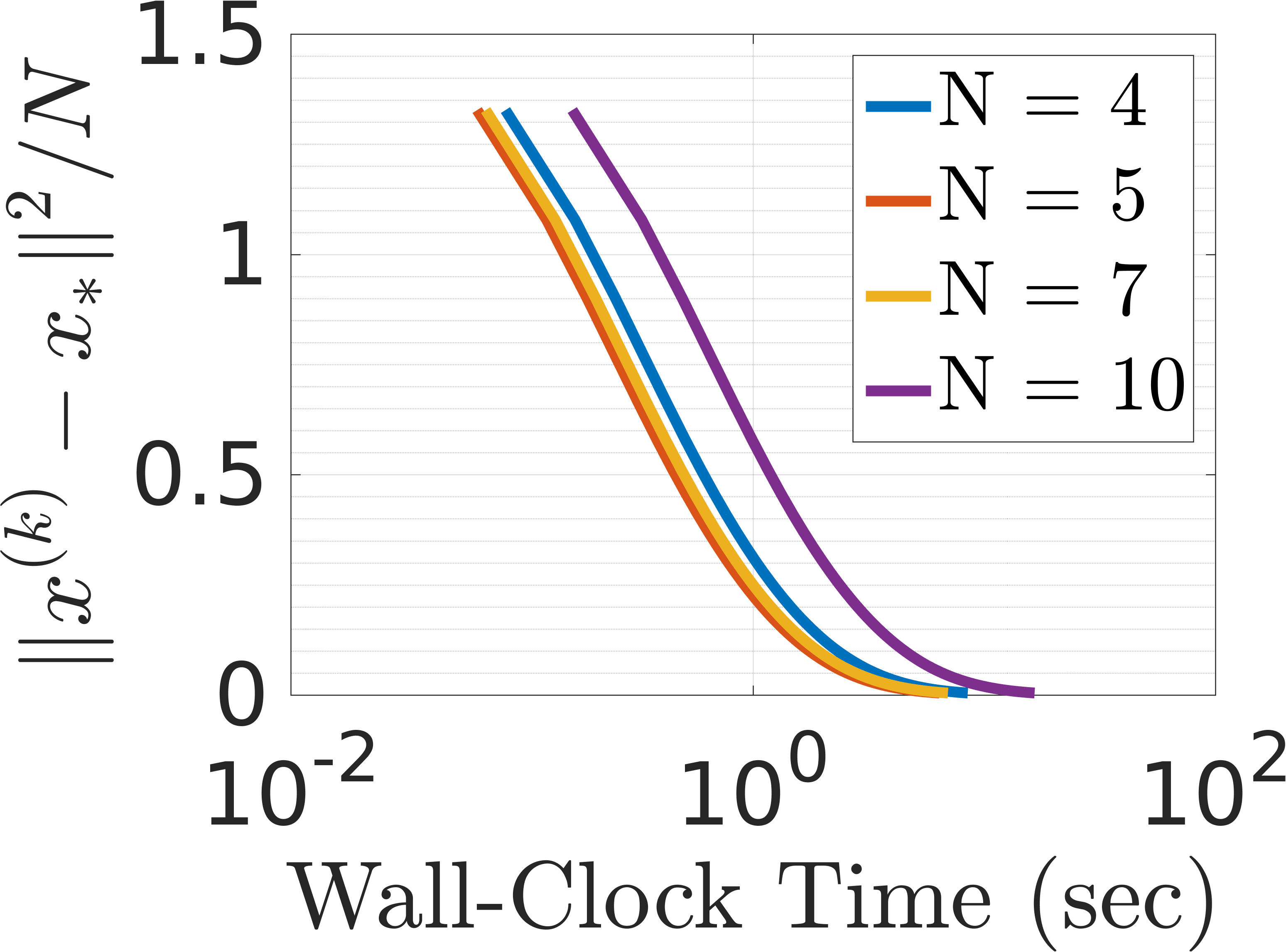} \label{fig:mnist_dsg_process} }
	\subfigure[MNIST D-ASG ]
    {\includegraphics[width=0.23\columnwidth]{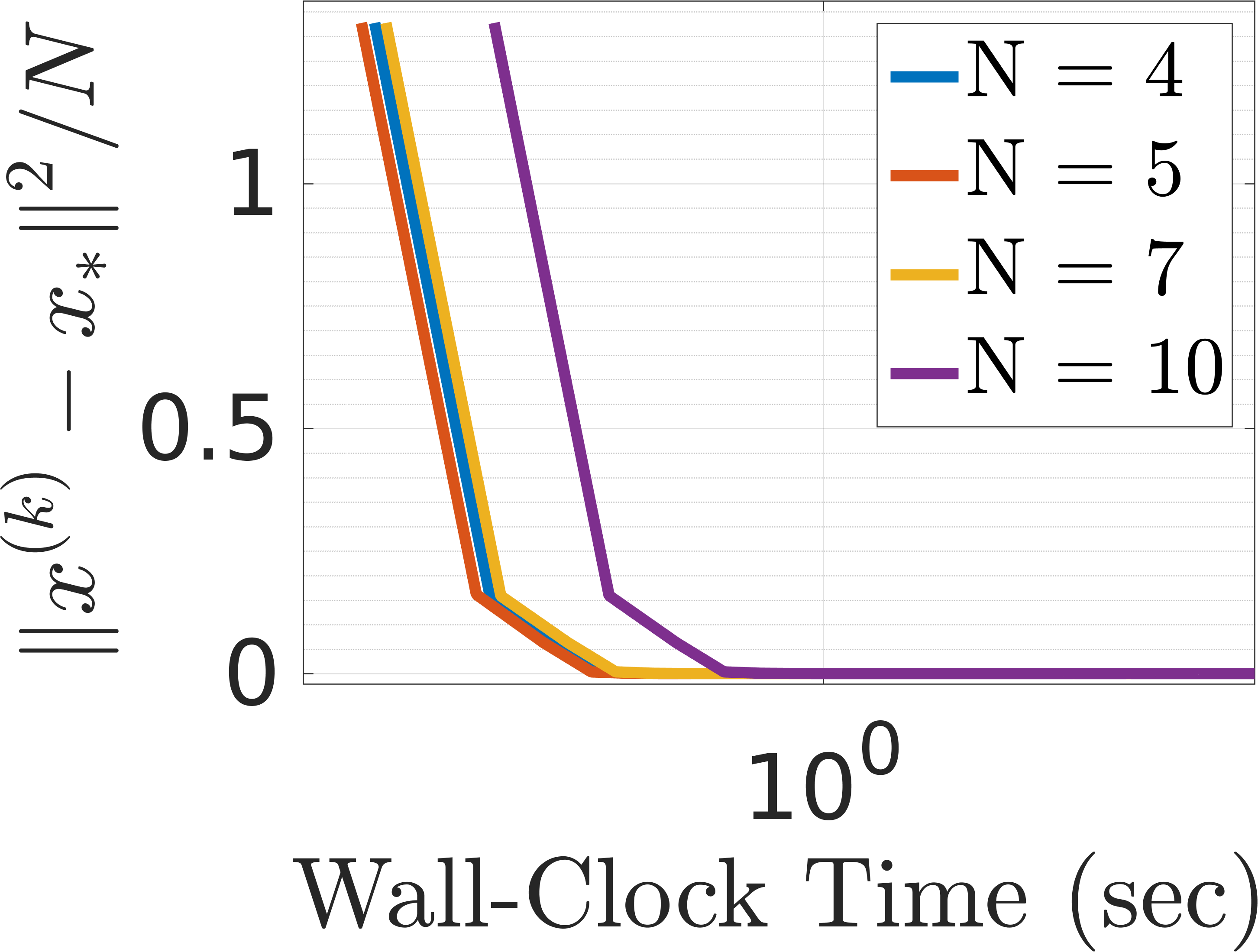} \label{fig:mnist_dasg_process} }
    \caption{Investigation of the computational requirements. }
    \label{fig:exp_log_reg_numproc}
\end{figure} 

Next, we investigate the computational aspects of the aforementioned algorithms. In Figures~\ref{fig:mnist_comp} and~\ref{fig:eps_comp} we measure the average times that the algorithms spend in terms of computation and communication per iteration. We observe that in both cases, the computation times of the algorithms is similar to each other. On the other hand, when the dimension of the problem is smaller (in the case of MNIST), the communication cost of D-SGT and D-CS dominates the overall complexity\footnote{In this experiment, the number of inner iterations of D-CS is set to $1$.}. However, when the dimension of the problem increases (in the case of Epsilon), the computation time increases superlinearly with the increasing dimension, which results in a similar proportion of computation/communication for all the algorithms. Combined with the performance comparison results (e.g.\ Figure~\ref{fig:exp_log_reg_real_sota}), this experiment suggests that D-ASG achieves a good balance between computational complexity and accuracy: while having similar computational complexity to D-SG and D-DA, it is able to provide better performance than D-SGT and D-CS, which have larger computational costs. 

In our final experiment, we investigate the behavior of D-SG and D-ASG on the increasing number computation nodes $N$ (while keeping all the other parameters unchanged). Figures~\ref{fig:mnist_dsg_process} and~\ref{fig:mnist_dasg_process} show the results. We observe that, the convergence behavior improves when we increase $N$ from $4$ to $5$; however, further increasing $N$ results in a degraded performance, since the overall computation time is dominated by the communication cost, a typical situation observed in synchronized distributed optimization \citep{kaya2019framework,csimcsekli2018asynchronous}.

{\color{black}

\begin{figure}[t] 

\color{black}
    \centering
    \subfigure[Varying $L$ for D-SG \label{fig:exp_muL_sg}]{\includegraphics[width=0.24\columnwidth]{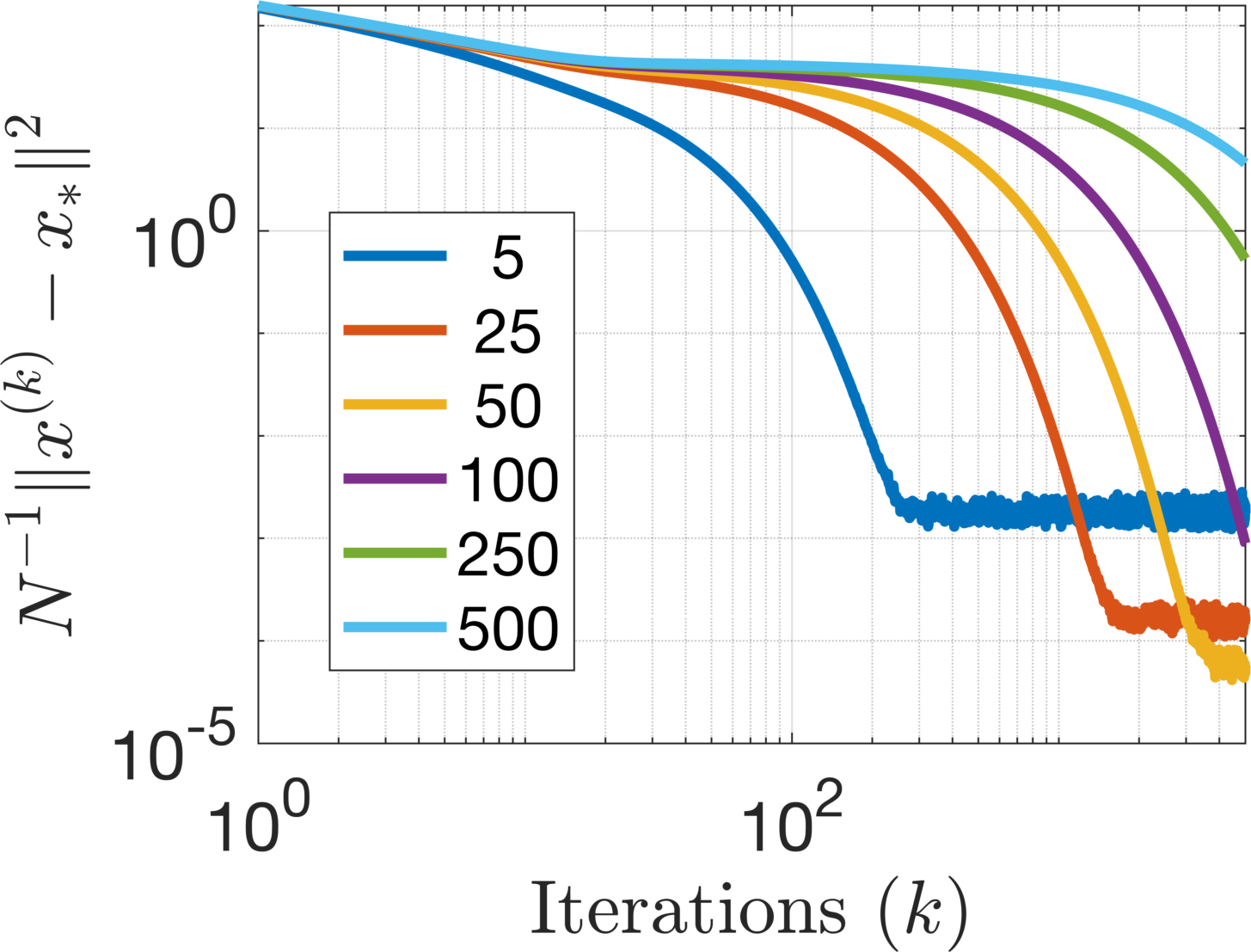}}
    \subfigure[Varying $L$ (left) and $\mu$ (right) for D-ASG \label{fig:exp_muL_asg}]{\includegraphics[width=0.24\columnwidth]{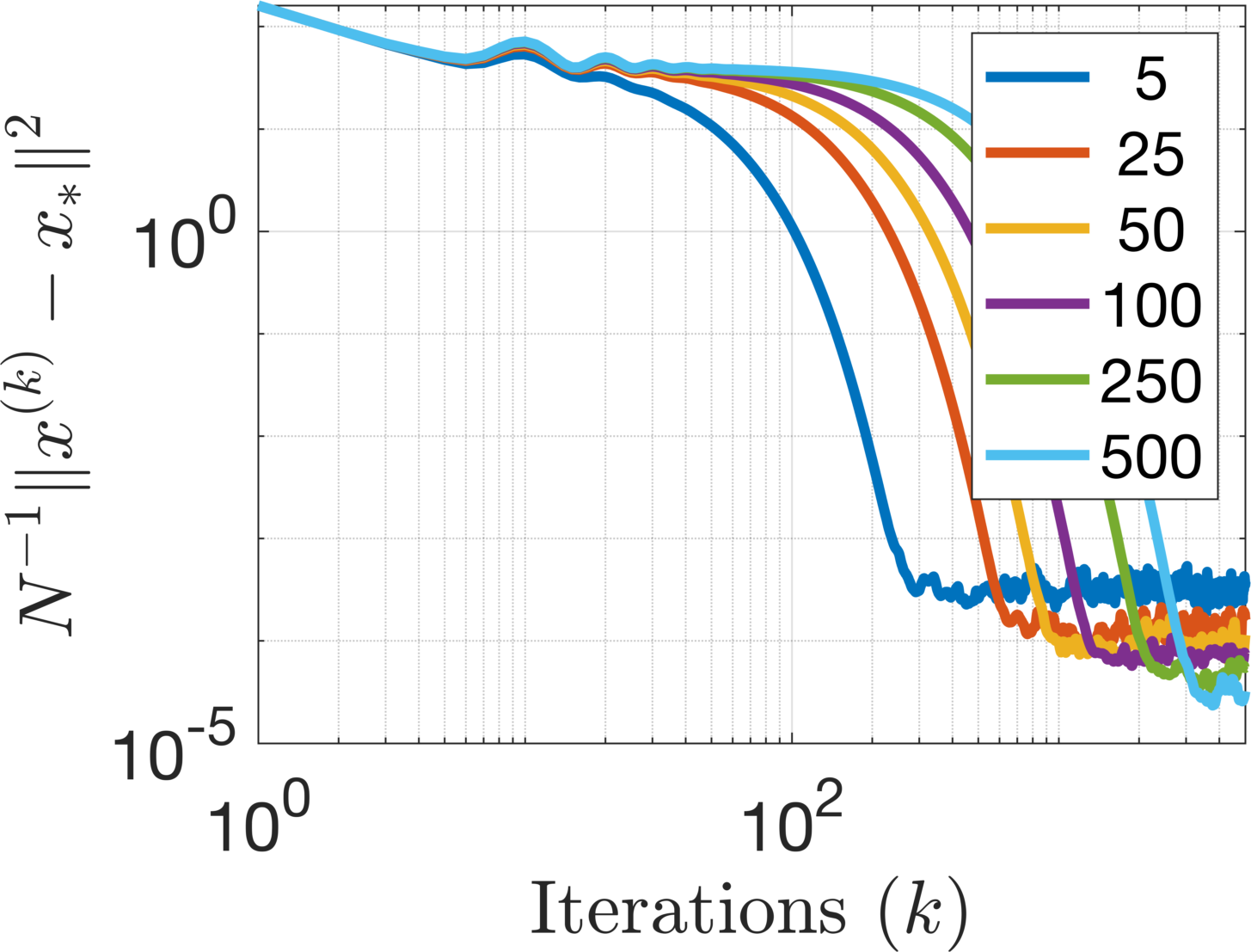}
    \includegraphics[width=0.24\columnwidth]{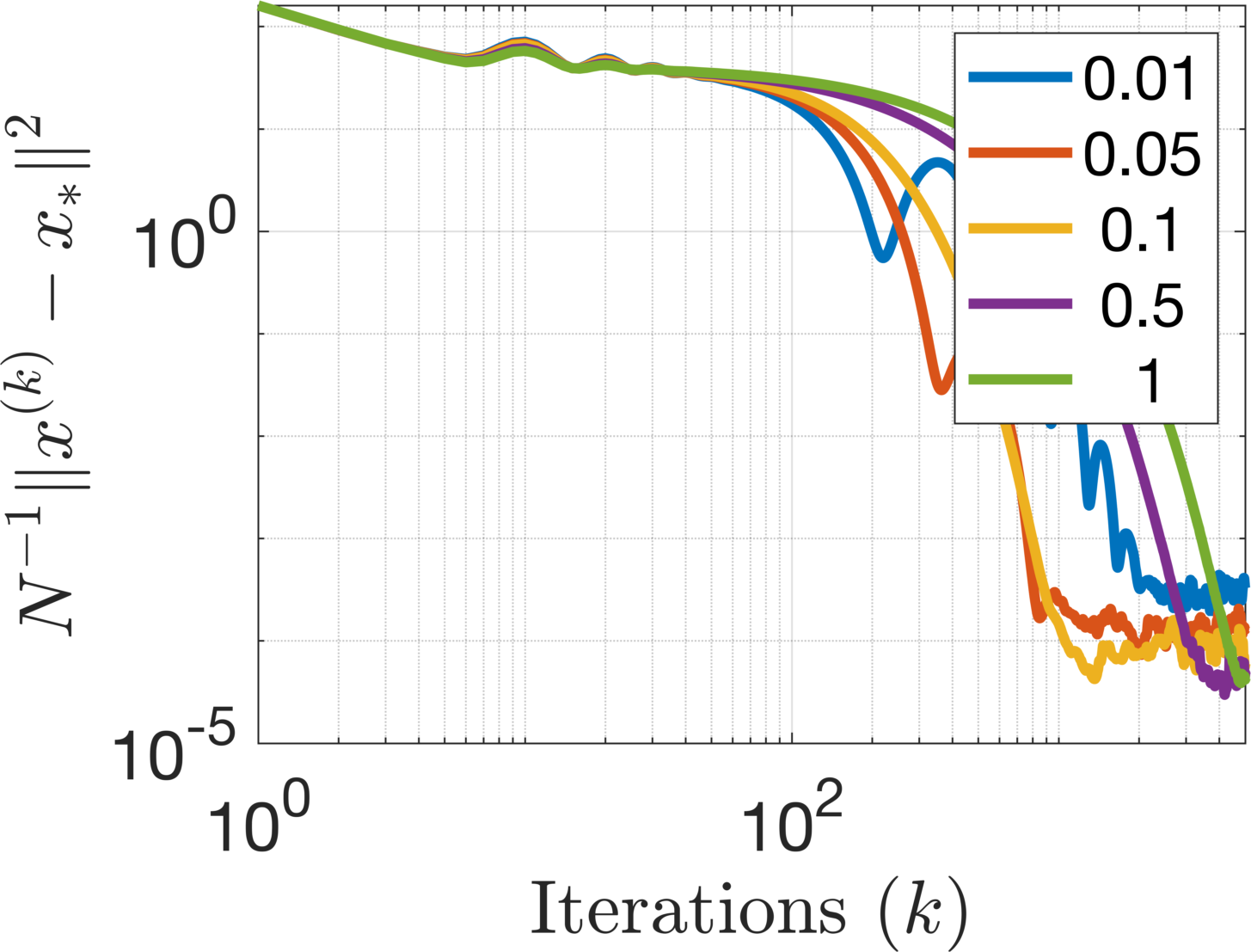}}
    \subfigure[Varying $L$ (left) and $\mu$ (right) for D-MASG \label{fig:exp_muL_masg}]{\includegraphics[width=0.24\columnwidth]{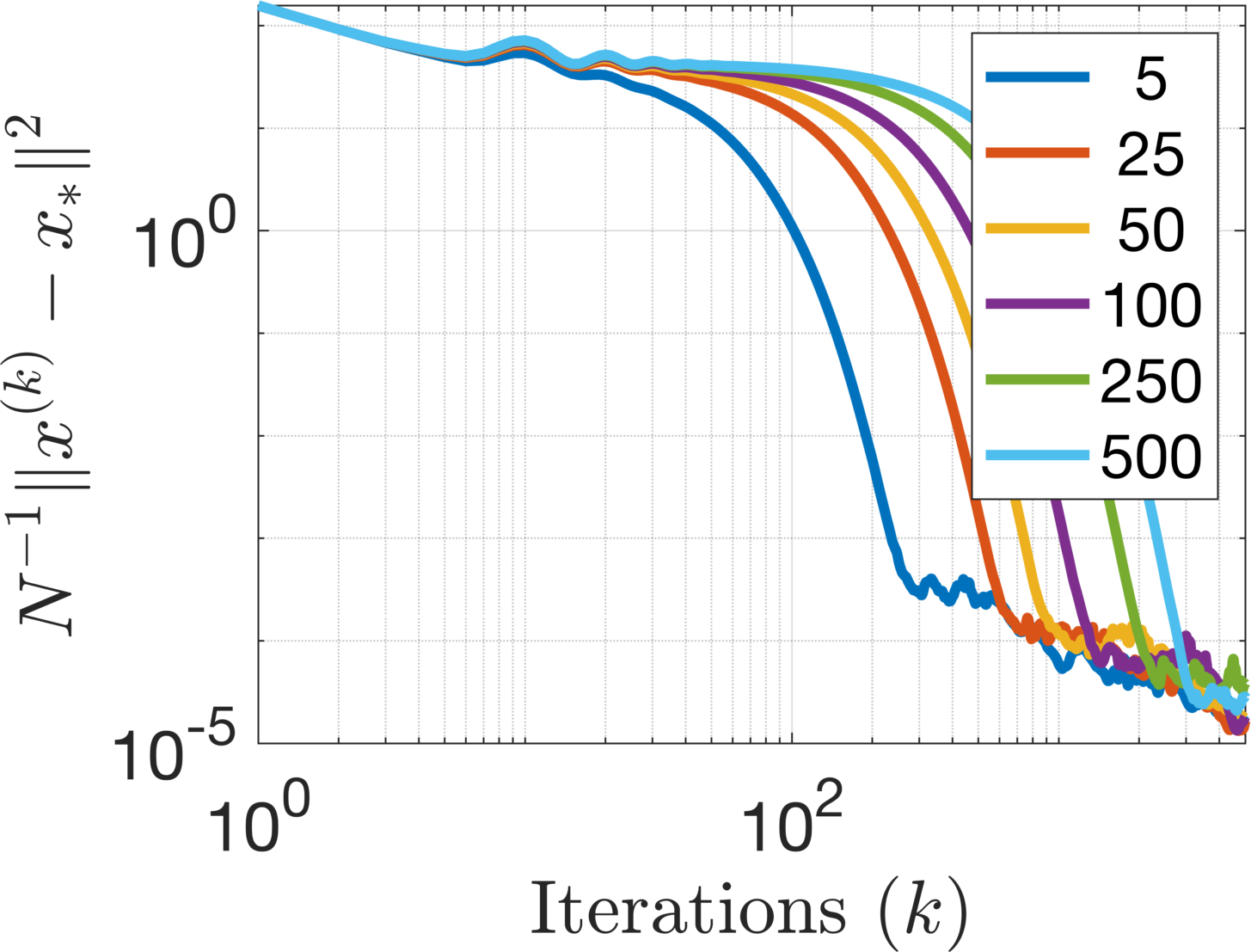}
    \includegraphics[width=0.24\columnwidth]{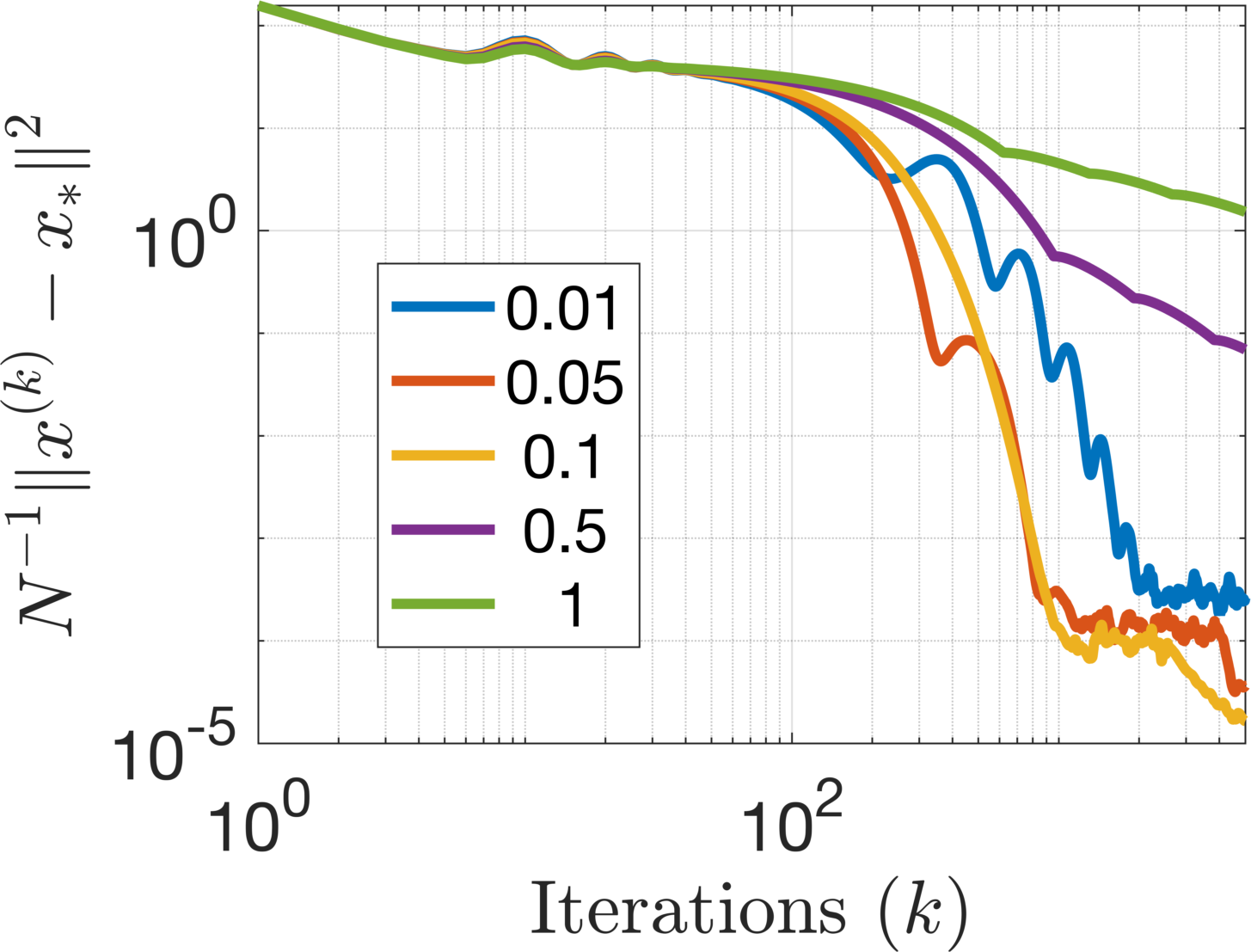}}
    \caption{Change in performance with respect to inaccurate estimates for $L$ and $\mu$.}
    \label{fig:exp_muL} 
\end{figure}

\subsection{Robustness to hyperparameters}

In our last set of experiments, we aim at investigating the performance of our algorithms in the case where the problem constants $L$ and $\mu$ cannot be estimated accurately. Here, we re-consider the MNIST dataset in the simulated distributed environment and run the three proposed algorithms for different estimates for $L$ and $\mu$. For D-SG we set the step-size $\alpha = (1+\lambda_{N}^{W}) / L$, whereas for D-ASG we set $\alpha = \lambda_{N}^{W} / L$ and $\beta=\frac{1-\sqrt{\alpha \mu}}{1+\sqrt{\alpha \mu}}$. Finally, for D-MASG we set $\tilde{\kappa}= \frac{L+\mu}{\mu \lambda_N^W} $ as in \eqref{tilde_kappa} and use the setting reported in Proposition~\ref{main_DMASG_result}.

In this problem, we first compute an estimate for $L$ and $\mu$ from the data, where we obtain $L \approx 50$ and $\mu \approx 0.1$. Then, we vary $L$ from $5$ to $500$ by fixing $\mu =0.1$, and we vary $\mu$ from $0.01$ to $1$ by fixing $L=50$. Accordingly, we run the algorithms with hyperparameters that are computed with these values for $L$ and $\mu$. Figure~\ref{fig:exp_muL} visualizes the results. In Figure~\ref{fig:exp_muL_sg}, we observe that, when $L$ is set close to $50$, D-SG performs similarly, whereas for lower or higher values of $L$ the performance degrades. On the other hand, in Figure~\ref{fig:exp_muL_asg}, we observe that D-ASG is also 
robust to the values of $L$ and $\mu$: the performance of the algorithm 
does not significantly vary for varying $L$ and $\mu$. Finally, Figure~\ref{fig:exp_muL_masg} illustrates the performance of D-MASG. Here, in terms of varying $L$, we again observe a robust behavior, where the performance of the algorithm stays almost the same for different values of $L$. On the other hand, we also observe that the algorithm has a strong dependency on the estimate of $\mu$, where an overestimation of the value of $\mu$ might significantly slow down the convergence. 

We conclude that, when a reasonably good estimate for $L$ and $\mu$ can be obtained, D-SG and D-MASG perform well. We also observe that on this dataset the performance of D-ASG is robust when subject to changes in the parameters $L$ and $\mu$.

}



\section{Conclusion}
Stochastic gradient (SG) methods are workhorse algorithms in machine learning practice. There is an increasing need to run stochastic gradient methods in distributed environments, either because the data is inherently distributed (for instance when  collected by autonomous units such as smart phones or sensors) and processing it in a non-distributed way is impractical for real-time decision making, or the data is non-distributed but due to its volume  distributing the data to multiple computational units become unavoidable for scalability reasons. This motivates the study of the performance of SG methods on arbitrary networks where there the performance depends on the interplay between the bias, variance and network effects. In this paper, we focused on distributed stochastic gradient (D-SG) and its accelerated version (D-ASG) with constant and decaying stepsize. We provided a number of convergence results for D-SG and D-ASG that improve the existing convergence results. Our performance bounds captures the trade-offs in the bias, variance terms and the network effects and are illustrated by our numerical experiments. We also proposed a multi-stage variant of D-ASG with an optimal dependency to bias and variance terms. 
In this work, we considered synchronous algorithms which require nodes to update their local copies synchronously. As part of future work, it would be interesting to study momentum acceleration in the context of asynchronous stochastic gradient algorithms where the nodes can do updates without requiring synchronization between the nodes.}


\section*{Acknowledgements}
\rev{The authors are also grateful to the Associate Editor and three anonymous referees for helpful suggestions
and comments.
Mert G\"{u}rb\"{u}zbalaban's research is supported in part by the grants Office of Naval Research Award Number
N00014-21-1-2244, National Science Foundation (NSF)
CCF-1814888, NSF DMS-2053485, NSF DMS-1723085. 
Umut \c{S}im\c{s}ekli's research is partly supported by the French government under management of Agence Nationale de la Recherche as part of the ``Investissements d’avenir'' program, reference ANR-19-P3IA-0001 (PRAIRIE 3IA Institute).
Lingjiong Zhu is grateful to the partial support from a Simons Foundation Collaboration Grant
and the grant NSF DMS-2053454 from the National Science Foundation.}


\bibliography{robust}

\newpage
\appendix

\section{Intermediate Results}
\begin{lemma}\label{lem:C1:gamma}\cite[Corollary 9]{yuan2016convergence} 
Recall the definition of 
$$x^\infty = 
\left[\left(x_{1}^{\infty}\right)^{T},\left(x_{2}^{\infty}\right)^{T},
\cdots,\left(x_{N}^{\infty}\right)^{T}\right]^{T},
$$
which is the unique fixed point of
\begin{equation}
(I_{Nd}-\calW) x^\infty +\alpha \nabla F(x^\infty)= 0.
\end{equation}
If $\alpha\leq \min\{\frac{1+\lambda_N^W}{L}, \frac{1}{L+\mu}\}$, then 
\begin{equation*}
\| x_i^{\infty} - x_*\| \leq C_1 \frac{\alpha}{1-\gamma}=  \mathcal{O}\left(\frac{\alpha}{1-\gamma}\right),
\end{equation*} 
where $x_*$ is the solution to the optimization problem \eqref{opt-pbm} and recall the definitions of $C_1, f_i^*,$ and $\gamma$:
\begin{equation*}
C_1 = \sqrt{2L\sum_{i=1}^N \left(f_i\left(0\right) -f_i^*\right)}\cdot \left(1 + \frac{2(L+\mu)}{\mu   }\right), \quad f_i^* = \min_{x\in\mathbb{R}^d} f_i(x), \quad \gamma = \max\left\{\left|\lambda_2^W\right|, \left|\lambda_N^W\right|\right\}.
\end{equation*}
\end{lemma}
\begin{proof}
According to Corollary~9 in \cite{yuan2016convergence}, 
\begin{equation*}
\Vert x_{i}^{\infty}-x_*\Vert
\leq
\frac{c_{4}}{\sqrt{1-c_{3}^{2}}}
+\frac{\alpha\hat{D}}{1-\gamma},
\end{equation*}
where
\begin{equation*}
\hat{D}:=\sqrt{2L\sum_{i=1}^N \left(f_i\left(0\right) -f_i^*\right)},
\end{equation*}
and
\begin{equation*}
c_{4}:=\alpha^{3/2}\sqrt{\alpha+\delta^{-1}}\frac{L\hat{D}}{1-\gamma},
\qquad
c_{3}:=\sqrt{1-\alpha c_{2}+\alpha\delta-\alpha^{2}\delta c_{2}},
\end{equation*}
where
\begin{equation*}
\delta:=\frac{c_{2}}{2(1-\alpha c_{2})},
\qquad
c_{2}:=\frac{\mu L}{\mu+L}.
\end{equation*}
Hence, we can compute that
\begin{align*}
\frac{c_{4}}{\sqrt{1-c_{3}^{2}}}
&= \frac{\sqrt{2}\alpha\sqrt{\alpha+\delta^{-1}}L\hat{D}}{\sqrt{c_{2}}{(1-\gamma)}}=
\frac{\sqrt{2}\alpha\sqrt{\frac{2(\mu+L)}{\mu L}-\alpha}}{\sqrt{\frac{\mu L}{\mu+L}}}{\frac{L\hat{D}}{1-\gamma}}
\\
&\leq\frac{\sqrt{2}\alpha\sqrt{\frac{2(\mu+L)}{\mu L}}}{\sqrt{\frac{\mu L}{\mu+L}}}{\frac{L\hat{D}}{1-\gamma}
=\frac{2(\mu+L)}{\mu}\frac{\alpha\hat{D}}{1-\gamma}},
\end{align*}
where the first equality follows from the fact that 
\begin{equation*}
1-c_3^2 = \alpha c_2 + \alpha^2 \delta c_2 - \alpha \delta = \alpha c_2\left(1+\alpha \delta - \frac{\delta}{c_2}\right) = \frac{\alpha c_2}{2}.  
\end{equation*}
The proof is complete.
\end{proof}
\begin{lemma}\label{lem:average}
Recall the definitions of $x^{(k)}$ and $\bar{x}^{(k)}$ as
\begin{equation}
\begin{aligned}
x^{(k)} &= \left[\left(x_{1}^{(k)}\right)^{T},\left(x_{2}^{(k)}\right)^{T},\ldots,\left(x_{N}^{(k)}\right)^{T}\right]^{T}\in\mathbb{R}^{Nd}, \\
\bar{x}^{(k)} &=\frac{1}{N}\sum_{i=1}^{N}x_{i}^{(k)}\in\mathbb{R}^{d}.
\end{aligned}
\end{equation}
Then, for any $k\in\mathbb{N}$, we have
\begin{equation*}
\mathbb{E}\left\Vert\bar{x}^{(k)}-x_{\ast}\right\Vert^{2}
\leq\frac{1}{N}\mathbb{E}\left\Vert x^{(k)}-x^{\ast}\right\Vert^{2},
\end{equation*}
where $x_*$ is the solution to the optimization problem \eqref{opt-pbm} and $x^{\ast}=[x_{\ast}^{T},\ldots,x_{\ast}^{T}]^{T}$.
\end{lemma}
\begin{proof}
Note that the function $x\mapsto\Vert x-x_{\ast}\Vert^{2}$ is convex.
Therefore, by Jensen's inequality, 
\begin{equation*}
\left\Vert \bar{x}^{(k)}-x^{\ast}\right\Vert^{2}
=\left\Vert\frac{1}{N}\sum_{i=1}^{N}x_{i}^{(k)}-x^{\ast}\right\Vert^{2}
\leq\frac{1}{N}\sum_{i}^{N}\left\Vert x_{i}^{(k)}-x_{\ast}\right\Vert^{2}
=\frac{1}{N}\left\Vert x^{(k)}-x^{\ast}\right\Vert^{2}.
\end{equation*}
By taking the expectations, we obtain the desired result.
\end{proof}
\section{Proofs of Main Results in Section~\ref{sec:strongly:convex}}\label{sec:main:proofs}

\subsection{Proofs of Main Results in Section~\ref{sec:DGD}}\label{proof_sec:DGD}

Before we proceed to the proof of Theorem~\ref{thm:DGD:explicit},
let us first state the following result from \cite{aybat2019universally} which is stated for Nesterov's accelerated stochastic gradient method but holds for stochastic gradient descent as well, as it is the special case of Nesterov's algorithm for $\beta=0$.

\begin{lemma}[Lemma B.1, \cite{aybat2019universally}]\label{lemma_B1_MASG} Let $P = p \otimes I_{Nd}$ where $p \geq 0$ and recall the Lyapunov function $V_P(\xi)= \xi^ \top P  \xi$. Then we have
\begin{equation}\label{storage_update_Expec}
\begin{aligned}
& \E[{V}_P(\xi_{k+1})] - \rho^2 \E[{V}_P(\xi_{k})] \\
& \quad \leq \E\left[\begin{bmatrix} 
	\xi_k\\
    \nabla F\left(x^{(k)}\right)
\end{bmatrix}^\top
\begin{bmatrix} 
	A^\top P A - \rho^2 P & A^\top P B\\
    B^\top P A & B^\top P B
\end{bmatrix}
\begin{bmatrix} 
	\xi_k\\
    \nabla F\left(x^{(k)}\right)
\end{bmatrix}\right]+ N\sigma^2 \alpha^2 p.
\end{aligned}
\end{equation}	
\end{lemma}

Now, we are ready to prove Theorem~\ref{thm:DGD:explicit}.

\subsubsection{Proof of Theorem~\ref{thm:DGD:explicit}}

\begin{proof}
First note that $F_{\mathcal{W},\alpha}$ is $\mu$-strongly convex and $L_{\alpha}$-smooth
where $L_{\alpha}=\frac{1-\lambda_{N}^{W}}{\alpha}+L$. 

Next, note that, as it is shown in \cite{lessard2016analysis}, for every $\alpha \in (0, 2/L_\alpha)$, which is equivalent to $\alpha \in (0,(1+\lambda_N^W)/L)$, there exists $p > 0$ such that the following matrix inequality holds with $\rho(\alpha) = \max\{|1-\alpha \mu|, |1-\alpha L_\alpha|\} = \max\{|1-\alpha \mu|, |\lambda_N^W - \alpha L|\}$:
\begin{equation*}
\begin{bmatrix} 
	2\mu L_\alpha I_d & -(\mu+L_\alpha)I_d\\
    -(\mu+L_\alpha)I_d & 2I_d
\end{bmatrix}
\succeq 
\begin{bmatrix} 
	A^\top P A - \rho(\alpha)^2 P & A^\top P B\\
    B^\top P A & B^\top P B
\end{bmatrix}.	
\end{equation*}
As a consequence, and by using Lemma~\ref{lemma_B1_MASG}, we have
\begin{equation}\label{ineq1_B1}
\begin{aligned}
&\begin{bmatrix} 
	\xi_k\\
    \nabla F\left(x^{(k)}\right)
\end{bmatrix}^\top
\begin{bmatrix} 
	2\mu L_\alpha I_d & -(\mu+L_\alpha)I_d\\
    -(\mu+L_\alpha)I_d & 2I_d
\end{bmatrix}
\begin{bmatrix} 
	\xi_k\\
    \nabla F\left(x^{(k)}\right)
\end{bmatrix} \\
&\geq
\begin{bmatrix} 
	\xi_k\\
    \nabla F\left(x^{(k)}\right)
\end{bmatrix}^\top 
\begin{bmatrix} 
	A^\top P A - \rho(\alpha)^2 P & A^\top P B\\
    B^\top P A & B^\top P B
\end{bmatrix}
\begin{bmatrix} 
	\xi_k\\
    \nabla F\left(x^{(k)}\right)
\end{bmatrix}\\
& \geq \E[{V}_P(\xi_{k+1})] - \rho(\alpha)^2 \E[{V}_P(\xi_{k})] - \sigma^2 \alpha^2 p.
\end{aligned}
\end{equation}
Finally, note that, by using Theorem 2.1.12 in \cite{nesterov_convex}, we obtain
\begin{equation*}
\begin{bmatrix} 
	\xi_k\\
    \nabla F\left(x^{(k)}\right)
\end{bmatrix}^\top
\begin{bmatrix} 
	2\mu L_\alpha I_d & -(\mu+L_\alpha)I_d\\
    -(\mu+L_\alpha)I_d & 2I_d
\end{bmatrix}
\begin{bmatrix} 
	\xi_k\\
    \nabla F\left(x^{(k)}\right)
\end{bmatrix} \leq 0.	
\end{equation*}
Plugging this in \eqref{ineq1_B1} and dividing both sides by $p$, implies
\begin{equation}
\rho(\alpha)^2 \E \left[\left\|x^{(k)} - x^\infty\right\|^2 \right ] + N \sigma^2 \alpha^2	\geq \E \left [\left\|x^{(k+1)} - x^\infty\right\|^2 \right ].
\end{equation}
Finally, by iterating over $k$, we obtain
\begin{equation*}
\mathbb{E}\left[\left\Vert x^{(k)}-x^{\infty}\right\Vert^{2}\right]
\leq \rho(\alpha)^{2k}\left\Vert x^{(0)}-x^{\infty}\right\Vert^{2}
+\alpha^{2}\sigma^{2}N\frac{1-\rho(\alpha)^{2k}}{1-\rho(\alpha)^{2}}.
\end{equation*}
We also achieve the bound on robustness using the definition of $J_{\infty}(\alpha)$:
\begin{equation*}
J_{\infty}(\alpha)=\frac{1}{\sigma^2 N} \limsup_{k\to\infty} \mbox{Var}\left(x^{(k)}-x^{\infty}\right)
\leq\frac{1}{\sigma^2 N} \limsup_{k\to\infty}\mathbb{E}\left[\left\Vert x^{(k)}-x^{\infty}\right\Vert^{2}\right].
\end{equation*}
The proof is complete.
\end{proof}

\subsubsection{Proof of Corollary~\ref{DSG_cor}}

\begin{proof}
Note that we have
$\left\Vert x^{(k)}-x^{\ast}\right\Vert^{2}
\leq 
2\left\Vert x^{(k)}-x^{\infty}\right\Vert^{2}
+2\left\Vert x^{\infty}-x^{\ast}\right\Vert^{2}$,
and, for the case that $\alpha < \frac{1}{L+\mu}$, we also have $\Vert x^{\infty}-x^{\ast}\Vert\leq\frac{\alpha C_{1}\sqrt{N}}{(1-\gamma)}$
from \eqref{eqn:asymp_suboptim}, 
which yields that
\begin{equation*}
\mathbb{E}\left[\left\Vert x^{(k)}-x^{\ast}\right\Vert^{2}\right]
\leq
2\rho^{2k}\left\Vert x^{(0)}-x^{\infty}\right\Vert^{2}
+2\alpha^{2}\sigma^{2}N\frac{1-\rho^{2k}}{1-\rho^{2}}
+2\frac{\alpha^{2}C_{1}^{2}N}{(1-\gamma)^{2}}.
\end{equation*}
\end{proof}

\subsubsection{Proof of Proposition~\ref{prop:dsg-average-optimal-rate}}

{\color{black}
We recall that the average at $k$-th iteration is given by
$\bar{x}^{(k)}:=\frac{1}{N}\sum_{i=1}^{N}x_{i}^{(k)}$.
Since $\mathcal{W}$ is doubly stochastic, we get
\begin{equation} \label{eq:avg-iterates}
\bar{x}^{(k+1)}=\bar{x}^{(k)}-\alpha\frac{1}{N}\sum_{i=1}^{N}\nabla f_{i}\left(x_{i}^{(k)}\right)
-\alpha\bar{\xi}^{(k+1)},
\end{equation}
where 
\begin{equation*}
\bar{\xi}^{(k+1)}:=\frac{1}{N}\sum_{i=1}^{N}\left(\tilde{\nabla}f_{i}\left(x_{i}^{(k)}\right)-\nabla f_{i}\left(x_{i}^{(k)}\right)\right),
\end{equation*}
satisfies 
\begin{equation}\label{bar:grad:noise}
\mathbb{E}\left[\bar{\xi}^{(k+1)}\Big|\mathcal{F}_{k}\right]=0,
\qquad
\mathbb{E}\left\Vert\bar{\xi}^{(k+1)}\right\Vert^{2}\leq\frac{\sigma^{2}}{N}.
\end{equation}
We can deduce from \eqref{eq:avg-iterates} that
\begin{equation*}
\bar{x}^{(k+1)}=\bar{x}^{(k)}-\alpha \nabla f\left(\bar{x}^{(k)}\right)
+\alpha\mathcal{E}_{k+1}-\alpha\bar{\xi}^{(k+1)},
\end{equation*}
where
\begin{equation*}
\mathcal{E}_{k+1}
:= \nabla f\left(\bar{x}^{(k)}\right)-\frac{1}{N}\sum_{i=1}^{N}\nabla f_{i}\left(x_{i}^{(k)}\right).
\end{equation*}

First, we will show that the error term $\mathcal{E}_{k+1}$ is small.
The following result essentially follows from Lemma~7 in \cite{gurbuzbalaban2020decentralized}
and hence the proof is omitted here.

\begin{lemma}[Lemma~7 in \cite{gurbuzbalaban2020decentralized}]\label{lem:2}
Assume that $\alpha<\frac{1+\lambda_{N}^{W}}{L}$ and $\mu\alpha(1+\lambda_{N}^{W}-\alpha L)<1$.  
For any $k$, we have
\begin{equation*}
\mathbb{E}\left\Vert\mathcal{E}_{k+1}\right\Vert^{2}
\leq
\frac{4L^{2}\gamma^{2k}}{N}\mathbb{E}\left\Vert x^{(0)}\right\Vert^{2}
+\frac{4L^{2}D^{2}\alpha^{2}}{N(1-\gamma)^{2}}
+\frac{4L^{2}\sigma^{2}\alpha^{2}}{(1-\gamma^{2})},
\end{equation*} 
where $D^{2}$ is defined in \eqref{eqn:D1}.
\end{lemma}

Let us define $x_{k}$ as the iterates of the centralized algorithm:
\begin{equation*}
x_{k+1}=x_{k}-\alpha  \nabla f\left(x_{k}\right)
-\alpha\bar{\xi}^{(k+1)},
\end{equation*}
with $x_{0}=\bar{x}^{(0)}$.
In the next lemma, we will show that the average of iterates $\bar{x}^{(k)}$
and the iterates of the centralized algorithm $x_{k}$ are close to each other.

\begin{lemma}\label{lem:central:approx}
Assume that $\alpha<\frac{1+\lambda_{N}^{W}}{L}$ and $\mu\alpha(1+\lambda_{N}^{W}-\alpha L)<1$.  
For any $k$, we have
\begin{align}
\mathbb{E}\left\Vert\bar{x}^{(k)}-x_{k}\right\Vert^{2}
&\leq\alpha\left(\frac{\alpha}{\mu(1-\frac{\alpha L}{2})}+\frac{(1+\alpha L)^{2}}{\mu^{2}(1-\frac{\alpha L}{2})^{2}}\right)
\left(\frac{4L^{2}D^{2}\alpha}{N(1-\gamma)^{2}}
+\frac{4L^{2}\sigma^{2}\alpha}{(1-\gamma^{2})}\right)
\nonumber
\\
&\qquad\qquad
+\frac{\gamma^{2k}-
\left(1-\alpha\mu\left(1-\frac{\alpha L}{2}\right)\right)^{k}}
{\gamma^{2}-1+\alpha\mu\left(1-\frac{\alpha L}{2}\right)}
\frac{4L^{2}\gamma^{2}}{N}\mathbb{E}\left\Vert x^{(0)}\right\Vert^{2}.\label{bar:x:x:k}
\end{align}
\end{lemma}

\begin{proof}
{\color{black}The proof of Lemma~\ref{lem:central:approx} will be provided
in Appendix~\ref{sec:technical}.}
\end{proof}


Next, we quote the following result from \cite{gurbuzbalaban2020decentralized}
which provides an bound on the distance between $x_{i}^{(k)}$ and $\bar{x}^{(k)}$
for every $1\leq i\leq N$.

\begin{lemma}[Lemma~6 in \cite{gurbuzbalaban2020decentralized}]\label{bar:x:i:distance}
In the setting of Lemma \ref{lem:2}, for any $k$ and $i=1,\ldots,N$, we have
\begin{equation*}
\mathbb{E}\left\Vert x_{i}^{(k)}-\bar{x}^{(k)}\right\Vert^{2}
\leq
\sum_{i=1}^{N}\mathbb{E}\left\Vert x_{i}^{(k)}-\bar{x}^{(k)}\right\Vert^{2}
\leq
4\gamma^{2k}\mathbb{E}\left\Vert x^{(0)}\right\Vert^{2}
+\frac{4D^{2}\alpha^{2}}{(1-\gamma)^{2}}
+\frac{4\sigma^{2}N\alpha^{2}}{(1-\gamma^{2})},
\end{equation*}
where $D$ is defined in \eqref{eqn:D1}. 
\end{lemma}

\textbf{Completing the proof of Proposition~\ref{prop:dsg-average-optimal-rate}.} Finally, we are ready to complete the proof of Proposition~\ref{prop:dsg-average-optimal-rate}. 
First, we notice that 
\begin{align*}
\mathbb{E}\left\Vert\bar{x}^{(k)}-x_{\ast}\right\Vert^{2}
\leq
2\mathbb{E}\left\Vert\bar{x}^{(k)}-x_{k}\right\Vert^{2}
+2\mathbb{E}\left\Vert x_{k}-x_{\ast}\right\Vert^{2},
\end{align*}
and for every $i=1,2,\ldots,N$,
\begin{align*}
\mathbb{E}\left\Vert x_{i}^{(k)}-x_{\ast}\right\Vert^{2}
&\leq
2\mathbb{E}\left\Vert\bar{x}^{(k)}-x_{\ast}\right\Vert^{2}
+2\mathbb{E}\left\Vert x_{i}^{(k)}-\bar{x}^{(k)}\right\Vert^{2}
\\
&\leq
4\mathbb{E}\left\Vert\bar{x}^{(k)}-x_{k}\right\Vert^{2}
+4\mathbb{E}\left\Vert x_{k}-x_{\ast}\right\Vert^{2}
+2\mathbb{E}\left\Vert x_{i}^{(k)}-\bar{x}^{(k)}\right\Vert^{2}.
\end{align*}
By Proposition~4.3. in \cite{StrConvex}, we have
for any $\alpha\leq\frac{2}{L+\mu}$,
\begin{align}
\mathbb{E}\Vert x_{k}-x_{\ast}\Vert^{2}
&\leq(1-\alpha\mu)^{2k}
\Vert x_{0}-x_{\ast}\Vert^{2}
+\frac{1-(1-\alpha\mu)^{2k}}{1-(1-\alpha\mu)^{2}}\frac{\alpha^{2}\sigma^{2}}{N}
\nonumber
\\
&=(1-\alpha\mu)^{2k}
\Vert x_{0}-x_{\ast}\Vert^{2}
+\frac{1}{2}\frac{1-(1-\alpha\mu)^{2k}}{\mu(1-\alpha\mu/2)}\frac{\alpha\sigma^{2}}{N}.
\end{align}
By \eqref{bar:x:x:k}, we have
an upper bound for $\mathbb{E}\left\Vert\bar{x}^{(k)}-x_{k}\right\Vert^{2}$, 
and by Lemma~\ref{bar:x:i:distance}, we 
have an upper bound for $\mathbb{E}\left\Vert x_{i}^{(k)}-\bar{x}^{(k)}\right\Vert^{2}$, 
which completes the proof. 
}

\subsection{Proofs of Main Results in Section~\ref{subsec:dasg}}

Before we proceed to the proof of Theorem~\ref{thm:general:DASG},
let us state the following result from \cite{aybat2019universally}.
\begin{lemma}[Lemma 2.2, \cite{aybat2019universally}]\label{lem:aybat} 
Consider to ASG iterates to minimize the function $F_{\mathcal{W},\alpha}$ in \eqref{def-pen-obj}.
Assume there exist $\rho\in(0,1)$ and a positive semi-definite $2\times 2$
matrix $\tilde{P}$ such that
\begin{equation*}
\rho^{2}\tilde{X}_{1}+(1-\rho^{2})\tilde{X}_{2}
\succeq 
\left[\begin{array}{cc}
\tilde{A}_{\text{dasg}}^{T}\tilde{P}\tilde{A}_{\text{dasg}}-\rho^{2}\tilde{P} & \tilde{A}_{\text{dasg}}^{T}\tilde{P}\tilde{B}_{\text{dasg}}
\\
\tilde{B}_{\text{dasg}}^{T}\tilde{P}\tilde{A}_{\text{dasg}} & \tilde{B}_{\text{dasg}}^{T}\tilde{P}\tilde{B}_{\text{dasg}}
\end{array}\right],
\end{equation*}
where
\begin{equation*}
\tilde{X}_{1}:=
\left[\begin{array}{ccc}
\frac{\beta^{2}\mu}{2} & \frac{-\beta^{2}\mu}{2} & \frac{-\beta}{2}
\\
\frac{-\beta^{2}\mu}{2} & \frac{\beta^{2}\mu}{2} & \frac{\beta}{2}
\\
\frac{-\beta}{2} & \frac{\beta}{2} & \frac{\alpha(2-L_{\alpha}\alpha)}{2}
\end{array}\right],
\quad
\tilde{X}_{2}:=
\left[\begin{array}{ccc}
\frac{(1+\beta)^{2}\mu}{2} & \frac{-\beta(1+\beta)\mu}{2} & \frac{-(1+\beta)}{2}
\\
\frac{-\beta(1+\beta)\mu}{2} & \frac{\beta^{2}\mu}{2} & \frac{\beta}{2}
\\
\frac{-(1+\beta)}{2} & \frac{\beta}{2} & \frac{\alpha(2-L_{\alpha}\alpha)}{2}
\end{array}\right].
\end{equation*}
Let $P=\tilde{P}\otimes I_{Nd}$. Then, for every $k\geq 0$,
\begin{equation*}
\mathbb{E}[V_{P, \alpha, 1}(\xi_{k})]\leq
\rho^{2}\mathbb{E}[V_{P, \alpha, 1}(\xi_{k-1})]
+\alpha^{2}\sigma^{2}N\left(\tilde{P}_{11}+\frac{L_{\alpha}}{2}\right).
\end{equation*}
\end{lemma}

Now, we are ready to prove Theorem~\ref{thm:general:DASG}.

\subsubsection{Proof of Theorem~\ref{thm:general:DASG}}

\begin{proof}
First recall that $F_{\mathcal{W},\alpha}$
is $\mu$-strongly convex and $L_{\alpha}$-smooth
where $L_{\alpha}=\frac{1-\lambda_{N}^{W}}{\alpha}+L$. 

Next, Lemma~\ref{lem:aybat} implies that
\begin{equation*}
\mathbb{E}\left[V_{P, \alpha, 1}(\xi_{k})\right]\leq
\rho^{2k}V_{P, \alpha, 1}(\xi_{0})
+\frac{1}{1-\rho^{2}}\alpha^{2}\sigma^{2}N\left(\tilde{P}_{11}+\frac{L_{\alpha}}{2}\right).
\end{equation*}
By the $\mu$-strong convexity of $F_{\calW,\alpha}$ and the fact that $\nabla F_{\calW,\alpha}(x^\infty)=0$ , we have 
\begin{equation*}
\left\| x^{(k)} -x^\infty \right\|^2 \leq \frac{2}{\mu}\left[ F_{\calW,\alpha}\left(x^{(k)}\right) - F_{\calW,\alpha}\left(x^\infty\right)\right]\leq 2\frac{V_{P,\alpha,1}(\xi_{k})}{\mu}.
\end{equation*}
Also, by the definition of $J_{\infty}(\alpha)$,
\begin{equation*}
J_{\infty}(\alpha)=\frac{1}{\sigma^2 N} \limsup_{k\to\infty} \mbox{Var}\left(x^{(k)}-x^{\infty}\right)
\leq\frac{1}{\sigma^2 N} \limsup_{k\to\infty}\mathbb{E}\left[\left\Vert x^{(k)}-x^{\infty}\right\Vert^{2}\right].
\end{equation*}
The proof is complete.
\end{proof}

\subsubsection{Proof of Corollary~\ref{cor:general:DASG}}

\begin{proof}
{\color{black}If $\alpha \leq \frac{1}{L+\mu}$, then we have
\begin{equation*}
\left\Vert x^{(k)}-x^{\ast}\right\Vert^{2}
\leq 
2\left\Vert x^{(k)}-x^{\infty}\right\Vert^{2}
+2\left\Vert x^{\infty}-x^{\ast}\right\Vert^{2},
\end{equation*}
and $\Vert x^{\infty}-x^{\ast}\Vert\leq\frac{\alpha C_{1}\sqrt{N}}{(1-\gamma)}$
from \eqref{eqn:asymp_suboptim}.
Also, by the proof of Lemma~\ref{lem:average}, 
\begin{equation*}
\left\Vert\bar{x}^{(k)}-x_{\ast}\right\Vert^{2}
\leq\frac{1}{N}\left\Vert x^{(k)}-x^{\ast}\right\Vert^{2}.
\end{equation*}
The proof is complete.}
\end{proof}

\subsubsection{Proof of Theorem~\ref{thm-rate-dasg}}

\begin{proof}
D-ASG reduces to the iterations \eqref{eq-centr-asg} which are equivalent to applying non-distributed ASG to minimize the function $F_{\calW,\alpha} \in S_{\mu,L_\alpha}(\R^{Nd})$. Therefore, applying \cite[Proposition 4.6]{StrConvex} and \cite[Corollary 4.9]{StrConvex} from the literature for non-distributed ASG, we obtain 
\begin{equation}\label{ineq-dasg-lyapunov-2}
\E \left[V_{S,\alpha}\left(\xi_{k+1}\right)\right] 
\leq \left(1-\sqrt{\alpha \mu}\right)  \mathbb{E}V_{S,\alpha}\left(\xi_{k}\right)
+ \frac{\sigma^2 N \alpha}{2}\left(1+\alpha L_\alpha \right),
\end{equation}
which yields
\begin{equation*}
\E \left[V_{S,\alpha}\left(\xi_{k}\right)\right] 
\leq \left(1-\sqrt{\alpha \mu}\right)^{k}  V_{S,\alpha}\left(\xi_{0}\right)
+ \frac{\sigma^2 N \alpha}{2\sqrt{\alpha\mu}}\left(1+\alpha L_\alpha \right),
\end{equation*}
provided that $ \alpha \in (0,\frac{1}{L_\alpha}]$ where $L_\alpha = \frac{1-\lambda_N^W}{\alpha} + L$ is the smoothness constant of $F_{\calW,\alpha}$. It can be checked that if $\alpha \in (0, \frac{\lambda_N^W}{L}]$ then, the condition $ \alpha \in (0,\frac{1}{L_\alpha}]$ is satisfied. Plugging the value of $L_\alpha$ into \eqref{ineq-dasg-lyapunov-2} proves 
\beq \E \left[V_{S,\alpha}\left(\xi_{k}\right)\right] \leq \left(1-\sqrt{\alpha \mu}\right)^{k}  V_{S,\alpha}\left(\xi_{0}\right)
+ \frac{\sigma^2 N \sqrt{\alpha}}{2\sqrt{\mu}}\left(2-\lambda_N^W +\alpha L \right),\label{ineq-dasg-lyapunov}
\eeq
for any $k\geq 0$.


By the $\mu$-strong convexity of $F_{\calW,\alpha}$ and the fact that $\nabla F_{\calW,\alpha}(x^\infty)=0$ , we have 
$$\left\| x^{(k)} -x^\infty \right\|^2 \leq \frac{2}{\mu}\left[ F_{\calW,\alpha}\left(x^{(k)}\right) - F_{\calW,\alpha}\left(x^\infty\right)\right].$$
Therefore, \eqref{ineq-dasg-lyapunov} implies \eqref{ineq-to-prove-1}. 
Finally, by the definition of $J_{\infty}(\alpha)$,
\begin{equation*}
J_{\infty}(\alpha)=\frac{1}{\sigma^2 N} \limsup_{k\to\infty} \mbox{Var}\left(x^{(k)}-x^{\infty}\right)
\leq\frac{1}{\sigma^2 N} \limsup_{k\to\infty}\mathbb{E}\left[\left\Vert x^{(k)}-x^{\infty}\right\Vert^{2}\right].
\end{equation*}
The proof is complete.
\end{proof}


\subsubsection{Proof of Corollary~\ref{cor:rate:dasg}}

\begin{proof}
{\color{black}If $\alpha \leq \frac{1}{L+\mu}$, then we obtain \eqref{ineq-to-prove-2} by applying \eqref{eqn:asymp_suboptim} 
and moreover, we obtain \eqref{ineq-to-prove-1} 
by applying Lemma~\ref{lem:average}.}
\end{proof}

\subsubsection{Proof of Proposition~\ref{prop:dsg-average-optimal-rate-DASG}}

{\color{black}We recall that the averages at $k$-th iteration are given by
$\bar{x}^{(k)}:=\frac{1}{N}\sum_{i=1}^{N}x_{i}^{(k)}$
and $\bar{y}^{(k)}:=\frac{1}{N}\sum_{i=1}^{N}y_{i}^{(k)}$.
Since $\mathcal{W}$ is doubly stochastic, we get
\begin{align}\label{eq:avg-iterates:DASG}
&\bar{x}^{(k+1)}=\bar{y}^{(k)}-\alpha\frac{1}{N}\sum_{i=1}^{N}\nabla f_{i}\left(y_{i}^{(k)}\right)
-\alpha\bar{\xi}^{(k+1)},
\\
&\bar{y}^{(k)}=(1+\beta)\bar{x}^{(k)}-\beta\bar{x}^{(k-1)},
\nonumber
\end{align}
where 
\begin{equation*}
\bar{\xi}^{(k+1)}:=\frac{1}{N}\sum_{i=1}^{N}\left(\tilde{\nabla}f_{i}\left(y_{i}^{(k)}\right)-\nabla f_{i}\left(y_{i}^{(k)}\right)\right),
\end{equation*}
satisfies 
\begin{equation}\label{bar:grad:noise:DASG}
\mathbb{E}\left[\bar{\xi}^{(k+1)}\Big|\mathcal{F}_{k}\right]=0,
\qquad
\mathbb{E}\left\Vert\bar{\xi}^{(k+1)}\right\Vert^{2}\leq\frac{\sigma^{2}}{N}.
\end{equation}

We will establish several lemmas for completing the proof of Proposition \ref{prop:dsg-average-optimal-rate-DASG}. We start with stating and proving the following lemma which provides a bound on the $L_2$ distance between the local variables $x_i^{(k)}$ and the node averages $\bar{x}^{(k)}$.}

\begin{lemma}\label{gradient:average:error:DASG}
{\color{black}
Assume the conditions in Proposition~\ref{prop:dsg-average-optimal-rate-DASG} hold. 
For any $k$, we have
\begin{align*}
\sum_{i=1}^{N}\mathbb{E}\left\Vert x_{i}^{(k)}-\bar{x}^{(k)}\right\Vert^{2} 
&\leq
8\gamma^{2k}\left(4\frac{ V_{S,\alpha}\left(\xi_{0}\right)}{\mu}
+ \frac{2\sigma^2 N \sqrt{\alpha}}{\mu\sqrt{\mu}}\left(2-\lambda_N^W +\alpha L \right) +  \frac{2C_{1}^{2}N\alpha^2}{(1-\gamma)^2}
+\Vert x^{\ast}\Vert^{2}\right)
\\
&\qquad
+\frac{4D_{y}^{2}\alpha^{2}}{(1-\gamma)^{2}}
+\frac{4\sigma^{2}N\alpha^{2}}{(1-\gamma)^{2}}
+\frac{8C_{0}\alpha}{(1-\gamma)^{2}},
\end{align*} 
where $C_{0}$ is defined in Lemma~\ref{lem:difference} and $D_{y}$ is defined in \eqref{D:y:eqn}.
}
\end{lemma}

\begin{proof}
{\color{black}
The proof of Lemma~\ref{gradient:average:error:DASG} will be provided
in Appendix~\ref{sec:technical}.}
\end{proof}

{\color{black}We can deduce from \eqref{eq:avg-iterates:DASG} that
\begin{align*}
&\bar{x}^{(k+1)}=\bar{x}^{(k)}-\alpha \nabla f\left(\bar{y}^{(k)}\right)
+\alpha\mathcal{E}_{k+1}-\alpha\bar{\xi}^{(k+1)},
\\
&\bar{y}^{(k)}=(1+\beta)\bar{x}^{(k)}-\beta\bar{x}^{(k-1)},
\end{align*}
where
\begin{equation}
\mathcal{E}_{k+1}
:= \nabla f\left(\bar{y}^{(k)}\right)-\frac{1}{N}\sum_{i=1}^{N}\nabla f_{i}\left(y_{i}^{(k)}\right).
\label{def-mathcalEk}
\end{equation}
Notice that we are abusing
the notation here; and $\mathcal{E}_{k+1}$
is used to denote the error term for D-SG as well.
Next, we will show that the error term $\mathcal{E}_{k+1}$ is small.} 

\begin{lemma}\label{lem:E:D-ASG}
{\color{black}Assume the conditions in Proposition~\ref{prop:dsg-average-optimal-rate-DASG} hold. 
For any $k$, we have
\begin{align*}
\mathbb{E}\left\Vert\mathcal{E}_{k+1}\right\Vert^{2}
&\leq
\frac{2}{N}L^{2}
\left((1+\beta)^{2}+\beta^{2}\right)
\Bigg[4D_{y}^{2}\alpha^{2}\frac{1}{(1-\gamma)^{2}}
+\frac{4\sigma^{2}N\alpha^{2}}{(1-\gamma)^{2}}
+\frac{8C_{0}}{(1-\gamma)^{2}}\alpha
\\
&\qquad
+8\gamma^{2(k-1)}\left(4\frac{ V_{S,\alpha}\left(\xi_{0}\right)}{\mu}
+ \frac{2\sigma^2 N \sqrt{\alpha}}{\mu\sqrt{\mu}}\left(2-\lambda_N^W +\alpha L \right) +  \frac{2C_{1}^{2}N\alpha^2}{(1-\gamma)^2}
+\Vert x^{\ast}\Vert^{2}\right)\Bigg],
\end{align*} 
}
{\color{black}
where $C_{0}$ is defined in Lemma~\ref{lem:difference}
and $\mathcal{E}_{k}$ is defined in \eqref{def-mathcalEk}}.
\end{lemma}

\begin{proof}
{\color{black}
The proof of Lemma~\ref{lem:E:D-ASG} will be provided
in Appendix~\ref{sec:technical}.}
\end{proof}

{\color{black}We recall from \eqref{eqn:V:S:alpha:bar}
that
\begin{equation}
V_{\bar{S},\alpha}(\bar{\xi}) := \bar{\xi}^T  \bar{S}_{\alpha} \bar{\xi}+ f\left(\bar{T}\bar{\xi} + x_{\ast}\right) - f(x_{\ast}),
\label{def-Lyap-fun-V-barS-alpha}
\end{equation}
where $\bar{S}_{\alpha}=\tilde{S}_{\alpha}\otimes I_{d}$.
We can represent the average of the D-ASG iterations as the dynamical system 
\begin{equation}\label{unified-dyn-sys-average}
\bar{\xi}_{k+1}=A\bar{\xi}_{k}+B \tilde{\nabla} f\left(C\bar{\xi}_{k}\right)+D_{k+1},
\end{equation} 
where $\xi_{k}$ is the state, and $A, B, C$ are system matrices that are appropriately chosen such that
\begin{equation}\label{eqn:xi:DASG:average}
\xi_{k}:=\left[\left(\bar{x}^{(k)}-x_{\ast}\right)^{T},\left(\bar{x}^{(k-1)}-x_{\ast}\right)^{T}\right]^{T}, 
\end{equation}
and 
$A = \tilde{A}_{\text{dasg}}\otimes I_{d}, B:=\tilde{B}_{\text{dasg}}\otimes I_{d}, C:=\tilde{C}_{\text{dasg}}\otimes I_{d}$
where
\begin{equation}\label{ABC:DASG:average}
\tilde{A}_{\text{dasg}}
:=\left[\begin{array}{cc}
1+\beta & -\beta
\\
1 & 0
\end{array}\right],
\quad
\tilde{B}_{\text{dasg}}
:=\left[\begin{array}{c}
-\alpha
\\
0
\end{array}\right],
\quad
\tilde{C}_{\text{dasg}}
:=\left[\begin{array}{cc}
1+\beta & -\beta
\end{array}\right],
\end{equation}
and
\begin{equation}
D_{k}:=\left[\begin{array}{c}
\alpha\mathcal{E}_{k}
\\
0
\end{array}\right],\label{def-D-k-vector}
\end{equation}
where $\mathcal{E}_{k}$ is defined in \eqref{def-mathcalEk}.}
{\color{black}We will make use of the Lyapunov function \eqref{def-Lyap-fun-V-barS-alpha} to establish the convergence of the averaged iterates in the remaining part of the proof. In the next result, we will obtain a helper lemma that shows that the difference between the consecutive iterates $x^{(k)}-x^{(k-1)}$ is bounded in $L_2$ with a bound that is proportional to the stepsize $\alpha$.}

\begin{lemma}\label{lem:difference}
{\color{black}Consider running D-ASG method with $\alpha \in (0,\frac{\lambda_N^W}{L}]$, $\beta = \frac{ 1-\sqrt{\alpha\mu}}{1+\sqrt{\alpha\mu}}$ and initialization $x^{(0)}=x^{(-1)}=0$. Then, we have
\begin{equation}
\sup_{k\geq 0}\mathbb{E}\left\Vert x^{(k)}-x^{(k-1)}\right\Vert^{2}
\leq\frac{2C_{0}}{\beta^{2}}\alpha,
\end{equation}
for a positive constant $C_0$ that can be made explicit. Furthermore, $C_0$ is such that $C_0 = \mathcal{O}(1)$ as $\alpha \to 0$.} 
\end{lemma}

\begin{proof}
{\color{black}
The proof of Lemma~\ref{lem:difference} will be provided
in Appendix~\ref{sec:technical}.}
\end{proof}

\begin{lemma}\label{lem:C:epsilon}{\color{black}Assume the conditions in Proposition~\ref{prop:dsg-average-optimal-rate-DASG} hold.} 
{\color{black}For any $\epsilon>0$, there exists $C_{\epsilon}>0$ such that
\begin{equation}
\E \left[V_{\bar{S},\alpha}\left(\bar{\xi}_{k+1}\right)\right] 
\leq 
(1+\epsilon)\E \left[V_{\bar{S},\alpha}\left(\bar{\xi}_{k+1}-D_{k+1}\right)\right]
+C_{\epsilon}\mathbb{E}\Vert D_{k+1}\Vert^{2}, \label{ineq-lyap-pert}
\end{equation}
where
\begin{equation}
C_{\epsilon}:=\frac{1}{2\epsilon}\max\left(4\left(\frac{1}{\alpha} + \frac{\mu}{2}-\frac{\sqrt{\mu}}{\sqrt{\alpha}}\right), \frac{L^{2}}{\mu}\right) + \frac{L}{2}+\frac{1}{\alpha} + \frac{\mu}{2}-\frac{\sqrt{\mu}}{\sqrt{\alpha}},
\end{equation}
}
{\color{black}$\bar{\xi}_{k+1}$ is defined by \eqref{unified-dyn-sys-average} and $D_{k+1}$ is defined by \eqref{def-D-k-vector}}.
\end{lemma}

\begin{proof}
{\color{black}
The proof of Lemma~\ref{lem:C:epsilon} will be provided
in Appendix~\ref{sec:technical}.}
\end{proof}

{\color{black}
\textbf{Completing the proof of Proposition~\ref{prop:dsg-average-optimal-rate-DASG}.}
D-ASG reduces to the iterations which are equivalent to applying non-distributed ASG to minimize the function $f \in S_{\mu,L_\alpha}(\R^{d})$. Therefore, applying \cite[Proposition 4.6]{StrConvex} and \cite[Corollary 4.9]{StrConvex} from the literature for non-distributed ASG, we obtain
\begin{equation}\label{ineq-dasg-lyapunov-2-average}
\E \left[V_{\bar{S},\alpha}\left(\bar{\xi}_{k+1}-D_{k+1}\right)\right] 
\leq \left(1-\sqrt{\alpha \mu}\right)  \mathbb{E}V_{\bar{S},\alpha}\left(\bar{\xi}_{k}\right)
+ \frac{\sigma^2  \alpha}{2N}\left(1+\alpha L \right),
\end{equation}
which yields
\begin{align*}
\E \left[V_{\bar{S},\alpha}\left(\bar{\xi}_{k+1}\right)\right] 
&\leq 
(1+\epsilon)\E \left[V_{\bar{S},\alpha}\left(\bar{\xi}_{k+1}-D_{k+1}\right)\right]
+C_{\epsilon}\mathbb{E}\Vert D_{k+1}\Vert^{2}
\\
&\leq
(1+\epsilon)\left(\left(1-\sqrt{\alpha \mu}\right)  \mathbb{E}V_{\bar{S},\alpha}\left(\bar{\xi}_{k}\right)
+ \frac{\sigma^2  \alpha}{2N}\left(1+\alpha L \right)\right)
+C_{\epsilon}\mathbb{E}\Vert D_{k+1}\Vert^{2}.
\end{align*}
Let us take $\epsilon=\frac{1}{2}\sqrt{\alpha\mu}$, then we get
\begin{align*}
\E \left[V_{\bar{S},\alpha}\left(\bar{\xi}_{k+1}\right)\right] 
&\leq
\left(1+\frac{\sqrt{\alpha\mu}}{2}\right)\left(\left(1-\sqrt{\alpha \mu}\right)  \mathbb{E}V_{\bar{S},\alpha}\left(\bar{\xi}_{k}\right)
+ \frac{\sigma^2  \alpha}{2N}\left(1+\alpha L \right)\right)
+C_{\frac{\sqrt{\alpha\mu}}{2}}\mathbb{E}\Vert D_{k+1}\Vert^{2}
\\
&\leq
\left(1-\frac{\sqrt{\alpha\mu}}{2}\right)\mathbb{E}V_{\bar{S},\alpha}\left(\bar{\xi}_{k}\right)
+\left(1+\frac{\sqrt{\alpha\mu}}{2}\right)\frac{\sigma^2  \alpha}{2N}\left(1+\alpha L \right)
+C_{\frac{\sqrt{\alpha\mu}}{2}}\mathbb{E}\Vert D_{k+1}\Vert^{2},
\end{align*}
and by Lemma~\ref{lem:C:epsilon}, we have
\begin{align*}
C_{\frac{\sqrt{\alpha\mu}}{2}}
&=\frac{1}{\sqrt{\alpha\mu}}\max\left(4\left(\frac{1}{\alpha} + \frac{\mu}{2}-\frac{\sqrt{\mu}}{\sqrt{\alpha}}\right), \frac{L^{2}}{\mu}\right) + \frac{L}{2}+\frac{1}{\alpha} + \frac{\mu}{2}-\frac{\sqrt{\mu}}{\sqrt{\alpha}}
\\
&\leq
\frac{1}{\sqrt{\alpha\mu}}\left(4\left(\frac{1}{\alpha} + \frac{\mu}{2}-\frac{\sqrt{\mu}}{\sqrt{\alpha}}\right)+\frac{L^{2}}{\mu}\right) + \frac{L}{2}+\frac{1}{\alpha} + \frac{\mu}{2}-\frac{\sqrt{\mu}}{\sqrt{\alpha}}
\\
&=\frac{1}{2\alpha\sqrt{\alpha\mu}}H_{1},
\end{align*}
where
\begin{equation}
H_{1}=8\left(1 + \frac{\mu\alpha}{2}-\sqrt{\alpha\mu}\right)+\frac{2L^{2}\alpha}{\mu} + \left(L\alpha+2 + \mu\alpha\right)\sqrt{\alpha\mu}-2\mu\alpha.
\end{equation}
By Lemma~\ref{lem:E:D-ASG}, we have
\begin{align*}
\mathbb{E}\Vert D_{k+1}\Vert^{2}
\leq
\alpha^{2}
\left[\alpha H_{2}+8\gamma^{2(k-1)}H_{3}\right],
\end{align*}
where $D_k$ is defined by \eqref{def-D-k-vector},
\begin{align*}
&H_{2}=\frac{2}{N}L^{2}
\left((1+\beta)^{2}+\beta^{2}\right)\left(4D_{y}^{2}\alpha\frac{1}{(1-\gamma)^{2}}
+\frac{4\sigma^{2}N\alpha}{(1-\gamma)^{2}}+\frac{8C_{0}}{(1-\gamma)^{2}}\right),
\\
&H_{3}=\frac{2}{N}L^{2}
\left((1+\beta)^{2}+\beta^{2}\right)
\left(4\frac{ V_{S,\alpha}\left(\xi_{0}\right)}{\mu}
+ \frac{2\sigma^2 N \sqrt{\alpha}}{\mu\sqrt{\mu}}\left(2-\lambda_N^W +\alpha L \right) +  \frac{2C_{1}^{2}N\alpha^2}{(1-\gamma)^2}
+\Vert x^{\ast}\Vert^{2}\right).
\end{align*}
Therefore, 
\begin{align*}
\E \left[V_{\bar{S},\alpha}\left(\bar{\xi}_{k+1}\right)\right] 
&\leq
\left(1-\frac{\sqrt{\alpha\mu}}{2}\right)\mathbb{E}V_{\bar{S},\alpha}\left(\bar{\xi}_{k}\right)
+\left(1+\frac{\sqrt{\alpha\mu}}{2}\right)\frac{\sigma^2  \alpha}{2N}\left(1+\alpha L \right)
\\
&\qquad
+\frac{1}{2\alpha\sqrt{\alpha\mu}}H_{1}\alpha^{2}
\left[\alpha H_{2}+8\gamma^{2(k-1)}H_{3}\right]
\\
&=\left(1-\frac{\sqrt{\alpha\mu}}{2}\right)\mathbb{E}V_{\bar{S},\alpha}\left(\bar{\xi}_{k}\right)
+\left(1+\frac{\sqrt{\alpha\mu}}{2}\right)\frac{\sigma^2  \alpha}{2N}\left(1+\alpha L \right)
\\
&\qquad
+\frac{1}{2\sqrt{\mu}}\alpha\sqrt{\alpha}H_{1}H_{2}+\frac{4}{\gamma^{2}\sqrt{\mu}}\sqrt{\alpha}H_{1}H_{3}\gamma^{2k},
\end{align*}
which implies that
\begin{align*}
\E \left[V_{\bar{S},\alpha}\left(\bar{\xi}_{k}\right)\right] 
&\leq
\left(1-\frac{\sqrt{\alpha\mu}}{2}\right)^{k}V_{\bar{S},\alpha}\left(\bar{\xi}_{0}\right)
+\frac{4}{\gamma^{2}\sqrt{\mu}}\sqrt{\alpha}H_{1}H_{3}\sum_{i=0}^{k-1}\left(1-\frac{\sqrt{\alpha\mu}}{2}\right)^{i}\left(\gamma^{2}\right)^{k-1-i}
\\
&\qquad
+\sum_{i=0}^{k-1}\left(1-\frac{\sqrt{\alpha\mu}}{2}\right)^{i}
\left(\left(1+\frac{\sqrt{\alpha\mu}}{2}\right)\frac{\sigma^2  \alpha}{2N}\left(1+\alpha L \right)+\frac{1}{2\sqrt{\mu}}\alpha\sqrt{\alpha}H_{1}H_{2}\right)
\\
&\leq
\left(1-\frac{\sqrt{\alpha\mu}}{2}\right)^{k}V_{\bar{S},\alpha}\left(\bar{\xi}_{0}\right)
+\frac{4}{\gamma^{2}\sqrt{\mu}}\sqrt{\alpha}H_{1}H_{3}\frac{\gamma^{2k}-(1-\sqrt{\alpha\mu}/2)^{k}}{\gamma^{2}-(1-\sqrt{\alpha\mu}/2)}
\\
&\qquad
+\frac{2}{\sqrt{\alpha\mu}}
\left(\left(1+\frac{\sqrt{\alpha\mu}}{2}\right)\frac{\sigma^2  \alpha}{2N}\left(1+\alpha L \right)+\frac{1}{2\sqrt{\mu}}\alpha\sqrt{\alpha}H_{1}H_{2}\right).
\end{align*}
By the $\mu$-strong convexity of $f$ and the fact that $\nabla f(x_{\ast})=0$ , we have 
$$\left\| \bar{x}^{(k)} -x_{\ast} \right\|^2 \leq \frac{2}{\mu}\left[ f\left(\bar{x}^{(k)}\right) - f\left(x_{\ast}\right)\right]
\leq\frac{2}{\mu}V_{\bar{S},\alpha}\left(\bar{\xi}_{k}\right),$$
which implies that
\begin{align*}
&\mathbb{E}\left\| \bar{x}^{(k)} -x_{\ast} \right\|^2
\\
&\leq\left(1-\frac{\sqrt{\alpha\mu}}{2}\right)^{k}\frac{2V_{\bar{S},\alpha}\left(\bar{\xi}_{0}\right)}{\mu}
+\frac{4}{\mu\sqrt{\mu}}
\left(\left(1+\frac{\sqrt{\alpha\mu}}{2}\right)\frac{\sigma^2  \sqrt{\alpha}}{2N}\left(1+\alpha L \right)+\frac{1}{2\sqrt{\mu}}\alpha H_{1}H_{2}\right)
\\
&\qquad
+\frac{8}{\gamma^{2}\mu\sqrt{\mu}}\sqrt{\alpha}H_{1}H_{3}\frac{\gamma^{2k}-(1-\sqrt{\alpha\mu}/2)^{k}}{\gamma^{2}-(1-\sqrt{\alpha\mu}/2)}.
\end{align*}
Finally, for every $1\leq i\leq N$,
\begin{align*}
\mathbb{E}\left\| x_{i}^{(k)} -x_{\ast} \right\|^2
\leq
2\mathbb{E}\left\| x_{i}^{(k)} -\bar{x}^{(k)} \right\|^2
+2\mathbb{E}\left\| \bar{x}^{(k)} -x_{\ast} \right\|^2,
\end{align*}
and by applying Lemma~\ref{gradient:average:error:DASG}, the proof is complete.}

\section{Quadratic Objectives}\label{sec:quadratic}

In this section, we analyze the special case when $f_i$ is quadratic at every node $i$ under the same Assumption~\ref{assump_1} with the main text. We assume
\begin{equation*}
f_i(x)=\frac{1}{2}x^{T}Q_ix-p_i^{T}x+r_i,
\end{equation*}
where $Q_i$ is an $d\times d$ symmetric positive definite matrix, $p_{i}\in\mathbb{R}^{d}$ and $r_{i}\in\mathbb{R}$ for $i=1,2,\dots,N$. In this special case, 
the optimum to the \eqref{opt-pbm} is explicitly given by
\begin{equation*}
x^* = \left(\sum_{i=1}^{N} Q_i\right)^{-1} \sum_{i=1}^{N}p_{i}. 
\end{equation*} 
Furthermore, the function $F$ defined 
as $F(x):=F(x_{1},\ldots,x_{N}):=\frac{1}{N}\sum_{i=1}^{N}f_{i}(x_{i})$ is also a  quadratic function of the form
\begin{equation}\label{eqn:quadratic:assump}
F(x)=\frac{1}{2}x^{T}Qx-p^{T}x+r,
\end{equation}
where $Q = \textbf{\mbox{diag}}(\{Q_i\}_{i=1}^N)$ is an $Nd\times Nd$ symmetric positive definite matrix:
\begin{equation}\label{eqn:Q}
Q := \begin{bmatrix} Q_1 & 0_d & \dots & 0_d \\
                    0_d & Q_2 & \ddots & 0_d \\
                    \vdots & \ddots & \ddots& 0_d\\
                     0_d   & \hdots & 0_d& Q_N
    \end{bmatrix},
\end{equation}
and
\begin{equation}\label{eqn:P}
p:=\left[p_1^T ~p_2^T \dots ~p_N^T\right]^T \in\mathbb{R}^{Nd}
\end{equation}
is a column vector and $r=\sum_{i=1}^{N} r^i\in\mathbb{R}$
is a scalar.
Moreover, the gradient
of $F$ is given by $\nabla F(x)=Qx-p$.

Throughout this section, and to simplify the derivations for quadratic functions, we focus on the case of additive noise. More formally, we consider the following noise assumption for this section:
\begin{assumption}\label{assump_quad}
At iteration $k$, node $i$ has access to $\tilde{\nabla} f_i \left(x_i^{(k)}, w_i^{(k)}\right)$ which is an estimate of $\nabla f_i \left(x_i^{(k)}\right)$ and satisfies the conditions given in Assumption~\ref{assump_1}. In addition, we assume this randomness is in the form of additive noise, i.e., $\tilde{\nabla} f_i \left(x_i^{(k)}, w_i^{(k)}\right) = \nabla f_i \left(x_i^{(k)}\right) + w_i^{(k)}$.
\end{assumption}
Also, similar to \eqref{column_x-nonave}, we define the vector $w^{(k)}$ as:
\begin{equation}\label{column_w}
w^{(k)}=\left[\left(w_{1}^{(k)}\right)^{T},\left(w_{2}^{(k)}\right)^{T},\ldots,\left(w_{N}^{(k)}\right)^{T}\right]^{T}\in\mathbb{R}^{Nd}.
\end{equation}

\subsection{Distributed stochastic gradient (D-SG)}\label{sec:quadratic:DSG}
The network-wide D-SG update \eqref{unified-dyn-sys} reduces to a linear recursion 
\begin{eqnarray*}
x^{(k+1)}= \left(W\otimes I_d\right) x^{(k)}-\alpha\left[Q x^{(k)} + p \right] -\alpha w^{(k+1)},
\end{eqnarray*}
where $Q$ and $p$ are defined in \eqref{eqn:Q} and \eqref{eqn:P}.
Then, the network-wide update \eqref{unified-dyn-sys} reduces to
\begin{equation*} 
\xi_{k+1} = A_Q \xi_{k} -\alpha w^{(k+1)},
\end{equation*}
where {\color{black}$\xi_{k}=x^{(k)}-x^{\infty}$ and} 
\begin{equation*}
A_{Q}=\mathcal{W}-\alpha Q.
\end{equation*}
By the assumption that $f_{i}$'s are $\mu$-strongly convex with $L$-Lipschitz gradients, we have $\mu I_{Nd}\preceq Q \preceq LI_{Nd}$. Since the stepsize $\alpha >0$, it is easy to see that
\begin{equation}\label{ineq-eigval-bound} 
\left(\lambda_N^W - \alpha L\right) I_{Nd} \preceq A_Q \preceq (1- \alpha \mu) I_{Nd}.
\end{equation}
The next result is on the \emph{spectral radius} of $A_Q$ which is defined as the maximum of the Euclidean norm of the eigenvalues of $A_Q$.

\begin{proposition}\label{thm:quadratic:rate:DGD}
For any stepsize $\alpha > 0$,
\begin{equation}\label{ineq-rho-AQ-ub}
\rho(A_{Q})= \| \mathcal{W} - \alpha Q\| =  \max\left\{\left|1-\alpha\mu\right|, \left|\lambda_N^W - \alpha L\right|\right\}.
\end{equation}
where $\rho$ denotes the spectral radius of $A_Q$. 
In particular, if $\alpha \in (0, \frac{1 + \lambda_N^W}{L+\mu}] $, then 
 $$\rho(A_Q) = 1-\alpha\mu \in [0,1). $$
\end{proposition}

\begin{proof}
The equality \eqref{ineq-rho-AQ-ub} follows directly from \eqref{ineq-eigval-bound}. The second part, note that we have $1-\alpha \mu > \lambda_N^W - \alpha L$ as $\mu \leq L$ and $\lambda_N^W < 1$. Furthermore, for $\alpha>0$ small enough, it is easy to see from \eqref{ineq-rho-AQ-ub} that $\rho(A_Q) = 1-\alpha \mu = |1-\alpha\mu |$. The proof follows after checking that $1-\alpha\mu = |1-\alpha\mu| \geq  |\lambda_N^W - \alpha L |$ for $\alpha \in [0, \frac{1 + \lambda_N^W}{L+\mu}] $. 
\end{proof}

\begin{remark}\label{remark-dgd-quad-tightness} 
In the noiseless case (when $\sigma = 0$), we have 
\begin{equation*} 
\left\| \xi_{k}\right\| = \left\|A_Q\right\|^k \left\|\xi_{0}\right\|, 
\end{equation*}
provided that $x^{(0)}$ is chosen as an eigenvector corresponding to a largest singular value of the    
$A_Q$ matrix. Therefore, by Proposition~\ref{thm:quadratic:rate:DGD}, this gives 
\begin{equation*}
\left\| \xi_{k}\right\| = \rho(\alpha)^k \left\|\xi_{0}\right\|, 
\end{equation*}
where $\rho(\alpha)$ is as in Theorem~\ref{thm:DGD:explicit}. This shows that the analysis of Theorem~\ref{thm:DGD:explicit} is tight in the sense that the convergence rate it provides for strongly convex objectives are attained for quadratics for particular choices of the initialization $\xi_{0}$ when $\sigma=0$.
\end{remark}

A consequence of Theorem~\ref{thm:DGD:explicit} for strongly convex objectives is that for $\rho(\alpha)=\rho(A_Q)< 1$, 
the robustness measure, or equivalently the variance of the iterates in the limit for the quadratic objectives, 
satisfies the bound
\begin{align*}
J_{\infty}(\alpha)&=
\frac{1}{\sigma^{2}N}\lim_{k\rightarrow\infty}\mbox{Var} \left(\xi_{k}\right)  = \frac{1}{\sigma^{2}N}\limsup_{k\to\infty}\mathbb{E}\left[\left\Vert x^{(k)}-x^{\infty}\right\Vert^{2}\right]\\
&\leq \frac{1}{1-\rho(\alpha)^{2}}\alpha^{2} = \frac{\alpha^{2}}{1-\max\left\{|1-\alpha\mu|, |\lambda_N^W - \alpha L|\right\}^2}.
\end{align*}
If we assume more structure on the noise, we can get tighter bounds. Consider the following assumption which says that the noise has a fixed covariance structure; this assumption is clearly stronger than Assumption ~\ref{assump_quad}.

\begin{assumption}\label{assump:iid}
The noise $w_{i}^{(k)}$ are independent, identically distributed (i.i.d.) for every $i$ and $k$ with zero mean and covariance matrix $\Sigma_{w_i} :=\mathbb{E}\left[w_{i}^{(k)} \left(w_{i}^{(k)}\right)^T\right] = \frac{\sigma^2}{d} I_d$. 
\end{assumption}

The next theorem shows that we can get a tighter explicit representation of the variance of the iterates in terms of the eigenvalues of the iteration matrix $A_Q$. 

\begin{theorem}\label{prop:quadratic:var:DGD} Under Assumption~\ref{assump_quad} and Assumption~\ref{assump:iid}, if $\alpha \in (0,\frac{1+\lambda_{N}^{W}}{\mu+L}]$, the D-SG iterates given by \eqref{DGD:xi} satisfy
\begin{equation}\label{eq-asymp-var-dgd}
\lim_{k\rightarrow\infty}\mbox{Var} \left(\xi_{k}\right) 
= \alpha^2 \frac{\sigma^{2}}{d} \sum_{i=1}^{Nd} \frac{1}{1-\mu_i^2},
\end{equation}
where $\mu_{i}$ are eigenvalues of $A_Q = \mathcal{W}-\alpha Q$, and hence the robustness measure is given by 
\begin{equation*}
J_{\infty}(\alpha)= \alpha^2 \frac{1}{Nd} \sum_{i=1}^{Nd} \frac{1}{1-\mu_i^2}.
\end{equation*}
\end{theorem}
\begin{proof}
Note that the matrix $A_Q$ is symmetric with real eigenvalues. Furthermore, by Proposition~\ref{thm:quadratic:rate:DGD}, we have $|\mu_i| \leq \rho(A_Q)<1$ for every $i$. Therefore, the quantity on the right-hand side of \eqref{eq-asymp-var-dgd} is well-defined. Define the covariance matrix
$$\Sigma_k = \E \left[ \xi_{k}\xi_{k}^T \right].$$
We have the recursion
\begin{equation}\label{eq-recursion-cov-matrix}
\Sigma_{k+1} = A_{Q} \Sigma_k A_{Q}^T + \alpha^2 \Sigma_w ,
\end{equation}
where $\Sigma_w = \textbf{\mbox{diag}}([\Sigma_{w_i}]_{i=1}^N)$ is the covariance matrix of the noise,
which is equal to $\frac{\sigma^{2}}{d}I_{Nd}$ by Assumption~\ref{assump:iid}. 

Let $W = V D V^T$ be an eigenvalue decomposition of $W$. Assume without loss of generality, that diagonal of $D$ contains the 
eigenvalues in decreasing order, i.e. $D_{ii}=\lambda_i^W$. In this case, $j$-th column of $V$, say $v_j$ is an eigenvector corresponding to $\lambda_j^W$. Note that the eigenvalues of $\mathcal{W} = W \otimes I_d$ are $\lambda_j^W$ each with multiplicity $d$
and we can choose the corresponding eigenvectors as $v_j \otimes e_i$ for $j=1,2,\dots,N$ and $i=1,2,\dots,d$ where $e_i$ is the standard basis. In other words, we can write
 $$ \mathcal{W} = \mathcal{V} \mathcal{D} \mathcal{V}^T, \quad \mbox{where} \quad \mathcal{D} = \begin{bmatrix} \lambda_1^W I_d & 0_d &    \dots&  ~0_d \\
  0_d & \lambda_2^W I_d &  \ddots & ~0_d \\
  \vdots & \ddots & \ddots & ~\vdots \\
  0_d    &  \hdots  &  0_d  & \lambda_N^W I_d
 \end{bmatrix}\,,
 $$
for some $\mathcal{V}$. We will write the D-SG iterations \eqref{eq-iter-dgd} with respect to this basis. Let
 $$\hat{Q} := \mathcal{V}^T Q \mathcal{V}, \quad
 \hat{\xi}_{k} := \mathcal{V}^T \xi_{k} \mathcal{V}, \quad  \hat{\Sigma}_k := \mathcal{V}^T \Sigma_k \mathcal{V}.$$
For $\Sigma_w = (\sigma^2/d) I_{Nd}$, we can write \eqref{eq-recursion-cov-matrix} as 
\begin{equation}\label{eq-recursion-cov-matrix-2}
{\hat{\Sigma}}_{k+1} = A_{\hat{Q}} {\hat\Sigma}_k A_{\hat{Q}}^T 
+ \alpha^2 (\sigma^2/d) I_{Nd},
\end{equation} 
where 
$$
A_{\hat Q} := \mathcal{D}-\alpha \hat{Q}.
$$
We obtain 
$$
\lim_{k\rightarrow\infty} {\hat{\Sigma}}_{k} 
=  \alpha^2 (\sigma^{2}/d)\sum_{k=0}^{\infty} A_{\hat Q}^{2k}  
= \alpha^2 (\sigma^{2}/d)\left(I -A_{\hat Q}^2\right)^{-1}, 
$$
where ${\hat\mu}_i$ are the eigenvalues of $A_{\hat Q}$. Therefore,
$$
\lim_{k\rightarrow\infty}\mbox{Var} \left(\xi_{k}\right) = \lim_{k\rightarrow\infty}\mbox{trace} (\Sigma_k) = \lim_{k\rightarrow\infty}\mbox{trace}( {\hat\Sigma}_k) = \alpha^2 (\sigma^{2}/d)\sum_{i=1}^{Nd} \frac{1}{1-\mu_i^2},
$$ 
where $\mu_i$ are the eigenvalues of $A_{\hat Q}$ or equivalently of $A_Q$.
\end{proof}

\begin{proposition}\label{prop:var:bar:x:DSG}
{\color{black}Assume that Assumption~\ref{assump_quad} and Assumption~\ref{assump:iid} hold.
For any $j=1,\ldots,d$ and any $k\in\mathbb{N}$,
$$ \lim_{k\rightarrow\infty}\mbox{Var}\left( \bar{x}^{(k)}(j) \right) \leq \frac{\sigma^2}{Nd} \max_{i=1,2,\dots,Nd} \frac{\alpha^{2}}{1-\mu_{i}^{2}},
$$
where $\bar{x}^{(k)}(j)$ denotes the $j$-th entry of the node average $\bar{x}^{(k)}$ and $\mu_{i}$ are the eigenvalues of $\mathcal{W}-\alpha Q$.}
\end{proposition}

\begin{proof}
{\color{black}D-SG can be viewed a special case of D-ASG when the momentum parameter $\beta=0$. The conclusion follows from the more general result for D-ASG in Proposition~\ref{prop:var:bar:x:DASG}.}
\end{proof}


Next, for D-SG iterates, we provide 
bounds on $\mathbb{E}\left[\left\Vert x^{(k)}-x^{\infty}\right\Vert^{2}\right]$ and
$\mathbb{E}\left[\left\Vert x^{(k)}-x^{\ast}\right\Vert^{2}\right]$.

\begin{theorem}\label{thm:quadratic:DSG}
Consider the D-SG iterates under Assumption~\ref{assump_quad} and Assumption~\ref{assump:iid}. For every $k\in\mathbb{N}$,
\begin{align}
&\mathbb{E}\left[\left\Vert x^{(k)}-x^{\infty}\right\Vert^{2}\right]
\leq
\rho_{\text{dsg}}^{2k}\left(\left\Vert\xi_{0}\xi_{0}^{T}\right\Vert
+\frac{\alpha^{2}\sigma^{2}N}{1-\rho_{\text{dsg}}^{2}}\right)
+\alpha^{2}\frac{\sigma^{2}}{d}\sum_{i=1}^{Nd}\frac{1}{1-\mu_{i}^{2}},
\label{quadratic:dgd:main:1}
\end{align}
where $\rho_{\text{dsg}}:=\max_{1\leq i\leq Nd}|\mu_{i}|$,
where $\mu_{i}$ are eigenvalues of $A_{Q}$. 

In particular, if $\alpha\in(0,\frac{1+\lambda_{N}^{W}}{L+\mu}]$,
then
\begin{align*}
&\mathbb{E}\left[\left\Vert x^{(k)}-x^{\infty}\right\Vert^{2}\right]
\leq
(1-\alpha\mu)^{2k}\left(\left\Vert\xi_{0}\xi_{0}^{T}\right\Vert
+\frac{\alpha^{2}\sigma^{2}N}{1-(1-\alpha\mu)^{2}}\right)
+\alpha^{2}\frac{\sigma^{2}}{d}\sum_{i=1}^{Nd}\frac{1}{1-\mu_{i}^{2}}.
\end{align*}
In addition, if we have $\alpha\in(0,\frac{1}{L+\mu}]$, then
\begin{align}
\mathbb{E}\left[\left\Vert x^{(k)}-x^{\ast}\right\Vert^{2}\right]
\leq
(1-\alpha\mu)^{2k}\left(2\left\Vert\xi_{0}\xi_{0}^{T}\right\Vert
+\frac{2\alpha^{2}\sigma^{2}N}{1-(1-\alpha\mu)^{2}}\right)
+\frac{2\alpha^{2}\sigma^{2}}{d}\sum_{i=1}^{Nd}\frac{1}{1-\mu_{i}^{2}}
+\frac{2\alpha^{2}C_{1}^{2}N}{(1-\gamma)^{2}}. \label{quadratic:dgd:main:2}
\end{align}
\end{theorem}
\begin{proof}
We recall that with $\xi_{k}=x^{(k)}-x^{\infty}$,
\begin{equation*}
\xi_{k+1}=A_{Q}\xi_{k}-\alpha w^{(k+1)},
\end{equation*}
and therefore, we get:
\begin{equation}\label{eqn:iterate:k:DSG}
\mathbb{E}\left[\xi_{k}\xi_{k}^{T}\right]
=
A_{Q}\mathbb{E}\left[\xi_{k-1}\xi_{k-1}^{T}\right](A_{Q})^{T}
+
\alpha^{2}\frac{\sigma^{2}}{d}I_{Nd},
\end{equation}
Therefore, 
\begin{equation*}
X:=\mathbb{E}\left[\xi_{\infty}\xi_{\infty}^{T}\right]
\end{equation*}
satisfies the discrete Lyapunov equation:
\begin{equation*}
X=A_{Q}X(A_{Q})^{T}+\alpha^{2}\frac{\sigma^{2}}{d}I_{Nd}.
\end{equation*}
By Theorem~\ref{prop:quadratic:var:DGD}, we have
\begin{equation*}
\mbox{trace}(X)=\alpha^{2}\frac{\sigma^{2}}{d}\sum_{i=1}^{Nd}\frac{1}{1-\mu_{i}^{2}}.
\end{equation*}

Next by iterating equation \eqref{eqn:iterate:k:DSG} over $k$, we immediately obtain
\begin{equation*}
\mathbb{E}\left[\xi_{k}\xi_{k}^{T}\right]
=\left(A_{Q}\right)^{k}\xi_{0}\xi_{0}^{T}
\left((A_{Q})^{T}\right)^{k}
+
\sum_{j=0}^{k-1}
\left(A_{Q}\right)^{j}
\alpha^{2}\frac{\sigma^{2}}{d}I_{Nd}
\left((A_{Q})^{T}\right)^{j},
\end{equation*}
so that
\begin{equation*}
\mathbb{E}\left[\xi_{k}\xi_{k}^{T}\right]
=\mathbb{E}\left[\xi_{\infty}\xi_{\infty}^{T}\right]
+\left(A_{Q}\right)^{k}
\xi_{0}\xi_{0}^{T}\left((A_{Q})^{T}\right)^{k}
-\sum_{j=k}^{\infty}
\left(A_{Q}\right)^{j}
\alpha^{2}\frac{\sigma^{2}}{d}I_{Nd}
\left((A_{Q})^{T}\right)^{j},
\end{equation*}
which implies that
\begin{align*}
\mbox{trace}\left(\mathbb{E}\left[\xi_{k}\xi_{k}^{T}\right]\right)
&=\mbox{trace}\left(\mathbb{E}\left[\xi_{\infty}\xi_{\infty}^{T}\right]\right)
+\left(A_{Q}\right)^{k}\xi_{0}\xi_{0}^{T}\left((A_{Q})^{T}\right)^{k}
\\
&\qquad
-\sum_{j=k}^{\infty}
\left(A_{Q}\right)^{j}
\alpha^{2}\frac{\sigma^{2}}{d}I_{Nd}
\left((A_{Q})^{T}\right)^{j}
\\
&\leq
\mbox{trace}(X)
+\left\Vert (A_{Q})^{k}\right\Vert^{2}
\left\Vert\xi_{0}\xi_{0}^{T}\right\Vert
+\sum_{j=k}^{\infty}\left\Vert (A_{Q})^{j}\right\Vert^{2}\alpha^{2}\sigma^{2}N
\\
&\leq\mbox{trace}(X)
+\rho_{\text{dsg}}^{2k}\left\Vert\xi_{0}\xi_{0}^{T}\right\Vert
+\alpha^{2}\sigma^{2}N\frac{\rho_{\text{dsg}}^{2k}}{1-\rho_{\text{dsg}}^{2}},
\end{align*}
where we used the estimate:
\begin{equation*}
\left\Vert A_{Q}^{k}\right\Vert
\leq\Vert V\Vert^{2}\left(\max_{1\leq i\leq Nd}|\mu_{i}|\right)^{k}
=\left(\max_{1\leq i\leq Nd}|\mu_{i}|\right)^{k}
=\rho_{\text{dsg}}^{k},
\end{equation*}
where we used the fact that $A_{Q}=\mathcal{W}-\alpha Q$
is symmetric with the decomposition $A_{Q}=V \textbf{\mbox{diag}}\left([\mu_i]_{i=1}^{Nd}\right)V^{T}$,
where $\mu_i$ are the eigenvalues of $A_{Q}$
and the fact that $\Vert V\Vert=1$ since $V$
is orthogonal. Note that $\xi_{k}=x^{(k)}-x^{\infty}$,
and this proves \eqref{quadratic:dgd:main:1}.

Finally, when $\alpha\in(0,\frac{1+\lambda_{N}^{W}}{L+\mu}]$, by Proposition~\ref{thm:quadratic:rate:DGD}, we get
\begin{equation*}
\rho_{\text{dsg}}=\max_{1\leq i\leq Nd}|\mu_{i}|=1-\alpha\mu.
\end{equation*}
Moreover,
\begin{equation*}
\left\Vert x^{(k)}-x^{\ast}\right\Vert^{2}
\leq 2\left\Vert x^{(k)}-x^{\infty}\right\Vert^{2}
+2\left\Vert x^{\infty}-x^{\ast}\right\Vert^{2},
\end{equation*}
and together with \eqref{eqn:asymp_suboptim},
it proves \eqref{quadratic:dgd:main:2}.
The proof is complete.
\end{proof}

\subsection{Distributed accelerated stochastic gradient (D-ASG)}\label{sec:quadratic:DASG}

First, let us recall that the network-wide update for D-ASG is given by
\begin{align}
&x^{(k+1)}=\mathcal{W}y^{(k)}-\alpha\left[\nabla F\left(y^{(k)}\right)+w^{(k+1)}\right],\label{DASG-Q-1}
\\
&y^{(k)}=(1+\beta)x^{(k)}-\beta x^{(k-1)},\label{DASG-Q-2}
\end{align}
where $F:\mathbb{R}^{Nd}\rightarrow\mathbb{R}$, is defined as
$F(y):=F(y_{1},\ldots,y_{N})
=\sum_{i=1}^{N}f_{i}(y_{i})$,
and the noise $w^{(k+1)}$ satisfies \eqref{assump}.

In the quadratic case, i.e. $F$ is quadratic and defined in \eqref{eqn:quadratic:assump},
we can re-write the D-ASG iterates \eqref{DASG-Q-1}-\eqref{DASG-Q-2} as
\begin{equation}\label{eq-iter-quad-dasg}
\xi_{k+1}=A_{\text{dasg},Q}\xi_{k}+B_{\text{dasg}}w^{(k+1)},
\end{equation}
where 
\begin{equation*}
\xi_{k}:=\left[\left(x^{(k)}-x^{\infty}\right)^{T},\left(x^{(k-1)}-x^{\infty}\right)^{T}\right]^{T},
\end{equation*}
and
\begin{equation}
A_{\text{dasg},Q}:=
\left[
\begin{array}{cc}
(1+\beta)(\mathcal{W}-\alpha Q) & -\beta(\mathcal{W}-\alpha Q)
\\
I_{Nd} & 0_{Nd}
\end{array}
\right],
\label{def-A-dasg-Q}
\end{equation}
and $B_{\text{dasg}}$ is defined in Section~\ref{sec:reformulate}
and $Q$ is given in \eqref{eqn:Q}.
Next, we obtain the spectral radius of $A_{\text{dasg},Q}$, that is
the maximum of the Euclidean norm of the eigenvalues of $A_{\text{dasg},Q}$.

\begin{proposition}\label{thm:DASG:rho}
Let $\mu_{i}$, $1\leq i\leq Nd$, be the eigenvalues
of $\mathcal{W}-\alpha Q$ listed in non-increasing order. We have
\begin{equation*}
\rho(A_{\text{dasg},Q})=
\max_{1\leq i\leq Nd}
\left\{\left|\frac{(1+\beta)\mu_{i}\pm\sqrt{(1+\beta)^{2}\mu_{i}^{2}-4\beta\mu_{i}}}{2}\right|\right\}.
\end{equation*}
\end{proposition}
\begin{proof}
Consider the eigenvalue decomposition 
\begin{equation*}
\mathcal{W} -\alpha Q = \rev{R} \textbf{\mbox{diag}}\left([\mu_i]_{i=1}^{Nd}\right)\rev{R^{T}},
\end{equation*} 
where \rev{$R$} is real orthogonal and the eigenvalues $\mu_i$ are listed in non-increasing order. Next, we introduce the matrix
\beq U = \textbf{\mbox{diag}}(\rev{R,R}),
\label{def-U-matrix}
\eeq
and the permutation matrix $P_{\pi}$
associated with the permutation $\pi$ over $\{1,2,\ldots,2Nd\}$
that satisfies
\begin{equation}
\pi(i)=
\begin{cases}
2i-1 &\text{if $1\leq i\leq Nd$},
\\
2(i-Nd) &\text{if $Nd+1\leq i\leq 2Nd$}.
\end{cases}
\label{def-perm-matrix}
\end{equation}
By definition, $P_{\pi}^{-1}=P_{\pi}^{T}=P_{\pi^{-1}}$. Then, we can write
\rev{
\begin{align}
U A_{\text{dasg},Q} U^T &= \left[
\begin{array}{cc}
(1+\beta)\textbf{\mbox{diag}}\left([\mu_i]_{i=1}^{Nd}\right) & -\beta\textbf{\mbox{diag}}\left([\mu_i]_{i=1}^{Nd}\right)
\\
I_{Nd} & 0_{Nd}
\end{array} 
\right]
\\
&= P_\pi^T \textbf{\mbox{diag}}\left([\tilde{T}_i]_{i=1}^{Nd}\right) P_\pi,
\label{eq-blockdiag-ASG-iter-mat}
\end{align}
and
}
\begin{equation} 
\tilde{T}_i :=\left[
\begin{array}{cc}
(1+\beta)\mu_{i} & -\beta\mu_{i}
\\
1 & 0
\end{array}
\right] \in {\mathbb{R}}^{2\times 2},
\qquad 1\leq i\leq Nd.
\label{def-Ti-matrices}
\end{equation}
Therefore, the eigenvalues of $A_{\text{dasg},Q}$ coincide with the eigenvalues of $\tilde{T}_i$ which can be computed explicitly as $\frac{(1+\beta)\mu_{i}\pm\sqrt{(1+\beta)^{2}\mu_{i}^{2}-4\beta\mu_{i}}}{2}$. This completes the proof.
\end{proof}
\begin{remark} 
In the noiseless case (when $\sigma = 0$ and $w^k = 0$), we have 
\begin{equation*} 
\left\| \xi_{k}\right\| = \left\|A_{\text{dasg},Q}^k\right\| \left\|\xi_{0}\right\|, 
\end{equation*}
provided that $x^{(0)}$ is chosen as an eigenvector corresponding to a largest singular value of the
$A_{\text{dasg},Q}$ matrix. By Gelfand's formula, we have
$$
\rho(A_{\text{dasg},Q}) = \lim_{k\to\infty} \left({\left\| \xi_{k}\right\|}/{ \left\|\xi_{0}\right\|}\right)^{1/k}.
$$
Therefore, Proposition~\ref{thm:DASG:rho} gives an explicit characterization of the asymptotic convergence rate. 

\end{remark}

For $\rho=\rho(A_{\text{dasg},Q})< 1$, it is clear from the iterations \eqref{eq-iter-quad-dasg} that the second moments $\mathbb{E}\left[\left\Vert x^{(k)}-x^{\infty}\right\Vert^{2}\right]
$ will stay bounded over $k$. In fact, a consequence of Theorem~\ref{thm:general:DASG} for strongly convex objectives is that, the variance of the iterates satisfies 
\begin{equation*}
\limsup_{k\to\infty}\mathbb{E}\left[\left\Vert x^{(k)}-x^{\infty}\right\Vert^{2}\right]
\leq
\frac{1}{1-\rho^{2}}\alpha^{2}\frac{2\sigma^{2}N}{\mu}\left(\tilde{P}_{11}+\frac{1-\lambda_{N}^{W}+\alpha L}{2\alpha}\right),
\end{equation*}
and hence the robustness measure satisfies
\begin{equation*}
J_{\infty}(\alpha)
\leq\frac{1}{1-\rho^{2}}\alpha^{2}\frac{2}{\mu}\left(\tilde{P}_{11}+\frac{1-\lambda_{N}^{W}+\alpha L}{2\alpha}\right).
\end{equation*}

The next theorem shows that we can get a tighter explicit representation of the variance of the iterates in terms of the eigenvalues of the iteration matrix $A_{\text{dasg},Q}$. 

\begin{theorem}\label{prop:quadratic:var:acc}
Assume that Assumption~\ref{assump_quad} and Assumption~\ref{assump:iid} hold.
Let $\mu_{i}$ be the eigenvalues of $\mathcal{W}-\alpha Q$. 
Then we have
\begin{equation}\label{formula:Q:var:acc}
\lim_{k\rightarrow\infty}\text{Var}\left(x^{(k)}-x^{\infty}\right)
=\frac{\sigma^{2}}{d}\sum_{i=1}^{Nd}\frac{\alpha^{2}(1+\beta\mu_{i})}{(1-\mu_{i})(1-\beta\mu_{i})(2+2\beta-(1-\mu_{i})(1+2\beta))},
\end{equation}
and hence the robustness measure is given by
\begin{equation*}
J_{\infty}(\alpha)
=\frac{1}{Nd}\sum_{i=1}^{Nd}\frac{\alpha^{2}(1+\beta\mu_{i})}{(1-\mu_{i})(1-\beta\mu_{i})(2+2\beta-(1-\mu_{i})(1+2\beta))}.
\end{equation*}
\end{theorem}
\begin{proof}
\rev{Similar to the D-SG case, the equilibrium covariance matrix $X =\lim_{k\to\infty} \mathbb{E}[\xi_k \xi_k^T]$ of the D-ASG iterates satisfies the corresponding discrete Lyapunov equation
$$ A_{dasg,Q} X A_{dasg,Q}^T + \frac{\sigma^2}{d} B B^T = 0.$$
where $A_{dasg,Q}$ is as in \eqref{def-A-dasg-Q}. The proof will be based on constructing a solution to this equation by block diagonalizing the matrix $A_{dasg,Q}$ with a change of variable technique. More specifically, if we introduce the matrix $Y = P_{\pi} (U X U^T) P_{\pi}^{T}$, where $U$ is an orthogonal matrix defined by \eqref{def-U-matrix} and $P_\pi$ is the permutation matrix defined in \eqref{def-perm-matrix}. It follows from \eqref{eq-blockdiag-ASG-iter-mat} that $Y$ satisfies the discrete Lyapunov equation:
 \begin{equation*}   
 \textbf{\mbox{diag}}\left([\tilde{T}_i]_{i=1}^{Nd}\right) Y \left[\textbf{\mbox{diag}}\left([\tilde{T}_i]_{i=1}^{Nd}\right)\right]^T -Y
 +\frac{\sigma^{2}}{d}P_{\pi}U^T BB^{T}UP_{\pi}^{T}=0,
 \end{equation*}
where $T_i$ is defined by \eqref{def-Ti-matrices}}. \rev{Furthermore, $\mbox{trace}(Y)=\mbox{trace}{(X)}$ since $U$ is orthogonal}. 
Similar as in the proof of Proposition 3.7. \cite{StrConvex}, we can solve for $Y$ which takes the block diagonal matrix form:
\begin{equation}
Y=
\left[\begin{array}{cccc}
Y_{1} & 0_{Nd} & \cdots & 0_{Nd}
\\
0_{Nd} & Y_{2} & \cdots & 0_{Nd}
\\
\vdots & \vdots & \ddots & \vdots
\\
0_{Nd} & 0_{Nd} & \cdots & Y_{Nd}
\end{array}\right],
\label{def-blockdiag-Y-mat}
\end{equation}
where $Y_{i}$ satisfies the equation
\begin{equation*}
\left[
\begin{array}{cc}
(1+\beta)\mu_{i} & -\beta\mu_{i}
\\
1 & 0
\end{array}
\right]
Y_{i}
\left[
\begin{array}{cc}
(1+\beta)\mu_{i} & 1
\\
-\beta\mu_{i} & 0
\end{array}
\right]
-Y_{i}
+
\frac{\sigma^{2}}{d}\left[
\begin{array}{cc}
\alpha^{2} & 0
\\
0 & 0
\end{array}
\right]
=0.
\end{equation*}
We can explicitly solve for $Y_{i}$ and get
\begin{equation}
Y_{i}=\frac{\sigma^{2}}{d}
\left[
\begin{array}{cc}
\frac{\alpha^{2}(1+\beta\mu_{i})}{(1-\mu_{i})(1-\beta\mu_{i})(2+2\beta-(1-\mu_{i})(1+2\beta))} & \frac{\alpha^{2}(1+\beta)\mu_{i}}{(1-\mu_{i})(1-\beta\mu_{i})(2+2\beta-(1-\mu_{i})(1+2\beta))}
\\
\frac{\alpha^{2}(1+\beta)\mu_{i}}{(1-\mu_{i})(1-\beta\mu_{i})(2+2\beta-(1-\mu_{i})(1+2\beta))} & \frac{\alpha^{2}(1+\beta\mu_{i})}{(1-\mu_{i})(1-\beta\mu_{i})(2+2\beta-(1-\mu_{i})(1+2\beta))}
\end{array}
\right].
\label{def-Yi}
\end{equation}
Since $\xi_{k}=[(x^{(k)}-x^{\infty})^{T},(x^{(k-1)}-x^{\infty})^{T}]^{T}$, we have
\begin{align*}
\lim_{k\rightarrow\infty}\text{Var}\left(x^{(k)}-x^{\infty}\right)
&=\lim_{k\rightarrow\infty}\frac{1}{2}\text{Var}(\xi_{k}) = \rev{\frac{1}{2} \mbox{trace}(X)} =  \rev{\frac{1}{2} \mbox{trace}(Y)}
\\
&=\frac{\sigma^{2}}{d}\sum_{i=1}^{Nd}\frac{\alpha^{2}(1+\beta\mu_{i})}{(1-\mu_{i})(1-\beta\mu_{i})(2+2\beta-(1-\mu_{i})(1+2\beta))},
\end{align*}
\rev{which completes the proof}.
\end{proof}



\begin{proposition}\label{prop:var:bar:x:DASG}
{\color{black}Assume that Assumption~\ref{assump_quad} and Assumption~\ref{assump:iid} hold.
For any $j=1,\ldots,d$ and any $k\in\mathbb{N}$,
$$ \lim_{k\rightarrow\infty}\mbox{Var}\left( \bar{x}^{(k)}(j) \right) \leq \frac{\sigma^2}{Nd} \max_{i=1,2,\dots,Nd} \frac{\alpha^{2}(1+\beta\mu_{i})}{(1-\mu_{i})(1-\beta\mu_{i})(2+2\beta-(1-\mu_{i})(1+2\beta))},
$$
where $\bar{x}^{(k)}(j)$ denotes the $j$-th entry of the node average $\bar{x}^{(k)}$ and $\mu_{i}$ are the eigenvalues of $\mathcal{W}-\alpha Q$.}
\end{proposition}

\begin{proof}
{\color{black}
It follows from the proof of Proposition \ref{prop:quadratic:var:acc} that the covariance matrix has the form 
\begin{equation} 
\mathbb{E}\left[\xi_{\infty} {\xi_{\infty}}^T\right] =  Z Y Z^T,
\end{equation}
where $Y$ is as in \eqref{def-blockdiag-Y-mat}, $Z = U P_\pi$ is orthogonal, 
where $\xi_\infty$ is a random vector whose distribution coincides with the distribution of $\xi_{k}$ in the limit as $k\to \infty$. For a random vector $q$ with mean zero, let $\mbox{Cov}(q)$ denote the covariance matrix of $q$, i.e. $\mbox{Cov}(q) = \mathbb{E} [q q^T ]$.
It follows that 
    $$ \mbox{Cov}(Z^T \xi_\infty) = Z^T  \mathbb{E}\left[\xi_{\infty} {\xi_{\infty}}^T\right] Z = Y,$$ 
where 
$$ Z^T \xi_\infty =  \lim_{k\to\infty}\begin{bmatrix} r_1^T x^{(k)} \\
                                     r_1^T x^{(k-1)}  \\
                                     r_2^T x^{(k)} \\
                                     r_2^T x^{(k-1)} \\
                                     \vdots \\
                                     r_{Nd}^T x^{(k)} \\
                                     r_{Nd}^T x^{(k-1)} \\
                        \end{bmatrix},
$$
where $r_i$ are the columns of $R$ in the eigenvalue decomposition 
$\mathcal{W} -\alpha Q = R \textbf{\mbox{diag}}\left([\mu_i]_{i=1}^{Nd}\right)R^{T}$. In other words, $r_i$ are the eigenvectors corresponding to the eigenvalues $\mu_i$. Using the block diagonal structure of $Y$ with blocks $Y_i$, this shows that 
\begin{align} 
&\lim_{k\to\infty}\mbox{Cov} \left( 
                \begin{bmatrix} r_i^T x^{(k)} \\
                                     r_i^T x^{(k-1)}  
                \end{bmatrix}
             \right) = Y_i \in \mathbb{R}^2, \nonumber
             \\
             &\lim_{k\rightarrow\infty}\mathbb{E}\left(\left(r_i^T x^{(k)}\right)\left(r_j^T x^{(k)}\right)\right) = Y_{2i-1,2j-1} = 0 \quad \mbox{for} \quad i\neq j,
             \label{eq-zero-covar}
\end{align}         
where $Y_i$ is given by \eqref{def-Yi} and the matrix $Y$ is given by \eqref{def-blockdiag-Y-mat}. Therefore, 
    \beq \lim_{k\to\infty}\mbox{Var}\left(v_i^T x^{(k)}\right) = \begin{bmatrix} 1 & 0 \end{bmatrix} Y_i \begin{bmatrix} 1\\0 \end{bmatrix} = \frac{\sigma^2}{d} \frac{\alpha^{2}(1+\beta\mu_{i})}{(1-\mu_{i})(1-\beta\mu_{i})(2+2\beta-(1-\mu_{i})(1+2\beta))},
    \label{eq-var-projections}
    \eeq
where $\mbox{Var}$ denotes the variance and we used \eqref{def-Yi}. The eigenvectors $v_i$ are not explicitly available, but we know they are orthogonal forming a basis; therefore for any unit vector $u\in \mathbb{R}^{Nd}$, we can express it in a unique way as linear combinations of the basis vectors $v_i$, i.e.  
$$ u = \sum_{i=1}^{Nd} m_i v_i, \qquad m_{i} = \langle u, v_i \rangle,$$
for some scalars $m_i$ that are not all zero. Since $u$ has unit norm in $\mathbb{R}^{Nd}$, we have also
  $$ \|u\|^2 = 1 = \sum_{i=1}^{Nd}m_i^2.$$
Consequently, 
\begin{align} 
&\lim_{k\rightarrow\infty} \mbox{Var}\left(u^T x^{(k)}\right) 
\nonumber
\\
&=  \lim_{k\rightarrow\infty}\mbox{Var}\left( \left(\sum_{i=1}^{Nd} m_i r_i^T\right) x^{(k)}\right) \nonumber
\\
&= \sum_{i=1}^{Nd} \alpha_i^2 \lim_{k\rightarrow\infty}\mbox{Var}\left(r_i^T x^{(k)}\right) + 2\sum_{1\leq i<j\leq Nd} m_i m_j \lim_{k\rightarrow\infty}\mathbb{E}\left[\left(r_i^T x^{(k)}\right) \left(r_j^T x^{(k)}\right)\right]  \nonumber\\
&= \sum_{i=1}^{Nd} m_i^2 \lim_{k\rightarrow\infty}\mbox{Var}\left(r_i^T x^{(k)}\right) \label{eq-intermediate}\\
&= \sum_{i=1}^{Nd} (\langle u, v_i\rangle)^2 \frac{\sigma^2}{d} \frac{\alpha^{2}(1+\beta\mu_{i})}{(1-\mu_{i})(1-\beta\mu_{i})(2+2\beta-(1-\mu_{i})(1+2\beta))}\nonumber,
\end{align}
where we used \eqref{eq-zero-covar}. This formula expresses the asymptotic variance along any unit direction $u$. However, we can also obtain an upper bound from \eqref{eq-intermediate},
\begin{align} 
\lim_{k\rightarrow\infty}\mbox{Var}\left(u^T x^{(k)}\right) &\leq \max_{1\leq i\leq Nd} \lim_{k\rightarrow\infty}\mbox{Var}\left(r_i^T x^{(k)}\right) \sum_{i=1}^{Nd}{m_i^2} \\
&= \max_{1\leq i\leq Nd} \lim_{k\rightarrow\infty}\mbox{Var}\left(r_i^T x^{(k)}\right) \\
&= \frac{\sigma^2}{d} \max_{i=1,2,\dots,Nd} \frac{\alpha^{2}(1+\beta\mu_{i})}{(1-\mu_{i})(1-\beta\mu_{i})(2+2\beta-(1-\mu_{i})(1+2\beta))}, 
\end{align}
for any unit vector $u\in\mathbb{R}^{Nd}$ where we used \eqref{eq-var-projections}. Furthermore, this bound does not grow with $N$ as the eigenvalues $\mu_i$ are bounded satisfying $\lambda_N^W - \alpha L \leq  \mu_i \leq 1-\alpha\mu$. 
If we choose the vector $u$ such that its entries are
$u_j = 1/\sqrt{N}$ if $j\in\{1,d+1,2d+1,\ldots,(N-1)d+1\}$ else 0. Then, $u$ is a unit vector satisfying
   $$ u^T x^{(k)} = \sqrt{N} \bar{x}^{(k)}(1),$$
where $\bar{x}^{(k)}(1)$ denotes the first entry of the node average $\bar{x}^{(k)}$. Therefore, we get
\begin{align*}
\lim_{k\rightarrow\infty}\mbox{Var}\left(u^T x^{(k)}\right) 
&= N \lim_{k\rightarrow\infty}\mbox{Var}\left( \bar{x}^{(k)}(1) \right) 
\\
&\leq \frac{\sigma^2}{d} \max_{i=1,2,\dots,Nd} \frac{\alpha^{2}(1+\beta\mu_{i})}{(1-\mu_{i})(1-\beta\mu_{i})(2+2\beta-(1-\mu_{i})(1+2\beta))}. 
\end{align*}
Consequently, 
$$ \lim_{k\rightarrow\infty}\mbox{Var}\left( \bar{x}^{(k)}(1) \right) \leq \frac{\sigma^2}{Nd} \max_{i=1,2,\dots,Nd} \frac{\alpha^{2}(1+\beta\mu_{i})}{(1-\mu_{i})(1-\beta\mu_{i})(2+2\beta-(1-\mu_{i})(1+2\beta))}.
$$
Similarly, choosing $u$ appropriately, we can obtain
$$ \lim_{k\rightarrow\infty}\mbox{Var}\left( \bar{x}^{(k)}(j)\right) \leq \frac{\sigma^2}{Nd} \max_{i=1,2,\dots,Nd} \frac{\alpha^{2}(1+\beta\mu_{i})}{(1-\mu_{i})(1-\beta\mu_{i})(2+2\beta-(1-\mu_{i})(1+2\beta))},
$$
for any $j=1,\ldots,d$, which completes the proof.}
\end{proof}

Next, for D-ASG iterates, we provide 
bounds on $\mathbb{E}\left[\left\Vert x^{(k)}-x^{\infty}\right\Vert^{2}\right]$ and
$\mathbb{E}\left[\left\Vert x^{(k)}-x^{\ast}\right\Vert^{2}\right]$.

\begin{theorem}\label{thm:quadratic:DASG}
Assume that Assumption~\ref{assump_quad} and Assumption~\ref{assump:iid} hold.
Consider the D-ASG iterates. For every $k\in\mathbb{N}$,
\begin{align}
&\mathbb{E}\left[\left\Vert x^{(k)}-x^{\infty}\right\Vert^{2}\right]
\leq
(C_{k})^{2}\rho_{\text{dasg}}^{2k}
\left(\left\Vert\xi_{0}\xi_{0}^{T}\right\Vert
+\frac{\alpha^{2}\sigma^{2}N}{1-\rho_{\text{dasg}}^{2}}\right)
\nonumber
\\
&\qquad\qquad\qquad\qquad\qquad
+\frac{\alpha^{2}\sigma^{2}}{d}\sum_{i=1}^{Nd}\frac{(1+\beta\mu_{i})}{(1-\mu_{i})(1-\beta\mu_{i})(2+2\beta-(1-\mu_{i})(1+2\beta))}.
\label{quadratic:acc:main:1}
\end{align}
In addition, if $\alpha \leq \frac{1}{L+\mu}$, we have:
\begin{align}
&\mathbb{E}\left[\left\Vert x^{(k)}-x^{\ast}\right\Vert^{2}\right]
\leq
(C_{k})^{2}\rho_{\text{dasg}}^{2k}\left(2\left\Vert\xi_{0}\xi_{0}^{T}\right\Vert
+\frac{2\alpha^{2}\sigma^{2}N}{1-\rho_{\text{dasg}}^{2}}\right)
+2\frac{\alpha^{2}C_{1}^{2}N}{(1-\gamma)^{2}}
\nonumber
\\
&\qquad\qquad\qquad\qquad\qquad
+\frac{\alpha^{2}\sigma^{2}}{d}\sum_{i=1}^{Nd}\frac{(1+\beta\mu_{i})}{(1-\mu_{i})(1-\beta\mu_{i})(2+2\beta-(1-\mu_{i})(1+2\beta))},
\label{quadratic:acc:main:2}
\end{align}
where $C_{k}$, $\rho_{\text{dasg}}$
are defined in Lemma~\ref{lem:DASG:Q}
and $\mu_{i}$ are the eigenvalues of $\mathcal{W}-\alpha Q$.

In particular, when $\beta=\frac{1-\sqrt{\alpha\mu}}{1+\sqrt{\alpha\mu}}$,
$\lambda_{N}^{W}>0$ and $\alpha\in \left (0,\min\{\frac{1}{L+\mu},\frac{\lambda_{N}^{W}}{L}\}\right ]$,
we have
\begin{align}
&\mathbb{E}\left[\left\Vert x^{(k)}-x^{\infty}\right\Vert^{2}\right]
\leq
(C_{k})^{2}(1-\sqrt{\alpha\mu})^{2k}
\left(\left\Vert\xi_{0}\xi_{0}^{T}\right\Vert
+\frac{\alpha^{2}\sigma^{2}N}{1-(1-\sqrt{\alpha\mu})^{2}}\right)
\nonumber
\\
&\qquad\qquad\qquad
+\frac{\alpha^{2}\sigma^{2}}{d}\sum_{i=1}^{Nd}\frac{(1+\sqrt{\alpha\mu})(1+\sqrt{\alpha\mu}+(1-\sqrt{\alpha\mu})\mu_{i})}{(1-\mu_{i})(1+\sqrt{\alpha\mu}-(1-\sqrt{\alpha\mu})\mu_{i})(4-(1-\mu_{i})(3-\sqrt{\alpha\mu}))},
\label{quadratic:acc:main:3}
\\
&\mathbb{E}\left[\left\Vert x^{(k)}-x^{\ast}\right\Vert^{2}\right]
\leq
(C_{k})^{2}(1-\sqrt{\alpha\mu})^{2k}
\left(2\left\Vert\xi_{0}\xi_{0}^{T}\right\Vert
+\frac{2\alpha^{2}\sigma^{2}N}{1-(1-\sqrt{\alpha\mu})^{2}}\right)
+2\frac{\alpha^{2}C_{1}^{2}N}{(1-\gamma)^{2}}
\nonumber
\\
&\qquad\qquad\qquad
+\frac{\alpha^{2}\sigma^{2}}{d}\sum_{i=1}^{Nd}\frac{(1+\sqrt{\alpha\mu})(1+\sqrt{\alpha\mu}+(1-\sqrt{\alpha\mu})\mu_{i})}{(1-\mu_{i})(1+\sqrt{\alpha\mu}-(1-\sqrt{\alpha\mu})\mu_{i})(4-(1-\mu_{i})(3-\sqrt{\alpha\mu}))},
\label{quadratic:acc:main:4}
\end{align}
where $\mu_{i}$ are the eigenvalues of $\mathcal{W}-\alpha Q$
and
\begin{equation*}
C_{k}=\max\left\{2k-1,\max_{i:0<\mu_{i}<1-\alpha\mu}
\frac{1+\sqrt{\alpha\mu}+(1-\sqrt{\alpha\mu})\mu_{i}}{2\sqrt{\mu_{i}(1-\alpha\mu-\mu_{i})}}\right\}.
\end{equation*}
\end{theorem}
Before we proceed to the proof of Theorem~\ref{thm:quadratic:DASG}, let us first derive
the following lemma providing an upper bound
on the norm of $A_{\text{dasg},Q}^{k}$
for every $k\in\mathbb{N}$, which will be used later.

\begin{lemma}\label{lem:DASG:Q}
For any $k\in\mathbb{N}$,
\begin{equation*}
\left\Vert A_{\text{dasg},Q}^{k}\right\Vert\leq
C_{k}\rho_{\text{dasg}}^{k},
\end{equation*}
where
\begin{align*}
&C_{k}:=
\max\left\{
2k-1
,\max_{i:\gamma_{i,+}\neq\gamma_{i,-}}
\frac{1+\max\{|\gamma_{i,+}|,|\gamma_{i,-}|\}^{2}}{|\gamma_{i,+}-\gamma_{i,-}|}
\right\},
\\
&\rho_{\text{dasg}}:=
\max_{1\leq i\leq Nd}\max\{|\gamma_{i,+}|,|\gamma_{i,-}|\},
\end{align*}
where
$\gamma_{i,\pm}:=\frac{(1+\beta)\mu_{i}\pm\sqrt{(1+\beta)^{2}\mu_{i}^{2}-4\beta\mu_{i}}}{2}$,
and $\mu_{i}$ are the eigenvalues of $A_{Q}=\mathcal{W}-\alpha Q$.

In particular, when $\beta=\frac{1-\sqrt{\alpha\mu}}{1+\sqrt{\alpha\mu}}$,
$\lambda_{N}^{W}>0$ and $\alpha\in(0,\frac{\lambda_{N}^{W}}{L}]$,
we have $\rho_{\text{dasg}}=1-\sqrt{\alpha\mu}$,
and
\begin{equation*}
C_{k}=\max\left\{2k-1,\max_{i:0<\mu_{i}<1-\alpha\mu}
\frac{1+\sqrt{\alpha\mu}+(1-\sqrt{\alpha\mu})\mu_{i}}{2\sqrt{\mu_{i}(1-\alpha\mu-\mu_{i})}}\right\}.
\end{equation*}
\end{lemma}

\begin{proof}
{\color{black}The proof of Lemma~\ref{lem:DASG:Q} will be provided in Appendix~\ref{sec:technical}.}
\end{proof}

Now, we are ready to prove Theorem~\ref{thm:quadratic:DASG}.

\subsubsection{Proof of Theorem~\ref{thm:quadratic:DASG}}

\begin{proof}
We recall that
\begin{equation*}
\xi_{k+1}=A_{\text{dasg},Q}\xi_{k}+B_{\text{dasg}}w^{(k+1)},
\end{equation*}
and therefore, we get:
\begin{equation}\label{eqn:iterate:k:DASG}
\mathbb{E}\left[\xi_{k}\xi_{k}^{T}\right]
=
A_{\text{dasg},Q}\mathbb{E}\left[\xi_{k-1}\xi_{k-1}^{T}\right](A_{\text{dasg},Q})^{T}
+
\left(\begin{array}{cc}
\alpha^{2}\frac{\sigma^{2}}{d}I_{Nd} & 0_{Nd}
\\
0_{Nd} & 0_{Nd}
\end{array}\right),
\end{equation}
Therefore, 
\begin{equation*}
X_{\text{dasg}}:=\mathbb{E}\left[\xi_{\infty}\xi_{\infty}^{T}\right]
\end{equation*}
satisfies the discrete Lyapunov equation:
\begin{equation*}
X_{\text{dasg}}=A_{\text{dasg},Q}X_{\text{dasg}}(A_{\text{dasg},Q})^{T}
+\left(\begin{array}{cc}
\alpha^{2}\frac{\sigma^{2}}{d}I_{Nd} & 0_{Nd}
\\
0_{Nd} & 0_{Nd}
\end{array}\right).
\end{equation*}
By Theorem~\ref{prop:quadratic:var:acc}
we have
\begin{equation*}
\mbox{trace}(X_{\text{dasg}})=\frac{\sigma^{2}}{d}\sum_{i=1}^{Nd}\frac{\alpha^{2}(1+\beta\mu_{i})}{(1-\mu_{i})(1-\beta\mu_{i})(2+2\beta-(1-\mu_{i})(1+2\beta))}.
\end{equation*}

Next by iterating equation \eqref{eqn:iterate:k:DASG} over $k$, we immediately obtain
\begin{align*}
\mathbb{E}\left[\xi_{k}\xi_{k}^{T}\right]
&=\left(A_{\text{dasg},Q}\right)^{k}\xi_{0}\xi_{0}^{T}
\left((A_{\text{dasg},Q})^{T}\right)^{k}
\\
&\qquad\qquad\qquad
+
\sum_{j=0}^{k-1}
\left(A_{\text{dasg},Q}\right)^{j}\left(\begin{array}{cc}
\alpha^{2}\frac{\sigma^{2}}{d}I_{Nd} & 0_{Nd}
\\
0_{Nd} & 0_{Nd}
\end{array}\right)\left((A_{\text{dasg},Q})^{T}\right)^{j},
\end{align*}
so that
\begin{align*}
\mathbb{E}\left[\xi_{k}\xi_{k}^{T}\right]
&=\mathbb{E}\left[\xi_{\infty}\xi_{\infty}^{T}\right]
+\left(A_{\text{dasg},Q}\right)^{k}
\xi_{0}\xi_{0}^{T}\left(\left(A_{\text{dasg},Q}\right)^{T}\right)^{k}
\\
&\qquad\qquad
-\sum_{j=k}^{\infty}
\left(A_{\text{dasg},Q}\right)^{j}\left(\begin{array}{cc}
\alpha^{2}\frac{\sigma^{2}}{d}I_{Nd} & 0_{Nd}
\\
0_{Nd} & 0_{Nd}
\end{array}\right)\left(\left(A_{\text{dasg},Q}\right)^{T}\right)^{j},
\end{align*}
which implies that
\begin{align*}
\mbox{trace}\left(\mathbb{E}\left[\xi_{k}\xi_{k}^{T}\right]\right)
&=\mbox{trace}\left(\mathbb{E}\left[\xi_{\infty}\xi_{\infty}^{T}\right]\right)
+\left(A_{\text{dasg},Q}\right)^{k}\xi_{0}\xi_{0}^{T}\left((A_{\text{dasg},Q})^{T}\right)^{k}
\\
&\qquad
-\sum_{j=k}^{\infty}
\left(A_{\text{dasg},Q}\right)^{j}\left(\begin{array}{cc}
\alpha^{2}\frac{\sigma^{2}}{d}I_{Nd} & 0_{Nd}
\\
0_{Nd} & 0_{Nd}
\end{array}\right)\left((A_{\text{dasg},Q})^{T}\right)^{j}
\\
&\leq
\mbox{trace}(X_{\text{dasg}})
+\left\Vert (A_{\text{dasg},Q})^{k}\right\Vert^{2}
\Vert\xi_{0}\xi_{0}^{T}\Vert
+\sum_{j=k}^{\infty}\left\Vert (A_{\text{dasg},Q})^{j}\right\Vert^{2}\alpha^{2}\sigma^{2}N
\\
&\leq\mbox{trace}(X_{\text{dasg}})
+(C_{k})^{2}(\rho_{\text{dasg}})^{2k}\Vert\xi_{0}\xi_{0}^{T}\Vert
+\alpha^{2}\sigma^{2}N (C_{k})^{2}\frac{(\rho_{\text{dasg}})^{2k}}{1-(\rho_{\text{dasg}})^{2}},
\end{align*}
where we used the estimate 
from the proof of Lemma~\ref{lem:DASG:Q}.

Note that $\xi_{k}=x^{(k)}-x^{\infty}$,
this proves \eqref{quadratic:acc:main:1}.
Moreover,
\begin{equation*}
\left\Vert x^{(k)}-x^{\ast}\right\Vert^{2}
\leq 2\left\Vert x^{(k)}-x^{\infty}\right\Vert^{2}
+2\left\Vert x^{\infty}-x^{\ast}\right\Vert^{2},
\end{equation*}
and together with \eqref{eqn:asymp_suboptim},
it proves \eqref{quadratic:acc:main:2}.

Finally, when $\beta=\frac{1-\sqrt{\alpha\mu}}{1+\sqrt{\alpha\mu}}$, $\lambda_{N}^{W}>0$
and $\alpha\in(0,\frac{\lambda_{N}^{W}}{L}]$, 
we have $\rho_{\text{dasg}}=1-\sqrt{\alpha\mu}$,
and
\begin{equation*}
\left\Vert A_{\text{dasg},Q}^{k}\right\Vert
\leq C_{k}\cdot(1-\sqrt{\alpha\mu})^{k},
\end{equation*}
where
\begin{equation*}
C_{k}=\max\left\{2k-1,\max_{i:0<\mu_{i}<1-\alpha\mu}
\frac{1+\sqrt{\alpha\mu}+(1-\sqrt{\alpha\mu})\mu_{i}}{2\sqrt{\mu_{i}(1-\alpha\mu-\mu_{i})}}\right\},
\end{equation*}
and by Theorem~\ref{prop:quadratic:var:acc}
with $\beta=\frac{1-\sqrt{\alpha\mu}}{1+\sqrt{\alpha\mu}}$
we have
\begin{align*}
\mbox{trace}(X_{\text{dasg}})
&=\frac{\sigma^{2}}{d}\sum_{i=1}^{Nd}\frac{\alpha^{2}(1+\beta\mu_{i})}{(1-\mu_{i})(1-\beta\mu_{i})(2+2\beta-(1-\mu_{i})(1+2\beta))}
\\
&=\frac{\alpha^{2}\sigma^{2}}{d}\sum_{i=1}^{Nd}\frac{(1+\sqrt{\alpha\mu})(1+\sqrt{\alpha\mu}+(1-\sqrt{\alpha\mu})\mu_{i})}{(1-\mu_{i})(1+\sqrt{\alpha\mu}-(1-\sqrt{\alpha\mu})\mu_{i})(4-(1-\mu_{i})(3-\sqrt{\alpha\mu}))}.
\end{align*}
The proof is complete.
\end{proof}


\section{Proofs of Main Results in Section~\ref{sec:multi}}

\subsubsection{Proof of Proposition~\ref{thm_one_stage}}

\begin{proof}
Recall that, from Theorem~\ref{thm-rate-dasg}, we have
\begin{equation}\label{thm_1_DMASG_ineq1}
\E\left[ \left\| x^{(k)} -x^* \right\|^2\right]
\leq
 4\left(1-\sqrt{\alpha \mu}\right)^{k} \frac{ V_{S,\alpha}\left(\xi_{0}\right)}{\mu}
+ \frac{2\sigma^2 N \alpha}{\mu\sqrt{\alpha\mu}}\left(2-\lambda_N^W +\alpha L \right) +  \frac{2C_{1}^{2}N\alpha^2}{(1-\gamma)^2}.    
\end{equation}
Next, note that, {\color{black}$\xi_{0} = \left [{\left(x^{(0)}-x^{\infty}\right)}^\top,{\left(x^{(0)}-x^{\infty}\right)}^\top \right ]^\top$}, and therefore,
\begin{align*}
V_{S,\alpha}\left(\xi_{0}\right) &= \rev{\xi_{0}^\top S_{\alpha}\xi_{0}} \nonumber\\
&= \left\|x^{(0)} - x^\infty \right\|^2 \left (\frac{1}{2\alpha}+\left(\sqrt{\frac{\mu}{2}}-\sqrt{\frac{1}{2\alpha}}\right)^2+\frac{\sqrt{2}}{\sqrt{\alpha}}\left(\sqrt{\frac{\mu}{2}}-\sqrt{\frac{1}{2\alpha}}\right)\right ) \nonumber\\
&= \frac{\mu}{2} \left\|x^{(0)} - x^\infty \right\|^2 \nonumber\\
& \leq \mu \left\|x^{(0)} - x^* \right\|^2 + \mu \left\|x^\infty - x^* \right\|^2 \nonumber \\
& \leq \mu \left ( \left\|x^{(0)} - x^* \right\|^2 + \frac{C_{1}^{2}N\alpha^2}{(1-\gamma)^2} \right ), \nonumber
\end{align*}
where the last inequality follows from \eqref{eqn:asymp_suboptim}. Plugging this bound in \eqref{thm_1_DMASG_ineq1} along with these straightforward inequalities
\begin{equation*}
1- \sqrt{\alpha \mu} \leq \exp(- \sqrt{\alpha \mu}), \quad 2-\lambda_N^W +\alpha L \leq 3, \quad 2 +4\left(1-\sqrt{\alpha \mu}\right)^{k} \leq 6,
\end{equation*}
completes the proof.
\end{proof}

\subsubsection{Proof of Corollary~\ref{cor_one_stage}}

\begin{proof}
First of all, notice that $x\mapsto\frac{\log x}{x}$ is decreasing for any $x\geq e$.
To simplify the notation, let $\hat{k} = \max\{2 \log(p \sqrt{\tilde{{\kappa}}}), e \}$. First note that, since $k \geq p \sqrt{\tilde{{\kappa}}} \hat{k}$, we have
\begin{align}
\frac{p \sqrt{\tkappa} \log k}{k} & \leq  \frac{p \sqrt{\tkappa} \log (p \sqrt{\tkappa} \hat{k})}{p \sqrt{\tkappa}\hat{k} } = \frac{\log (p \sqrt{\tkappa}) + \log \hat{k}}{\hat{k}}\leq \frac{1}{2} + \frac{\log \hat{k}}{\hat{k}} \leq 1, 
\end{align}
where the second inequality follows from $\hat{k} \geq 2 \log (p \sqrt{\tkappa})$ and the last inequality is obtained using $\hat{k} \geq e$. Hence, $\alpha_1$ satisfies the condition $\alpha_1 \leq \min\{{\lambda_N^W}/{L}, 1/(L+\mu)\}$ in Proposition~\ref{thm_one_stage}. In addition, note that $\alpha_1$ can be written as
\begin{equation*}
\alpha_1 = \frac{\lambda_N^W}{L+\mu} \left ({p \sqrt{\tilde{\kappa}} \log k}/{k}\right )^2 = \frac{1}{\mu} \left ({p \log k}/{k}\right )^2.    
\end{equation*}
Plugging this into Proposition~\ref{thm_one_stage} completes the proof.
\end{proof}


\subsubsection{Proof of Proposition~\ref{main_DMASG_result}}

\begin{proof}
We show this result by using induction. First note that, for $t=0$, the argument holds using Proposition~\ref{thm_one_stage}. Now, assume it holds for $t$ and we show it for $t+1$. Using Proposition~\ref{thm_one_stage}, and taking expectation from both sides, we have
\begin{align}
\E&\left[ \left\| x^{(L_{t+2})} -x^* \right\|^2\right] \nonumber \\
\leq & 4 \exp\left(-k_{t+1} \sqrt{\alpha_{t+1} \mu}\right) \E\left[ \left\| x^{(L_{t+1})} -x^* \right\|^2\right]
+ 6N \left ( \frac{\sqrt{\alpha_{t+1}}}{\mu\sqrt{\mu}} \sigma^2 +  \frac{C_{1}^{2}\alpha_{t+1}^2}{(1-\gamma)^2} \right) \nonumber \\
=& \frac{1}{2^{p-2}} \E\left[ \left\| x^{(L_{t+1})} -x^* \right\|^2\right]
+ \frac{6N}{2^{t+1}} \sqrt{\frac{\lambda_N^W }{(L+\mu) \mu^3}} \sigma^2 + \frac{6N}{2^{4(t+1)}} \left (\frac{C_1 \lambda_N^W}{(L+\mu)(1-\gamma)} \right )^2 \label{thm_multi1_eq1} \\ 
\leq & \frac{4}{2^{(p-2)(t+1)}} \exp\left(-\frac{k_1}{\sqrt{\tilde{\kappa}}}\right) \left\| x^{(0)} - x^* \right\|^2 + 12N\left(\frac{1/2^{(p-2)}}{2^t}+\frac{1/2}{2^{t+1}}\right) \frac{\sigma^2}{\mu^2 \sqrt{\tilde{\kappa}}} \nonumber \\ 
& \quad + 12N\left(\frac{1/2^{(p-2)}}{2^{4t}}+\frac{1/2}{2^{4(t+1)}}\right) \left (\frac{C_1 \lambda_N^W}{(L+\mu)(1-\gamma)} \right )^2 \label{thm_multi1_ineq1} \\ 
\leq & \frac{4}{2^{(p-2)(t+1)}} \exp\left(-\frac{k_1}{\sqrt{\tilde{\kappa}}}\right) \left\| x^{(0)} - x^* \right\|^2 + \frac{12N}{2^{t+1}} \frac{\sigma^2}{\mu^2 \sqrt{\tilde{\kappa}}} + \frac{12N}{2^{4(t+1)}} \left (\frac{C_1 \lambda_N^W}{L(1-\gamma)} \right )^{2}, \label{thm_multi1_ineq2}
\end{align}
where \eqref{thm_multi1_eq1} follows from substituting $\alpha_{t+1}$ and $k_{t+1}$ and \eqref{thm_multi1_ineq1} is obtained using the induction hypothesis for $t$. Finally, \eqref{thm_multi1_ineq2} is obtained by replacing $L+\mu$ by $L$ in \eqref{thm_multi1_ineq1} along with the assumption $p \geq 7$ so that the term $12N\left(\frac{1/2^{(p-2)}}{2^{4t}}+\frac{1/2}{2^{4(t+1)}}\right)$ in network effect in \eqref{thm_multi1_ineq1} is bounded by $\frac{12N}{2^{4(t+1)}}$ in \eqref{thm_multi1_ineq2}, which completes the proof.
\end{proof}


\subsubsection{Proof of Proposition~\ref{D-MASG_any_n}}

\begin{proof}
Let $T$ denote the largest $t$ such that $k \geq L_t$. In particular, we have
$$ L_T \leq k < L_{T+1}.$$
Now, using Proposition~\ref{thm_one_stage}, we have
\begin{align}
\E & \left[ \left\| x^{(k)} -x^* \right\|^2\right] \leq
 4 \E \left[\left\|x^{(L_T)} -x^*\right\|^2 \right] + 6N \left ( \frac{\sqrt{\alpha_{T+1}}}{\mu\sqrt{\mu}} \sigma^2 +  \frac{C_{1}^{2}\alpha_{T+1}^2}{(1-\gamma)^2} \right) \nonumber \\
= & 4 \E \left[\left\|x^{(L_T)} -x^*\right\|^2 \right] + \frac{6N}{2^{T+1}} \sqrt{\frac{\lambda_N^W }{(L+\mu) \mu^3}} \sigma^2 + \frac{6N}{2^{4(T+1)}} \left (\frac{C_1 \lambda_N^W}{(L+\mu)(1-\gamma)} \right )^2 \nonumber \\
\leq & \bigO(1) \left ( \frac{1}{2^{(p-2)T}} \exp\left(-\frac{k_1}{\sqrt{\tilde{\kappa}}}\right) \left\| x^{(0)} - x^* \right\|^2 + \frac{N \sigma^2}{2^T \mu^2 \sqrt{\tilde{\kappa}}}  + \frac{N}{2^{4T}} \left (\frac{C_1 \lambda_N^W}{L(1-\gamma)} \right )^2  \right ), \label{thm_multi2_ineq1}
\end{align}
where the last inequality follows from Proposition~\ref{main_DMASG_result}.
Next, note that $$k-k_1 \leq L_{T+1} - k_1 \leq 2 (L_T - k_1) \leq p 2^{T+3} \log(2) \sqrt{\tilde{\kappa}},$$
where the last two inequalities follows from the special pattern of the sequence $\{k_i\}_i$.
Therefore, we have
$$ \frac{1}{2^T} \leq \frac{8 p \log(2) \sqrt{\tilde{\kappa}}}{k-k_1} \leq \frac{6 p \sqrt{\tilde{\kappa}}}{k-k_1}.$$
Plugging this bound in \eqref{thm_multi2_ineq1} completes the proof.
\end{proof}


\subsubsection{Proof of Corollary~\ref{cor_D-MASG:bar}}

\begin{proof}
{\color{black}
By Proposition~\ref{prop:dsg-average-optimal-rate-DASG}, 
for any $k$, we have
\begin{align*}
&\mathbb{E}\left\Vert\bar{x}^{(k)}-x_{\ast}\right\Vert^{2}
\\
&\leq
\left(1-\frac{\sqrt{\alpha\mu}}{2}\right)^{k}\frac{2V_{\bar{S},\alpha}\left(\bar{\xi}_{0}\right)}{\mu}
+\frac{4}{\mu\sqrt{\mu}}
\left(\left(1+\frac{\sqrt{\alpha\mu}}{2}\right)\frac{\sigma^2  \sqrt{\alpha}}{2N}\left(1+\alpha L \right)+\frac{1}{2\sqrt{\mu}}\alpha H_{1}H_{2}\right)
\\
&\qquad
+\frac{8}{\gamma^{2}\mu\sqrt{\mu}}\sqrt{\alpha}H_{1}H_{3}\frac{\gamma^{2k}-(1-\sqrt{\alpha\mu}/2)^{k}}{\gamma^{2}-(1-\sqrt{\alpha\mu}/2)}.
\end{align*}
where $V_{\bar{S},\alpha}$ is defined by \eqref{eqn:V:S:alpha:bar}. As $\alpha\rightarrow 0$, one can check
that $H_{1}=\mathcal{O}(1)$, 
$H_{2}=\mathcal{O}(1)\frac{L^{2}}{N}\frac{C_{0}}{(1-\gamma)^{2}}$
and $H_{3}=\mathcal{O}(1)\frac{L^{2}}{N}$
since it follows from the proof of Proposition~\ref{thm_one_stage} that $V_{S,\alpha}(\xi_{0})\leq\mu\Vert x^{(0)}-x^{\ast}\Vert^{2}+\mu\frac{C_{1}^{2}N\alpha^{2}}{(1-\gamma)^{2}}$.
When $\alpha$ is sufficiently small, 
\begin{align*}
\frac{8}{\gamma^{2}\mu\sqrt{\mu}}\sqrt{\alpha}H_{1}H_{3}\frac{\gamma^{2k}-(1-\sqrt{\alpha\mu}/2)^{k}}{\gamma^{2}-(1-\sqrt{\alpha\mu}/2)}
&\leq
\frac{8}{\gamma^{2}\mu\sqrt{\mu}}\sqrt{\alpha}H_{1}H_{3}\frac{(1-\sqrt{\alpha\mu}/2)^{k}}{(1-\sqrt{\alpha\mu}/2)-\gamma^{2}}
\\
&\leq
\left(1-\frac{\sqrt{\alpha\mu}}{2}\right)^{k}\frac{2V_{\bar{S},\alpha}\left(\bar{\xi}_{0}\right)}{\mu}.
\end{align*}
Moreover, it follows similarly as in the proof of Proposition~\ref{thm_one_stage} that $V_{\bar{S},\alpha}(\bar{\xi}_{0})\leq\frac{\mu}{N}\Vert x^{(0)}-x^{\ast}\Vert^{2}+\mu\frac{C_{1}^{2}\alpha^{2}}{(1-\gamma)^{2}}$.
Hence, as $\alpha\rightarrow 0$, we have
\begin{align*}
&\mathbb{E}\left\Vert\bar{x}^{(k)}-x_{\ast}\right\Vert^{2}
\\
&\leq
\mathcal{O}(1)\left(\left(1-\frac{\sqrt{\alpha\mu}}{2}\right)^{k}\frac{1}{N}\Vert x^{(0)}-x^{\ast}\Vert^{2}
+\frac{1}{\mu\sqrt{\mu}}
\left(\frac{\sigma^2  \sqrt{\alpha}}{N}+\frac{1}{\sqrt{\mu}}\alpha \frac{L^{2}}{N}\frac{C_{0}}{(1-\gamma)^{2}}\right)\right).
\end{align*}
Then, similar as in Corollary~\ref{cor_D-MASG}, we can show
that by choosing $k_1 = \ceil{(p-2) \log(6 p \tilde{\kappa})\sqrt{\tilde{\kappa}}}$, we have
\begin{align*}
\E \left[ \left\| \bar{x}^{(k)} -x_* \right\|^2\right] \leq \bigO(1) \left ( \frac{1}{k^{p-2}} \frac{\left\| x^{(0)} - x^* \right\|^2}{N} + \frac{p \sigma^2}{N\mu\sqrt{\mu} k}  + \frac{p^4 C_{0}L^{2}( 1-\gamma)^{-2}}{N\mu^2 k^4} \right ),
\end{align*}
for any $k \geq 2 k_1$.
Also, for a given number of iterations, $k$, by choosing $p=7$
and $k_1 = \ceil{\frac{k}{C}}$ for some constant $C \geq 2$, we have
\begin{align*}
\E \left[ \left\| \bar{x}^{(k)} -x_* \right\|^2\right] \leq \bigO(1) \left ( \exp\left(-\frac{k}{C\sqrt{\tilde{\kappa}}}\right)\frac{\left\| x^{(0)} - x^* \right\|^2}{N} + \frac{  \sigma^2}{N\mu\sqrt{\mu}k}  + \frac{C_{0}L^{2}( 1-\gamma)^{-2}}{N\mu^2 k^4} \right ),
\end{align*}
for any $k \geq 2 \sqrt{\tilde{\kappa}}$, {\color{black}where $C_{1},\gamma$ are given in \eqref{eqn:asymp_suboptim} and $\tilde{\kappa}$ is given in \eqref{tilde_kappa}}.

Moreover, we recall from Proposition~\ref{prop:dsg-average-optimal-rate-DASG} that
for every $i=1,2,\ldots,N$ and any $k$,
\begin{align*}
&\mathbb{E}\left\Vert x_{i}^{(k)}-x_{\ast}\right\Vert^{2}
\\
&\leq
\left(1-\frac{\sqrt{\alpha\mu}}{2}\right)^{k}\frac{4V_{\bar{S},\alpha}\left(\bar{\xi}_{0}\right)}{\mu}
+\frac{8}{\mu\sqrt{\mu}}
\left(\left(1+\frac{\sqrt{\alpha\mu}}{2}\right)\frac{\sigma^2  \sqrt{\alpha}}{2N}\left(1+\alpha L \right)+\frac{1}{2\sqrt{\mu}}\alpha H_{1}H_{2}\right)
\\
&\qquad
+\frac{16}{\gamma^{2}\mu\sqrt{\mu}}\sqrt{\alpha}H_{1}H_{3}\frac{\gamma^{2k}-(1-\sqrt{\alpha\mu}/2)^{k}}{\gamma^{2}-(1-\sqrt{\alpha\mu}/2)}
\\
&\qquad
+16\gamma^{2k}\left(4\frac{ V_{S,\alpha}\left(\xi_{0}\right)}{\mu}
+ \frac{2\sigma^2 N \sqrt{\alpha}}{\mu\sqrt{\mu}}\left(2-\lambda_N^W +\alpha L \right) +  \frac{2C_{1}^{2}N\alpha^2}{(1-\gamma)^2}
+\Vert x^{\ast}\Vert^{2}\right)
\\
&\qquad\qquad
+\frac{8D_{y}^{2}\alpha^{2}}{(1-\gamma)^{2}}
+\frac{8\sigma^{2}N\alpha^{2}}{(1-\gamma)^{2}}
+\frac{16C_{0}\alpha}{(1-\gamma)^{2}}.
\end{align*}
When $\alpha$ is sufficiently small, 
\begin{align*}
&16\gamma^{2k}\left(4\frac{ V_{S,\alpha}\left(\xi_{0}\right)}{\mu}
+ \frac{2\sigma^2 N \sqrt{\alpha}}{\mu\sqrt{\mu}}\left(2-\lambda_N^W +\alpha L \right) +  \frac{2C_{1}^{2}N\alpha^2}{(1-\gamma)^2}
+\Vert x^{\ast}\Vert^{2}\right)
\\
&\leq
\left(1-\frac{\sqrt{\alpha\mu}}{2}\right)^{k}\frac{4V_{\bar{S},\alpha}\left(\bar{\xi}_{0}\right)}{\mu}.
\end{align*}
Similar as before, we can show that
as $\alpha\rightarrow 0$, we have
\begin{align*}
&\mathbb{E}\left\Vert x_{i}^{(k)}-x_{\ast}\right\Vert^{2}
\\
&\leq
\mathcal{O}(1)\left(\left(1-\frac{\sqrt{\alpha\mu}}{2}\right)^{k}\frac{\Vert x^{(0)}-x^{\ast}\Vert^{2}}{N}
+\frac{1}{\mu\sqrt{\mu}}
\left(\frac{\sigma^2  \sqrt{\alpha}}{N}+\frac{1}{\sqrt{\mu}}\alpha \frac{L^{2}}{N}\frac{C_{0}}{(1-\gamma)^{2}}\right)
+\frac{C_{0}\alpha}{(1-\gamma)^{2}}\right),
\end{align*}
and thus similar as in Corollary~\ref{cor_D-MASG}, we can show
that by choosing $k_1 = \ceil{(p-2) \log(6 p \tilde{\kappa})\sqrt{\tilde{\kappa}}}$, we have
\begin{align*}
&\E \left[ \left\| x_{i}^{(k)} -x_* \right\|^2\right] 
\\
&\leq \bigO(1) \left ( \frac{1}{k^{p-2}} \frac{\left\| x^{(0)} - x^* \right\|^2}{N} + \frac{p \sigma^2}{N\mu\sqrt{\mu} k}  + \left(\frac{L^{2}}{N\mu^{2}}+1\right)\frac{p^4 C_{0}( 1-\gamma)^{-2}}{k^4} \right ),
\end{align*}
for any $k \geq 2 k_1$ and $i=1,2,\ldots,N$.
Also, for a given number of iterations, $k$, by choosing $p=7$
and $k_1 = \ceil{\frac{k}{C}}$ for some constant $C \geq 2$, we have
\begin{align*}
&\E \left[ \left\| x_{i}^{(k)} -x_* \right\|^2\right] 
\\
&\leq \bigO(1) \left ( \exp\left(-\frac{k}{C\sqrt{\tilde{\kappa}}}\right)\frac{\left\| x^{(0)} - x^* \right\|^2}{N} + \frac{  \sigma^2}{N\mu\sqrt{\mu}k}  + \left(\frac{L^{2}}{N\mu^{2}}+1\right)\frac{C_{0}( 1-\gamma)^{-2}}{k^4} \right ),
\end{align*}
for any $k \geq 2 \sqrt{\tilde{\kappa}}$ and $i=1,2,\ldots,N$, {\color{black}where $C_{1},\gamma$ are given in \eqref{eqn:asymp_suboptim} and $\tilde{\kappa}$ is given in \eqref{tilde_kappa}}.
The proof is complete.}
\end{proof}



%




\section{{\color{black}Results for More General Noise Setting}}\label{sec:general:noise}

{\color{black}Consider the following assumption on noise which is more general than Assumption~\ref{assump_1},
and we will show that the main results in this paper for D-SG and D-ASG still hold
under this more general assumption on gradient noise.}

\begin{assumption}\label{assump:unbounded}
{\color{black}Recall that $x_i^{(k)}$ denotes the decision variable of node $i$ at iteration $k$. We assume at iteration $k$, node $i$ has access to $\tilde{\nabla} f_i \left(x_i^{(k)}, w_i^{(k)}\right)$ which is an estimate of $\nabla f_i \left(x_i^{(k)}\right)$ where $w_{i}^{(k)}$ is a random variable independent of $\left\{w_{j}^{(t)}\right\}_{j=1,\ldots,N, t=1,\ldots,k-1}$ and $\left\{w_{j}^{(k)}\right\}_{j \neq i}$. Moreover, we assume $\mathbb{E}\left[\tilde{\nabla} f_i \left(x_i^{(k)}, w_i^{(k)}\right)\Big|x_i^{(k)}\right]= \nabla f_i \left(x_i^{(k)}\right)$ and
\begin{equation*}
\mathbb{E}\left[\left\Vert \tilde{\nabla} f_i \left(x_i^{(k)}, w_i^{(k)}\right) - \nabla f_i \left(x_i^{(k)}\right) \right\Vert^{2}\Big|x_i^{(k)}\right]\leq\sigma^{2}
+\frac{\eta^{2}}{2}\left\Vert x_{i}^{(k)}-x_{\ast}\right\Vert^{2}. 
\end{equation*}	
for some constant $\eta > 0$. To simplify the notation, we suppress the $w_i^{(k)}$ dependence, and denote $\tilde{\nabla} f_i \left(x_i^{(k)}, w_i^{(k)}\right)$ by $\tilde{\nabla} f_i \left(x_i^{(k)}\right)$.}
\end{assumption}
\rev{Such assumptions could hold if gradients are estimates from batches (randomly selected subset of data points) in the context of empirical risk minimization problems \citep{jain2018accelerating, gurbuzbalaban2020decentralized}. The constant $\eta^2$ is often inversely proportional to the batch size (see e.g. \cite{raginsky2017non})}.
\subsection{Distributed stochastic gradient (D-SG)}

{\color{black}Let us recall the D-SG in \eqref{eqn:dsg_update},
which takes the equivalent form \eqref{eq-iter-dgd}.
Then it follows from Assumption~\ref{assump:unbounded} that $\mathbb{E}\left[\tilde{\nabla}F\left(x^{(k)}\right)\Big|x^{(k)}\right]= \nabla F\left(x^{(k)}\right)$ and
\begin{equation}
\mathbb{E}\left[\left\Vert \tilde{\nabla}F\left(x^{(k)}\right) - \nabla F\left(x^{(k)}\right) \right\Vert^{2}\Big|x^{(k)}\right]
\leq\sigma^{2}N+\frac{\eta^{2}}{2}\left\Vert x^{(k)}-x^{\ast}\right\Vert^{2}.
\end{equation}
We recall that $\Vert x^{\infty}-x^{\ast}\Vert\leq\frac{\alpha C_{1}\sqrt{N}}{(1-\gamma)}$
from \eqref{eqn:asymp_suboptim}.
Therefore, we have
\begin{align}
\mathbb{E}\left[\left\Vert \tilde{\nabla}F\left(x^{(k)}\right) - \nabla F\left(x^{(k)}\right) \right\Vert^{2}\Big|x^{(k)}\right]
&\leq\sigma^{2}N+\eta^{2}\left\Vert x^{(k)}-x^{\infty}\right\Vert^{2}
+\eta^{2}\left\Vert x^{\infty}-x^{\ast}\right\Vert^{2}
\nonumber
\\
&\leq
\left(\sigma^{2}+\eta^{2}\frac{\alpha^{2}(C_{1})^{2}}{(1-\gamma)^{2}}\right)N+\eta^{2}\left\Vert x^{(k)}-x^{\infty}\right\Vert^{2}.
\end{align}}

{\color{black}We have the following explicit performance bounds on the convergence and the robustness of D-SG iterates.}

\begin{theorem}\label{thm:DSG:unbounded}
{\color{black}Assume that $\alpha\leq\frac{\frac{1}{2}+\lambda_{N}^{W}}{L+\frac{\eta^{2}}{\mu}}$.
For any $k\geq 0$, 
\begin{align*}
\mathbb{E}\left\Vert x^{(k)}-x^{\ast}\right\Vert^{2}
\leq
2\left(1-\alpha\mu/2\right)^{k}\mathbb{E}\left\Vert x^{(0)}-x^{\infty}\right\Vert^{2}
+\frac{4\alpha}{\mu}\left(\sigma^{2}+\eta^{2}\frac{\alpha^{2}(C_{1})^{2}}{(1-\gamma)^{2}}\right)N
+\frac{2\alpha^{2}(C_{1})^{2}N}{(1-\gamma)^{2}}.
\end{align*}}
\end{theorem}

\begin{proof}
{\color{black}The D-SG iterates are given by
\begin{equation}
x^{(k+1)}=x^{(k)}-\alpha\nabla F_{\mathcal{W},\alpha}\left(x^{(k)}\right)-\alpha\xi^{(k+1)},
\end{equation}
where $\mathbb{E}\left[\xi^{(k+1)}|\mathcal{F}_{k}\right]=0$
and
\begin{equation}\label{E:upper:bound:unbounded}
\mathbb{E}\left[\left\Vert\xi^{(k+1)}\right\Vert^{2}\Big|\mathcal{F}_{k}\right]
\leq
\left(\sigma^{2}+\eta^{2}\frac{\alpha^{2}(C_{1})^{2}}{(1-\gamma)^{2}}\right)N+\eta^{2}\left\Vert x^{(k)}-x^{\infty}\right\Vert^{2}.
\end{equation}
Therefore, we can compute that
\begin{align*}
\mathbb{E}\left\Vert x^{(k+1)}-x^{\infty}\right\Vert^{2}
&=\mathbb{E}\left\Vert x^{(k)}-x^{\infty}-\alpha\nabla F_{\mathcal{W},\alpha}\left(x^{(k)}\right)-\alpha\xi^{(k+1)}\right\Vert^{2}
\\
&=\mathbb{E}\left\Vert x^{(k)}-x^{\infty}-\alpha\nabla F_{\mathcal{W},\alpha}\left(x^{(k)}\right)\right\Vert^{2}
+\alpha^{2}\mathbb{E}\left\Vert\xi^{(k+1)}\right\Vert^{2}
\\
&=\mathbb{E}\left\Vert x^{(k)}-x^{\infty}\right\Vert^{2}
+\alpha^{2}\mathbb{E}\left\Vert\nabla F_{\mathcal{W},\alpha}\left(x^{(k)}\right)\right\Vert^{2}
\\
&\qquad\qquad
-2\alpha\mathbb{E}\left\langle x^{(k)}-x^{\infty},\nabla F_{\mathcal{W},\alpha}\left(x^{(k)}\right)\right\rangle
+\alpha^{2}\mathbb{E}\left\Vert\xi^{(k+1)}\right\Vert^{2}
\\
&=\mathbb{E}\left\Vert x^{(k)}-x^{\infty}\right\Vert^{2}
+\alpha^{2}L_{\alpha}\mathbb{E}\left\langle x^{(k)}-x^{\infty},\nabla F_{\mathcal{W},\alpha}\left(x^{(k)}\right)\right\rangle
\\
&\qquad\qquad
-2\alpha\mathbb{E}\left\langle x^{(k)}-x^{\infty},\nabla F_{\mathcal{W},\alpha}\left(x^{(k)}\right)\right\rangle
+\alpha^{2}\mathbb{E}\left\Vert\xi^{(k+1)}\right\Vert^{2}
\\
&\leq
\left(1-2\alpha\mu\left(1-\frac{\alpha L_{\alpha}}{2}\right)\right)\mathbb{E}\left\Vert x^{(k)}-x^{\infty}\right\Vert^{2}
+\alpha^{2}\mathbb{E}\left\Vert\xi^{(k+1)}\right\Vert^{2},
\end{align*}
where we used the fact that $\mathcal{F}_{\mathcal{W},\alpha}$ is $L_{\alpha}$-smooth
and $\mu$-strongly convex and the assumption 
$\alpha\leq\frac{\frac{1}{2}+\lambda_{N}^{W}}{L+\frac{\eta^{2}}{\mu}}<\frac{1+\lambda_{N}^{W}}{L}$ so that $\alpha L_{\alpha}=1-\lambda_{N}^{W}+\alpha L<2$.
By applying \eqref{E:upper:bound:unbounded}, 
we get
\begin{align*}
\mathbb{E}\left\Vert x^{(k+1)}-x^{\infty}\right\Vert^{2}
&\leq
\left(1-2\alpha\mu\left(1-\frac{\alpha L_{\alpha}}{2}-\frac{\alpha\eta^{2}}{2\mu}\right)\right)\mathbb{E}\left\Vert x^{(k)}-x^{\infty}\right\Vert^{2}
\\
&\qquad\qquad\qquad
+\alpha^{2}\left(\sigma^{2}+\eta^{2}\frac{\alpha^{2}(C_{1})^{2}}{(1-\gamma)^{2}}\right)N.
\end{align*}
We recall that assumption $\alpha\leq\frac{\frac{1}{2}+\lambda_{N}^{W}}{L+\frac{\eta^{2}}{\mu}}$
so that $1-\frac{\alpha L_{\alpha}}{2}-\frac{\alpha\eta^{2}}{2\mu}\geq\frac{1}{4}$.
Therefore, we have
\begin{align*}
\mathbb{E}\left\Vert x^{(k+1)}-x^{\infty}\right\Vert^{2}
\leq
\left(1-\alpha\mu/2\right)\mathbb{E}\left\Vert x^{(k)}-x^{\infty}\right\Vert^{2}
+\alpha^{2}\left(\sigma^{2}+\eta^{2}\frac{\alpha^{2}(C_{1})^{2}}{(1-\gamma)^{2}}\right)N,
\end{align*}
which implies that
\begin{align*}
\mathbb{E}\left\Vert x^{(k)}-x^{\infty}\right\Vert^{2}
\leq
\left(1-\alpha\mu/2\right)^{k}\mathbb{E}\left\Vert x^{(0)}-x^{\infty}\right\Vert^{2}
+\frac{2\alpha}{\mu}\left(\sigma^{2}+\eta^{2}\frac{\alpha^{2}(C_{1})^{2}}{(1-\gamma)^{2}}\right)N.
\end{align*}
Hence, we conclude that
\begin{align*}
\mathbb{E}\left\Vert x^{(k)}-x^{\ast}\right\Vert^{2}
\leq
2\left(1-\alpha\mu/2\right)^{k}\mathbb{E}\left\Vert x^{(0)}-x^{\infty}\right\Vert^{2}
+\frac{4\alpha}{\mu}\left(\sigma^{2}+\eta^{2}\frac{\alpha^{2}(C_{1})^{2}}{(1-\gamma)^{2}}\right)N
+\frac{2\alpha^{2}(C_{1})^{2}N}{(1-\gamma)^{2}}.
\end{align*}}
\end{proof}

{\color{black}
Next, we will provide the performance bounds for the average iterates and individual iterates.
Before we proceed, let us first introduce and prove a few technical lemmas.
Let us recall that
\begin{equation*}
\mathcal{E}_{k+1}
:= \nabla f\left(\bar{x}^{(k)}\right)-\frac{1}{N}\sum_{i=1}^{N}\nabla f_{i}\left(x_{i}^{(k)}\right).
\end{equation*}
Next, we will show that the error term $\mathcal{E}_{k+1}$ is small
and for every $i=1,2,\ldots,N$, $x_{i}^{(k)}$ is close to the average $\bar{k}^{(k)}$.}

\begin{lemma}\label{E:upper:bound:DSG:unbounded}
\rev{Assume that $\alpha\leq\frac{\frac{1}{2}+\lambda_{N}^{W}}{L+\frac{\eta^{2}}{\mu}}$ and $\alpha\mu(1+\lambda_{N}^{W}-\alpha L)<1$.
For any $k$ and $i=1,2,\ldots,N$, we have
\begin{equation}
\mathbb{E}\left\Vert x_{i}^{(k)}-\bar{x}^{(k)}\right\Vert^{2}
\leq
\sum_{i=1}^{N}\mathbb{E}\left\Vert x_{i}^{(k)}-\bar{x}^{(k)}\right\Vert^{2}
\leq
4\gamma^{2k}\mathbb{E}\left\Vert x^{(0)}\right\Vert^{2}
+\frac{4D_{1}^{2}\alpha^{2}}{(1-\gamma)^{2}}
+\frac{4N\alpha^{2}}{(1-\gamma^{2})}D_{2}^{2},
\end{equation}
and for any $k$, we have
\begin{align*}
\mathbb{E}\left\Vert\mathcal{E}_{k+1}\right\Vert^{2}
\leq
\frac{4L^{2}\gamma^{2k}}{N}\mathbb{E}\left\Vert x^{(0)}\right\Vert^{2}
+\frac{4L^{2}D_{1}^{2}\alpha^{2}}{N(1-\gamma)^{2}}
+\frac{4L^{2}\alpha^{2}}{(1-\gamma^{2})}D_{2}^{2},
\end{align*} 
where 
\begin{align}
&D_{1}^{2}:=L^{2}\mathbb{E}\left\Vert x^{(0)}-x^{\infty}\right\Vert^{2}
+L^{2}\frac{2\alpha}{\mu}\left(\sigma^{2}+\eta^{2}\frac{\alpha^{2}(C_{1})^{2}}{(1-\gamma)^{2}}\right)N,
\\
&D_{2}^{2}:=\left(\sigma^{2}+\eta^{2}\frac{\alpha^{2}(C_{1})^{2}}{(1-\gamma)^{2}}\right)\frac{\mu+2\alpha}{\mu}
+\frac{\eta^{2}}{N}\mathbb{E}\left\Vert x^{(0)}-x^{\infty}\right\Vert^{2}.
\end{align}} 
\end{lemma}

\begin{proof}
{\color{black}The proof of Lemma~\ref{E:upper:bound:DSG:unbounded} will be provided
in Appendix~\ref{sec:technical}.}
\end{proof}

{\color{black}Let us define $x_{k}$ as the iterates of the decentralized algorithm:
\begin{equation*} 
x_{k+1}=x_{k}-\alpha  \nabla f\left(x_{k}\right)
-\alpha\bar{\xi}^{(k+1)},
\end{equation*}
with $x_{0}=\bar{x}^{(0)}$, where we recall that
\begin{equation*}
\bar{\xi}^{(k+1)}:=\frac{1}{N}\sum_{i=1}^{N}\left(\tilde{\nabla}f_{i}\left(x_{i}^{(k)}\right)-\nabla f_{i}\left(x_{i}^{(k)}\right)\right),
\end{equation*}
so that we have 
\begin{equation}\label{bar:xi:ineq}
\mathbb{E}\left[\bar{\xi}^{(k+1)}\Big|\mathcal{F}_{k}\right]=0,
\qquad
\mathbb{E}\left\Vert\bar{\xi}^{(k+1)}\Big|\mathcal{F}_{k}\right\Vert^{2}
\leq\frac{\sigma^{2}}{N}
+\frac{\eta^{2}}{2N^{2}}\left\Vert x^{(k)}-x^{\ast}\right\Vert^{2}.
\end{equation}

Next, we will show that $x_{k}$ and the average iterates $\bar{x}^{(k)}$ are close
to each other in the $L^{2}$ norm.}

\begin{lemma}\label{lem:central:approx:unbounded}
{\color{black}Assume that $\alpha\leq\frac{\frac{1}{2}+\lambda_{N}^{W}}{L+\frac{\eta^{2}}{\mu}}$ and $\alpha\mu(1+\lambda_{N}^{W}-\alpha L)<1$.
For any $k$, we have
\begin{align*}
\mathbb{E}\left\Vert\bar{x}^{(k)}-x_{k}\right\Vert^{2}
&\leq\alpha\left(\frac{\alpha}{\mu(1-\frac{\alpha L}{2})}+\frac{(1+\alpha L)^{2}}{\mu^{2}(1-\frac{\alpha L}{2})^{2}}\right)
\left(\frac{4L^{2}D_{1}^{2}\alpha}{N(1-\gamma)^{2}}
+\frac{4L^{2}\alpha}{(1-\gamma^{2})}D_{2}^{2}\right)
\\
&\qquad\qquad
+\frac{\gamma^{2k}-
\left(1-\alpha\mu\left(1-\frac{\alpha L}{2}\right)\right)^{k}}
{\gamma^{2}-1+\alpha\mu\left(1-\frac{\alpha L}{2}\right)}
\frac{4L^{2}\gamma^{2}}{N}\mathbb{E}\left\Vert x^{(0)}\right\Vert^{2},
\end{align*}
where $D_{1},D_{2}$ are defined in Lemma~\ref{E:upper:bound:DSG:unbounded}.}
\end{lemma}

\begin{proof}
{\color{black}The proof is similar to that of Lemma~\ref{lem:central:approx}
and is hence omitted here.}
\end{proof}

{\color{black}Finally, we are ready to present the performance bounds for the average iterates and individual iterates.}

\begin{proposition}
{\color{black}Assume that $\alpha\leq\frac{\frac{1}{2}+\lambda_{N}^{W}}{L+\frac{\eta^{2}}{\mu}}$ and $\alpha\mu(1+\lambda_{N}^{W}-\alpha L)<1$.
For any $k\geq 0$,
\begin{align*}
\mathbb{E}\left\Vert\bar{x}^{(k)}-x_{\ast}\right\Vert^{2}
&\leq
2(1-\alpha\mu)^{k}\mathbb{E}\left\Vert\bar{x}^{(0)}-x_{\ast}\right\Vert^{2}
+\frac{2\alpha}{\mu}\frac{\sigma^{2}}{N}
\\
&\quad
+\frac{2\alpha\eta^{2}}{\mu N^{2}}\left(\mathbb{E}\left\Vert x^{(0)}-x^{\infty}\right\Vert^{2}
+\frac{2\alpha}{\mu}\left(\sigma^{2}+\eta^{2}\frac{\alpha^{2}(C_{1})^{2}}{(1-\gamma)^{2}}\right)N
+\frac{\alpha^{2}(C_{1})^{2}N}{(1-\gamma)^{2}}\right)
\\
&\qquad
+\alpha\left(\frac{\alpha}{\mu(1-\frac{\alpha L}{2})}+\frac{(1+\alpha L)^{2}}{\mu^{2}(1-\frac{\alpha L}{2})^{2}}\right)
\left(\frac{8L^{2}D_{1}^{2}\alpha}{N(1-\gamma)^{2}}
+\frac{8L^{2}\alpha}{(1-\gamma^{2})}D_{2}^{2}\right)
\\
&\qquad\qquad
+\frac{\gamma^{2k}-
\left(1-\alpha\mu\left(1-\frac{\alpha L}{2}\right)\right)^{k}}
{\gamma^{2}-1+\alpha\mu\left(1-\frac{\alpha L}{2}\right)}
\frac{8L^{2}\gamma^{2}}{N}\mathbb{E}\left\Vert x^{(0)}\right\Vert^{2},
\end{align*}
and for any $k\geq 0$, $i=1,2,\ldots,N$,
\begin{align*}
\mathbb{E}\left\Vert x_{i}^{(k)}-x_{\ast}\right\Vert^{2}
&\leq
4(1-\alpha\mu)^{k}\mathbb{E}\left\Vert\bar{x}^{(0)}-x_{\ast}\right\Vert^{2}
+\frac{4\alpha}{\mu}\frac{\sigma^{2}}{N}
\\
&\quad
+\frac{4\alpha\eta^{2}}{\mu N^{2}}\left(\mathbb{E}\left\Vert x^{(0)}-x^{\infty}\right\Vert^{2}
+\frac{2\alpha}{\mu}\left(\sigma^{2}+\eta^{2}\frac{\alpha^{2}(C_{1})^{2}}{(1-\gamma)^{2}}\right)N
+\frac{\alpha^{2}(C_{1})^{2}N}{(1-\gamma)^{2}}\right)
\\
&\qquad
+\alpha\left(\frac{\alpha}{\mu(1-\frac{\alpha L}{2})}+\frac{(1+\alpha L)^{2}}{\mu^{2}(1-\frac{\alpha L}{2})^{2}}\right)
\left(\frac{16L^{2}D_{1}^{2}\alpha}{N(1-\gamma)^{2}}
+\frac{16L^{2}\alpha}{(1-\gamma^{2})}D_{2}^{2}\right)
\\
&\qquad\qquad
+\frac{\gamma^{2k}-
\left(1-\alpha\mu\left(1-\frac{\alpha L}{2}\right)\right)^{k}}
{\gamma^{2}-1+\alpha\mu\left(1-\frac{\alpha L}{2}\right)}
\frac{16L^{2}\gamma^{2}}{N}\mathbb{E}\left\Vert x^{(0)}\right\Vert^{2}
\\
&\qquad
+8\gamma^{2k}\mathbb{E}\left\Vert x^{(0)}\right\Vert^{2}
+\frac{8D_{1}^{2}\alpha^{2}}{(1-\gamma)^{2}}
+\frac{8N\alpha^{2}}{(1-\gamma^{2})}D_{2}^{2},
\end{align*}
where $D_{1},D_{2}$ are defined in Lemma~\ref{E:upper:bound:DSG:unbounded}.}
\end{proposition}

\begin{proof}
{\color{black}By following the proof of Theorem~\ref{thm:DSG:unbounded}, we get
for any $\alpha\leq\frac{1}{L}$, 
\begin{align*}
\mathbb{E}\left\Vert x_{k+1}-x_{\ast}\right\Vert^{2}
&\leq
\left(1-2\alpha\mu\left(1-\frac{\alpha L}{2}\right)\right)
\mathbb{E}\left\Vert x_{k}-x_{\ast}\right\Vert^{2}
+\alpha^{2}\mathbb{E}\left\Vert\bar{\xi}^{(k+1)}\right\Vert^{2}
\\
&\leq
\left(1-\alpha\mu\right)
\mathbb{E}\left\Vert x_{k}-x_{\ast}\right\Vert^{2}
+\alpha^{2}\mathbb{E}\left\Vert\bar{\xi}^{(k+1)}\right\Vert^{2}.
\end{align*}
By \eqref{bar:xi:ineq} and Theorem~\ref{thm:DSG:unbounded}, we get
\begin{align*}
\mathbb{E}\left\Vert\bar{\xi}^{(k+1)}\right\Vert^{2}
&\leq\frac{\sigma^{2}}{N}
+\frac{\eta^{2}}{2N^{2}}\mathbb{E}\left\Vert x^{(k)}-x^{\ast}\right\Vert^{2}
\\
&\leq\frac{\sigma^{2}}{N}
+\frac{\eta^{2}}{2N^{2}}\left(2\mathbb{E}\left\Vert x^{(0)}-x^{\infty}\right\Vert^{2}
+\frac{4\alpha}{\mu}\left(\sigma^{2}+\eta^{2}\frac{\alpha^{2}(C_{1})^{2}}{(1-\gamma)^{2}}\right)N
+\frac{2\alpha^{2}(C_{1})^{2}N}{(1-\gamma)^{2}}\right).
\end{align*}
Therefore, we obtain
\begin{align*}
\mathbb{E}\left\Vert x_{k}-x_{\ast}\right\Vert^{2}
&\leq
(1-\alpha\mu)^{k}\mathbb{E}\left\Vert\bar{x}^{(0)}-x_{\ast}\right\Vert^{2}
+\frac{\alpha}{\mu}\frac{\sigma^{2}}{N}
\\
&\quad
+\frac{\alpha\eta^{2}}{\mu N^{2}}\left(\mathbb{E}\left\Vert x^{(0)}-x^{\infty}\right\Vert^{2}
+\frac{2\alpha}{\mu}\left(\sigma^{2}+\eta^{2}\frac{\alpha^{2}(C_{1})^{2}}{(1-\gamma)^{2}}\right)N
+\frac{\alpha^{2}(C_{1})^{2}N}{(1-\gamma)^{2}}\right).
\end{align*}
Finally, by applying
\begin{equation}
\mathbb{E}\left\Vert\bar{x}^{(k)}-x_{\ast}\right\Vert^{2}
\leq
2\mathbb{E}\left\Vert x_{k}-x_{\ast}\right\Vert^{2}
+2\mathbb{E}\left\Vert x_{k}-\bar{x}^{(k)}\right\Vert^{2},
\end{equation}
and for every $i=1,2,\ldots,N$,
\begin{align*}
\mathbb{E}\left\Vert x_{i}^{(k)}-x_{\ast}\right\Vert^{2}
&\leq
2\mathbb{E}\left\Vert\bar{x}^{(k)}-x_{\ast}\right\Vert^{2}
+2\mathbb{E}\left\Vert x_{i}^{(k)}-\bar{x}^{(k)}\right\Vert^{2}
\\
&\leq
4\mathbb{E}\left\Vert\bar{x}^{(k)}-x_{k}\right\Vert^{2}
+4\mathbb{E}\left\Vert x_{k}-x_{\ast}\right\Vert^{2}
+2\mathbb{E}\left\Vert x_{i}^{(k)}-\bar{x}^{(k)}\right\Vert^{2},
\end{align*}
and by applying Lemma~\ref{lem:central:approx:unbounded}, we complete the proof.}
\end{proof}

{\color{black}
\subsection{Distributed accelerated stochastic gradient (D-ASG)}

Let us recall the D-ASG \eqref{def-dasg-iters-0}. Define 
\begin{equation*}
\bar{\xi}^{(k+1)}:=\frac{1}{N}\sum_{i=1}^{N}\left(\tilde{\nabla}f_{i}\left(y_{i}^{(k)}\right)-\nabla f_{i}\left(y_{i}^{(k)}\right)\right),
\end{equation*}
so that by Assumption~\ref{assump:unbounded}, we have
\begin{equation}
\mathbb{E}\left[\bar{\xi}^{(k+1)}\Big|\mathcal{F}_{k}\right]=0,
\qquad
\mathbb{E}\left\Vert\bar{\xi}^{(k+1)}\Big|\mathcal{F}_{k}\right\Vert^{2}
\leq\frac{\sigma^{2}}{N}
+\frac{\eta^{2}}{2N^{2}}\left\Vert y^{(k)}-x^{\ast}\right\Vert^{2}.
\end{equation}

Let us define
\begin{equation}
V_{Q,\alpha}(\xi) := \xi^T  Q_{\alpha} \xi+ F_{\calW,\alpha}(T\xi + x^\infty) - F_{\calW,\alpha}(x^\infty),
\end{equation}
where $F_{\calW,\alpha}$ is defined in \eqref{def-pen-obj} and $Q_{\alpha}:=\tilde{Q}_{\alpha}\otimes I_{Nd}$ with
\begin{equation}\label{defn:Q:tilde}
\tilde{Q}_{\alpha}:=
\left[
\begin{array}{c}
\sqrt{\frac{1}{2\alpha}}
\\
\sqrt{\frac{\mu}{2}}-\sqrt{\frac{1}{2\alpha}}
\end{array}
\right]
\left[
\begin{array}{cc}
\sqrt{\frac{1}{2\alpha}} &\sqrt{\frac{\mu}{2}}-\sqrt{\frac{1}{2\alpha}}
\end{array}
\right]
+2\alpha\eta^{2}
\left[
\begin{array}{c}
1+\beta
\\
-\beta
\end{array}
\right]
\left[
\begin{array}{cc}
1+\beta & -\beta
\end{array}
\right].
\end{equation}

We have the following explicit performance bounds on the convergence and the robustness of D-ASG iterates.

\begin{theorem}\label{thm-rate-dasg-unbounded}
Assume $\kappa\geq 4$.
Consider running D-ASG method with $\alpha \in (0,\hat{\alpha}]$ and $\beta = \frac{ 1-\sqrt{\alpha\mu}}{1+\sqrt{\alpha\mu}}$ with 
\begin{equation}\label{eqn:hat:alpha}
\hat{\alpha}:=
\begin{cases}
\min\left\{\frac{\lambda_{N}^{W}}{L},\frac{\mu^{3}}{(60\eta^{2})^{2}}\right\} &\text{if $\eta>0$},
\\
\frac{\lambda_{N}^{W}}{L} &\text{if $\eta=0$}.
\end{cases}
\end{equation}
Then, for any $k\geq 0$, we have
\begin{align}
&\E \left[\left\|  x^{(k)} - x^\infty \right\|^2\right]
\leq
2\left(1-\sqrt{\alpha \mu}/3\right)^{k}  \frac{V_{Q,\alpha}\left(\xi_{0}\right)}{\mu}
+ \frac{12\sqrt{\alpha}}{\mu\sqrt{\mu}}\left(\sigma^{2}+\eta^{2}\frac{\alpha^{2}(C_{1})^{2}}{(1-\gamma)^{2}}\right)N.
\end{align}
In addition, if $\alpha \leq \frac{1}{L+\mu}$, we have
\begin{align}
&\E\left[ \left\| x^{(k)} -x^* \right\|^2\right]
\leq
4\left(1-\sqrt{\alpha \mu}/3\right)^{k}  \frac{V_{Q,\alpha}\left(\xi_{0}\right)}{\mu}
+ \frac{24\sqrt{\alpha}}{\mu\sqrt{\mu}}\left(\sigma^{2}+\eta^{2}\frac{\alpha^{2}(C_{1})^{2}}{(1-\gamma)^{2}}\right)N +  \frac{2C_{1}^{2}N\alpha^2}{(1-\gamma)^2},
\end{align}
where $C_{1},\gamma$ are given in \eqref{eqn:asymp_suboptim}.
\end{theorem}    

\begin{proof}
D-ASG reduces to the iterations \eqref{eq-centr-asg} which are equivalent to applying non-distributed ASG to minimize the function $F_{\calW,\alpha} \in S_{\mu,L_\alpha}(\R^{Nd})$, where $F_{\calW,\alpha}$ is defined in \eqref{def-pen-obj} and $L_\alpha = \frac{1-\lambda_N^W}{\alpha} + L$. Therefore, 
applying Theorem~K.1 in \cite{aybat2019universally} from the literature for non-distributed ASG, 
for any $\alpha \in (0,\bar{\alpha}]$ and $\beta = \frac{ 1-\sqrt{\alpha\mu}}{1+\sqrt{\alpha\mu}}$ with
\begin{equation}
\bar{\alpha}:=
\begin{cases}
\min\left\{\frac{1}{L_{\alpha}},\frac{\mu^{3}}{(60\eta^{2})^{2}}\right\} &\text{if $\eta>0$},
\\
\frac{1}{L_{\alpha}} &\text{if $\eta=0$},
\end{cases}
\end{equation}
where $L_\alpha = \frac{1-\lambda_N^W}{\alpha} + L$, 
we obtain 
\begin{equation}
\E \left[V_{Q,\alpha}\left(\xi_{k+1}\right)\right] 
\leq \left(1-\sqrt{\alpha \mu}/3\right)  \mathbb{E}V_{Q,\alpha}\left(\xi_{k}\right)
+ 2\alpha\left(\sigma^{2}+\eta^{2}\frac{\alpha^{2}(C_{1})^{2}}{(1-\gamma)^{2}}\right)N,
\end{equation}
which yields
\begin{equation*}
\E \left[V_{Q,\alpha}\left(\xi_{k}\right)\right] 
\leq \left(1-\sqrt{\alpha \mu}/3\right)^{k}  V_{Q,\alpha}\left(\xi_{0}\right)
+ \frac{6\sqrt{\alpha}}{\sqrt{\mu}}\left(\sigma^{2}+\eta^{2}\frac{\alpha^{2}(C_{1})^{2}}{(1-\gamma)^{2}}\right)N.
\end{equation*}
Note that the condition $\alpha\in(0,\bar{\alpha}]$ 
is equivalent to $\alpha\in(0,\hat{\alpha}]$,
where
\begin{equation}
\hat{\alpha}:=
\begin{cases}
\min\left\{\frac{\lambda_{N}^{W}}{L},\frac{\mu^{3}}{(60\eta^{2})^{2}}\right\} &\text{if $\eta>0$},
\\
\frac{\lambda_{N}^{W}}{L} &\text{if $\eta=0$}.
\end{cases}
\end{equation}
The rest of the proof is similar to that of Theorem~\ref{thm-rate-dasg}
and is omitted here.
\end{proof}
}

{\color{black}Next, we show that the individual iterates and the average iterates are close.}

\begin{lemma}\label{gradient:average:error:DASG:general:noise}
{\color{black} 
Consider running D-ASG method under the assumptions in Theorem~\ref{thm-rate-dasg-unbounded}.
For any $k$ and $i=1,\ldots,N$, we have
\begin{align*}
\mathbb{E}\left\Vert x_{i}^{(k)}-\bar{x}^{(k)}\right\Vert^{2} 
&\leq
\sum_{i=1}^{N}\mathbb{E}\left\Vert x_{i}^{(k)}-\bar{x}^{(k)}\right\Vert^{2} 
\\
&\leq
8\gamma^{2k}\left(4\frac{V_{Q,\alpha}\left(\xi_{0}\right)}{\mu}
+ \frac{24\sqrt{\alpha}}{\mu\sqrt{\mu}}\left(\sigma^{2}+\eta^{2}\frac{\alpha^{2}(C_{1})^{2}}{(1-\gamma)^{2}}\right)N +  \frac{2C_{1}^{2}N\alpha^2}{(1-\gamma)^2}
+\Vert x^{\ast}\Vert^{2}\right)
\\
&\qquad
+\frac{4\tilde{D}_{y}^{2}\alpha^{2}}{(1-\gamma)^{2}}
+\frac{4\tilde{E}_{y}^{2}\alpha^{2}}{(1-\gamma^{2})}
+\frac{8\tilde{C}_{0}\alpha}{(1-\gamma)^{2}},
\end{align*} 
and for any $k$, we have
\begin{align*}
\mathbb{E}\left\Vert\mathcal{E}_{k+1}\right\Vert^{2}
&\leq
\frac{2}{N}L^{2}
\left((1+\beta)^{2}+\beta^{2}\right)
\Bigg[\frac{4\tilde{D}_{y}^{2}\alpha^{2}}{(1-\gamma)^{2}}
+\frac{4\tilde{E}_{y}^{2}\alpha^{2}}{(1-\gamma^{2})}
+\frac{8\tilde{C}_{0}\alpha}{(1-\gamma)^{2}}
\\
&\quad
+8\gamma^{2(k-1)}\left(4\frac{V_{Q,\alpha}\left(\xi_{0}\right)}{\mu}
+ \frac{24\sqrt{\alpha}}{\mu\sqrt{\mu}}\left(\sigma^{2}+\eta^{2}\frac{\alpha^{2}(C_{1})^{2}}{(1-\gamma)^{2}}\right)N +  \frac{2C_{1}^{2}N\alpha^2}{(1-\gamma)^2}
+\Vert x^{\ast}\Vert^{2}\right)\Bigg],
\end{align*} 
where $\tilde{C}_{0}$ is defined in \eqref{eqn:difference:general}, $\tilde{D}_{y}$ is defined in \eqref{D:y:eqn:general:noise}
and $\tilde{E}_{y}$ is defined in \eqref{E:y:eqn:general:noise}.
}
\end{lemma}

\begin{proof}
{\color{black}
The proof of Lemma~\ref{gradient:average:error:DASG:general:noise} will be provided in Appendix~\ref{sec:technical}.}
\end{proof}

{\color{black}Finally, we provide the performance bounds for the average iterates and individual iterates. We first define
\begin{equation}
V_{\bar{Q},\alpha}(\bar{\xi}) := \bar{\xi}^T  \bar{Q}_{\alpha} \bar{\xi}+ f\left(T\bar{\xi} + x_{\ast}\right) - f(x_{\ast}),
\label{def-barQ-alpha}
\end{equation}
where $\bar{Q}_{\alpha}=\tilde{Q}_{\alpha}\otimes I_{d}$ and $\tilde{Q}_{\alpha}$
is defined in \eqref{defn:Q:tilde}.
Before we proceed, let us first
prove a technical lemma.}

\begin{lemma}\label{lem:C:epsilon-generalized}
{\color{black}
Consider running D-ASG method under the assumptions in Theorem~\ref{thm-rate-dasg-unbounded}.
For any $\epsilon>0$, there exists $M_{\epsilon}>0$ such that
\begin{equation}
\E \left[V_{\bar{Q},\alpha}\left(\bar{\xi}_{k+1}\right)\right]
\leq 
(1+\epsilon)\E \left[V_{\bar{Q},\alpha}\left(\bar{\xi}_{k+1}-D_{k+1}\right)\right]
+M_{\epsilon}\mathbb{E}\Vert D_{k+1}\Vert^{2}, \label{ineq-lyap-pert-generalized}
\end{equation}
where
\begin{equation}
M_{\epsilon}:= \frac{\max(4/m_2,L^2/\mu)}{2\epsilon} + \frac{L}{2} + \frac{1}{\alpha} + \frac{\mu}{2}-\frac{\sqrt{\mu}}{\sqrt{\alpha}} + 2\alpha\eta^2 \left((1+\beta)^2 + \beta^2 \right), 
\end{equation}
where $m_2 >0$ is the smallest eigenvalue of $\bar{Q}_\alpha$.}
\end{lemma}

\begin{proof}
{\color{black}
The proof of Lemma~\ref{lem:C:epsilon-generalized} will be provided in Appendix~\ref{sec:technical}.}
\end{proof}

{\color{black}
Finally, we are ready to present the performance bounds for the average iterates and individual iterates.
}

\begin{proposition}\label{prop:dsg-average-optimal-rate-DASG:general:noise}
{\color{black} 
Consider running D-ASG method under the assumptions in Theorem~\ref{thm-rate-dasg-unbounded}.
For any $k$, we have
\begin{align*}
&\mathbb{E}\left\Vert\bar{x}^{(k)}-x_{\ast}\right\Vert^{2}
\\
&\leq
\left(1-\frac{\sqrt{\alpha\mu}}{6}\right)^{k}\frac{2V_{\bar{Q},\alpha}\left(\bar{\xi}_{0}\right)}{\mu} 
+\frac{12}{\mu\sqrt{\mu}}
\left(\left(1+\frac{\sqrt{\alpha\mu}}{6}\right)\frac{2\sqrt{\alpha}}{N}\left(\sigma^{2}+\eta^{2}\frac{\alpha^{2}(C_{1})^{2}}{(1-\gamma)^{2}}\right)
+\frac{\alpha\tilde{H}_{1}\tilde{H}_{2}}{2\sqrt{\mu}}\right)
\\
&\qquad
+\frac{8}{\gamma^{2}\mu\sqrt{\mu}}\sqrt{\alpha}\tilde{H}_{1}\tilde{H}_{3}\frac{\gamma^{2k}-(1-\sqrt{\alpha\mu}/6)^{k}}{\gamma^{2}-(1-\sqrt{\alpha\mu}/6)},
\end{align*}
and for every $i=1,2,\ldots,N$ and any $k$,
\begin{align*}
&\mathbb{E}\left\Vert x_{i}^{(k)}-x_{\ast}\right\Vert^{2}
\\
&\leq
\left(1-\frac{\sqrt{\alpha\mu}}{6}\right)^{k}\frac{4V_{\bar{Q},\alpha}\left(\bar{\xi}_{0}\right)}{\mu}
+\frac{24}{\mu\sqrt{\mu}}
\left(\left(1+\frac{\sqrt{\alpha\mu}}{6}\right)\frac{2\sqrt{\alpha}}{N}\left(\sigma^{2}+\eta^{2}\frac{\alpha^{2}(C_{1})^{2}}{(1-\gamma)^{2}}\right)
+\frac{\alpha\tilde{H}_{1}\tilde{H}_{2}}{2\sqrt{\mu}}\right)
\\
&\qquad
+\frac{16}{\gamma^{2}\mu\sqrt{\mu}}\sqrt{\alpha}\tilde{H}_{1}\tilde{H}_{3}\frac{\gamma^{2k}-(1-\sqrt{\alpha\mu}/6)^{k}}{\gamma^{2}-(1-\sqrt{\alpha\mu}/6)}
\\
&\qquad
+16\gamma^{2k}\left(4\frac{V_{Q,\alpha}\left(\xi_{0}\right)}{\mu}
+ \frac{24\sqrt{\alpha}}{\mu\sqrt{\mu}}\left(\sigma^{2}+\eta^{2}\frac{\alpha^{2}(C_{1})^{2}}{(1-\gamma)^{2}}\right)N +  \frac{2C_{1}^{2}N\alpha^2}{(1-\gamma)^2}
+\Vert x^{\ast}\Vert^{2}\right)
\\
&\qquad
+\frac{8\tilde{D}_{y}^{2}\alpha^{2}}{(1-\gamma)^{2}}
+\frac{8\tilde{E}_{y}^{2}\alpha^{2}}{(1-\gamma^{2})}
+\frac{16\tilde{C}_{0}\alpha}{(1-\gamma)^{2}},
\end{align*}
where $\tilde{C}_{0}$ is some constant
such that $\tilde{C}_{0}=\mathcal{O}(1)$ as $\alpha\rightarrow 0$, and
\begin{align}
&\tilde{D}_{y}^{2}:=4L^{2}\left((1+\beta)^{2}+\beta^{2}\right)\left(\frac{2V_{Q,\alpha}\left(\xi_{0}\right)}{\mu}
+ \frac{12\sqrt{\alpha}}{\mu\sqrt{\mu}}\left(\sigma^{2}+\eta^{2}\frac{\alpha^{2}(C_{1})^{2}}{(1-\gamma)^{2}}\right)N\right)
+2\Vert\nabla F(x^{\ast})\Vert^{2},\label{D:y:eqn:general:noise}
\\
&\tilde{E}_{y}^{2}
:=\left(\sigma^{2}+\eta^{2}\frac{\alpha^{2}(C_{1})^{2}}{(1-\gamma)^{2}}\right)N\nonumber
\\
&\qquad\qquad
+2\eta^{2}\left((1+\beta)^{2}+\beta^{2}\right)\left(\frac{2V_{Q,\alpha}\left(\xi_{0}\right)}{\mu}
+ \frac{12\sqrt{\alpha}}{\mu\sqrt{\mu}}\left(\sigma^{2}+\eta^{2}\frac{\alpha^{2}(C_{1})^{2}}{(1-\gamma)^{2}}\right)N\right),\label{E:y:eqn:general:noise}
\end{align}
and
\begin{align*}
&\tilde{H}_{1}:=6\alpha\max(4/m_2,L^2/\mu)+ \left(L\alpha+2 + \mu\alpha\right)\sqrt{\alpha\mu}-2\mu\alpha
+ 4\alpha^{2}\sqrt{\alpha\mu}\eta^2 \left((1+\beta)^2 + \beta^2 \right),
\\
&\tilde{H}_{2}:=\frac{2}{N}L^{2}
\left((1+\beta)^{2}+\beta^{2}\right)\left(\frac{4\tilde{D}_{y}^{2}\alpha}{(1-\gamma)^{2}}
+\frac{4\tilde{E}_{y}^{2}\alpha}{(1-\gamma)^{2}}+\frac{8\tilde{C}_{0}}{(1-\gamma)^{2}}\right),
\\
&\tilde{H}_{3}:=\frac{2}{N}L^{2}
\left((1+\beta)^{2}+\beta^{2}\right)
\\
&\qquad\qquad\cdot\left(4\frac{V_{Q,\alpha}\left(\xi_{0}\right)}{\mu}
+ \frac{24\sqrt{\alpha}}{\mu\sqrt{\mu}}\left(\sigma^{2}+\eta^{2}\frac{\alpha^{2}(C_{1})^{2}}{(1-\gamma)^{2}}\right)N +  \frac{2C_{1}^{2}N\alpha^2}{(1-\gamma)^2}
+\Vert x^{\ast}\Vert^{2}\right),
\end{align*}
where $m_{2}>0$ is the smallest eigenvalue of $\bar{Q}_{\alpha}$.
}
\end{proposition}

\begin{proof}
{\color{black}
The proof is similar to the proof of Proposition~\ref{prop:dsg-average-optimal-rate-DASG}.
Similar to Lemma~\ref{lem:difference}, we can show
that
\begin{equation}\label{eqn:difference:general}
\sup_{k}\mathbb{E}\left\Vert x^{(k)}-x^{(k-1)}\right\Vert^{2}
\leq\frac{2\tilde{C}_{0}}{\beta^{2}}\alpha,
\end{equation}
for some $\tilde{C}_{0}$ such that $\tilde{C}_{0}=\mathcal{O}(1)$ as $\alpha\rightarrow 0$.
By applying \eqref{eqn:difference:general} and Theorem~K.1 in \cite{aybat2019universally} from the literature for non-distributed ASG, 
for any $\alpha \in (0,\bar{\alpha}]$ and $\beta = \frac{ 1-\sqrt{\alpha\mu}}{1+\sqrt{\alpha\mu}}$, as well as Lemma~\ref{lem:C:epsilon-generalized}, we have
\begin{align*}
&\E \left[V_{\bar{Q},\alpha}\left(\bar{\xi}_{k+1}\right)\right] 
\\
&\leq 
(1+\epsilon)\E \left[V_{\bar{Q},\alpha}\left(\bar{\xi}_{k+1}-D_{k+1}\right)\right]
+M_{\epsilon}\mathbb{E}\Vert D_{k+1}\Vert^{2}
\\ 
&\leq
(1+\epsilon)\left(\left(1-\sqrt{\alpha \mu}/3\right)  \mathbb{E}V_{\bar{Q},\alpha}\left(\bar{\xi}_{k}\right)
+ \frac{2\alpha}{N}\left(\sigma^{2}+\eta^{2}\frac{\alpha^{2}(C_{1})^{2}}{(1-\gamma)^{2}}\right)\right)
+M_{\epsilon}\mathbb{E}\Vert D_{k+1}\Vert^{2}.
\end{align*}
Let us take $\epsilon=\frac{1}{6}\sqrt{\alpha\mu}$, then we get
\begin{align*}
&\E \left[V_{\bar{Q},\alpha}\left(\bar{\xi}_{k+1}\right)\right] 
\\
&\leq
\left(1-\frac{\sqrt{\alpha\mu}}{6}\right)\mathbb{E}V_{\bar{Q},\alpha}\left(\bar{\xi}_{k}\right)
+\left(1+\frac{\sqrt{\alpha\mu}}{6}\right)\frac{2\alpha}{N}\left(\sigma^{2}+\eta^{2}\frac{\alpha^{2}(C_{1})^{2}}{(1-\gamma)^{2}}\right)
+M_{\frac{\sqrt{\alpha\mu}}{6}}\mathbb{E}\Vert D_{k+1}\Vert^{2},
\end{align*}
and by Lemma~\ref{lem:C:epsilon-generalized}, we have
\begin{align*}
M_{\frac{\sqrt{\alpha\mu}}{6}}
&=\frac{3\max(4/m_2,L^2/\mu)}{\sqrt{\alpha\mu}} + \frac{L}{2} + \frac{1}{\alpha} + \frac{\mu}{2}-\frac{\sqrt{\mu}}{\sqrt{\alpha}} + 2\alpha\eta^2 \left((1+\beta)^2 + \beta^2 \right)
\\
&\leq\frac{1}{2\alpha\sqrt{\alpha\mu}}\tilde{H}_{1},
\end{align*}
where
\begin{equation}
\tilde{H}_{1}=6\alpha\max(4/m_2,L^2/\mu)+ \left(L\alpha+2 + \mu\alpha\right)\sqrt{\alpha\mu}-2\mu\alpha
+ 4\alpha^{2}\sqrt{\alpha\mu}\eta^2 \left((1+\beta)^2 + \beta^2 \right).
\end{equation}
By Lemma~\ref{gradient:average:error:DASG:general:noise}, we have
\begin{align*}
\mathbb{E}\Vert D_{k+1}\Vert^{2}
\leq
\alpha^{2}
\left[\alpha\tilde{H}_{2}+8\gamma^{2(k-1)}\tilde{H}_{3}\right],
\end{align*}
where
\begin{align*}
&\tilde{H}_{2}=\frac{2}{N}L^{2}
\left((1+\beta)^{2}+\beta^{2}\right)\left(\frac{4\tilde{D}_{y}^{2}\alpha}{(1-\gamma)^{2}}
+\frac{4\tilde{E}_{y}^{2}\alpha}{(1-\gamma)^{2}}+\frac{8\tilde{C}_{0}}{(1-\gamma)^{2}}\right),
\\
&\tilde{H}_{3}=\frac{2}{N}L^{2}
\left((1+\beta)^{2}+\beta^{2}\right)
\\
&\qquad\qquad\cdot\left(4\frac{V_{Q,\alpha}\left(\xi_{0}\right)}{\mu}
+ \frac{24\sqrt{\alpha}}{\mu\sqrt{\mu}}\left(\sigma^{2}+\eta^{2}\frac{\alpha^{2}(C_{1})^{2}}{(1-\gamma)^{2}}\right)N +  \frac{2C_{1}^{2}N\alpha^2}{(1-\gamma)^2}
+\Vert x^{\ast}\Vert^{2}\right).
\end{align*}
Therefore, 
\begin{align*}
\E \left[V_{\bar{Q},\alpha}\left(\bar{\xi}_{k+1}\right)\right] 
&\leq
\left(1-\frac{\sqrt{\alpha\mu}}{6}\right)\mathbb{E}V_{\bar{Q},\alpha}\left(\bar{\xi}_{k}\right)
+\left(1+\frac{\sqrt{\alpha\mu}}{6}\right)\frac{2\alpha}{N}\left(\sigma^{2}+\eta^{2}\frac{\alpha^{2}(C_{1})^{2}}{(1-\gamma)^{2}}\right)
\\
&\qquad
+\frac{1}{2\sqrt{\mu}}\alpha^{2}\tilde{H}_{1}\tilde{H}_{2}+\frac{4}{\gamma^{2}\sqrt{\mu}}\sqrt{\alpha}\tilde{H}_{1}\tilde{H}_{3}\gamma^{2k},
\end{align*}
which by following the similar argument in the proof of Proposition~\ref{prop:dsg-average-optimal-rate-DASG} implies that
\begin{align*}
\E \left[V_{\bar{Q},\alpha}\left(\bar{\xi}_{k}\right)\right] 
&\leq
\left(1-\frac{\sqrt{\alpha\mu}}{6}\right)^{k}V_{\bar{Q},\alpha}\left(\bar{\xi}_{0}\right)
+\frac{4}{\gamma^{2}\sqrt{\mu}}\sqrt{\alpha}\tilde{H}_{1}\tilde{H}_{3}\frac{\gamma^{2k}-(1-\sqrt{\alpha\mu}/6)^{k}}{\gamma^{2}-(1-\sqrt{\alpha\mu}/6)}
\\
&\qquad
+\frac{6}{\sqrt{\alpha\mu}}
\left(\left(1+\frac{\sqrt{\alpha\mu}}{6}\right)\frac{2\alpha}{N}\left(\sigma^{2}+\eta^{2}\frac{\alpha^{2}(C_{1})^{2}}{(1-\gamma)^{2}}\right)
+\frac{1}{2\sqrt{\mu}}\alpha\sqrt{\alpha}\tilde{H}_{1}\tilde{H}_{2}\right).
\end{align*}
The rest of the proof is similar to that of Proposition~\ref{prop:dsg-average-optimal-rate-DASG}.}
\end{proof}

\section{Proofs of Technical Lemmas}\label{sec:technical}

\subsection{Proofs of Technical Results in Appendix~\ref{sec:main:proofs}}

\subsubsection{Proof of Lemma~\ref{lem:central:approx}}

\begin{proof}
{\color{black}
We can compute that
\begin{equation}
\bar{x}^{(k+1)}-x_{k+1}=\bar{x}^{(k)}-x_{k}
-\alpha\left[\nabla f\left(\bar{x}^{(k)}\right)-\nabla f(x_{k})\right]
+\alpha\mathcal{E}_{k+1},
\label{eq-to-refer-to-noisy}
\end{equation}
where
$\mathcal{E}_{k+1}
:= \nabla f\left(\bar{x}^{(k)}\right)-\frac{1}{N}\sum_{i=1}^{N}\nabla f_{i}\left(x_{i}^{(k)}\right)$.

If the term $\mathcal{E}_{k+1}$ were not present in the recursion \eqref{eq-to-refer-to-noisy}, we could rely on standard analysis techniques for analyzing a gradient step in order to bound $\|\bar{x}^{(k+1)}-x_{k+1}\|^2$ with $\|\bar{x}^{(k)}-x_{k}\|^2$. However, in the presence of $\mathcal{E}_{k+1}$, we need to control this error term based on Lemma~\ref{lem:2}. To be more precise, we have
\begin{align}
&\left\Vert\bar{x}^{(k+1)}-x_{k+1}\right\Vert^{2}
\nonumber
\\
&=\left\Vert\bar{x}^{(k)}-x_{k}
-\alpha\left[\nabla f\left(\bar{x}^{(k)}\right)-\nabla f(x_{k})\right]\right\Vert^{2}
+\alpha^{2}\left\Vert\mathcal{E}_{k+1}\right\Vert^{2}
\nonumber
\\
&\qquad\qquad
+2\left\langle\bar{x}^{(k)}-x_{k}
-\alpha\left[\nabla f\left(\bar{x}^{(k)}\right)-\nabla f(x_{k})\right]
,\alpha\mathcal{E}_{k+1}\right\rangle
\nonumber
\\
&=\left\Vert\bar{x}^{(k)}-x_{k}\right\Vert^{2}
+\alpha^{2}\left\Vert\left[\nabla f\left(\bar{x}^{(k)}\right)-\nabla f(x_{k})\right]\right\Vert^{2}
\nonumber
\\
&\qquad
-2\left\langle\bar{x}^{(k)}-x_{k},\alpha\left[\nabla f\left(\bar{x}^{(k)}\right)-\nabla f(x_{k})\right]\right\rangle
+\alpha^{2}\left\Vert\mathcal{E}_{k+1}\right\Vert^{2}
\nonumber
\\
&\qquad\qquad
+2\left\langle\bar{x}^{(k)}-x_{k}
-\alpha\left[\nabla f\left(\bar{x}^{(k)}\right)-\nabla f(x_{k})\right]
,\alpha\mathcal{E}_{k+1}\right\rangle
\nonumber
\\
&\leq
\left\Vert\bar{x}^{(k)}-x_{k}\right\Vert^{2}
+\alpha^{2}L\left\langle\bar{x}^{(k)}-x_{k},\left[\nabla f\left(\bar{x}^{(k)}\right)-\nabla f(x_{k})\right]\right\rangle
\nonumber
\\
&\qquad
-2\left\langle\bar{x}^{(k)}-x_{k},\alpha\left[\nabla f\left(\bar{x}^{(k)}\right)-\nabla f(x_{k})\right]\right\rangle
+\alpha^{2}\left\Vert\mathcal{E}_{k+1}\right\Vert^{2}
\nonumber
\\
&\qquad\qquad
+2\left\langle\bar{x}^{(k)}-x_{k}
-\alpha\left[\nabla f\left(\bar{x}^{(k)}\right)-\nabla f(x_{k})\right]
,\alpha\mathcal{E}_{k+1}\right\rangle
\nonumber
\\
&=
\left\Vert\bar{x}^{(k)}-x_{k}\right\Vert^{2}
-2\alpha \left(1 - \frac{\alpha L}{2} \right) \left\langle\bar{x}^{(k)}-x_{k},\left[\nabla f\left(\bar{x}^{(k)}\right)-\nabla f(x_{k})\right]\right\rangle
\nonumber
\\
&\qquad\qquad
+\alpha^{2}\left\Vert\mathcal{E}_{k+1}\right\Vert^{2} +2\left\langle\bar{x}^{(k)}-x_{k}
-\alpha\left[\nabla f\left(\bar{x}^{(k)}\right)-\nabla f(x_{k})\right]
,\alpha\mathcal{E}_{k+1}\right\rangle
\nonumber
\\
&\leq
\left(1-2\alpha\mu\left(1-\frac{\alpha L}{2}\right)\right)\left\Vert\bar{x}^{(k)}-x_{k}\right\Vert^{2}
+\alpha^{2}\left\Vert\mathcal{E}_{k+1}\right\Vert^{2}
\nonumber
\\
&\qquad\qquad
+2\left\langle\bar{x}^{(k)}-x_{k}
-\alpha\left[\nabla f\left(\bar{x}^{(k)}\right)-\nabla f(x_{k})\right]
,\alpha\mathcal{E}_{k+1}\right\rangle,\label{take:expect}
\end{align} 
where we used \citet[Theorem 2.1.5]{nesterov_convex} on $L$-smooth and convex functions 
to obtain the second term after
the first inequality above 
and $\mu$-strong convexity of $f$ and the assumption that $\alpha< 2/L$ to obtain the first term after the second inequality above.
By taking expectations in \eqref{take:expect}, we get
\begin{align*}
\mathbb{E}\left\Vert\bar{x}^{(k+1)}-x_{k+1}\right\Vert^{2}
&\leq
\left(1-2\alpha\mu\left(1-\frac{\alpha L}{2}\right)\right)\mathbb{E}\left\Vert\bar{x}^{(k)}-x_{k}\right\Vert^{2}
+\alpha^{2}\mathbb{E}\left\Vert\mathcal{E}_{k+1}\right\Vert^{2}
\nonumber
\\
&\qquad\qquad
+\mathbb{E}\left[2\left\langle\bar{x}^{(k)}-x_{k}
-\alpha\left[\nabla f\left(\bar{x}^{(k)}\right)-\nabla f(x_{k})\right]
,\alpha\mathcal{E}_{k+1}\right\rangle\right]
\\
&=\left(1-2\alpha\mu\left(1-\frac{\alpha L}{2}\right)\right)\mathbb{E}\left\Vert\bar{x}^{(k)}-x_{k}\right\Vert^{2}
+\alpha^{2}\mathbb{E}\left\Vert\mathcal{E}_{k+1}\right\Vert^{2}
\nonumber
\\
&\qquad\qquad
+\mathbb{E}\left[2\left\langle\bar{x}^{(k)}-x_{k}
-\alpha\left[\nabla f\left(\bar{x}^{(k)}\right)-\nabla f(x_{k})\right]
,\alpha\mathcal{E}_{k+1}\right\rangle\right]
\\
&\leq
\left(1-2\alpha\mu\left(1-\frac{\alpha L}{2}\right)\right)\mathbb{E}\left\Vert\bar{x}^{(k)}-x_{k}\right\Vert^{2}
+\alpha^{2}\mathbb{E}\left\Vert\mathcal{E}_{k+1}\right\Vert^{2}
\nonumber
\\
&\qquad\qquad
+2(1+\alpha L)
\alpha\mathbb{E}\left[\left\Vert\bar{x}^{(k)}-x_{k}\right\Vert\cdot\left\Vert\mathcal{E}_{k+1}\right\Vert\right],
\end{align*}
where we used $L$-smoothness of $f$. 

For any $x,y\geq 0$ and $c>0$, we have the inequality $2xy\leq cx^{2}+\frac{y^{2}}{c}$,
which implies that
\begin{align*}
&\mathbb{E}\left\Vert\bar{x}^{(k+1)}-x_{k+1}\right\Vert^{2}
\\
&\leq
\left(1-2\alpha\mu\left(1-\frac{\alpha L}{2}\right)\right)\mathbb{E}\left\Vert\bar{x}^{(k)}-x_{k}\right\Vert^{2}
+\alpha^{2}\mathbb{E}\left\Vert\mathcal{E}_{k+1}\right\Vert^{2}
\\
&\qquad
+(1+\alpha L)\alpha
\left(\frac{\mu(1-\frac{\alpha L}{2})}{1+\alpha L}\mathbb{E}\left\Vert\bar{x}^{(k)}-x_{k}\right\Vert^{2}
+\frac{1+\alpha L}{\mu(1-\frac{\alpha L}{2})}\mathbb{E}\left\Vert\mathcal{E}_{k+1}\right\Vert^{2}\right)
\\
&=\left(1-\alpha\mu\left(1-\frac{\alpha L}{2}\right)\right)\mathbb{E}\left\Vert\bar{x}^{(k)}-x_{k}\right\Vert^{2}
+\alpha\left(\alpha+\frac{(1+\alpha L)^{2}}{\mu(1-\frac{\alpha L}{2})}\right)\mathbb{E}\left\Vert\mathcal{E}_{k+1}\right\Vert^{2}.
\end{align*}
By applying Lemma~\ref{lem:2}, we get
\begin{align*}
&\mathbb{E}\left\Vert\bar{x}^{(k+1)}-x_{k+1}\right\Vert^{2}
\\
&\leq\left(1-\alpha\mu\left(1-\frac{\alpha L}{2}\right)\right)
\mathbb{E}\left\Vert\bar{x}^{(k)}-x_{k}\right\Vert^{2}
\\
&\quad
+\alpha\left(\alpha+\frac{(1+\alpha L)^{2}}{\mu(1-\frac{\alpha L}{2})}\right)
\left(\frac{4L^{2}\rev{\gamma^{2k}}}{N}\mathbb{E}\left\Vert x^{(0)}\right\Vert^{2}
+\frac{4L^{2}D^{2}\alpha^{2}}{N(1-\gamma)^{2}}
+\frac{4L^{2}\sigma^{2}\alpha^{2}}{(1-\gamma^{2})}\right),
\end{align*}
for every $k$. Note that $\mathbb{E}\left\Vert\bar{x}^{(0)}-x_{0}\right\Vert^{2}=0$.
By iterating the above equation, we get
\begin{align*}
&\mathbb{E}\left\Vert\bar{x}^{(k)}-x_{k}\right\Vert^{2}
\\
&\leq
\sum_{i=0}^{k-1}
\left(1-\alpha\mu\left(1-\frac{\alpha L}{2}\right)\right)^{i}
\cdot\alpha\left(\alpha+\frac{(1+\alpha L)^{2}}{\mu(1-\frac{\alpha L}{2})}\right)
\left(\frac{4L^{2}D^{2}\alpha^{2}}{N(1-\gamma)^{2}}
+\frac{4L^{2}\sigma^{2}\alpha^{2}}{(1-\gamma^{2})}
\right)
\\
&\qquad
+\sum_{i=0}^{k-1}
\left(1-\alpha\mu\left(1-\frac{\alpha L}{2}\right)\right)^{i}
\alpha\left(\alpha+\frac{(1+\alpha L)^{2}}{\mu(1-\frac{\alpha L}{2})}\right)
\frac{4L^{2}\rev{\gamma^{2(k-i)}}}{N}\mathbb{E}\left\Vert x^{(0)}\right\Vert^{2}
\\
&=\frac{1-
\left(1-\alpha\mu\left(1-\frac{\alpha L}{2}\right)\right)^{k}}
{1-\left(1-\alpha\mu\left(1-\frac{\alpha L}{2}\right)\right)}
\cdot\alpha\left(\alpha+\frac{(1+\alpha L)^{2}}{\mu(1-\frac{\alpha L}{2})}\right)
\left(\frac{4L^{2}D^{2}\alpha^{2}}{N(1-\gamma)^{2}}
+\frac{4L^{2}\sigma^{2}\alpha^{2}}{(1-\gamma^{2})}
\right)
\\
&\qquad
+\rev{\frac{\gamma^{2k}-
\left(1-\alpha\mu\left(1-\frac{\alpha L}{2}\right)\right)^{k}}
{1-\left(1-\alpha\mu\left(1-\frac{\alpha L}{2}\right)\right)(\gamma)^{-2}}
\frac{4L^{2}}{N}\mathbb{E}\left\Vert x^{(0)}\right\Vert^{2}.}
\end{align*}
By our assumption on stepsize $\alpha$, we have
$1-\alpha\mu\left(1-\frac{\alpha L}{2}\right)\in[0,1)$.
Hence, we conclude that for every $k$,
\begin{align*}
\mathbb{E}\left\Vert\bar{x}^{(k)}-x_{k}\right\Vert^{2}
&\leq\frac{\alpha\left(\alpha+\frac{(1+\alpha L)^{2}}{\mu(1-\frac{\alpha L}{2})}\right)
\left(\frac{4L^{2}D^{2}\alpha^{2}}{N(1-\gamma)^{2}}
+\frac{4L^{2}\sigma^{2}\alpha^{2}}{(1-\gamma^{2})}
\right)}{
1-\left(1-\alpha\mu\left(1-\frac{\alpha L}{2}\right)\right)}
\\
&\qquad\qquad
+\rev{\frac{\gamma^{2k}-
\left(1-\alpha\mu\left(1-\frac{\alpha L}{2}\right)\right)^{k}}
{1-\left(1-\alpha\mu\left(1-\frac{\alpha L}{2}\right)\right)(\gamma)^{-2}}
\frac{4L^{2}}{N}\mathbb{E}\left\Vert x^{(0)}\right\Vert^{2}}
\\
&=\frac{\alpha\left(\alpha+\frac{(1+\alpha L)^{2}}{\mu(1-\frac{\alpha L}{2})}\right)
\left(\frac{4L^{2}D^{2}\alpha}{N(1-\gamma)^{2}}
+\frac{4L^{2}\sigma^{2}\alpha}{(1-\gamma^{2})}\right)}{
\mu\left(1-\frac{\alpha L}{2}\right)}
\\
&\qquad\qquad
+\rev{\frac{\gamma^{2k}-
\left(1-\alpha\mu\left(1-\frac{\alpha L}{2}\right)\right)^{k}}
{\gamma^{2}-1+\alpha\mu\left(1-\frac{\alpha L}{2}\right)}
\frac{4L^{2}\gamma^{2}}{N}\mathbb{E}\left\Vert x^{(0)}\right\Vert^{2}.}
\end{align*}
The proof is complete.
}
\end{proof}


\subsubsection{Proof of Lemma~\ref{gradient:average:error:DASG}}

\begin{proof}
{\color{black}By the definition of $x^{(k)}$ and $y^{(k)}$, we get
\begin{align*}
&x^{(k+1)}=(W\otimes I_{d})y^{(k)}-\alpha \nabla F\left(y^{(k)}\right)
-\alpha\xi^{(k+1)},
\\
&y^{(k)}=(1+\beta)x^{(k)}-\beta x^{(k-1)},
\end{align*}
which implies that
\begin{align*}
x^{(k+1)}=(W\otimes I_{d})x^{(k)}
+(W\otimes I_{d})\beta\left(x^{(k)}-x^{(k-1)}\right)-\alpha \nabla F\left(y^{(k)}\right)
-\alpha\xi^{(k+1)}.
\end{align*}
It follows that
\begin{align}
x^{(k)}&=\left(W^{k}\otimes I_{d}\right)x^{(0)}-\alpha\sum_{s=0}^{k-1}\left(W^{k-1-s}\otimes I_{d}\right)\nabla F\left(y^{(s)}\right)
\nonumber
\\
&\qquad\qquad
-\alpha\sum_{s=0}^{k-1}\left(W^{k-1-s}\otimes I_{d}\right)\xi^{(s+1)}
+\sum_{s=0}^{k-1}\left(W^{k-s}\otimes I_{d}\right)\beta\left(x^{(s)}-x^{(s-1)}\right).
\label{x:k:expression}
\end{align}
Let us define $\mathbf{\bar{x}}^{(k)}:=\left[\left(\bar{x}^{(k)}\right)^{T},\cdots,\left(\bar{x}^{(k)}\right)^{T}\right]^{T}\in\mathbb{R}^{Nd}$ 
and equivalently
\begin{equation}
\mathbf{\bar{x}}^{(k)}=\frac{1}{N}\left(\left(1_{N}1_{N}^{T}\right)\otimes I_{d}\right)x^{(k)}\,,
\label{def-bold-bar-xk}
\end{equation}
where $1_N \in \mathbb{R}^N$ is a vector of ones; i.e. it is a column vector with all entries equal to one and the superscript $^T$ denotes the vector transpose.
Therefore, we get
\begin{equation*}
\sum_{i=1}^{N}\left\Vert x_{i}^{(k)}-\bar{x}^{(k)}\right\Vert^{2}
=\left\Vert x^{(k)}-\mathbf{\bar{x}}^{(k)}\right\Vert^{2}
=\left\Vert x^{(k)}-\frac{1}{N}\left(\left(1_{N}1_{N}^{T}\right)\otimes I_d\right)x^{(k)}\right\Vert^{2}.
\end{equation*}
Note that it follows from \eqref{x:k:expression} that
\begin{align*}
&x^{(k)}-\frac{1}{N}\left(\left(1_{N}1_{N}^{T}\right)\otimes I_{d}\right)x^{(k)}
\\
&=\left(W^{k}\otimes I_{d}\right)x^{(0)}-\frac{1}{N}\left(\left(1_{N}1_{N}^{T}W^{k}\right)\otimes I_{d}\right)x^{(0)}
\\
&-\alpha\sum_{s=0}^{k-1}\left(W^{k-1-s}\otimes I_{d}\right)\nabla F\left(y^{(s)}\right)
+\alpha\sum_{s=0}^{k-1}\frac{1}{N}\left(\left(1_{N}1_{N}^{T}W^{k-1-s}\right)\otimes I_{d}\right)\nabla F\left(y^{(s)}\right)
\\
&\quad
-\alpha\sum_{s=0}^{k-1}\left(W^{k-1-s}\otimes I_{d}\right)\xi^{(s+1)}
+\alpha\sum_{s=0}^{k-1}\frac{1}{N}\left(\left(1_{N}1_{N}^{T}W^{k-1-s}\right)\otimes I_{d}\right)\xi^{(s+1)}
\\
&\quad
+\sum_{s=0}^{k-1}\left(W^{k-s}\otimes I_{d}\right)\beta\left(x^{(s)}-x^{(s-1)}\right)
-\sum_{s=0}^{k-1}\frac{1}{N}\left(\left(1_{N}1_{N}^{T}W^{k-s}\right)\otimes I_{d}\right)\beta\left(x^{(s)}-x^{(s-1)}\right).
\end{align*}
By the Cauchy-Schwarz inequality, we have
\begin{align*}
&\left\Vert x^{(k)}-\frac{1}{N}\left(\left(1_{N}1_{N}^{T}\right)\otimes I_{d}\right)x^{(k)}\right\Vert^{2}
\\
&\leq 
4\left\Vert(W^{k}\otimes I_{d})x^{(0)}-\frac{1}{N}\left(\left(1_{N}1_{N}^{T}W^{k}\right)\otimes I_{d}\right)x^{(0)}\right\Vert^{2}
\\
&\quad
+4\left\Vert-\alpha\sum_{s=0}^{k-1}\left(W^{k-1-s}\otimes I_{d}\right)\nabla F\left(y^{(s)}\right)
+\alpha\sum_{s=0}^{k-1}\frac{1}{N}\left(\left(1_{N}1_{N}^{T}W^{k-1-s}\right)\otimes I_{d}\right)\nabla F\left(y^{(s)}\right)\right\Vert^{2}
\\
&\quad
+4\left\Vert\alpha\sum_{s=0}^{k-1}\left(W^{k-1-s}\otimes I_{d}\right)\xi^{(s+1)}
-\alpha\sum_{s=0}^{k-1}\frac{1}{N}\left(\left(1_{N}1_{N}^{T}W^{k-1-s}\right)\otimes I_{d}\right)\xi^{(s+1)}\right\Vert^{2}
\\
&\quad
+4\left\Vert\sum_{s=0}^{k-1}\left(W^{k-s}\otimes I_{d}\right)\beta\left(x^{(s)}-x^{(s-1)}\right)
-\sum_{s=0}^{k-1}\frac{1}{N}\left(\left(1_{N}1_{N}^{T}W^{k-s}\right)\otimes I_{d}\right)\beta\left(x^{(s)}-x^{(s-1)}\right)\right\Vert^{2}
\\
&=4\left\Vert(W^{k}\otimes I_{d})x^{(0)}-\frac{1}{N}\left(\left(1_{N}1_{N}^{T}\right)\otimes I_{d}\right)x^{(0)}\right\Vert^{2}
\\
&\quad
+4\left\Vert-\alpha\sum_{s=0}^{k-1}\left(W^{k-1-s}\otimes I_{d}\right)\nabla F\left(y^{(s)}\right)
+\alpha\sum_{s=0}^{k-1}\frac{1}{N}\left(\left(1_{N}1_{N}^{T}\right)\otimes I_{d}\right)\nabla F\left(y^{(s)}\right)\right\Vert^{2}
\\
&\quad
+4\left\Vert\alpha\sum_{s=0}^{k-1}\left(W^{k-1-s}\otimes I_{d}\right)\xi^{(s+1)}
-\alpha\sum_{s=0}^{k-1}\frac{1}{N}\left(\left(1_{N}1_{N}^{T}\right)\otimes I_{d}\right)\xi^{(s+1)}\right\Vert^{2}
\\
&\quad
+4\left\Vert\sum_{s=0}^{k-1}\left(W^{k-s}\otimes I_{d}\right)\beta\left(x^{(s)}-x^{(s-1)}\right)
-\sum_{s=0}^{k-1}\frac{1}{N}\left(\left(1_{N}1_{N}^{T}\right)\otimes I_{d}\right)\beta\left(x^{(s)}-x^{(s-1)}\right)\right\Vert^{2},
\end{align*}
where we used the property that $W$ is doubly stochastic. Therefore, we get
\begin{align}
&\left\Vert x^{(k)}-\frac{1}{N}\left(\left(1_{N}1_{N}^{T}\right)\otimes I_{d}\right)x^{(k)}\right\Vert^{2}
\nonumber
\\
&\leq 4\left\Vert\left(\left(W^{k}-\frac{1}{N}1_{N}1_{N}^{T}\right)\otimes I_{d}\right)x^{(0)}\right\Vert^{2}
\nonumber
\\
&\qquad
+4\alpha^{2}\left\Vert\sum_{s=0}^{k-1}\left(\left(W^{k-1-s}-\frac{1}{N}1_{N}1_{N}^{T}\right)\otimes I_{d}\right)\nabla F\left(y^{(s)}\right)\right\Vert^{2}
\nonumber \\
&\qquad\qquad+4\alpha^{2}\left\Vert\sum_{s=0}^{k-1}\left(\left(W^{k-1-s}-\frac{1}{N}1_{N}1_{N}^{T}\right)\otimes I_{d}\right)\xi^{(s+1)}\right\Vert^{2}
\nonumber
\\
&\qquad\qquad\qquad\qquad
+4\left\Vert\sum_{s=0}^{k-1}\left(\left(W^{k-s}-\frac{1}{N}1_{N}1_{N}^{T}\right)\otimes I_{d}\right)\beta\left(x^{(s)}-x^{(s-1)}\right)\right\Vert^{2}. 
\label{eq:inter-estimate:DASG}
\end{align}
Note that
\begin{align}
&4\alpha^{2}\left\Vert\sum_{s=0}^{k-1}\left(\left(W^{k-1-s}-\frac{1}{N}1_{N}1_{N}^{T}\right)\otimes I_{d}\right)\nabla F\left(y^{(s)}\right)\right\Vert^{2} \nonumber
\\
&\leq
4\alpha^{2}\left(\sum_{s=0}^{k-1}\left\Vert\left(W^{k-1-s}-\frac{1}{N}1_{N}1_{N}^{T}\right)\otimes I_{d}\right\Vert
\cdot\left\Vert\nabla F\left(y^{(s)}\right)\right\Vert\right)^{2} \nonumber
\\
&\leq
4\alpha^{2}\left(\sum_{s=0}^{k-1}\left\Vert W^{k-1-s}-\frac{1}{N}1_{N}1_{N}^{T}\right\Vert
\cdot\left\Vert\nabla F\left(y^{(s)}\right)\right\Vert\right)^{2} \nonumber
\\
&=
4\alpha^{2}\left(\sum_{s=0}^{k-1}\gamma^{k-1-s}
\cdot\left\Vert\nabla F\left(y^{(s)}\right)\right\Vert\right)^{2} \nonumber
\\
&=4\alpha^{2}\left(\sum_{s=0}^{k-1}\gamma^{k-1-s}\right)^{2}\left(\frac{\sum_{s=0}^{k-1}\gamma^{k-1-s}
\cdot\left\Vert\nabla F\left(y^{(s)}\right)\right\Vert}{\sum_{s=0}^{k-1}\gamma^{k-1-s}}\right)^{2} \nonumber
\\
&\leq
4\alpha^{2}\left(\sum_{s=0}^{k-1}\gamma^{k-1-s}\right)^{2}
\sum_{s=0}^{k-1}\frac{\gamma^{k-1-s}}{\sum_{s=0}^{k-1}\gamma^{k-1-s}}\left\Vert\nabla F\left(y^{(s)}\right)\right\Vert^{2}, \label{last-ineq-dev-from-mean-DASG}
\end{align}
where we used Jensen's inequality in the last step above,
and the fact that $W^{k-1-s}$ has eigenvalues
$(\lambda_{i}^{W})^{k-1-s}$
with $1=\lambda_{1}^{W}>\lambda_{2}^{W}\geq\cdots\geq\lambda_{N}^{W}>-1$, and hence
$\left\Vert W^{k-1-s}-\frac{1}{N}1_{N}1_{N}^{T}\right\Vert
=\max\{|\lambda_{2}^{W}|^{k-1-s},|\lambda_{N}^W|^{k-1-s}\}=\gamma^{k-1-s}$.
Moreover, we can compute that for any $k$
\begin{align}
\mathbb{E}\left\Vert\nabla F\left(y^{(k)}\right)\right\Vert^{2}
&\leq
2\mathbb{E}\left\Vert\nabla F\left(y^{(k)}\right)-\nabla F\left(x^{\ast}\right)\right\Vert^{2}
+2\left\Vert\nabla F\left(x^{\ast}\right)\right\Vert^{2}\nonumber
\\
&\leq
2L^{2}\mathbb{E}\left\Vert y^{(k)}-x^{\ast}\right\Vert^{2}
+2\left\Vert\nabla F\left(x^{\ast}\right)\right\Vert^{2}\nonumber
\\
&=2L^{2}\mathbb{E}\left\Vert (1+\beta)\left(x^{(k)}-x^{\ast}\right)-\beta\left(x^{(k-1)}-x^{\ast}\right)\right\Vert^{2}
+2\left\Vert\nabla F\left(x^{\ast}\right)\right\Vert^{2}\nonumber
\\
&\leq
4L^{2}(1+\beta)^{2}\mathbb{E}\left\Vert x^{(k)}-x^{\ast}\right\Vert^{2}
+4L^{2}\beta^{2}\mathbb{E}\left\Vert x^{(k-1)}-x^{\ast}\right\Vert^{2}
+2\left\Vert\nabla F\left(x^{\ast}\right)\right\Vert^{2}\nonumber
\\
&\leq
D_{y}^{2}, \label{ineq-stoc-grad-bound-ASG}
\end{align}
where $D_{y}^{2}$ is defined in \eqref{D:y:eqn}
and we used Corollary~\ref{cor:rate:dasg} to obtain the last line above. 
Therefore, by \eqref{last-ineq-dev-from-mean-DASG}, we have
\begin{align*}
&4\alpha^{2}\mathbb{E}\left[\left\Vert\sum_{s=0}^{k-1}\left(\left(W^{k-1-s}-\frac{1}{N}1_{N}1_{N}^{T}\right)\otimes I_{d}\right)\nabla F\left(y^{(s)}\right)\right\Vert^{2}\right]
\\
&\leq
4D_{y}^{2}\alpha^{2}\left(\sum_{s=0}^{k-1}\gamma^{k-1-s}\right)^{2}
\sum_{s=0}^{k-1}\frac{\gamma^{k-1-s}}{\sum_{s=0}^{k-1}\gamma^{k-1-s}}
\leq 
4D_{y}^{2}\alpha^{2}\frac{1}{(1-\gamma)^{2}}.
\end{align*}
Similarly, we can show that
\begin{align*}
4\left\Vert\left(\left(W^{k}-\frac{1}{N}1_{N}1_{N}^{T}\right)\otimes I_{d}\right)x^{(0)}\right\Vert^{2}
&\leq
4\left\Vert\left(W^{k}-\frac{1}{N}1_{N}1_{N}^{T}\right)\otimes I_{d}\right\Vert^{2} \left\Vert x^{(0)}\right\Vert^{2}
\\
&\leq 4\gamma^{2k}\left\Vert x^{(0)}\right\Vert^{2}.
\end{align*}
This implies that
\begin{align*}
&4\mathbb{E}\left\Vert\left(\left(W^{k}-\frac{1}{N}1_{N}1_{N}^{T}\right)\otimes I_{d}\right)x^{(0)}\right\Vert^{2}
\\
&\leq 8\gamma^{2k}\mathbb{E}\left\Vert x^{(0)}-x^{\ast}\right\Vert^{2}
+8\gamma^{2k}\Vert x^{\ast}\Vert^{2}
\\
&\leq 8\gamma^{2k}\left(4\frac{ V_{S,\alpha}\left(\xi_{0}\right)}{\mu}
+ \frac{2\sigma^2 N \sqrt{\alpha}}{\mu\sqrt{\mu}}\left(2-\lambda_N^W +\alpha L \right) +  \frac{2C_{1}^{2}N\alpha^2}{(1-\gamma)^2}
+\Vert x^{\ast}\Vert^{2}\right),
\end{align*}
where we used Corollary~\ref{cor:rate:dasg}.

In addition, by applying Lemma~\ref{lem:difference}, we can show that
\begin{align*}
&4\sum_{s=0}^{k-1}\mathbb{E}\left\Vert\left(\left(W^{k-s}-\frac{1}{N}1_{N}1_{N}^{T}\right)\otimes I_{d}\right)\beta\left(x^{(s)}-x^{(s-1)}\right)\right\Vert^{2}
\\
&\leq
4\frac{\beta^{2}}{(1-\gamma)^{2}}\sup_{s\geq 0}\mathbb{E}\left\Vert\left(x^{(s)}-x^{(s-1)}\right)\right\Vert^{2}
\\
&\leq
4\frac{\beta^{2}}{(1-\gamma)^{2}}\frac{2C_{0}}{\beta^{2}}\alpha
=\frac{8C_{0}}{(1-\gamma)^{2}}\alpha.
\end{align*}
Finally, we can show that
\begin{align*}
&4\alpha^{2}\sum_{s=0}^{k-1}\mathbb{E}\left\Vert\left(\left(W^{k-1-s}-\frac{1}{N}1_{N}1_{N}^{T}\right)\otimes I_{d}\right)\xi^{(s+1)}\right\Vert^{2}
\\
&\leq
4\alpha^{2}\sum_{s=0}^{k-1}\left\Vert W^{k-1-s}-\frac{1}{N}1_{N}1_{N}^{T}\right\Vert^{2}
\mathbb{E}\left\Vert \xi^{(s+1)}\right\Vert^{2}
\leq
\frac{4\sigma^{2}N\alpha^{2}}{(1-\gamma)^{2}}.
\end{align*}
Hence, it follows from \eqref{eq:inter-estimate:DASG} that 
\begin{align*}
&\sum_{i=1}^{N}\mathbb{E}\left\Vert x_{i}^{(k)}-\bar{x}^{(k)}\right\Vert^{2} 
\\
&=\left\Vert x^{(k)}-\frac{1}{N}\left(\left(1_{N}1_{N}^{T}\right)\otimes I_{d}\right)x^{(k)}\right\Vert^{2}
\\
&\leq
4\gamma^{2k}\mathbb{E}\left\Vert x^{(0)}\right\Vert^{2}
+4D_{y}^{2}\alpha^{2}\frac{1}{(1-\gamma)^{2}}
+4\alpha^{2}\sum_{s=0}^{k-1}\mathbb{E}\left\Vert\left(\left(W^{k-1-s}-\frac{1}{N}1_{N}1_{N}^{T}\right)\otimes I_{d}\right)\xi^{(s+1)}\right\Vert^{2}
\\
&\qquad
+4\sum_{s=0}^{k-1}\mathbb{E}\left\Vert\left(\left(W^{k-s}-\frac{1}{N}1_{N}1_{N}^{T}\right)\otimes I_{d}\right)\beta\left(x^{(s)}-x^{(s-1)}\right)\right\Vert^{2}
\\
&\leq
8\gamma^{2k}\left(4\frac{ V_{S,\alpha}\left(\xi_{0}\right)}{\mu}
+ \frac{2\sigma^2 N \sqrt{\alpha}}{\mu\sqrt{\mu}}\left(2-\lambda_N^W +\alpha L \right) +  \frac{2C_{1}^{2}N\alpha^2}{(1-\gamma)^2}
+\Vert x^{\ast}\Vert^{2}\right)
\\
&\qquad
+4D_{y}^{2}\alpha^{2}\frac{1}{(1-\gamma)^{2}}
+\frac{4\sigma^{2}N\alpha^{2}}{(1-\gamma)^{2}}
+\frac{8C_{0}}{(1-\gamma)^{2}}\alpha.
\end{align*} 
The proof is complete.}
\end{proof}

\subsubsection{Proof of Lemma~\ref{lem:E:D-ASG}}

\begin{proof}
{\color{black}We notice that
\begin{align*}
&\mathbb{E}\left\Vert\mathcal{E}_{k+1}\right\Vert^{2}
\\
&=\mathbb{E}\left\Vert
\frac{1}{N}\sum_{i=1}^{N}\left(\nabla f_{i}\left(y_{i}^{(k)}\right)
-\nabla f_{i}\left(\bar{y}^{(k)}\right)\right)\right\Vert^{2}
\\
&\leq
\frac{1}{N^{2}}\sum_{i=1}^{N}
N\mathbb{E}\left\Vert
\nabla f_{i}\left(y_{i}^{(k)}\right)
-\nabla f_{i}\left(\bar{y}^{(k)}\right)\right\Vert^{2}
\\
&\leq\frac{1}{N}L^{2}\sum_{i=1}^{N}
\mathbb{E}\left\Vert
y_{i}^{(k)}
-\bar{y}^{(k)}\right\Vert^{2}
\\
&\leq\frac{2}{N}L^{2}\sum_{i=1}^{N}
\left((1+\beta)^{2}\mathbb{E}\left\Vert
x_{i}^{(k)}
-\bar{x}^{(k)}\right\Vert^{2}
+\beta^{2}\mathbb{E}\left\Vert
x_{i}^{(k-1)}
-\bar{x}^{(k-1)}\right\Vert^{2}\right)
\\
&\leq
\frac{2}{N}L^{2}
\left((1+\beta)^{2}+\beta^{2}\right)
\Bigg[
4D_{y}^{2}\alpha^{2}\frac{1}{(1-\gamma)^{2}}
+\frac{4\sigma^{2}N\alpha^{2}}{(1-\gamma)^{2}}
+\frac{8C_{0}}{(1-\gamma)^{2}}\alpha
\\
&\qquad\qquad
+8\gamma^{2(k-1)}\left(4\frac{ V_{S,\alpha}\left(\xi_{0}\right)}{\mu}
+ \frac{2\sigma^2 N \sqrt{\alpha}}{\mu\sqrt{\mu}}\left(2-\lambda_N^W +\alpha L \right) +  \frac{2C_{1}^{2}N\alpha^2}{(1-\gamma)^2}
+\Vert x^{\ast}\Vert^{2}\right)\Bigg],
\end{align*} 
where we used Lemma~\ref{gradient:average:error:DASG}.
The proof is complete.}
\end{proof}

\subsubsection{Proof of Lemma~\ref{lem:difference}}

\begin{proof} 
{\color{black}We first rewrite the D-ASG iterations \eqref{eq-centr-asg} as
\begin{equation}
\begin{aligned}
&z^{(k+1)}
= \left(\tilde{M} \otimes I_d\right) z^{(k)} - \alpha \begin{bmatrix}
       \tilde \nabla F\left(Cz^{(k)}\right) \\             0
\end{bmatrix}, 
\quad
z^{(k)} = \begin{bmatrix}
          x^{(k)}   \\ y^{(k)}      
\end{bmatrix},
\end{aligned}
\end{equation}
where 
\beq \tilde M = \begin{bmatrix}
(1+\beta) W & -\beta W \\
I_{N} & 0_{N}
\end{bmatrix}.
\label{eq:tildeM-decomposition}
\eeq
By a reasoning similar to the proof of Proposition \ref{thm:DASG:rho}, we observe that 
$\tilde{M}$ is block diagonalizable with $2\times 2$ blocks satisfying
   \beq \tilde{M} = \tilde{O} \textbf{\mbox{diag}}\left(\{{Z}_i\}_{i=1}^{N}\right)\tilde{O}^T, 
   \label{eq:decomposition-tildeM}
   \eeq
where   
   \begin{equation*} 
\tilde{Z}_i =\left[
\begin{array}{cc}
(1+\beta)\lambda_{i}^{W} & -\beta\lambda_{i}^{W}
\\
1 & 0
\end{array}
\right] \in {\mathbb{R}}^{2\times 2},
\qquad 1\leq i\leq N,
\end{equation*}
$\lambda_{i}^{W}$ are the eigenvalues of $W$ in decreasing order, $\tilde{O} = \tilde{U}\tilde{P}_{\tilde\pi}$ is orthogonal with $\tilde{U}$ and $\tilde{P}_{\tilde{\pi}}$ are defined
as
$\tilde U = \mbox{\textbf{diag}}( V,V)$ where $W = VDV^T$ is the eigenvalue decomposition of $W$ and
$\tilde{P}_{\tilde\pi}$ is the permutation matrix
associated with the permutation $\tilde\pi$ over $\{1,2,\ldots,2N\}$
that satisfies
\begin{equation}
\tilde\pi(i)=
\begin{cases}
2i-1 &\text{if $1\leq i\leq N$},
\\
2(i-N) &\text{if $N+1\leq i\leq 2N$}.
\end{cases}
\label{def-perm-P}
\end{equation}
 We also observe that $\tilde{Z}_i$ has eigenvalues
$$\mu_{i,\pm}:=\frac{(1+\beta)\lambda_{i}^{{W}}\pm\sqrt{(1+\beta)^{2}(\lambda_{i}^{{W}})^{2}-4\beta\lambda_{i}^W}}{2}.$$
In particular, in the special case when $i=1$, we have $\lambda_1^W = 1$ and $\tilde{Z}_1$ has two eigenvalues $\mu_{1,+} = 1$ and $\mu_{1,-} = \beta<1$ for $i=1,2,\dots, d$, admitting the Jordan decomposition
$$ \tilde{Z}_i = S_1 \begin{bmatrix} 1 & 0\\ 0 & \beta
\end{bmatrix}
S_1^{-1}, \quad \mbox{for} \quad i=1,2,\dots,d,
$$
where
 $$ S_1 = \begin{bmatrix}
                     1 & \beta\\
                     1 & 1
 \end{bmatrix}, \quad
 S_1^{-1} = \frac{1}{1-\beta}\begin{bmatrix}
                    1 & -\beta \\
                    -1 & 1
 \end{bmatrix}. 
 $$
Similarly, for $i>1$, we can also write the Jordan decomposition of $\tilde{Z}_i$ as 
\beq\tilde{Z}_i = S_i J_i S_i^{-1} \quad \mbox{for} \quad 1<i\leq N, \label{eq-tilde-Zi}
\eeq
where 
\begin{align}
&S_i = \begin{cases}
  \begin{bmatrix}
        \mu_{i,+} & \mu_{i,-} \\
        1 & 1
  \end{bmatrix} & \mbox{if} \quad \mu_{i,+} \neq \mu_{i,-}, \\
  \begin{bmatrix}
        \mu_{i,+} & 1 \\
        1 & 1
  \end{bmatrix} & \mbox{if} \quad \mu_{i,+} = \mu_{i,-},
\end{cases},
\\
&S_i^{-1} = \begin{cases}
  \frac{1}{\mu_{i,+}-\mu_{i,-}}\begin{bmatrix}
         1& -\mu_{i,-} \\
        -1 & \mu_{i,+}
  \end{bmatrix} & \mbox{if} \quad \mu_{i,+} \neq \mu_{i,-}, \\
  \frac{1}{\mu_{i,+}-1}\begin{bmatrix}
        1 & -1 \\
        -1 & \mu_{i,+}
  \end{bmatrix} & \mbox{if} \quad \mu_{i,+} = \mu_{i,-},
\end{cases}
\end{align}
and 
$$ J_i = \begin{cases}
    \begin{bmatrix}
         \mu_{i,+} & 0 \\
                        0 & \mu_{i,-}
    \end{bmatrix}
& \mbox{if} \quad \mu_{i,+} \neq \mu_{i,-},\\
    \begin{bmatrix}
         \mu_{i,+} & 1 \\
                        0 & \mu_{i,+}
    \end{bmatrix}
& \mbox{if} \quad \mu_{i,+} = \mu_{i,-}.
\end{cases} 
$$
Basically, the structure of the Jordan blocks $J_i$ will depend on the multiplicity of the eigenvalues $s \mu_{i,+}, \mu_{i,-}$ which itself depends on the stepsize chosen and the eigenvalues of the matrix $W$. Next, we introduce 
$$ \bar{Z}_i := \begin{cases}
S_1 \begin{bmatrix} 1 & 0\\ 0 & 0
\end{bmatrix}
S_1^{-1} & \mbox{if} \quad i=1,\\
0 & \mbox{if} \quad 1<i \leq N,
\end{cases}
$$
as well as
$$Z:=\textbf{\mbox{diag}}\left(\{{Z}_i\}_{i=1}^{N}\right), \quad \bar{Z} := \textbf{\mbox{diag}}\left(\{\bar{Z}_i\}_{i=1}^{N}\right).  
$$
Note that the matrices $Z$ and $\bar{Z}$ are block diagonal with $2\times 2$ blocks and they have both $1$ as a simple eigenvalue with the same eigenvectors; however other eigenvalues of $\bar{Z}$ is set to zero. We have also
\beq Z_i - \bar{Z}_i = 
\begin{cases} S_1
\begin{bmatrix} 0 & 0\\ 0 & \beta
\end{bmatrix}
S_1^{-1} & \mbox{if} \quad i=1,\\
\tilde{Z}_i & \mbox{if} \quad 1<i \leq N.
\end{cases} \label{eq-diff-mat-Z}
\eeq
It can also be computed that
$$ \bar{M}:= \tilde O \bar{Z} \tilde{O}^T = \frac{1}{1-\beta}
\begin{bmatrix}
      v_1v_1^T & -\beta v_1 v_1^T  \\
       v_1 v_1^T        &    - \beta v_1 v_1^T  
\end{bmatrix},
$$
where $v_1 = \frac{1}{\sqrt{N}}\ones \in \mathbb{R}^N$ is the eigenvector of $W$ corresponding to the eigenvalue 1.  
Consider
\beq 
\left(\bar{M}\otimes I_d\right) z^{(k)} = \frac{1}{1-\beta}
\begin{bmatrix}
 \bar{x}^{(k)} - \beta \bar{x}^{(k-1)} \\  
 \quad \vdots \quad  \\
 \bar{x}^{(k)} - \beta \bar{x}^{(k-1)} 
\end{bmatrix} = 
\frac{1}{1-\beta}
\begin{bmatrix}
\mathbf{\bar{x}}^{(k)}  - \beta \mathbf{\bar{x}}^{(k-1)}   \\
\mathbf{\bar{x}}^{(k)}  - \beta \mathbf{\bar{x}}^{(k-1)} 
\end{bmatrix}
\in \mathbb{R}^{2Nd},
\label{eq-weighted-average-subtracted}
\eeq
which can be viewed as a weighted average of $\bar{x}^{(k)}$ and $\bar{x}^{(k-1)}$ with $\mathbf{\bar{x}}^{(k)}$ defined as in \eqref{def-bold-bar-xk}. From the definition of $\bar{M}$ and the decomposition \eqref{eq:decomposition-tildeM}, we have 
\beq 
\tilde{M} - \bar{M} = \tilde O \left(Z - \bar{Z}\right) \tilde{O}^T, 
\label{eq-diff-iter-mat}
\eeq
and this matrix has the spectral radius
\begin{align} 
r:&= \rho\left(\tilde{M}-\bar{M}\right) =\max\left(\beta, \max_{j\geq2}\rho(J_j)\right) \\
&= \max\left(\beta,\left|\frac{(1+\beta)\lambda_{2}^W+\sqrt{(1+\beta)^{2}(\lambda_{2}^W)^{2}-4\beta\lambda_{2}^W}}{2}\right|\right)
\\
&< r_3: = \max\left(\beta, \sqrt{\lambda_2^W}\right) \label{ineq:upper bound on spectral rad}
< 1,
\end{align}
where in the last inequality we used the facts that 
the function 
\beq g(\lambda,\beta) := \left|\frac{(1+\beta)\lambda+\sqrt{(1+\beta)^{2}(\lambda)^{2}-4\beta\lambda}}{2}\right|
\label{def-function-g}
\eeq
defined for $\lambda, \beta \in [0,1]$ is increasing in $\lambda$ on the interval $[0,1]$ for fixed $\beta \in [0,1]$ and is increasing in $\beta$ on the interval $[0,1]$ for fixed $\alpha \in [0,1]$ which results in the inequalities $g(\lambda_i^W,\beta) \leq g(\lambda_2^W,\beta) \leq g(\lambda_2^W,1) = \sqrt{\lambda_2^W}$  for $i\geq 2$. From \eqref{ineq:upper bound on spectral rad}, it follows that
\beq \left\| (M-\bar{M})^k \right\| \leq C_3 r_3^k \quad \mbox{for all} \quad k\geq 0, 
\label{ineq-to-satisfy-exp-decay}
\eeq
for some positive constant $C_3$. The constant $C_3$ will depend on the parameter $\alpha$ in general (as $\beta$ and $r_3$ are functions of $\alpha$) but it will not depend on $k$. A natural question would be how the constants $C_3$ and $r_3$ change as a function of $\alpha$ as $\alpha \to 0$. It follows from Lemma \ref{lem:constants-in-alpha} that one can choose $c_3$ and $r_3$ such that 
\beq C_3 = \Theta\left(\frac{1}{\sqrt{\alpha}}\right), \quad r_3 = 1- \Theta(\sqrt{\alpha}) \quad \mbox{as} \quad \alpha \to 0.
\label{eq-order-alpha-exp bounds}
\eeq
Furthermore,
\begin{align} 
&\left( \left(I_{2N}-\bar{M}\right)\otimes I_d\right)z^{(k+1)}
\nonumber
\\
&=  \left( \left(I_{2N}-\bar{M}\right) \tilde{M} \otimes I_d \right) z^{(k)} - \left(\left(I_{2N}-\bar{M}\right)\otimes I_d \right) \alpha \begin{bmatrix}
       \tilde \nabla F\left(Cz^{(k)}\right) \\             0
\end{bmatrix}
\nonumber
\\
&= \left( \left(\tilde M - \bar{M}\right)\left(I_{2N} - \bar{M}\right) \otimes I_d\right) z^{(k)} - \left(\left(I_{2N}-\bar{M}\right)\otimes I_d \right) \alpha \begin{bmatrix}
       \tilde \nabla F\left(Cz^{(k)}\right) \\             0
\end{bmatrix},
\end{align}
where we used the facts that $\tilde M\bar{M} = \bar{M} \tilde M = \bar{M}$ and $\bar{M}^2 = \bar{M}$. 
On the other hand, 
\begin{align}
    \big( (I_{2N}-\bar{M})\otimes I_d \big) z^{(k+1)} &= \big((\tilde M-\bar{M})^{k+1} (I_{2N}-\bar{M}) \otimes I_d \big) z^{(0)}\nonumber \\
&\qquad\qquad- \alpha \sum_{j=0}^k \big((\tilde M-\bar{M})^j (I_{2N}-\bar{M}) \otimes I_d \big) \begin{bmatrix}
       \tilde \nabla F\left(Cz^{(k-j)}\right) \\             0
\end{bmatrix}. 
\end{align}
Therefore, for $z^{(0)} = 0$, 
\begin{align} 
&\mathbb{E}\left\Vert \big( (I_{2N}-\bar{M})\otimes I_d \big) z^{(k+1)} \right\Vert^2\nonumber
\\
&= 
 \alpha^2  \mathbb{E} \left\|\sum_{j=0}^k \bigg( \left(\tilde M-\bar{M}\right)^j \left(I_{2N}-\bar{M}\right) \otimes I_d \bigg) \begin{bmatrix}
       \tilde \nabla  F\left(Cz^{(k-j)}\right) \\             0
\end{bmatrix} \right\|^2 \nonumber\\
&\leq\alpha^2 \mathbb{E} \left\|\sum_{j=0}^k  \bigg( \left(\tilde M-\bar{M}\right)^j \left(I_{2N}-\bar{M}\right) \otimes I_d \bigg) \begin{bmatrix}
       \tilde \nabla  F\left(Cz^{(k-j)}\right) \\             0
\end{bmatrix} \right\|^2 \nonumber\\
&=\alpha^2 \mathbb{E} \left\|\sum_{j=0}^k \bigg( \left( \left(\tilde M-\bar{M}\right)^{j+1} +\left(\tilde M-\bar{M}\right)^{j}\left(I_{2N}-\tilde M\right) \right) \otimes I_d \bigg) \begin{bmatrix}
       \tilde \nabla  F\left(Cz^{(k-j)}\right) \\             0
\end{bmatrix} \right\|^2 \nonumber\\
&\leq \alpha^2\mathbb{E} \left|\sum_{j=0}^k  c_3 r_3^j \left(1 + \left\|I_{2N}-\tilde M\right\|\right) 
\left\| 
       \tilde \nabla  F\left(Cz^{(k-j)}\right) \right\|
       \right|^2 \nonumber\\
&\leq\alpha^2 \left(\sum_{j=0}^k a_j\right)^2 \mathbb{E} \left|\sum_{j=0}^k  \frac{a_j}{(\sum_{j=0}^k a_j)} \left\| 
       \tilde \nabla  F\left(Cz^{(k-j)}\right) \right\|
       \right|^2\nonumber\\
&\leq \alpha^2 \left(\sum_{j=0}^k a_j\right)^2        \sum_{j=0}^k \frac{a_j}{(\sum_{j=0}^k a_j)} \mathbb{E}\left\| \tilde \nabla  F\left(Cz^{(k-j)}\right)\right\|^2 \nonumber\\
&\leq\alpha^2 \left(\sum_{j=0}^k a_j\right)^2 D_y^2, \label{ineq-final-alpha-dep-of-average}
\end{align}
where we used \eqref{ineq-stoc-grad-bound-ASG}, $D_y^2$ is defined in \eqref{D:y:eqn} and
$$ a_j := c_3 r_3^j \left(1+\left\|I_{2N}-\tilde M\right\|\right). $$
Note that
\begin{align}  
\left\|I_{2N}-\tilde M\right\| &\leq \|I_{2N}\| + \left\|\tilde M\right\|\nonumber\\
&\leq  1 + \left\|\tilde M\right\|_F \nonumber\\
&= 1 + \sqrt{\left((1+\beta)^2 + \beta^2\right)\|W\|_F^2 +N},
\end{align}
where we used the definition of $\tilde{M}$. Consequently,
$$ \sum_{j=0}^\infty a_j \leq \sum_{j\geq 0}c_3 r_3^j  \left(1 + \left\|I_{2N}-\tilde{M}\right\|\right) \leq c_3 \frac{1}{(1-r_3)} \left(1 + \sqrt{\left((1+\beta)^2 + \beta^2\right)\|W\|_F^2 +N}\right).
$$
Therefore, we conclude from \eqref{eq-weighted-average-subtracted} and  \eqref{ineq-final-alpha-dep-of-average} that
\begin{align} 
\sup_k \mathbb{E} \left\Vert \big( (I_{2N}-\bar{M}) \otimes I_d \big) z^{(k)}\right\Vert^2 &=\sup_{k}\mathbb{E}\left\|
           \begin{bmatrix}
 x^{(k)} -\frac{1}{1-\beta} \left(\mathbf{\bar{x}}^{(k)} - \beta \mathbf{\bar{x}}^{(k-1)}\right) \nonumber\\    y^{(k)} - \frac{1}{1-\beta}\left(\mathbf{\bar{x}}^{(k)} - \beta \mathbf{\bar{x}}^{(k-1)}\right) 
\end{bmatrix}   \right\|^2 \nonumber\\      
&\leq D_y^2 \alpha^2 \frac{c_3}{(1-r_3)} \left(1 + \sqrt{\left((1+\beta)^2 + \beta^2\right)\|W\|_F^2 +N}\right)\nonumber\\
&\leq D_y^2 \alpha^2 \frac{c_3}{(1-r_3)} \left(1 + \sqrt{5\|W\|_F^2 +N}\right),  \label{ineq-to-prove-alpha-dependency}
\end{align}
where $\mathbf{\bar{x}}^{(k)}$ is defined by \eqref{def-bold-bar-xk} and we used the fact that $\beta\leq 1$. From \eqref{eq-order-alpha-exp bounds}, we observe that the term $\frac{c_3}{1-r_3} = \mathcal{O}(\frac{1}{\alpha})$. We conclude that the right hand-side of \eqref{ineq-to-prove-alpha-dependency} is $O(\alpha)$.
Hence, we conclude that
\begin{align*}
&\sup_{k}\mathbb{E}\left\Vert x^{(k)}-y^{(k)}\right\Vert^{2}
\\
&\leq 
2\sup_{k}\left(\mathbb{E}\left\Vert x^{(k)}-\frac{1}{1-\beta}\left(\mathbf{\bar{x}}^{(k)} - \beta \mathbf{\bar{x}}^{(k-1)}\right)\right\Vert^{2}
+\mathbb{E}\left\Vert y^{(k)}-\frac{1}{1-\beta}\left(\mathbf{\bar{x}}^{(k)} - \beta \mathbf{\bar{x}}^{(k-1)}\right)\right\Vert^{2}\right)
\\
&\leq 
2C_{0}\alpha,
\end{align*}
where $C_{0}$ is some constant such that
$C_{0}=\mathcal{O}(1)$ as $\alpha\rightarrow 0$.
Finally, we notice that 
$y^{(k)}=(1+\beta)x^{(k)}-\beta x^{(k-1)}$,
and this implies that,
\begin{equation}
\sup_{k}\mathbb{E}\left\Vert x^{(k)}-x^{(k-1)}\right\Vert^{2}
=\sup_{k}\mathbb{E}\left\Vert\frac{1}{\beta}\left(x^{(k)}-y^{(k)}\right)\right\Vert^{2}
\leq\frac{2C_{0}}{\beta^{2}}\alpha,
\end{equation}
which completes the proof.
}
\end{proof}

\begin{lemma}\label{lem:constants-in-alpha} 
{\color{black}In the setting of the proof of Lemma  \ref{lem:difference}, 
$$c_3 = \Theta\left(\frac{1}{\sqrt{\alpha}}\right), \quad r_3 = 1- \Theta(\sqrt{\alpha}) \quad \mbox{as} \quad \alpha \to 0.$$}
\end{lemma}

\begin{proof}
{\color{black}Note that we have from \eqref{eq-diff-iter-mat}, \eqref{eq-tilde-Zi}, \eqref{eq-diff-mat-Z},
\begin{align}
\left\|(M - \bar{M})^k\right\| 
&= \left\| \tilde{O} \left(Z - \bar{Z}\right)^k \tilde{O}^T\right\| = \max_i \left\| (Z_i - \bar{Z}_i)^k\right\|\\
&\leq \max\left(\|S_1\| \left\|S_1^{-1}\right\| \beta^k,  \max_{2\leq i \leq N} \| S_i\| \left\|J_i^k\right\| \left\|S_i^{-1}\right\|\right),
\label{ineq: decay of the normalized iter matrix}
\end{align}
where we used the fact that $\tilde O$ is orthogonal. Also, 
\begin{equation}
    \|S_i\|_2 \leq \|S_i\|_F \leq 2 \quad \mbox{for} \quad 1\leq i\leq N,
    \label{ineq: Si norm bound}
\end{equation} 
where we used the definition of $S_i$ and the inequalities $|\mu_{i,+}| = g(\lambda_i^W,\beta) \leq g(1,1) = 1$ and $|\mu_{i,-}| \leq 1$ which follow from the definition \eqref{def-function-g} of the function $g(\lambda,\beta)$ and its monotonicity property with respect to $\lambda$ and $\beta$. Similarly,
\beq
\left\|S_1^{-1}\right\|_2 \leq  
\left\|S_1^{-1}\right\|_F \leq \frac{2}{1-\beta},
\label{ineq: S1 inverse norm bound}
\eeq 
and for $1<i\leq N$,
\beq \left\|S_i^{-1}\right\|_2 \leq\left\|S_i^{-1}\right\|_F \leq 
c_i := \begin{cases}
  \frac{2}{|\mu_{i,+}-\mu_{i,-}|} = \frac{2}{\sqrt{|(1+\beta)^{2}(\lambda_{i}^{{W}})^{2}-4\beta\lambda_{i}^W}|} & \mbox{if} \quad \mu_{i,+} \neq \mu_{i,-}, \\
  \frac{2}{|\mu_{i,+}-1|} & \mbox{if} \quad \mu_{i,+} = \mu_{i,-}.
\end{cases}
 \label{ineq: Si-inverse norm bound}
\eeq
In addition, 
$$ J_i^k = \begin{cases}
    \begin{bmatrix}
         \mu_{i,+}^k & 0 \\
                        0 & \mu_{i,-}^k
    \end{bmatrix}
& \mbox{if} \quad \mu_{i,+} \neq \mu_{i,-},\\
    \begin{bmatrix}
         \mu_{i,+}^k & k\mu_{i,+}^{k-1}  \\
                        0 & \mu_{i,+}^k
    \end{bmatrix}
& \mbox{if} \quad \mu_{i,+} = \mu_{i,-},
\end{cases} 
$$
used again the fact that $|\mu_{i,+}|\leq 1$, we have then
\begin{align} 
\left\|J_i^k\right\|_2 \leq \left\|J_i^k\right\|_F &=
\begin{cases}
|\mu_{i,+}|^k & \mbox{if}\quad \mu_{i,+} \neq \mu_{i,-}  \\
\sqrt{2|\mu_{i,+}|^{2k} + k^2 |\mu_{i,+}|^{2k-2}}& \mbox{if} \quad \mu_{i,+} = \mu_{i,-}
\end{cases} \nonumber\\
& \leq d_{i,k}|\mu_{i,+}|^{k-1}, 
\label{ineq: power of Jordan matrix}
\end{align} 
with
     $$ d_{i,k} := \begin{cases}
        1 & \mbox{if}\quad \mu_{i,+} \neq \mu_{i,-}, \\
       \sqrt{2}+k & \mbox{if}  \quad \mu_{i,+} = \mu_{i,-}.
     \end{cases}$$ 
Combining all the estimates \eqref{ineq: Si norm bound}, \eqref{ineq: S1 inverse norm bound}, \eqref{ineq: Si-inverse norm bound}, \eqref{ineq: power of Jordan matrix}  together, we obtain
     \begin{align} 
     \max_{2\leq i \leq N} \|S_i \| \left\|J_i^k\right\| \left\|S_i^{-1}\right\|
     &\leq 2 \left( \max_{2\leq i \leq N} c_i d_{i,k} \right) \left(\max_{2\leq i \leq N}  |\mu_{i,+}|^{k-1} \right)\nonumber \\
     &= 2 \left( \max_{2\leq i \leq N} c_i d_{i,k} \right)  |\mu_{2,+}|^{k-1}. 
     \label{ineq: expoential growth with poly prefac}
     \end{align}
Note that we have $|\mu_{i,+}| = g(\lambda_i^W, \beta) \leq g(\lambda_2^W, \beta) = \mu_{2,+} < 1$. Therefore, 
$$|\mu_{2,+}| < r_2 := \frac{1 + |\mu_{2,+}|}{2}.$$ 
Furthermore, $ \max_{2\leq i \leq N} c_i(\alpha) d_{i,k}  = \mathcal{O}(k)$. Consequently, we can bound  \eqref{ineq: expoential growth with poly prefac} as 
\beq \max_{2\leq i \leq N} \|S_i \| \left\|J_i^k\right\| \left\|S_i^{-1}\right\|
     \leq C_2r_2^{k-1} 
\eeq     
for some constant $C_2 = \mathcal{O}(1)$ that does not depend on $k$. 
Then, from \eqref{ineq: decay of the normalized iter matrix}, it follows that 
\begin{align}
\left\|\left(M - \bar{M}\right)^k\right\| 
&\leq \max\left(\|S_1\| \left\|S_1^{-1}\right\| \beta^k,  \max_{2\leq i \leq N} \| S_i\| \left\|J_i^k\right\| \left\|S_i^{-1}\right\|\right) \nonumber\\
&\leq \max\left( \frac{4}{1-\beta} \beta^k, C_2 r_2^{k-1}
\right) \nonumber\\
&\leq C_3 r_3^k,
\end{align}
where 
$$C_3:= \max\left(\frac{4}{1-\beta}, C_2\right), \quad r_3 = \max(\beta, r_2).$$
Also $|\mu_{2,+}| = g(\lambda_2^W,\beta) < g(\lambda_2^W,1) = \sqrt{\lambda_2^W} $. Therefore, 
$r_2 < \frac{1+\sqrt{\lambda_2^W}}{2}$. Since $\beta = 1-\Theta(\sqrt{\alpha})$, we observe that  
$$C_3 = \Theta\left(\frac{1}{\alpha}\right), \quad r_3 = 1- \Theta(\sqrt{\alpha})$$ as $\alpha \to 0$.
The proof is complete.} 
\end{proof}


\subsubsection{Proof of Lemma~\ref{lem:C:epsilon}}
{\color{black}
\begin{proof}The function $V_{\bar{S},\alpha}$ is $L_\alpha$ smooth where $\bar{L}_\alpha = L + 2\|\bar{S}_\alpha\|$. Note that 
$$ \bar{S}_\alpha = vv^T, \qquad v= 
\left[
\begin{array}{c}
\sqrt{\frac{1}{2\alpha}}
\\
\sqrt{\frac{\mu}{2}}-\sqrt{\frac{1}{2\alpha}}
\end{array}
\right].
$$
Therefore, 
\beq \left\|\bar{S}_\alpha\right\| = \|v\|^2 = \frac{1}{\alpha} + \frac{\mu}{2}-\frac{\sqrt{\mu}}{\sqrt{\alpha}}.
\label{ineq-S-alpha-norm}
\eeq
For any $\xi,\Delta\in\mathbb{R}^{2d}$, by the $\bar{L}_\alpha$-smoothness of $V_{\bar{S},\alpha}$ we have also
\begin{align} 
V_{\bar{S},\alpha}(\xi+\Delta) &\leq
V_{\bar{S},\alpha}(\xi) + \left\langle 
\nabla V_{\bar{S},\alpha}(\xi), \Delta \right\rangle + \frac{\bar{L}_\alpha}{2} \|\Delta\|^2 \nonumber\\
&\leq V_{\bar{S},\alpha}(\xi) + c_1 \left\|\nabla V_{\bar{S},\alpha}(\xi)\right\|^2 + \|\Delta\|^2/(4c_1)  + \frac{\bar{L}_\alpha}{2} \|\Delta\|^2, \label{ineq-helper}
\end{align}
for any $c_1>0$ where we used Cauchy-Schwarz in the last inequality. We have also
\begin{align} 
\left\|\nabla V_{\bar{S},\alpha}(\xi)\right\|^2 & \leq 2\|2\bar{S}_\alpha \xi\|^2 + 2\left\|\nabla f\left(\bar{T}\xi + x_*\right)\right\|^2\nonumber\\
&\leq
8 \xi^T \bar{S}_\alpha^2 \xi + 2L^2\left\|\bar{T}\xi\right\|^2\nonumber\\
&\leq
8\left(\frac{1}{\alpha} + \frac{\mu}{2}-\frac{\sqrt{\mu}}{\sqrt{\alpha}}\right) \xi^T \bar{S}_\alpha \xi + 2L^2\left\|\bar{T}\xi\right\|^2 \nonumber\\
&\leq 8\left(\frac{1}{\alpha} + \frac{\mu}{2}-\frac{\sqrt{\mu}}{\sqrt{\alpha}}\right) \xi^T \bar{S}_\alpha \xi + 2L^2 \frac{f(\bar{T}\xi+x_*) - f(x_*)}{\mu} \nonumber\\
&\leq c_2 V_{\bar{S},\alpha}(\xi),
\end{align}
where $$ c_2 := 2\max\left(4\left(\frac{1}{\alpha} + \frac{\mu}{2}-\frac{\sqrt{\mu}}{\sqrt{\alpha}}\right), \frac{L^2}{\mu}\right),$$
and we used the facts that
$$\bar{S}_\alpha^2 = \|v\|^2 \bar{S}_\alpha = \left(\frac{1}{\alpha} + \frac{\mu}{2}-\frac{\sqrt{\mu}}{\sqrt{\alpha}}\right)\bar{S}_\alpha.$$
Therefore, if choose $c_1 = \epsilon/c_2$; then we get from \eqref{ineq-helper}
\begin{align} 
V_{\bar{S},\alpha}(\xi+\Delta) 
\leq (1+\epsilon) V_{\bar{S},\alpha}(\xi)  + C_\epsilon \|\Delta\|^2, 
\end{align}
with $C_\epsilon = \frac{c_2}{4\epsilon} + \bar{L}_{\alpha}/2$. Finally, if we choose $\xi = \bar{\xi}_{k+1}-D_{k+1}$ and $\Delta= D_{k+1}$; we obtain \eqref{ineq-lyap-pert}. This completes the proof.
\end{proof}
}


\subsection{Proofs of Technical Results in Appendix~\ref{sec:quadratic}}

\subsubsection{Proof of Lemma~\ref{lem:DASG:Q}}
\begin{proof}
We recall from the proof of Proposition~\ref{thm:DASG:rho} that
\begin{equation*}
\mathcal{W} -\alpha Q = {\color{black}R} \textbf{\mbox{diag}}\left([\mu_i]_{i=1}^{Nd}\right){\color{black}R}^{T},
\end{equation*} 
where {\color{black}$R$} is real orthogonal and the eigenvalues $\mu_i$ are listed in non-increasing order. Then we can write
\begin{equation*}
\rev{P_\pi} U A_{\text{dasg},Q} U^T \rev{P_\pi^T} = \textbf{\mbox{diag}}\left(\tilde{T}_i\right) \quad \mbox{where} \quad U = \textbf{\mbox{diag}}\rev{(R,R)}
\end{equation*}
is orthogonal and
\begin{equation*} 
\tilde{T}_i =\left[
\begin{array}{cc}
(1+\beta)\mu_{i} & -\beta\mu_{i}
\\
1 & 0
\end{array}
\right] \in {\mathbb{R}}^{2\times 2},
\qquad 1\leq i\leq Nd.
\end{equation*}

Therefore, we have
\begin{equation*}
\left\Vert A_{\text{dasg},Q}^{k}\right\Vert
\leq\Vert V\Vert^{2}
\max_{1\leq i\leq Nd}
\left\Vert\left(\tilde{T}_{i}\right)^{k}\right\Vert
=\max_{1\leq i\leq Nd}
\left\Vert\left(\tilde{T}_{i}\right)^{k}\right\Vert,
\end{equation*}
where we used the fact that $\Vert V\Vert=1$
since $V$ is orthogonal. 
The remainder of the proof is devoted
to provide an upper bound on $\max_{1\leq i\leq Nd}
\Vert(\tilde{T}_{i})^{k}\Vert$.

Let $\gamma_{i,\pm}:=\frac{(1+\beta)\mu_{i}\pm\sqrt{(1+\beta)^{2}\mu_{i}^{2}-4\beta\mu_{i}}}{2}$
be the eigenvalues of $\tilde{T}_{i}$.

(i) If $\gamma_{i,+}\neq\gamma_{i,-}$,
then by the formula
of $k$-th power of $2\times 2$
matrix with distinct eigenvalues 
(see e.g. \cite{williams2by2}), 
we get
\begin{equation*}
\left(\tilde{T}_{i}\right)^{k}
=\frac{\gamma_{i,+}^{k}}{\gamma_{i,+}-\gamma_{i,-}}\left(\tilde{T}_{i}-\gamma_{i,-}I\right)
+\frac{\gamma_{i,-}^{k}}{\gamma_{i,-}-\gamma_{i,+}}\left(\tilde{T}_{i}-\gamma_{i,+}I\right).
\end{equation*}
This implies that
\begin{equation*}
\left\Vert\left(\tilde{T}_{i}\right)^{k}\right\Vert
\leq\frac{
\max\left\{\left\Vert\tilde{T}_{i}-\gamma_{i,-}I\right\Vert,\left\Vert\tilde{T}_{i}-\gamma_{i,+}I\right\Vert\right\}}{|\gamma_{i,+}-\gamma_{i,-}|}
\max\left\{|\gamma_{i,+}|,|\gamma_{i,-}|\right\}^{k}.
\end{equation*}
We can compute that
\begin{equation*}
\tilde{T}_{i}-\gamma_{i,-}I
=\left[\begin{array}{cc}
\gamma_{i,+} & -\gamma_{i,+}\gamma_{i,-}
\\
1 & -\gamma_{i,-}
\end{array}
\right],
\end{equation*}
which implies that
\begin{equation*}
\left\Vert\tilde{T}_{i}-\gamma_{i,-}I\right\Vert
\leq\left\Vert\left(
\begin{array}{c}
\gamma_{i,+}
\\
1
\end{array}
\right)\right\Vert
\left\Vert\left(\begin{array}{cc}
1 & -\gamma_{i,-}\end{array}\right)\right\Vert
\leq
1+\max\{|\gamma_{i,+}|,|\gamma_{i,-}|\}^{2}.
\end{equation*}
Similarly, we have
\begin{equation*}
\left\Vert\tilde{T}_{i}-\gamma_{i,+}I\right\Vert
\leq\left\Vert\left(
\begin{array}{c}
\gamma_{i,-}
\\
1
\end{array}
\right)\right\Vert
\left\Vert\left(\begin{array}{cc}
1 & -\gamma_{i,+}\end{array}\right)\right\Vert
\leq
1+\max\{|\gamma_{i,+}|,|\gamma_{i,-}|\}^{2}.
\end{equation*}

(ii) If $\gamma_{i,+}=\gamma_{i,-}
=\frac{(1+\beta)\mu_{i}}{2}$, then
by the formula for $k$-th
power of $2\times 2$ matrix
with two identical eigenvalues
(see e.g. \cite{williams2by2}), 
we get
\begin{equation*}
\left(\tilde{T}_{i}\right)^{k}
=\left(\frac{(1+\beta)\mu_{i}}{2}\right)^{k-1}\left(k\tilde{T}_{i}-(k-1)\frac{(1+\beta)\mu_{i}}{2}I\right),
\end{equation*}
so that
\begin{equation*}
\left\Vert\left(\tilde{T}_{i}\right)^{k}
\right\Vert\leq\left(\frac{(1+\beta)|\mu_{i}|}{2}\right)^{k-1}\left(k\Vert\tilde{T}_{i}\Vert+(k-1)\frac{(1+\beta)|\mu_{i}|}{2}\right).
\end{equation*}
Also notice that $\mu_{i}=0$, 
$\gamma_{i,+}=\gamma_{i,-}=0$
and $\tilde{T}_{i}^{k}=0$ for every $k\geq 2$.

Hence, we get
\begin{equation*}
\max_{1\leq i\leq Nd}
\left\Vert\left(\tilde{T}_{i}\right)^{k}
\right\Vert\leq C_{k}\cdot\rho_{\text{dasg}}^{k},
\end{equation*}
where
\begin{equation*}
C_{k}:=
\max\left\{
k\max_{i:\gamma_{i,+}=\gamma_{i,-}, \mu_{i}\neq 0}
\frac{2\Vert\tilde{T}_{i}\Vert}{(1+\beta)|\mu_{i}|}
+k-1
,\max_{i:\gamma_{i,+}\neq\gamma_{i,-}}
\frac{1+\max\{|\gamma_{i,+}|,|\gamma_{i,-}|\}^{2}}{|\gamma_{i,+}-\gamma_{i,-}|}
\right\},
\end{equation*}
and
\begin{equation*}
\rho_{\text{dasg}}=
\max_{1\leq i\leq Nd}\max\{|\gamma_{i,+}|,|\gamma_{i,-}|\}.
\end{equation*}
Moreover, when $\mu_{i}\neq 0$,
\begin{equation*}
\left\Vert\tilde{T}_{i}\right\Vert=\max\left\{|\gamma_{i,-}|,|\gamma_{i,+}|\right\}
=\frac{1+\beta}{2}\mu_{i}.
\end{equation*}
Hence
\begin{equation*}
C_{k}:=
\max\left\{
2k-1,\max_{i:\gamma_{i,+}\neq\gamma_{i,-}}
\frac{1+\max\{|\gamma_{i,+}|,|\gamma_{i,-}|\}^{2}}{|\gamma_{i,+}-\gamma_{i,-}|}
\right\}.
\end{equation*}

Next, let us assume that
$\beta=\frac{1-\sqrt{\alpha\mu}}{1+\sqrt{\alpha\mu}}$,
$\lambda_{N}^{W}>0$ and $\alpha\in(0,\frac{\lambda_{N}^{W}}{L}]$.
Therefore, we get
\begin{equation*}
0\leq\lambda_{N}^{W}-\alpha L\leq\mu_{i}\leq 1-\alpha\mu.
\end{equation*}

When $0<\mu_{i}<1-\alpha\mu$, we claim that
\begin{equation*}
\Delta_{i}:=(1+\beta)^{2}\mu_{i}^{2}-4\beta\mu_{i}<0.
\end{equation*}
To see this, note that since $\mu_{i}>0$ it is equivalent to
\begin{equation*}
\mu_{i}<\frac{4\beta}{(1+\beta)^{2}}
=\frac{4\frac{1-\sqrt{\alpha\mu}}{1+\sqrt{\alpha\mu}}}{(\frac{2}{1+\sqrt{\alpha\mu}})^{2}}
=1-\alpha\mu.
\end{equation*}
Therefore, when $0<\mu_{i}<1-\alpha\mu$, 
we have $\Delta_{i}<0$ and both $\gamma_{i,+}$ and $\gamma_{i,-}$ 
are complex numbers.
In this case, 
\begin{equation*}
|\gamma_{i,-}|=|\gamma_{i,+}|=\sqrt{\beta\mu_{i}}, 
\end{equation*}
and
\begin{align*}
\max_{i:\gamma_{i,+}\neq\gamma_{i,-}}
\frac{1+\max\{|\gamma_{i,+}|,|\gamma_{i,-}|\}^{2}}{|\gamma_{i,+}-\gamma_{i,-}|}
&=\max_{i:0<\mu_{i}<1-\alpha\mu}
\frac{1+\beta\mu_{i}}{\sqrt{-(1+\beta)^{2}\mu_{i}^{2}+4\beta\mu_{i}}}
\\
&=\max_{i:0<\mu_{i}<1-\alpha\mu}
\frac{1+\sqrt{\alpha\mu}+(1-\sqrt{\alpha\mu})\mu_{i}}{2\sqrt{\mu_{i}(1-\alpha\mu-\mu_{i})}}.
\end{align*}
Moreover, when $0<\mu_{i}<1-\alpha\mu$
\begin{equation*}
|\gamma_{i,+}|=|\gamma_{i,-}|
=\sqrt{\beta\mu_{i}}
\leq\sqrt{\frac{1-\sqrt{\alpha\mu}}{1+\sqrt{\alpha\mu}}(1-\alpha\mu)}=1-\sqrt{\alpha\mu}.
\end{equation*}

Next, $\gamma_{i,-}=\gamma_{i,+}$
if and only if $\mu_{i}=0$ or $\mu_{i}=1-\alpha\mu$.
When $\mu_{i}=0$, $|\gamma_{i,-}|=|\gamma_{i,+}|=0$,
and when $\mu_{i}=1-\alpha\mu$,
\begin{equation*}
|\gamma_{i,-}|=|\gamma_{i,+}|
=\frac{(1+\beta)\mu_{i}}{2}
\leq\frac{1}{1+\sqrt{\alpha\mu}}(1-\alpha\mu)=1-\sqrt{\alpha\mu}.
\end{equation*}

Hence, we conclude that when $\beta=\frac{1-\sqrt{\alpha\mu}}{1+\sqrt{\alpha\mu}}$,
$\lambda_{N}^{W}>0$ and $\alpha\in(0,\frac{\lambda_{N}^{W}}{L}]$,
we have
$\rho_{\text{dasg}}=1-\sqrt{\alpha\mu}$,
and
\begin{equation*}
C_{k}=\max\left\{2k-1,\max_{i:0<\mu_{i}<1-\alpha\mu}
\frac{1+\sqrt{\alpha\mu}+(1-\sqrt{\alpha\mu})\mu_{i}}{2\sqrt{\mu_{i}(1-\alpha\mu-\mu_{i})}}\right\}.
\end{equation*}
The proof is complete.
\end{proof}

\subsection{Proofs of Technical Results in Appendix~\ref{sec:general:noise}}
\subsubsection{Proof of Lemma~\ref{E:upper:bound:DSG:unbounded}}

\begin{proof}
{\color{black}The proof follows from Lemma~6 and Lemma~7 in \cite{gurbuzbalaban2020decentralized}.
The difference is that in our case, by applying the proof of Theorem~\ref{thm:DSG:unbounded}
\begin{equation}
\mathbb{E}\left\Vert\nabla F\left(x^{(k)}\right)\right\Vert^{2}
\leq
L^{2}\mathbb{E}\left\Vert x^{(k)}-x^{\infty}\right\Vert^{2}
\leq D_{1}^{2},
\end{equation}
where
\begin{equation}
D_{1}^{2}:=L^{2}\mathbb{E}\left\Vert x^{(0)}-x^{\infty}\right\Vert^{2}
+L^{2}\frac{2\alpha}{\mu}\left(\sigma^{2}+\eta^{2}\frac{\alpha^{2}(C_{1})^{2}}{(1-\gamma)^{2}}\right)N,
\end{equation}
and moreover, by \eqref{E:upper:bound:unbounded}
and the proof of Theorem~\ref{thm:DSG:unbounded}, we get
\begin{align*}
\mathbb{E}\left[\left\Vert\xi^{(k+1)}\right\Vert^{2}\right]
&\leq
\left(\sigma^{2}+\eta^{2}\frac{\alpha^{2}(C_{1})^{2}}{(1-\gamma)^{2}}\right)N
+\eta^{2}\mathbb{E}\left\Vert x^{(k)}-x^{\infty}\right\Vert^{2}
\\
&\leq
\left(\sigma^{2}+\eta^{2}\frac{\alpha^{2}(C_{1})^{2}}{(1-\gamma)^{2}}\right)N
+\eta^{2}\left(\mathbb{E}\left\Vert x^{(0)}-x^{\infty}\right\Vert^{2}
+\frac{2\alpha}{\mu}\left(\sigma^{2}+\eta^{2}\frac{\alpha^{2}(C_{1})^{2}}{(1-\gamma)^{2}}\right)N\right)
\\
&=\left(\sigma^{2}+\eta^{2}\frac{\alpha^{2}(C_{1})^{2}}{(1-\gamma)^{2}}\right)\frac{\mu+2\alpha}{\mu}N
+\eta^{2}\mathbb{E}\left\Vert x^{(0)}-x^{\infty}\right\Vert^{2}.
\end{align*}
The rest of the proof is similar to Lemma~6 and Lemma~7 in \cite{gurbuzbalaban2020decentralized}
and is omitted here.}
\end{proof}

\subsubsection{Proof of Lemma~\ref{gradient:average:error:DASG:general:noise}}

\begin{proof}
{\color{black}The proof follows from Lemma~\ref{gradient:average:error:DASG} and Lemma~\ref{lem:E:D-ASG}.
The difference is that in our case, by applying Theorem~\ref{thm-rate-dasg-unbounded}
\begin{align}
\mathbb{E}\left\Vert\nabla F\left(y^{(k)}\right)\right\Vert^{2}
&\leq
4L^{2}(1+\beta)^{2}\mathbb{E}\left\Vert x^{(k)}-x^{\infty}\right\Vert^{2}
+4L^{2}\beta^{2}\mathbb{E}\left\Vert x^{(k-1)}-x^{\infty}\right\Vert^{2}
+2\Vert\nabla F(x^{\ast})\Vert^{2}
\nonumber
\\
&\leq\tilde{D}_{y}^{2},
\end{align}
where we recall that $\tilde{D}_{y}^{2}$ is defined in \eqref{D:y:eqn:general:noise} as follows:
\begin{equation}
\tilde{D}_{y}^{2}=4L^{2}\left((1+\beta)^{2}+\beta^{2}\right)\left(\frac{2V_{Q,\alpha}\left(\xi_{0}\right)}{\mu}
+ \frac{12\sqrt{\alpha}}{\mu\sqrt{\mu}}\left(\sigma^{2}+\eta^{2}\frac{\alpha^{2}(C_{1})^{2}}{(1-\gamma)^{2}}\right)N\right)
+2\Vert\nabla F(x^{\ast})\Vert^{2},
\end{equation}
and moreover, by applying \eqref{E:upper:bound:unbounded}
to the context of D-ASG
and Theorem~\ref{thm-rate-dasg-unbounded}, we get
\begin{align*}
&\mathbb{E}\left[\left\Vert\xi^{(k+1)}\right\Vert^{2}\right]
\\
&\leq
\left(\sigma^{2}+\eta^{2}\frac{\alpha^{2}(C_{1})^{2}}{(1-\gamma)^{2}}\right)N
+\eta^{2}\mathbb{E}\left\Vert y^{(k)}-x^{\infty}\right\Vert^{2}
\\
&\leq
\left(\sigma^{2}+\eta^{2}\frac{\alpha^{2}(C_{1})^{2}}{(1-\gamma)^{2}}\right)N
+2\eta^{2}(1+\beta)^{2}\mathbb{E}\left\Vert x^{(k)}-x^{\infty}\right\Vert^{2}
+2\eta^{2}\beta^{2}\mathbb{E}\left\Vert x^{(k-1)}-x^{\infty}\right\Vert^{2}
\\
&\leq
\left(\sigma^{2}+\eta^{2}\frac{\alpha^{2}(C_{1})^{2}}{(1-\gamma)^{2}}\right)N
\\
&\qquad\qquad
+2\eta^{2}\left((1+\beta)^{2}+\beta^{2}\right)\left(\frac{2V_{Q,\alpha}\left(\xi_{0}\right)}{\mu}
+ \frac{12\sqrt{\alpha}}{\mu\sqrt{\mu}}\left(\sigma^{2}+\eta^{2}\frac{\alpha^{2}(C_{1})^{2}}{(1-\gamma)^{2}}\right)N\right).
\end{align*}
The rest of the proof is similar to Lemma~\ref{gradient:average:error:DASG} and Lemma~\ref{lem:E:D-ASG} and is omitted here.}
\end{proof}
\subsubsection{Proof of Lemma~\ref{lem:C:epsilon-generalized}}

\begin{proof}
{\color{black}The functions $V_{\bar{Q},\alpha}$ and $V_{\bar{S},\alpha}$ have a similar structure and can be written as the sum of a quadratic term and a term involving the objective. Therefore, the proof of Lemma \ref{lem:C:epsilon} applies with minor modifications. 
We next provide the details for the sake of completeness. We start with observing that the function $V_{\bar{Q},\alpha}$ is $\bar{M}_\alpha$ smooth where $\bar{M}_\alpha = L + 2\|\bar{Q}_\alpha\|$ and  
\beqs
\left\|\bar{Q}_\alpha\right\| = \left\|\tilde{Q}_\alpha \right\| \leq \left\|\bar{S}_\alpha\right\|  +  2\alpha\eta^2 \left((1+\beta)^2 + \beta^2 \right) = \frac{1}{\alpha} + \frac{\mu}{2}-\frac{\sqrt{\mu}}{\sqrt{\alpha}} + 2\alpha\eta^2 \left((1+\beta)^2 + \beta^2 \right),
\eeqs
where we used \eqref{ineq-S-alpha-norm}. For any $\xi,\Delta\in\mathbb{R}^{2d}$, by the $\bar{M}_\alpha$-smoothness of $V_{\bar{Q},\alpha}$ we have also
\begin{align} 
V_{\bar{Q},\alpha}(\xi+\Delta) &\leq
V_{\bar{Q},\alpha}(\xi) + 
\left\langle \nabla V_{\bar{Q},\alpha}(\xi), \Delta \right\rangle + \frac{\bar{M}_\alpha}{2} \|\Delta\|^2 \nonumber\\
&\leq V_{\bar{Q},\alpha}(\xi) + m_1 \|\nabla V_{\bar{S},\alpha}(\xi)\|^2 + \|\Delta\|^2/(4m_1)  + \frac{\bar{M}_\alpha}{2} \|\Delta\|^2, \label{ineq-helper-generalized}
\end{align}
for any $m_1>0$ where we used Cauchy-Schwarz inequality 
in \eqref{ineq-helper-generalized}. Note that the matrix $\tilde{Q}_\alpha$ has some special structure as a sum of two rank-one matrices, i.e. we can write 
$$ \tilde{Q}_\alpha = \sigma_1 a a^T + \sigma_2 b b^T,$$
with 
\begin{align*}
&\sigma_1 :=  \frac{1}{\alpha} + \frac{\mu}{2}-\frac{\sqrt{\mu}}{\sqrt{\alpha}}, \qquad  a :=\frac{1}{\sqrt{\sigma_1}} \left[
\begin{array}{c}
\sqrt{\frac{1}{2\alpha}}
\\
\sqrt{\frac{\mu}{2}}-\sqrt{\frac{1}{2\alpha}} 
\end{array}
\right], \\
&\sigma_2 := 2\alpha\eta^2 \left( (1+\beta)^2 + \beta^2 \right),\qquad b := \frac{1}{\sqrt{(1+\beta)^2 + \beta^2}}\begin{bmatrix} 1+ \beta \\ -\beta \end{bmatrix},
\end{align*}
where $a$ and $b$ are both unit vectors with norm $\|a\|=\|b\|=1$. Using the definition of $\beta$, one can see that $a$ cannot be equal to $b$ up to a multiplicative constant. Therefore, $\bar{Q}_\alpha$ cannot be of rank one; and has to be of rank two and hence is positive definite. In other words, the smallest eigenvalue $m_2$ of $\tilde{Q}_\alpha$ is positive. In this case, we have
$$ \frac{1}{m_2}\tilde{Q}_\alpha \succeq I. $$
We have also
\begin{align} 
\left\|\nabla V_{\bar{Q},\alpha}(\xi)  \right\|^2 &\leq 2\left\|2\bar{Q}_\alpha \xi\right\|^2 + 2\left\|\nabla f\left(\bar{T}\xi + x_*\right)\right\|^2\nonumber\\
&\leq
8 \xi^T \bar{Q}_\alpha^2 \xi+2 L^2\left\|\bar{T}\xi\right\|^2 \nonumber\\
&\leq 8 \xi^T \bar{Q}_\alpha^2 \xi + 2L^2 \frac{f(\bar{T}\xi+x_*) - f(x_*)}{\mu} \\
&\leq m_3 V_{\bar{Q},\alpha}(\xi),
\end{align}
where $$ m_3 := 2\max\left(4/m_2, \frac{L^2}{\mu}\right).$$
Therefore, if choose $m_1 = \epsilon/m_3$; then we get from \eqref{ineq-helper-generalized}
\begin{align} 
V_{\bar{Q},\alpha}(\xi+\Delta) 
\leq (1+\epsilon) V_{\bar{S},\alpha}(\xi)  + M_\epsilon \|\Delta\|^2, 
\end{align}
with $M_\epsilon = \frac{m_3}{4\epsilon} + \bar{M}_{\alpha}/2$. Finally, if we choose $\xi = \bar{\xi}_{k+1}-D_{k+1}$ and $\Delta= D_{k+1}$; we obtain \eqref{ineq-lyap-pert-generalized}. This completes the proof.}
\end{proof}

\end{document}